\documentclass{article}
\usepackage{NCSC}

\title{A martingale approach to noncommutative stochastic calculus}
\author[1]{David A.\ Jekel\thanks{Supported by NSF grant DMS-2002826}}
\author[2]{Todd A.\ Kemp\thanks{Supported by NSF grants DMS-2055340 and DMS-1800733}}
\author[3]{Evangelos A.\ Nikitopoulos\thanks{Supported by NSF grant DGE-2038238 and partially supported by NSF grant DMS-2055340}}
\affil[1]{Department of Mathematical Sciences, University of Copenhagen\protect\\
Universitetsparken 5, 2100 København Ø (Denmark)\protect\\
{Email: \tt \href{mailto:daj@math.ku.dk}{daj@math.ku.dk}}\vspace{2mm}}
\affil[2]{Department of Mathematics, University of California San Diego\protect\\
\noindent 9500 Gilman Drive, La Jolla, CA 92093-0112 (USA)\protect\\
{Email: \tt \href{mailto:tkemp@ucsd.edu}{tkemp@ucsd.edu}}\vspace{2mm}}
\affil[3]{Department of Mathematics, University of Michigan\protect\\
\noindent 530 Church Street, Ann Arbor, MI 48109-1043 (USA)\protect\\
{Email: {\tt \href{mailto:enikitop@umich.edu}{enikitop@umich.edu}}}}
\date{\vspace{-5ex}}

\begin{document}

\maketitle

\begin{abstract}
We present a new approach to noncommutative stochastic calculus that is, like the classical theory, based primarily on the martingale property.
Using this approach, we introduce a general theory of stochastic integration and quadratic (co)variation for a certain class of noncommutative processes, analogous to semimartingales, that includes both the $q$-Brownian motions and classical matrix-valued Brownian motions.
As applications, we obtain Burkholder--Davis--Gundy inequalities (with $p \geq 2$) for continuous-time noncommutative martingales and a noncommutative It\^{o}'s formula for ``adapted $C^2$ maps,'' including trace $\ast$-polynomial maps and operator functions associated to the noncommutative $C^2$ scalar functions $\mathbb{R} \to \mathbb{C}$ introduced by Nikitopoulos, as well as the more general multivariate tracial noncommutative $C^2$ functions introduced by Jekel, Li, and Shlyakhtenko.

$\,$

\noindent \textbf{Keyphrases:} free probability, noncommutative stochastic analysis, stochastic integral, quadratic covariation, It\^{o}'s formula, noncommutative smooth functions

$\,$

\end{abstract}

\tableofcontents

\section{Introduction}\label{sec.intro}

Stochastic calculus is a cornerstone of modern probability theory, supporting the foundations of fields of quantitative research from statistical physics to mathematical finance.
Invented first by Kiyoshi It\^{o} to develop a differential model of the evolution of a Markov process, it fully burst onto the scene with It\^{o}'s 1951 paper, \textit{On a formula concerning stochastic differentials} \cite{Ito1951}.
That paper's main result is what is now known as It\^{o}'s formula or sometimes, due to McKean's choice of nomenclature in \cite{McKean1969}, It\^{o}'s lemma;
since the label ``lemma'' grossly understates the importance of this result, we stick firmly to ``formula.''
We now state a special case.
Let $W = (W_t)_{t \geq 0}$ be a standard Brownian motion, and suppose $X=(X_t)_{t\ge 0}$ is (what is now known as) an \textbf{It\^{o} process}, i.e.,
\[
X_t = X_0 + \int_0^t H_s\,\d W_s + \int_0^t K_s\,\d s
\]
for some appropriately nice stochastic processes $H = (H_t)_{t \geq 0}$ and $K = (K_t)_{t \geq 0}$ adapted to the natural filtration of $W$. 
(The first integral above is a Brownian stochastic integral developed by It\^{o} in \cite{Ito1944}.)
If $f \colon \R \to \R$ is a $C^2$ function, then the process $f(X) = (f(X_t))_{t \geq 0}$ satisfies
\begin{equation}
f(X_t) = f(X_0) + \int_0^t f'(X_s)\,\d X_s + \frac{1}{2}\int_0^t f''(X_s) \, (\d X_s)^2. \label{eq.Ito.1951}
\end{equation}
Above, $\d X_s$ stands for $H_s \,\d W_s + K_s\,\d s$, and $(\d X_s)^2$ stands for $H_s^2\,\d s$.
In other words, \eqref{eq.Ito.1951} says
\[
f(X_t) = f(X_0) + \int_0^t f'(X_s)H_s\,\d W_s + \int_0^t \left(f'(X_s)K_s + \frac{1}{2} f''(X_s)H_s^2\right)\d s.
\]
Actually, in \cite{Ito1951}, It\^{o} stated and proved a form of \eqref{eq.Ito.1951} for multivariate functions of It\^{o} processes driven by multidimensional Brownian motion.
In Section \ref{sec.phil}, we discuss a generalization of this formula in detail;
see Theorem \ref{thm.clIF} and \eqref{eq.clIF} below.

Over the subsequent two decades, It\^{o}'s stochastic calculus was expanded from these important but limited beginnings to its modern form.
Beginning with the work of Doob and the follow-up work of Meyer, the field's perspective shifted away from Brownian integrators to processes sharing a key orthogonal-increments property with Brownian motion:
processes now known as martingales.
Through further key contributions of Courr\`{e}ge \cite{Courrege1962} and Motoo--Watanabe \cite{MotooWatanabe1965}, the framework for stochastic calculus was moved almost completely to the world of martingales.
(Here, ``almost'' refers to the fact that some simplifying technical assumptions lingered from the roots of stochastic calculus as a tool to study Markov processes.)
In their influential 1967 paper \cite{KW1967}, Kunita and Watanabe made another major advancement by clarifying the role of the\pagebreak\ quadratic variation in It\^o's original formula, enabling a vast generalization thereof;
see Theorem \ref{thm.clQC} below for the definition of quadratic (co)variation.
After further development and refinement by Meyer \cite{Meyer1967x4} and Dol\'{e}ans-Dade--Meyer \cite{DDM1970} of the concept of quadratic covariation, thereby removing the remaining extraneous technical assumptions, the theory reached its modern form by 1970.\footnote{This paragraph firmly centers the development of stochastic calculus in Japan and France, 1944--1970.
In fact, there were important parallel developments in the same mathematical arena in the Soviet Union during the same period, owing to Dynkin, Girsanov, Skorohod, Stratonovich, and others.
These developments would not reach across the Iron Curtain until much later, and a richer theory of stochastic analysis grew out of the combined understanding of both worlds.
For a more thorough summary of the historical development of stochastic calculus, the reader should consult the excellent introduction by Varadhan and Stroock to a curated selection of It\^o's works \cite{VaradhanStroock1987} or the concise paper \cite{JarrowProtter2004} of Jarrow and Protter on which our discussion is based.}
Throughout this whole development, it was clear that It\^{o}'s original formula was the linchpin of the theory;
it is the key computational tool in the theory, like the fundamental theorem of calculus it generalizes from the world of smooth ``processes.''
Indeed, had it not been for It\^{o}'s humility and the diminishing names given to his formula by his colleagues, it may well have garnered the appropriate name ``fundamental theorem of stochastic calculus.''

In the mirror universe of noncommutative probability, there have been many developments of versions of stochastic calculus.
Perhaps the earliest major steps in this direction are the 1984 papers of Hudson--Parthasarathy \cite{HP1984} and Applebaum--Hudson \cite{AH1984}.
These highly cited papers develop rudimentary theories of stochastic calculus for certain noncommutative It\^{o}-type processes driven by ``quantum Brownian motions,'' i.e., one-parameter families of quantum field operators on the bosonic and fermionic Fock spaces, respectively.
Motivated by the work of Hudson--Parthasarathy and Applebaum--Hudson, K\"{u}mmerer and Speicher published in 1992 the paper \cite{KS1992}, which developed a similar theory over the full (or Boltzmann) Fock space.
Such frameworks relied heavily on the Fock space structure and were limited in scope;
for example, It\^{o}'s formula in these contexts only applied to products of (or polynomials in) their It\^{o}-type processes.

The Brownian character of these families of field operators and the special algebraic (free probabilistic) properties of the ones acting on the full Fock space led Biane to define in  \cite{Biane1} an abstract notion of ``free Brownian motion'' (Example \ref{ex.freeinc}), an example of which can be constructed using field operators acting on the full Fock space.
Biane and Speicher then joined forces in \cite{BS1998,BS2001} to formulate and apply a theory of stochastic calculus for It\^{o} processes driven by free Brownian motion.
Their foundational work on free stochastic calculus supports more than one hundred important papers from the last quarter century.
Here are just a few landmark accomplishments:
the theory of free unitary Brownian motion introduced by Biane in \cite{Biane1}, which is a central ingredient in Voiculescu's non-microstates approach to free entropy via the liberation process from \cite{VoiculescuVI};
applications to estimates on microstates free entropy, e.g., \cite{BS2001,Shlyakhtenko2009};
applications to deformation/rigidity theory of group von Neumann algebras \cite{Dabrowski2010};
and free analogs of coercive functional inequalities, e.g., Talagrand inequalities \cite{HiaiUeda2006}.

All the papers mentioned above, as well as those not mentioned,\footnote{There are at least two areas of research we have not mentioned.
First is the area of $q$-stochastic analysis ($-1 \leq q \leq 1$), which interpolates between the quantum Brownian motions (with $q=-1,0,1$ corresponding respectively to the fermionic, free, and bosonic cases) and was pioneered in \cite{BS1991,BKS1997}. 
Key works on $q$-stochastic calculus include \cite{DonatiMartinS2003,DS2018} for $q$-Brownian motion and \cite{Anshelevich2002,Anshelevich2004} for (free and) $q$-L\'{e}vy processes. 
Second is the realm of non-tracial noncommutative stochastic analysis, which makes use of Haagerup-type $L^p$ spaces and is of a different flavor.
See, e.g., the recent works \cite{ABDVG2022,DVFG2025,DVFGG2025}.}
rely on special properties of integrators with noncommutative Brownian/Gaussian or L\'{e}vy-process character, and there is no general theory of quadratic (co)variation to unite them.
Consequently, It\^{o}'s formula is approached in a somewhat \textit{ad hoc} and limited way in each particular context.
One standard approach is as follows:
1) For the class of processes (e.g., free It\^{o} processes) and functions (e.g., polynomial or other functional calculi) of interest, work out heuristically---using special properties of the processes of interest---how an It\^{o}-type formula \textit{ought} to work (see \cite[\S1.1]{NikitopoulosIto} for an example);
2) prove a product rule--type special case of the formula from the previous step;
and 3) extend the product rule from the previous step to the desired class of functions through a mix of combinatorial methods and limiting arguments. 
Such an approach is tantamount to treating the term $\frac{1}{2}\int_0^t f''(X_s)\,(\d X_s)^2 = \frac{1}{2}\int_0^t f''(X_s)H_s^2\,\d s$ in \eqref{eq.Ito.1951} as a \textit{single} entity, the ``It\^{o} correction term,'' depending on the pair $(f,X)$ instead of the combination of \textit{two} distinct entities:
one depending on $f$ (its second derivative) and one depending on $X$ (its quadratic variation).
Specifically, as we explain in Section \ref{sec.phil}, the modern statement of It\^{o}'s formula in a vector-valued setting is
\[
F(X_t) = F(X_0) + \int_0^t DF(X_s)[\d X_s] + \frac{1}{2}\int_0^t D^2F(X_s)[\d X_s, \d X_s]
\]
whenever $X$ is a (continuous vector-valued) semimartingale and $F$ is a $C^2$ map;
above, $DF$ and $D^2F$ are, respectively, the first and second Fr\'{e}chet derivatives of $F$, and $[\d X_s, \d X_s]$ denotes integration against the quadratic variation of $X$.
This crucial perspective from classical stochastic analysis has not yet made it into the noncommutative probability literature.
The central goal of the present paper is to incorporate this insight by developing a general theory of noncommutative stochastic calculus that follows as closely as possible the classical martingale-theoretic development of the subject.
As we summarize in more detail in Section \ref{sec.mainres}, the fruits of these labors are the first general theory of noncommutative quadratic (co)variation, continuous-time noncommutative Burkholder--Davis--Gundy inequalities (Theorem \ref{thm.NCBDG}), and a noncommutative It\^{o}'s formula (Theorem \ref{thm.NCIF}) in which the correction term is a quadratic variation integral of the second derivative as in the classical case.
Moreover, we show that other known instances of noncommutative It\^{o}'s formula arise as special cases of ours via derivative and quadratic variation calculations, thereby demonstrating a ``universality'' of our formula.
A key point is that we consider general maps $F$ defined on subsets of the operator algebra rather than highly specific classes of functions like those induced via functional calculus by a scalar function of a real variable.

The remainder of this paper is organized as follows.
In Section \ref{sec.phil}, we briefly outline the modern form of the classical (continuous semi)martingale-theoretic approach to stochastic integration.
In Section \ref{sec.mainres}, we give precise statements of the main constructions and results of this paper. 
Section \ref{sec.bg1} summarizes the necessary preliminaries:
background on noncommutative probability theory and notations for trace polynomials and various classes of multilinear maps.
Section \ref{sec.ncprocesses} introduces notions of adaptedness (of noncommutative $L^p$ space--valued processes and various multilinear map--valued processes) and special processes (martingales, FV processes, and decomposable processes) that are key to our development.

In Section \ref{sec.stochint}, we develop a general theory of stochastic integration with respect to the $L^2$-decomposable processes introduced in Section \ref{sec.NCsemi}.
We do so by adapting the classical method, explained in Section \ref{sec.phil}, of proving an It\^{o} isometry using the Dol\'{e}ans measure of a square-integrable martingale.
In Section \ref{sec.QC}, the most technically challenging part, we develop our theory of noncommutative quadratic covariation.
In Section \ref{sec.BDG}, we obtain our noncommutative Burkholder--Davis--Gundy inequalities.
We also compute examples of quadratic covariations in Section \ref{sec.QCexs} that shed new light on other calculations in the literature.

In Section \ref{sec.NCIF}, we introduce a concept of adapted $C^2$ maps (more generally, adapted $C^{k,\ell}$ maps) defined on subsets of operator algebras and prove our noncommutative It\^{o}'s formula for such maps.
Using ideas and results from \cite{JLS2022,NikitopoulosNCk}, we conclude in Sections \ref{sec.trsmoothmaps} and \ref{sec.NCk} with many examples of applying the formula, demonstrating its computational flexibility and recovering other noncommutative It\^{o} formulas from the literature as special cases.

Finally, Appendix \ref{sec.CstarLp} fills a small gap in the literature on noncommutative $L^p$ spaces---see Section \ref{sec.freeprob} for more information---and Appendix \ref{sec.nota} is a notation index for the reader's convenience.

\subsection{Philosophy of the approach}\label{sec.phil}

Using classical stochastic calculus as our guide, we describe the philosophy of the present paper's approach to noncommutative stochastic calculus.
For the duration of this discussion, we assume the reader is familiar with the basics of continuous-time stochastic processes;
see \cite{CW1990,EK1986,RY1999} for relevant background.
Aside from Theorem \ref{thm.matBMtilde} and Remark \ref{rem.Poisson} below, we shall not use this material elsewhere in the paper.

Fix a filtered probability space $(\Om,\sF,(\sF_t)_{t \geq 0},P)$ satisfying the usual conditions.\footnote{right-continuity and completeness:
$\sF_t = \bigcap_{u > t}\sF_u$ and $\{G \subseteq \Om : G \subseteq G_0$ for some $G_0 \in \sF$ with $P(G_0) = 0\} \subseteq \sF_0$\label{p.usual}}
An \textbf{FV process} is an adapted process $A = (A_t)_{t \geq 0} \colon \R_+ \times \Om \to \R$ whose paths almost surely have locally bounded variation.
A (\textbf{continuous}) \textbf{semimartingale} is an adapted continuous process $X$ such that $X = X_0 + M + A$ for some continuous local martingale $M$ and some continuous FV process $A$ with $M_0 = A_0 = 0$ almost surely.
In this case, $M$ and $A$ are unique up to indistinguishability, and we call $M$ the \textbf{martingale part} of $X$ and $A$ the \textbf{FV part} of $X$.
It might not be clear from this definition why a semimartingale is a useful object.
However, in a certain sense that can be made precise---see, e.g., \cite{Bichteler2002,Protter2005}---semimartingales are precisely the continuous stochastic processes against which it is possible to define ``well-behaved'' stochastic integrals.
For the present discussion, knowing Theorem \ref{thm.clSI} below suffices.

\begin{nota}[Partitions]\label{nota.part}
Suppose $-\infty < a  < b \leq \infty$, and write $I \coloneqq [a,b] \cap \R$.
If $b < \infty$, then a partition of $I$ is a finite subset $\Pi = \{a = t_0 < \cdots < t_n = b\} \subseteq I$.
A partition of $[a,\infty)$ is a collection $\Pi = \{t_n : n \in \N_0\}$ such that $t_0 = a$, $t_n < t_{n+1}$ for all $n \in \N_0$, and $t_n \to \infty$ as $n \to \infty$.
In general, $\cP_I$ is the set of partitions of $I$.
Now, fix $\Pi \in \cP_I$.
If $t \in \Pi$, then $t_- \in \Pi$ is the member of $\Pi$ to the left of $t$;
precisely, $a_- \coloneqq a$, and $t_- \coloneqq \max\{s \in \Pi : s < t\}$ for $t \in \Pi \setminus \{a\}$.
Also, $\Delta t \coloneqq t-t_-$, $|\Pi| \coloneqq \sup\{\Delta s : s \in \Pi\}$ is the mesh of $\Pi$, and $\Delta_tF \coloneqq F(t) - F(t_-)$ for a function $F$ from $I$ to a vector space.
Limits as $|\Pi| \to 0$ will be denoted by $\lim_{\Pi \in \cP_I}$;
see Fact \ref{fact.partlim} below.
\end{nota}

\begin{thm}[Stochastic integral]\label{thm.clSI}
If $X$ is a semimartingale and $H$ is an adapted continuous process, then there exists a unique-up-to-indistinguishability semimartingale $\into H_s \,\d X_s$ such that for all $t \geq 0$,
\[
\int_0^t H_s\,\d X_s = L^0\text{-}\lim_{\Pi \in \cP_{[0,t]}} \sum_{s \in \Pi} H_{s_-}\,\Delta_s X.
\]
The limit above is a limit in probability as $|\Pi| \to 0$.
We call $\into H_s \,\d X_s$ the \textbf{stochastic integral} of $H$ with respect to $X$.
\end{thm}

\begin{rem}
The choice of the left-endpoint evaluation scheme above matters in the sense that other evaluation schemes can yield different answers.
The standard example is when $X = H = B$ is a Brownian motion, in which case
\[
L^2\text{-}\lim_{\Pi \in \cP_{[0,t]}} \sum_{s \in \Pi} B_{s_-} \,\Delta_s B = \frac{1}{2}(B_t^2 - t) \; \text{ and } \; L^2\text{-}\lim_{\Pi \in \cP_{[0,t]}} \sum_{s \in \Pi} B_s \,\Delta_s B = \frac{1}{2}(B_t^2 + t).
\]
The left-endpoint choice ensures the probabilistically desirable property that if $M$ is a continuous local martingale, then so is  $\into H_s \, \d M_s$.
\end{rem}

One standard proof of this result proceeds as follows:
1) Use pathwise Stieltjes integration theory on the FV part of $X$ to reduce to the case in which $X = M$ is a continuous local martingale,
2) use stopping time ``localization'' arguments to reduce to the case in which $M$ and $H$ are bounded, and
3) use the It\^{o} isometry (\cite[Thm.\ 2.3]{CW1990}) to treat the latter case.
The third step is important in spirit for us, so we say a few more words about it.
Write $\cP \subseteq 2^{\R_+ \times \Om}$ for the $\sigma$-algebra generated by
\[
\{\{0\} \times F : F \in \sF_0\} \cup \{(s,t] \times F : 0 \leq s < t, \, F \in \sF_s\}.
\]
We call $\cP$ the $\sigma$-algebra of \textbf{predictable sets}.
If $M$ is a continuous $L^2$-martingale, then there exists a unique measure $\mu_M$ on $(\R_+ \times \Om, \cP)$, called the \textbf{Dol\'{e}ans measure} of $M$, such that $\mu_M(\{0\} \times F_0) = 0$ for all $F_0 \in \sF_0$~and
\[
\mu_M((s,t] \times F_s) = \E_P\big[1_{F_s}(M_t - M_s)^2\big] \qquad (0 \leq s < t, \; F_s \in \sF_s),
\]
where $1_S$ is the indicator function of $S$.\label{page.indicator}
(See \cite[\S 2.4 \& \S2.8]{CW1990}.)
Now, if
\[
H_t(\om) = \sum_{i=1}^n 1_{\{0\}}(t)\,Y_i(\om)  + \sum_{j=1}^m 1_{(s_j,t_j]}(t)\, Z_j(\om) 
\]
for bounded $\sF_0$-measurable random variables $Y_i$ and bounded $\sF_{s_j}$-measurable random variables $Z_j$, then $H$ is called an \textbf{elementary predictable process}, and we define
\[
I_M(H)_t \coloneqq \sum_{j=1}^m Z_j(M_{t_j \wedge t} - M_{s_j \wedge t}).
\]
For such $H$, $I_M(H)$ is a continuous $L^2$-martingale, and the It\^{o} isometry says that
\[
\E_P[I_M(H)_t^2] = \int_{[0,t] \times \Om} |H|^2 \,\d\mu_M.
\]
This enables the extension of the definition of $\into H_s\,\d M_s \coloneqq I_M(H)$ to the set of predictable, i.e., $\cP$-measurable, processes $H \colon \R_+ \times \Om \to \R$ such that $\int_{[0,t] \times \Om} |H|^2 \,\d\mu_M < \infty$ for all $t \geq 0$.
(See \cite[\S2.5]{CW1990} for details.)
Finally, if $H$ is bounded, continuous, and adapted, then the It\^{o} isometry and the approximation $H_r^{\Pi} \coloneqq \sum_{s \in \Pi} 1_{(s_-,s]}(r)\,H_{s_-}$ can be used to show that
\[
I_M(H)_t = L^2\text{-}\lim_{\Pi \in \cP_{[0,t]}} \sum_{s \in \Pi} H_{s_-} \Delta_sM,
\]
as desired.
\pagebreak

Theorem \ref{thm.clSI} is used to construct one of the most important objects in stochastic analysis:
the quadratic covariation of a pair of semimartingales.

\begin{thm}[Quadratic covariation]\label{thm.clQC}
If $X$ and $Y$ are semimartingales and $t \geq 0$, then
\[
L^0\text{-}\lim_{\Pi \in \cP_{[0,t]}} \sum_{s \in \Pi}\Delta_s X\, \Delta_sY = X_tY_t - X_0Y_0 - \int_0^t X_s \,\d Y_s - \int_0^t Y_s\,\d X_s.
\]
The process
\[
[X,Y] \coloneqq XY - X_0Y_0 - \into X_s \,\d Y_s - \into Y_s\,\d X_s
\]
is called the \textbf{quadratic covariation} of $X$ and $Y$.
It is a continuous FV process, and $[X,Y] = [M,N]$, where $M$ is the martingale part of $X$ and $N$ is the martingale part of $Y$.
Also, we write $[X] \coloneqq [X,X]$.
\end{thm}

Since\hspace{-0.32mm} $[X,\hspace{-0.32mm}Y]$\hspace{-0.32mm} is\hspace{-0.32mm} a\hspace{-0.32mm} continuous\hspace{-0.32mm} FV\hspace{-0.32mm} process,\hspace{-0.32mm} one\hspace{-0.32mm} can\hspace{-0.32mm} Stieltjes\hspace{-0.32mm} integrate\hspace{-0.32mm} against\hspace{-0.32mm} it\hspace{-0.32mm} pathwise.\hspace{-0.32mm}
It\hspace{-0.32mm} is\hspace{-0.32mm} common\hspace{-0.32mm} to\hspace{-0.32mm} write
\[
\int_0^tH_s\,\d X_s\,\d Y_s \coloneqq \int_0^t H_s\,\d[X,Y]_s
\]
for such integrals.
The quadratic covariation appears in two places relevant to our development.
First, one can use it to rewrite the Dol\'{e}ans measure and therefore the It\^{o} isometry:
\cite[Thm.\ 4.2(iv)]{CW1990} says that if $M$ is a continuous $L^2$-martingale, then
\[
\mu_M(G) = \E_P\Bigg[\int_{\R_+} 1_G(t,\cdot) \,\d[M]_t\Bigg] \qquad (G \in \cP).
\]
It follows that the It\^{o} isometry may be rewritten as
\begin{equation}
    \E_P\Bigg[\Bigg(\int_0^t H_s\,\d M_s\Bigg)^2\Bigg] = \E_P\Bigg[\int_0^t |H_s|^2\,\d[M]_s\Bigg] \label{eq.clII}
\end{equation}
for elementary predictable $H$.
Second, quadratic covariations show up in It\^{o}'s formula.

\begin{thm}[It\^{o}'s formula]\label{thm.clIF}
If $F \colon \R^n \to \R^m$ is twice continuously differentiable and $X = (X_1,\ldots,X_n)$ is a vector of semimartingales, then $F(X) = (F_1(X),\ldots,F_m(X))$ is a vector of semimartingales satisfying
\[
F_i(X) = F_i(X_0) + \sum_{j=1}^n\into \partial_jF_i(X_t)\,\d X_{j,t} + \frac{1}{2}\sum_{j,k=1}^n\into \partial_k\partial_j F_i(X_t) \,\d[X_j,X_k]_t \qquad (i=1,\ldots,m).
\]
In ``stochastic differential notation,''
\[
\d F_i(X_t) = \sum_{j=1}^n\partial_jF_i(X_t)\,\d X_{j,t} + \frac{1}{2}\sum_{j,k=1}^n\partial_k\partial_j F_i(X_t) \,\d X_{j,t}\,\d X_{k,t} \qquad (i=1,\ldots,m).
\]
\end{thm}

Let us rewrite this formula in a ``vector-valued way'' that is more conducive to interpretation in infinite-dimensional contexts.
First, we can write the terms $\sum_{j=1}^n\int_0^t \partial_jF_i(X_s)\,\d X_{j,s}$ ($i=1,\ldots,m$) together as a vector-valued stochastic integral
\[
\int_0^t DF(X_s)[\d X_s] = L^0\text{-}\lim_{\Pi \in \cP_{[0,t]}} \sum_{s \in \Pi} DF(X_{s_-})[\Delta_sX].
\]
Above, $DF$ is the Fr\'{e}chet derivative of $F$, so $DF(X)$ is a stochastic process with values in the space of linear maps from $\R^n$ to $\R^m$.
The quadratic covariation integral terms $\sum_{j,k=1}^n\int_0^t\partial_k\partial_j F_i(X_s) \,\d X_{j,s}\,\d X_{k,s}$ ($i=1,\ldots,m$) are a bit trickier.
Indeed, define
\[
\llbracket X \rrbracket_t \coloneqq \sum_{j,k=1}^n [X_j,X_k]_t\,e_j \otimes e_k = L^0\text{-}\lim_{\Pi \in \cP_{[0,t]}}\sum_{s \in \Pi} \Delta_s X \otimes \Delta_s X,
\]
where $e_1,\ldots,e_n$ is the standard basis of $\R^n$.
Then $\llbracket X \rrbracket$ is a continuous FV process with values in $\R^n \otimes \R^n$.
In terms of vector-valued stochastic integrals,
\[
\llbracket X \rrbracket = X \otimes X - X_0 \otimes X_0 - \into \d X_t \otimes X_t - \into X_t \otimes \d X_t.
\]
Now, if $D^2F$ is the second Fr\'{e}chet derivative of $F$ (i.e., $D^2F(x)[h,k] = \partial_k\partial_hF(x)$), then $D^2F(X)$ is a stochastic process with values in the space of bilinear maps $\R^n \times \R^n \to \R^m$.
Since $\R^n$ is finite-dimensional, we can equivalently view $D^2F(X)$ as a stochastic process with values in the space of linear maps from $\R^n \otimes \R^n$ to $\R^m$.
This allows us to write
\[
\Bigg(\sum_{j,k=1}^n\int_0^t\partial_k\partial_j F_1(X_s) \,\d X_{j,s}\,\d X_{k,s}, \ldots, \sum_{j,k=1}^n\int_0^t\partial_k\partial_j F_m(X_s) \,\d X_{j,s}\,\d X_{k,s}\Bigg) = \int_0^t D^2F(X_s)[\d\llbracket X \rrbracket_s].
\]
This is a nice interpretation.
However, since infinite-dimensional tensor products can be ill-behaved, it is desirable to remove tensor products from the picture.
This is possible with a small amount of extra work.
Indeed, one can show that
\[
\int_0^t D^2F(X_s)[\d\llbracket X \rrbracket_s] = L^0\text{-}\lim_{\Pi \in \cP_{[0,t]}} \sum_{s \in \Pi} D^2F(X_{s_-})[\Delta_sX, \Delta_sX].
\]
We emphasize that on the right-hand side, $D^2F(X_{s_-})[\Delta_sX, \Delta_sX]$ is the application of the bilinear map $D^2F(X_{s_-})$ to the pair $(\Delta_sX, \Delta_sX)$.
The right-hand side motivates the notation $\int_0^t D^2F(X_s)[\d X_s, \d X_s]$ for the left-hand side.
This enables us to rewrite It\^{o}'s formula as
\begin{equation}
    \d F(X_t) = DF(X_t)[\d X_t] + \frac{1}{2}D^2F(X_t)[\d X_t, \d X_t] \label{eq.clIF}
\end{equation}
in vector-valued stochastic differential notation.

From the preceding discussion emerges the following list of tasks for someone interested in developing a general theory of noncommutative stochastic calculus.
\begin{enumerate}[leftmargin=2\parindent]
\itemsep0em
    \item Define a noncommutative analog of a semimartingale.
    \item For each ``noncommutative semimartingale'' $X$ and sufficiently many ``adapted, linear map--valued processes'' $H$, construct a stochastic integral $\into H(t)[\d X(t)]$. 
    \item For each pair $(X,Y)$ of ``noncommutative semimartingales'' and sufficiently many ``adapted, bilinear map--valued processes $\Lambda$,'' construct a quadratic covariation integral $\into \Lambda(t)[\d X(t), \d Y(t)]$, preferably in a way that \eqref{eq.clII} has a noncommutative analog.
    \item Define a space of $C^2$ maps $F$ appropriate for a noncommutative analog of \eqref{eq.clIF}.
\end{enumerate}
Unsurprisingly, this list essentially forms an outline of the paper, which we summarize in the next section.

\subsection{Summary of main results}\label{sec.mainres}

Here, we summarize our results on noncommutative stochastic calculus.
Fix three filtered $\mathrm{W}^*$-probability spaces $(\cA,(\cA_t)_{t \geq 0},\E = \E_{\mathsmaller{\cA}})$, $(\cB,(\cB_t)_{t \geq 0},\E_{\mathsmaller{\cB}})$, and $(\cC,(\cC_t)_{t \geq 0},\E_{\mathsmaller{\cC}})$ (Definition \ref{def.filtr}).
The results we shall quote also hold for filtered $\mathrm{C}^*$-probability spaces, but the definitions are slightly more technical in that setting.
We therefore opt to restrict ourselves to the $\mathrm{W}^*$ setting in this section.
Also, let us point out our non-standard use of $\E_{\mathsmaller{\cA}}$, $\E_{\mathsmaller{\cB}}$, and $\E_{\mathsmaller{\cC}}$ for the states on our operator algebras (rather than more common lowercase Greek letters like $\varphi$ and $\tau$).
We have made this choice to emphasize our conceptual adherence to the classical approach to stochastic calculus.

Let $p \in [1,\infty]$.
A process $X \colon \R_+ \to L^p(\cA,\E) = L^p(\E)$ is \textbf{adapted} if $X(t) \in L^p(\cA_t,\E)$ for all $t \geq 0$.
An \textbf{$\boldsymbol{L^p}$-FV} process is an adapted process $X \colon \R_+ \to L^p(\E)$ with locally bounded variation with respect to the noncommutative $L^p$ norm.
By analogy with the classical notion of a semimartingale, we consider processes that can be decomposed as the sum of a martingale and an FV process:
$X$ is a (\textbf{continuous}) \textbf{$\boldsymbol{L^p}$-decomposable process} if $X = X(0) + M + A$ for some $L^p$-continuous martingale $M \colon \R_+ \to L^p(\E)$ (Definition \ref{def.processes}\ref{item.mart}) and some $L^p$-continuous $L^p$-FV process $A \colon \R_+ \to L^p(\E)$ with $M(0) = A(0) = 0$.
In Section \ref{sec.NCsemi}, we give many examples of $L^p$-decomposable processes and show that if $p \geq 2$, then the decomposition $X = X(0) + M + A$ is unique (Corollary \ref{cor.uniquedecomp}).
As in the classical case, we call $M$ the \textbf{martingale part} of $X$ and $A$ the \textbf{FV part} of $X$.

Now,\hspace{-0.1mm} we\hspace{-0.1mm} construct\hspace{-0.1mm} stochastic\hspace{-0.1mm} integrals\hspace{-0.1mm} of\hspace{-0.1mm} ``adapted,\hspace{-0.1mm} linear\hspace{-0.1mm} map--valued\hspace{-0.1mm} processes''\hspace{-0.1mm} against\hspace{-0.1mm} $L^2$-decomposable processes.
The key to doing so, as will also be the case for other parts of the development, is to find the right notion of adaptedness.
Fix $p,q \in [1,\infty]$ and a map $H \colon \R_+ \to B(L^p(\E_{\mathsmaller{\cA}});L^q(\E_{\mathsmaller{\cB}}))$, where $B(\cV;\cW)$ is the space of bounded real-linear maps $\cV \to \cW$ (Notation \ref{nota.nota}\ref{item.Bk}).
We say $H$ is \textbf{adapted} if
\[
u \geq t \geq 0, \; x \in L^p(\E_{\mathsmaller{\cA}}) \implies \E_{\mathsmaller{\cB}}[H(t)x \mid \cB_u] = H(t)\E_{\mathsmaller{\cA}}[x \mid \cA_u].
\]
Note that $H(t)x = H(t)[x]$ is the application of the linear map $H(t) \colon L^p(\E_{\mathsmaller{\cA}}) \to L^q(\E_{\mathsmaller{\cB}})$ to the vector $x \in L^p(\E_{\mathsmaller{\cA}})$.
We explore this concept of adaptedness---and more generally, a concept of adaptedness of multilinear map--valued processes---in more depth in Section \ref{sec.filtradap}.
For now, here is a motivating example.

\begin{ex}\label{ex.introadap1}
Take $(\cA,(\cA_t)_{t \geq 0},\E) = (\cB,(\cB_t)_{t \geq 0},\E_{\mathsmaller{\cB}})$ and $p=q$.
If $A,B \colon \R_+ \to \cA = L^{\infty}(\E)$ are adapted, then the processes $H,K \colon \R_+ \to B(L^p(\E))$ defined by
\[
H(t)x \coloneqq A(t)xB(t) \; \text{ and } \; K(t)x \coloneqq \E_{\mathsmaller{\cA}}[A(t)x]\,B(t)
\]
are adapted (Proposition \ref{prop.TkinFk}).
We encourage the reader to think through why this is.
\end{ex}

We are now prepared to state our first main result:
a noncommutative analog of Theorem \ref{thm.clSI}.

\begin{thm}[Noncommutative stochastic integral]\label{thm.NCSI}
Suppose $X \colon \R_+ \to L^2(\E_{\mathsmaller{\cA}})$ is an $L^2$-decomposable process.
If $H \colon \R_+ \to B(L^2(\E_{\mathsmaller{\cA}});L^2(\E_{\mathsmaller{\cB}}))$ is adapted and continuous and $t \geq 0$, then
\[
\int_0^t H(s)[\d X(s)] \coloneqq L^2\text{-}\lim_{\Pi \in \cP_{[0,t]}}\sum_{s \in \Pi}H(s_-)[\Delta_sX] \in L^2(\cB_t,\E_{\mathsmaller{\cB}})
\]
exists.
The limit above is a limit in the space $L^2(\cB_t,\E_{\mathsmaller{\cB}})$ as $|\Pi| \to 0$ (Notation \ref{nota.part}).
Moreover, the process $\into H(s)[\d X(s)] \colon \R_+ \to L^2(\E_{\mathsmaller{\cB}})$ is $L^2$-decomposable.
\end{thm}

This result is a special case of a combination of Theorem \ref{thm.stochint} and Proposition \ref{prop.LERSapprox}.
The relevant development is inspired by parts of the classical case described in the previous section.
For the FV part of $X$, we use vector-valued Stieltjes integration theory (Section \ref{sec.vint}).
For the martingale part $M$ of $X$, we use a noncommutative analog of the time marginal of the Dol\'{e}ans measure to bound the $L^2$ norm of the integral of ``elementary predictable processes'' against $M$.
This allows us to extend the ``elementary integral'' against $M$.
Finally, we approximate an adapted continuous process by elementary predictable processes to complete the proof.
See Sections \ref{sec.L2NCstochint} and \ref{sec.RS1} for the full development, including additional properties like the ``substitution formula'': $\into H(t)[\d U(t)] = \into H(t)K(t)[\d X(t)]$ when $U = \into K(t)[\d X(t)]$ (Theorem \ref{thm.subform}).

By combining Example \ref{ex.introadap1} and Theorem \ref{thm.NCSI}, we can make sense of the stochastic integrals
\[
\into A(t)\,\d X(t)\,B(t) \; \text{ and } \; \into \E[A(t)\,\d X(t)] \,B(t)
\]
whenever $A,B \colon \R_+ \to \cA$ are adapted and continuous.
These are special cases of integrals of trace biprocesses (Definition \ref{def.trkpr}), which we introduce at the end of Section \ref{sec.filtradap}.
The term ``trace biprocess'' is inspired by previous work on noncommutative stochastic calculus---specifically, that of Biane--Speicher \cite{BS1998}---in which integrals like $\into A(t) \,\d X(t)\,B(t)$ are treated by defining integrals of tensor-valued processes called ``biprocesses.'' 
Notably, however, integrals like $\into \E[A(t)\,\d X(t)] \, B(t)$ are not considered in previous work.

Per the list at the end of the previous section, our next goal is to define quadratic covariation integrals of bilinear map--valued processes.
As was the case with the stochastic integrals we just discussed, the key is the right notion of adaptedness.
Fix $p,q,r \in [1,\infty]$ and a map $\Lambda \colon \R_+ \to B_2(L^p(\E_{\mathsmaller{\cA}}) \times L^q(\E_{\mathsmaller{\cB}});L^r(\E_{\mathsmaller{\cC}}))$, where $B_2(\cU \times \cV ;\cW)$ is the space of bounded real-bilinear maps $\cU \times \cV \to \cW$.
We say $\Lambda$ is \textbf{adapted} if
\[
u \geq t \geq 0, \; (x,y) \in L^p(\E_{\mathsmaller{\cA}}) \times L^q(\E_{\mathsmaller{\cB}}) \implies \begin{cases}
x \in L^p(\cA_u,\E_{\mathsmaller{\cA}}) \Rightarrow \E_{\mathsmaller{\cC}}[\Lambda(t)[x,y] \mid \cC_u] =\Lambda(t)[x,\E_{\mathsmaller{\cB}}[y \mid \cB_u]] \\
y \in L^q(\cB_u,\E_{\mathsmaller{\cB}}) \Rightarrow \E_{\mathsmaller{\cC}}[\Lambda(t)[x,y] \mid \cC_u] = \Lambda(t)[\E_{\mathsmaller{\cA}}[x \mid \cA_u],y].
\end{cases}
\]
Here is a motivating example similar to Example \ref{ex.introadap1}.

\begin{ex}\label{ex.introadap2}
Assume that $(\cA,(\cA_t)_{t \geq 0},\E) = (\cB,(\cB_t)_{t \geq 0},\E_{\mathsmaller{\cB}}) = (\cC,(\cC_t)_{t \geq 0},\E_{\mathsmaller{\cC}})$ and $1/p + 1/q = 1/r$.
If $A,B,C \colon \R_+ \to \cA$ are adapted, then the processes $\Lambda,\Xi,\Sigma,\Om, \colon \R_+ \to B_2(L^p(\E) \times L^q(\E);L^r(\E))$ defined by
\begin{align*}
    \Lambda(t)[x,y] & \coloneqq A(t)xB(t)yC(t), \;\; \Xi(t)[x,y] \coloneqq \E[A(t)xB(t)y]\,C(t), \\
    \Sigma(t)[x,y] & \coloneqq \E[A(t)x]\,B(t)yC(t), \; \text{ and } \;\Om(t)[x,y] \coloneqq \E[A(t)x]\,\E[B(t)y]\,C(t)
\end{align*}
are adapted (Proposition \ref{prop.TkinFk}).
Once again, we encourage the reader to think through why this is.
The processes $\Lambda$, $\Xi$, $\Sigma$, and $\Om$ are special cases of trace triprocesses (Definition \ref{def.trkpr}).
\end{ex}

The second main result is a construction of noncommutative quadratic covariation.

\begin{thm}[Noncommutative quadratic covariation]\label{thm.NCQC}
Suppose $X \colon \R_+ \to L^2(\E_{\mathsmaller{\cA}})$ and $Y \colon \R_+ \to L^2(\E_{\mathsmaller{\cB}})$ are $L^2$-decomposable processes and $\Lambda \colon \R_+ \to B_2(L^2(\E_{\mathsmaller{\cA}}) \times L^2(\E_{\mathsmaller{\cB}}) ; L^1(\E_{\mathsmaller{\cC}}))$ is adapted and continuous.
Assume, in addition, that
\begin{enumerate}[label=(\roman*),font=\normalfont]
    \item if $t \geq 0$, then $\Lambda(t)[\cA,\cB] \subseteq \cC$, and
    \[
    \sup\bigg\{\norm{\Lambda(t)[x,y]}_r : 1 \leq p,q,r, \leq \infty, \; \frac1p + \frac1q = \frac1r, \; x \in \cA, \; \norm{x}_p \leq 1, \; y \in \cB, \; \norm{y}_q \leq 1\bigg\} < \infty,
    \]
    where $\norm{\cdot}_p$ is the noncommutative $L^p$ norm (Notation \ref{nota.Lp}); and\label{item.bbBcond}
    \item the martingale parts of $X$ and $Y$ are locally uniformly $L^2$-approximable by $L^{\infty}$-continuous martingales.
    (See Definitions \ref{def.processes}\ref{item.mart} and \ref{def.tildesemi} for a precise statement of this condition.)\label{item.Mtildecond}
\end{enumerate}
If $t \geq 0$, then
\[
\int_0^t \Lambda(s)[\d X(s), \d Y(s)] \coloneqq L^1\text{-}\lim_{\Pi \in \cP_{[0,t]}}\sum_{s \in \Pi}\Lambda(s_-)[\Delta_sX, \Delta_sY] \in L^1(\cC_t,\E_{\mathsmaller{\cC}})
\]
exists.
Moreover, the process $\into \Lambda(s)[\d X(s), \d Y(s)] \colon \R_+ \to L^1(\E_{\mathsmaller{\cC}})$ is $L^1$-continuous and $L^1$-FV, and
\[
\into \Lambda(s)[\d X(s), \d Y(s)] = \into \Lambda(s)[\d M(s), \d N(s)],
\]
where $M$ (resp., $N$) is the martingale part of $X$ (resp., $Y$).
\end{thm}

\begin{rem}
The condition in \ref{item.Mtildecond} may seem strange, but many interesting examples satisfy it, e.g., $q$-Brownian motions and classical $n \times n$ Hermitian matrix Brownian motions (Theorem \ref{thm.matBMtilde}).
The fact that matrix Brownian motions satisfy this condition has applications to random matrix theory that will be explored in future work.
\end{rem}

Theorem \ref{thm.NCQC} is a special case of a combination of Theorems \ref{thm.QC1} and \ref{thm.QC2}.
The most interesting part of the proof is the first step, in which we establish a ``noncommutative It\^o product rule'' (Theorem \ref{thm.NCIPR}) for $\Lambda[X,Y]$ when $\Lambda$ is sufficiently nice and $X$ and $Y$ are $L^{\infty}$-decomposable processes.
This product rule is a noncommutative analog of Theorem \ref{thm.clQC}, which corresponds to $\Lambda(t)[x,y] = xy$.
See Sections \ref{sec.NCIPR}--\ref{sec.RS2} for the full development of noncommutative quadratic covariation, including additional properties like
\[
\into \Lambda(t)[\d U(t), \d V(t)] = \into \Lambda(t)[H(t)[\d X(t)], K(t)[\d Y(t)]]
\]
when $U = \into H(t)[\d X(t)]$ and $V = \into K(t)[\d Y(t)]$ (Theorem \ref{thm.QCSI}).

By combining Example \ref{ex.introadap2} and Theorem \ref{thm.NCQC}, we can make sense of the quadratic covariation integrals
\begin{align*}
    & \into A(t)\,\d X(t)\,B(t)\,\d Y(t)\,C(t), \;\; \into \E[A(t)\,\d X(t)\,B(t)\,\d Y(t)]\,C(t), \\
    & \into \E[A(t)\,\d X(t)]\,B(t)\,\d Y(t)\,C(t), \text{ and }  \into \E[A(t)\,\d X(t)]\,\E[B(t)\,\d Y(t)]\,C(t)
\end{align*}
whenever $A,B,C \colon \R_+ \to \cA$ are adapted and continuous.
In Section \ref{sec.QCexs}, we explicitly calculate these (and much more general) quadratic covariation integrals for a class of $L^2$-decomposable processes $X,Y$ that includes $q$-Brownian motions and classical matrix Brownian motions.
The resulting formulas shed new light on related calculations done on a case-by-case basis (without a general theory) in the literature on $q$-stochastic calculus, e.g., \cite{BS1998,DonatiMartinS2003,DS2018,NikitopoulosIto}.
\pagebreak

We end this section by discussing two applications of our theory of quadratic covariation.
The first is a noncommutative (NC) Burkholder--Davis--Gundy (BDG) inequality (with $p \geq 2$) for noncommutative martingales in continuous time.

\begin{thm}[NC BDG inequalities]\label{thm.NCBDG}
There exist increasing families $(\alpha_p)_{p \geq 2}$ and $(\beta_p)_{p \geq 2}$ of strictly positive constants such that the following holds.
If $2 \leq p < \infty$ and $M \colon \R_+ \to L^p(\E)$ is a martingale that is locally uniformly $L^2$-approximable by $L^{\infty}$-continuous martingales, then $\int_0^t \d M^*(s)\,\d M(s)$ and $\int_0^t \d M(s)\,\d M^*(s)$ belong to $L^{p/2}(\E)$, and
\[
\alpha_p^{-1}\norm{M}_{\cH_t^p(\cA)} \leq \norm{M(t)}_p = \sup_{0 \leq s \leq t} \norm{M(s)}_p \leq \beta_p\norm{M}_{\cH_t^p(\cA)},
\]
where
\[
\norm{M}_{\cH_t^p(\cA)} \coloneqq \max\Bigg\{\Bigg\|M(0)^*M(0) + \int_0^t \d M^*(s)\,\d M(s)\Bigg\|_{\frac{p}{2}}^{\frac{1}{2}}, \Bigg\|M(0)M(0)^* + \int_0^t \d M(s)\,\d M^*(s)\Bigg\|_{\frac{p}{2}}^{\frac{1}{2}}\Bigg\}.
\]
Furthermore, $\norm{M(t)}_2^2 = \E\big[M^*(0)M(0) + \int_0^t \d M^*(s)\,\d M(s)\big] = \E\big[M(0)M^*(0) + \int_0^t \d M(s)\,\d M^*(s)\big]$.
\end{thm}

The families $(\alpha_p)_{p \geq 2}$ and $(\beta_p)_{p \geq 2}$ do not depend on $(\cA,(\cA_t)_{t \geq 0},\E)$.
See Theorem \ref{thm.BDG} and Example \ref{ex.M2tilde} for the proof of Theorem \ref{thm.NCBDG}, which makes use of the discrete-time NC BDG inequalities of Pisier--Xu \cite{PX1997}.
The latter have counterparts for $p \in (1,2)$, but our tools do not seem to allow us to prove continuous-time versions of them.

Our continuous-time NC BDG inequalities allow us to prove noncommutative $L^p$-norm estimates for stochastic integrals (Theorem \ref{thm.LpnormSI}), the $p=2$ case of which yields a noncommutative analog of the It\^{o} isometry in the form of \eqref{eq.clII} (Corollary \ref{cor.NCII}).
Such estimates are useful for the study of noncommutative stochastic differential equations, which we shall explore in future work.

As a second application of our theory, we formulate and prove a noncommutative analog of It\^{o}'s formula in the form of \eqref{eq.clIF}.
Per the list at the end of the previous section, it is necessary to identify an appropriate collection of $C^2$ maps $F \colon \cA \to \cB$ for which to prove the formula.
Historically, certain spaces of polynomials or operator functions (i.e., maps induced by scalar functional calculus) have been used.
Instead of considering only these special maps, we formulate an abstract notion of adaptedness for $C^k$ maps (Definition \ref{def.adapCkl}) that encompasses all such examples of interest.
Since the definition is slightly more involved than the definition of adaptedness of (bi)linear map--valued processes, we omit it from the present summary.
It suffices to know that the definition ensures that the stochastic and quadratic variation integrals in Theorem \ref{thm.NCIF} below make sense and to keep the following motivating example in mind.

\begin{ex}\label{ex.clCk}
Let $(\Om,\sF,(\sF_t)_{t \geq 0},P)$ be a classical filtered probability space, and consider the filtered $\mathrm{W}^*$-probability space
\[
(\cA,(\cA_t)_{t \geq 0},\E) = (L^{\infty}(\Om,\sF,P),(L^{\infty}(\Om,\sF_t,P))_{t \geq 0},\E_P).
\]
If $d,\ell \in \N$ and $f \colon \R^d \to \C^{\ell}$ is $k$-times continuously differentiable and $f_*(\mathbf{a}) \coloneqq f \circ \mathbf{a} \in \cA^{\ell}$ for $\mathbf{a} \in \cA_{\sa}^d$, then $f_* \colon \cA_{\sa}^d \to \cA^{\ell}$ is an adapted $C^k$ map (Proposition \ref{prop.clCk} with $n=m=1$).
\end{ex}

\begin{thm}[Noncommutative It\^{o}'s formula]\label{thm.NCIF}
Let $\cA_{\beta}$ be a fixed choice of $\cA$ or $\cA_{\sa}$.
If $\cU \subseteq \cA_{\beta}$ is an open set, $F \colon \cU \to \cB$ is an adapted $C^2$ map (Definition \ref{def.adapCkl}), and $X \colon \R_+ \to \cA$ is an $L^{\infty}$-decomposable process such that $X(t) \in \cU$ for all $t \geq 0$, then
\begin{align*}
    \d F(X(t)) & = DF(X(t))[\d X(t)] + \frac{1}{2}D^2F(X(t))[\d X(t), \d X(t)], \; \text{ i.e.,} \\
    F(X) & = F(X(0)) + \into DF(X(t))[\d X(t)] + \frac{1}{2}\into D^2F(X(t))[\d X(t), \d X(t)].
\end{align*}
\end{thm}

This result is the special case of Theorem \ref{thm.NCIFtimedep} explained at the end of Example \ref{ex.timedepIF}.
The main thing setting Theorem \ref{thm.NCIF} apart from existing noncommutative analogs of It\^{o}'s formula is that the objects in our formula are truly computed by doing (stochastic) \textit{calculus}.
Specifically, one must compute the first and second derivatives of $F$ and quadratic variation integrals $\into \Lambda(t)[\d X(t), \d X(t)]$ of $X$.
Other noncommutative It\^{o} formulas in the literature---e.g., \cite[Props.\ 4.3.2 \& 4.3.4]{BS1998}, \cite[Thm.\ 9]{Anshelevich2002}, the equation following \cite[Cor.\ 4.9]{DS2018}, and \cite[Thms.\ 3.4.4 \& 4.3.4]{NikitopoulosIto}---are stated in terms of combinatorial objects like Voiculescu's free difference\pagebreak\ quotients or analytic objects like multiple operator integrals.
Though these objects do appear in the relevant derivative formulas, the aforementioned It\^{o} formulas do not explicitly use this fact because they are proven by induction and/or polynomial approximation arguments, not via Taylor's theorem as in the classical case.
In contrast, we do use Taylor's theorem to prove our formula (Theorem \ref{thm.NCIF}).

Finally, in Sections \ref{sec.trsmoothmaps} and \ref{sec.NCk}, we employ ideas from \cite{JLS2022,NikitopoulosNCk} to show that many interesting maps are adapted $C^k$.
First, by taking $\cA$ and $\cB$ to be direct sums of filtered $\mathrm{W}^*$-probability spaces, one can prove a multivariate version of Theorem \ref{thm.NCIF} for adapted $C^2$ maps from open subsets $\cU \subseteq \cA_{\beta}^n$ to $\cA^m$ (Example \ref{ex.multivarIF}).
In Section \ref{sec.trsmoothmaps}, we show that a large class of multivariate functions $\cA_{\beta}^n \to \cA^m$ are adapted $C^k$, including trace $\ast$-polynomial maps and tracial noncommutative $C^k$ maps in the sense of Jekel--Li--Shlyakhtenko \cite{JLS2022}.
In Section \ref{sec.NCk}, we show that if $n=m=1$ and $\cA_{\beta} = \cA_{\sa}$, then the class in the previous sentence contains all the operator functions $\cA_{\sa} \ni a \mapsto f(a) \in \cA$ associated to scalar functions $f \colon \R \to \C$ that are noncommutative $C^k$ in the sense of Nikitopoulos \cite{NikitopoulosNCk}.
In particular, we recover the ``free It\^{o} formulas'' of Biane--Speicher \cite{BS1998} and Nikitopoulos \cite{NikitopoulosIto}.

\section{Preliminaries}\label{sec.bg1}

To begin, we set some notation for basic objects that we use freely throughout the paper.
A complete notation index is available in Appendix \ref{sec.nota}.

\begin{nota}\label{nota.nota}
Suppose $-\infty < a  < b \leq \infty$, and write $I \coloneqq [a,b] \cap \R$.
Also, let $\cV$ be a vector space over $\F \in \{\R,\C\}$ and $F \colon I \to \cV$ be a function.
When $\cV$ is assumed to be normed, $\norm{\cdot}_{\cV}$ is its norm.
\begin{enumerate}[label=(\roman*),font=\normalfont]
    \item Suppose $\cV$ is normed.
    Then
    \[
    V(F:[s,t]) \coloneqq \sup_{\Pi \in \cP_{[s,t]}} \sum_{r \in \Pi}\norm{\Delta_rF}_{\cV} \in [0,\infty]
    \]
    is the variation of $F$ on $[s,t] \subseteq I$.
    Also,
    \[
    V(F:[a,\infty)) = V(F:[a,\infty]) \coloneqq \sup_{c \in I} V(F : [a,c]) \in [0,\infty]
    \]
    when $b=\infty$.
    When we want to emphasize the space $\cV$, we shall write $V = V_{\cV}$.
    Recall that $F$ has bounded variation if $V(F : I) < \infty$ and locally bounded variation if $V(F : [a,c]) < \infty$ for all $c \in I$.\label{item.V1}
    \item Suppose $\cV_1,\ldots,\cV_k,\cV$ are normed $\F$-vector spaces.
    If $\Lambda \colon \cV_1 \times \cdots \times \cV_k \to \cV$ is a real--$k$-linear map, then
    \[
    \norm{\Lambda}_{B_k(\cV_1 \times \cdots \times \cV_k; \cV)} \coloneqq \sup\{\norm{\Lambda[v_1,\ldots,v_k]}_{\cV} : \norm{v_1}_{\cV_1} \leq 1, \ldots, \norm{v_k}_{\cV_k} \leq 1\} \in [0,\infty]
    \]
    is the operator norm of $\Lambda$, and $B_k(\cV_1 \times \cdots \times \cV_k; \cV)$ is the normed $\F$-vector space of real--$k$-linear maps $\Lambda$ with finite operator norm.
    Also, $B(\cV_1;\cV) \coloneqq B_1(\cV_1;\cV)$, $B(\cV) \coloneqq B(\cV;\cV)$, and $\norm{\cdot}_{\cV_1 \to \cV} \coloneqq \norm{\cdot}_{B(\cV_1;\cV)}$.\label{item.Bk}
    \item Suppose $\cV$ is a Hausdorff topological vector space.
    If $t \in I$, then
    \[
    F(t-) \coloneqq \lim_{s \nearrow t} F(s) \; \text{ and } \; F(t+) \coloneqq \lim_{u \searrow t} F(u)
    \]
    when such limits exist, with the convention that $F(a-) \coloneqq F(a)$ and $F(b+) \coloneqq F(b)$ when $b < \infty$.
    If $F$ has left/right limits on all of $I$, then $F_{\pm} \colon I \to \cV$ is the function defined by $t \mapsto F(t\pm)$.\label{item.l/rlim}
    \item If $(\Om,\sF,\mu)$ is a complex or signed measure space, then $|\mu|$ is the total variation measure of $\mu$.
    If $\cV$ is a Banach space, then $L^0(\Om,\mu;\cV)$ is the space of $|\mu|$-a.e.\ equivalence classes of strongly measurable maps $\Om \to \cV$.
    If $p \in [1,\infty]$, then $L^p(\Om,\mu;\cV)$ is the Banach space of $F \in L^0(\Om,\mu;\cV)$ such that
    \[
    \norm{F}_{L^p(|\mu|;\cV)} \coloneqq \Bigg(\int_{\Om} \norm{F(\om)}_{\cV}^p \, |\mu|(\d\om)\Bigg)^{\frac{1}{p}} < \infty,
    \]
    with the obvious adjustment when $p=\infty$.
    If $F \in L^1(\Om,\mu;\cV)$, then $\int_{\Om} F \,\d\mu = \int_{\Om} F(\om) \, \mu(\d\om) \in \cV$ is the Bochner $\mu$-integral of $F$;
    see \cite[App.\ E]{Cohn2013} for background on strong measurability and Bochner integrals.
    If $\Om$ is a Hausdorff topological space and $\sF = \cB_{\Om}$, then $L_{\loc}^p(\Om,\mu;\cV)$ is the space of $F \in L^0(\Om,\mu;\cV)$ such that $F|_K \in L^p(K,\mu|_K;\cV)$ for all compact $K \subseteq \Om$.\label{item.Bochner}
\end{enumerate}
\end{nota}

\subsection{Noncommutative probability}\label{sec.freeprob}

In this section, we discuss some basic definitions and facts about free probability and noncommutative $L^p$ spaces.
We assume the reader is familiar with these and recall only what is necessary for the present application.
See \cite{MS2017,NS2006} for a proper treatment of the basics of free probability.

A pair $(\cA,\E)$ is a \textbf{$\boldsymbol{\ast}$-probability space} if $\cA$ is a unital $\ast$-algebra and $\E \colon \cA \to \C$ is a \textbf{state}, i.e., $\E$ is $\C$-linear, unital ($\E[1]=1$), and positive ($\E[a^*a] \geq 0$ for all $a \in \cA$).
The state $\E$ is \textbf{tracial} if $\E[ab] = \E[ba]$ for all $a,b \in \cA$ and \textbf{faithful} if $\E[a^*a] = 0$ implies $a=0$.
A collection $(\cA_i)_{i \in I}$ of (not necessarily $\ast$-)subalgebras of $\cA$ is \textbf{freely independent}---\textbf{free} for short---if $\E[a_1\cdots a_n] = 0$ whenever $\E[a_1] = \cdots = \E[a_n] = 0$ and $a_1 \in \cA_{i_1},\ldots,a_n \in \cA_{i_n}$ with $i_1 \neq i_2,i_2 \neq i_3, \ldots, i_{n-2} \neq i_{n-1}, i_{n-1}\neq i_n$.
When applied to elements or subsets of $\cA$, the terms  ``($\ast$-)free'' or ``($\ast$-)freely independent'' refer to the ($\ast$-)subalgebras these elements or subsets generate, e.g., $a \in \cA$ and $S \subseteq \cA$ are ($\ast$-)free if the ($\ast$-)subalgebra generated by $a$ is free from the ($\ast$-)subalgebra generated by $S$.

Let $\cH$ be a complex Hilbert space and $B_{\C}(\cH) \coloneqq \{$bounded $\C$-linear maps $\cH \to \cH\}$.
A \textbf{$\boldsymbol{\mathrm{C}^*}$-algebra} is an operator norm--closed $\ast$-subalgebra of $B_{\C}(\cH)$.
A \textbf{von Neumann algebra} is a unital $\mathrm{C}^*$-algebra that is closed in the $\sigma$-weak operator topology ($\sigma$-WOT).
A $\ast$-probability space $(\cA,\E)$ is a \textbf{$\boldsymbol{\mathrm{C}^*}$-probability space} if $\cA$ is a unital $\mathrm{C}^*$-algebra and the state $\E$ is tracial and faithful.\label{page.Cstarprob}
(In this case, $\E$ is bounded and has operator norm $\E[1] = 1$.)
A $\mathrm{C}^*$-probability space $(\cA,\E)$ is a \textbf{$\boldsymbol{\mathrm{W}^*}$-probability space} if $\cA$ is a von Neumann algebra and $\E$ is normal ($\sigma$-WOT continuous).
All $\ast$-probability spaces considered in this paper will be $\mathrm{C}^*$-probability spaces;
sometimes, they will be $\mathrm{W}^*$-probability spaces.
See \cite{Conway1990,Conway2000,Dixmier1981} for background on operator algebras.

\begin{ex}[Random matrices]\label{ex.randommatrices}
Let $(\Om,\sF,P)$ be a (classical) probability space and $n \in \N$.
The algebra $\cA_n \coloneqq L^{\infty}(\Om,\sF,P;\MnC)$ of $P$-essentially bounded, $\MnC$-valued random variables (modulo $P$-a.e.\ equality) is a $\mathrm{W}^*$-probability space with the expected normalized matrix trace $\tau_n \coloneqq \E_P[\tr_n(\cdot)] \coloneqq \E_P[n^{-1}\Tr_n(\cdot)]$.
Here, $\cA_n$ is represented as multiplication operators on $\cH \coloneqq L^2(\Om,\sF,P;\C^n)$, i.e., $A \in \cA_n$ is viewed as the operator $\cH \ni v \mapsto Av \in \cH$.
Note that $(\cA_1,\tau_1) = (L^{\infty}(\Om,\sF,P),\E_P)$, and if $(\Om,\sF,P)$ is the one-point probability space, then $(\cA_1,\tau_1) = (\MnC,\tr_n)$.
\end{ex}

Fix a $\mathrm{C}^*$-probability space $(\cA,\E)$.
If $a \in \cA$ is normal ($a^*a = aa^*$), then the \textbf{$\boldsymbol{\ast}$-distribution} of $a$ is the Borel probability measure $\mu_a$ on the spectrum $\sigma(a) \subseteq \C$ satisfying
\[
\E[a^n(a^*)^m] = \int_{\sigma(a)} \lambda^n\bar{\lambda}^m\,\mu_a(\d\lambda) \qquad (n,m \in \N_0).
\]
When $(\cA,\E)$ is a $\mathrm{W}^*$-probability space, $\mu_a(\d\lambda) = \E[P^a(\d\lambda)]$, where $P^a \colon \cB_{\sigma(a)} \to \cA$ is the \textbf{projection-valued spectral measure} of $a$, i.e., the projection-valued measure that is characterized by the identity $a = \int_{\sigma(a)} \lambda \, P^a(\d\lambda)$ and is guaranteed to exist by the spectral theorem (\cite[Chap.\ IX]{Conway1990}).
If $(\cA,\E) = (\cA_n,\tau_n)$~as in Example \ref{ex.randommatrices} and $A \in \cA_n$ is normal, then $\mu_A$ is the $P$-expected empirical distribution of eigenvalues of~$A$.

Define $\mu^{\mathrm{sc}}_0 \coloneqq \delta_0$ and
\[
\mu^{\mathrm{sc}}_t(\d s) \coloneqq \frac{\sqrt{(4t-s^2)_+}}{2\pi t} \, \d s \qquad (t > 0)
\]
to be the semicircle distribution of variance $t$.
An element $a \in \cA_{\sa} \coloneqq \{b \in \cA : b^*=b\}$ is \textbf{semicircular with variance $\boldsymbol{t}$} if $\mu_a = \mu_t^{\mathrm{sc}}$.
An element $c \in \cA$ is \textbf{circular with variance $\boldsymbol{t}$} if $c = 2^{-1/2}(a_1+ia_2)$ for two free semicircular elements $a_1,a_2 \in \cA_{\sa}$ with variance $t$.
In this case, $c^*$ is also circular with variance $t$ because $-a_2$ is semicircular with variance $t$.

Now, we turn to noncommutative $L^p$ spaces. 
See \cite{daSilva2018} for a detailed development of the basic properties of noncommutative $L^p$ spaces in the von Neumann algebra setting.
Most of these basic properties still hold in the $\mathrm{C}^*$-algebra case---in fact, they may be deduced from the von Neumann algebra case---but we are unaware of a reference in which these facts are proven.
We fill this gap in Appendix \ref{sec.CstarLp}.

\begin{nota}[Noncommutative $L^p$ spaces]\label{nota.Lp}
If $p \in [1,\infty)$, then
\[
\norm{a}_p = \norm{a}_{L^p(\E)} \coloneqq \E[|a|^p]^{\frac{1}{p}} = \E\big[(a^*a)^{\frac{p}{2}}\big]^{\frac{1}{p}} \qquad (a \in \cA),
\]
and $L^p(\cA,\E) = L^p(\E)$ is the $\norm{\cdot}_p$-completion of $\cA$.
Also, write $L^{\infty}(\cA,\E) = L^{\infty}(\E) \coloneqq \cA$ and, for all $a \in \cA$, $\norm{a}_{\infty} \coloneqq \lim_{p \to \infty}\norm{a}_p = \norm{a}_{\cA} = \norm{a}$.
\end{nota}

\begin{rem}
The convention $L^{\infty}(\cA,\E) = \cA$ is conceptually inappropriate when $\cA$ is not a von Neumann algebra, as we observe in Remark \ref{rem.Linfty} below, but it makes the notation in the present paper work much~better.
\end{rem}

As in the classical case, the trace $\E \colon \cA \to \C$ extends uniquely to a bounded linear map $L^1(\E) \to \C$ with operator norm $1$;
we use the same notation for this extension.
Also, if $a \in \cA$ and $1 \leq p \leq q \leq \infty$, then $\norm{a}_p \leq \norm{a}_q$.
Moreover, the complex-linear contraction $L^q(\E) \to L^p(\E)$ extending the identity is injective, so we view $L^q(\E)$ as a subset of $L^p(\E)$.
We also have noncommutative H\"{o}lder's inequality: 
If $a_1,\ldots,a_n \in \cA$ and $p_1,\ldots,p_n,p \in [1,\infty]$ satisfy $1/p_1+\cdots+1/p_n \leq 1/p$, then $\norm{a_1 \cdots a_n}_p \leq \norm{a_1}_{p_1}\cdots \norm{a_n}_{p_n}$.
(The usual statement requires $1/p_1+\cdots+1/p_n = 1/p$.
However, using that $p \leq q$ implies $\norm{\cdot}_p \leq \norm{\cdot}_q$, one may generalize the usual statement to ours.)
This allows us to extend multiplication to a bounded complex--$n$-linear map $L^{p_1}(\E) \times \cdots \times L^{p_n}(\E) \to L^p(\E)$.
In addition, there is a dual characterization of the $L^p$ norm:
If $p,q \in [1,\infty]$ satisfy $1/p+1/q = 1$, then
\[
\norm{a}_p = \sup\{|\E[ab]| : b \in \cA, \;\norm{b}_q \leq 1\} \qquad (a \in \cA).
\]
This leads to the duality relationship $L^q(\E) \cong L^p(\E)^*$, via the map $a \mapsto (b \mapsto \E[ab])$, whenever $1/p+1/q=1$ and $p \not\in \{1,\infty\}$.
When $(\cA,\E)$ is a $\mathrm{W}^*$-probability space, the duality still holds with $p=1$ (but not generally with $p=\infty$).

Next, we discuss the notion of conditional expectation.

\begin{prop}[Conditional expectation]\label{prop.condexp}
Let $\cB \subseteq \cA$ be a \textbf{$\boldsymbol{\mathrm{C}^*}$-subalgebra}, i.e., a unital, operator norm--closed $\ast$-subalgebra.
\begin{enumerate}[label=(\roman*),font=\normalfont]
    \item If $a \in L^1(\cA,\E)$, then there exists a unique $b \in L^1(\cB,\E) \subseteq L^1(\cA,\E)$ such that
    \[
    \E[b_0b] = \E[b_0a] \qquad (b_0 \in \cB).
    \]
    The element $b$ is the \textbf{conditional expectation of $\boldsymbol{a}$ onto $\boldsymbol{\cB}$} and is written $\E[a \mid \cB]$.
    The map $L^1(\cA,\E) \ni a \mapsto \E[a \mid \cB] \in L^1(\cB,\E)$ is a linear projection that respects the $\ast$-operation.
    Moreover, it is a $\cB$-$\cB$ bimodule map:\label{item.condexpexist}
    \[
    \E[b_1ab_2 \mid \cB] = b_1\E[a \mid \cB]b_2 \qquad \big(b_1,b_2 \in \cB, \; a \in L^1(\cA,\E)\big).
    \]
    \item If $p \in [1,\infty)$ and $a \in L^p(\cA,\E)$, then $\E[a \mid \cB] \in L^p(\cB,\E)$ and $\norm{\E[a \mid \cB]}_p \leq \norm{a}_p$.
    If $(\cA,\E)$ is a $\mathrm{W}^*$-probability space and $\cB \subseteq \cA$ is a \textbf{$\boldsymbol{\mathrm{W}^*}$-subalgebra} (a $\sigma$-WOT--closed $\mathrm{C}^*$-subalgebra), then this is also true when $p=\infty$.\label{item.condexpLpbd}
    \item {\rm (Tower property)} If $\cC \subseteq \cB$ is another $\mathrm{C}^*$-subalgebra, then $\E[\E[\cdot \mid \cB] \mid \cC] = \E[\cdot \mid \cC]$.\label{item.condexpTower}
\end{enumerate}
\end{prop}

\begin{ex}[Scalars]\label{ex.scalarcond}
If $a \in L^1(\cA,\E)$, then $\E[a \mid \C] = \E[a]\,1 = \E[a]$.
Somewhat more generally, if $a \in \cA$ is free from $\cB$, then $\E[a \mid \cB] = \E[a]$.
\end{ex}

\begin{ex}[$p=2$]\label{ex.L2cond}
If $a \in L^2(\cA,\E)$, then $\E[a \mid \cB]$ is the orthogonal projection of $a$ onto the closed subspace $L^2(\cB,\E)$ of $L^2(\cA,\E)$.
\end{ex}

\begin{ex}[Random matrices]\label{ex.randommatricescond}
If $(\cA,\E) = (\cA_n,\tau_n)$ as in Example \ref{ex.randommatrices}, then
\[
(L^p(\tau_n),\norm{\cdot}_{L^p(\tau_n)}) = \Big(L^p(\Om,\sF,P;\MnC),\big\|\norm{\cdot}_{L^p(\tr_n)}\big\|_{L^p(\E_P)}\Big) \qquad (1 \leq p \leq \infty).
\]
Also, if $\sG \subseteq \sF$ is a sub--$\sigma$-algebra, then the algebra $\cB_n \coloneqq L^{\infty}(\Om,\sG,P;\MnC)$ of $\sG$-measurable elements of $\cA_n$ is a $\mathrm{W}^*$-subalgebra of $\cA_n$, and $\tau_n[A \mid \cB_n] = \E_P[A \mid \sG]$ for all $A \in L^1(\tau_n)$.
\end{ex}

Note that if $a \in \cB$ and $\E[a \mid \cB] \in \cB$, then
\[
\norm{\E[a \mid \cB]} = \lim_{p \to \infty}\norm{\E[a \mid \cB]}_p \leq \lim_{p \to \infty}\norm{a}_p = \norm{a}.
\]
If $\E[a \mid \cB] \in \cB$ for all $a \in \cA$, then $\cB$ is called \textbf{conditionable}.
Proposition \ref{prop.condexp}\ref{item.condexpLpbd} says that all $\mathrm{W}^*$-subalgebras of $\mathrm{W}^*$-probability spaces are conditionable.

Finally, we comment on the direct sum construction for $\mathrm{C}^*$-probability spaces, as it will be of use to us when considering ``multivariate'' instances of our results.
Suppose $(\cA_1,\E_1),\ldots,(\cA_n,\E_n)$ are $\mathrm{C}^*$-probability spaces, and write $\cA \coloneqq \cA_1 \oplus \cdots \oplus \cA_n$ for their $\mathrm{C}^*$-direct sum.
If
\[
\E[\a] \coloneqq \frac1n\sum_{i=1}^n \E_i[a_i] \qquad (\a = (a_1,\ldots,a_n) \in \cA),\pagebreak
\]
then $(\cA,\E)$ is a $\mathrm{C}^*$-probability space.
(If $(\cA_i,\E_i)$ is a $\mathrm{W}^*$-probability space for all $i=1,\ldots,n$, then $(\cA,\E)$ is a $\mathrm{W}^*$-probability space.)
If $p \in [1,\infty)$, then
$L^p(\E) = L^p(\E_1) \oplus \cdots \oplus L^p(\E_n)$, and
\[
\norm{\a}_p = \left(\frac1n\sum_{i=1}^n \norm{a_i}_p^p\right)^{\frac1p} \qquad \big(\a \in L^p(\E)\big).
\]
Of course,
\[
\norm{\a}_{\infty} = \max\{\norm{a_i}_{\infty} : i=1,\ldots,n\} \qquad (\a \in \cA)
\]
as well.
Moreover, if $\cB_i \subseteq \cA_i$ is a $\mathrm{C}^*$-subalgebra for all $i =1,\ldots,n$ and $\cB \coloneqq \cB_1 \oplus \cdots \oplus \cB_n$, then $\cB \subseteq \cA$ is a $\mathrm{C}^*$-subalgebra, and $\E[\a \mid \cB] = (\E_1[ a_1 \mid \cB_1],\ldots,\E_n[a_n \mid \cB_n])$ for all $\a \in L^1(\cA,\E)$.

\subsection{Multilinear maps on \texorpdfstring{$L^p$ spaces}{}}\label{sec.multilin}

We now set some notation for various classes of multilinear maps.
For the duration of this section, fix $k \in \N$ and, for each $i = 1,\ldots,k+1$, a $\mathrm{C}^*$-probability space $(\cA_i,\E_i)$.

\begin{conv}\label{conv.lin}
Henceforth, the terms ``linear,'' ``multilinear,'' and ``$k$-linear'' will refer to $\R$ as the base field unless otherwise specified, even if the vector spaces under consideration are complex.
For example, a ``linear map from $L^2(\E_1)$ to $L^2(\E_2)$'' is a \textit{real}-linear map from $L^2(\E_1)$ to $L^2(\E_2)$.
\end{conv}

Note that Convention \ref{conv.lin} is consistent with Notation \ref{nota.nota}\ref{item.Bk}, which the reader should now review.

\begin{nota}[Bounded multilinear maps on $L^p$ spaces]\label{nota.bddlin}
If $p_1,\ldots,p_k,p \in [1,\infty]$, then
\begin{align*}
    B_k^{p_1,\ldots,p_k;p} & \coloneqq B_k(L^{p_1}(\E_1) \times \cdots \times L^{p_k}(\E_k) ; L^p(\E_{k+1})) \, \text{ and} \\
    \norm{\cdot}_{p_1,\ldots,p_k;p} & \coloneqq \norm{\cdot}_{B_k(L^{p_1}(\E_1) \times \cdots \times L^{p_k}(\E_k) ; L^p(\E_{k+1}))}.
\end{align*}
Also, for $\Lambda \in B_k^{\infty,\ldots,\infty;\infty}$, we write
\begin{align*}
    \vertiii{\Lambda}_k & \coloneqq \sup\bigg\{\norm{\Lambda}_{p_1,\ldots,p_k;p} : \frac{1}{p_1}+\cdots+\frac{1}{p_k} = \frac{1}{p}\bigg\} \\
    & =  \sup\bigg\{\norm{\Lambda}_{p_1,\ldots,p_k;p} : \frac{1}{p_1}+\cdots+\frac{1}{p_k} \leq \frac{1}{p}\bigg\} \in [0,\infty] \, \text{ and} \\
    \mathbb{B}_k & = \mathbb{B}_k(\cA_1 \times \cdots \times \cA_k;\cA_{k+1}) \coloneqq \big\{\Xi \in B_k^{\infty,\ldots,\infty;\infty} : \vertiii{\Xi}_k < \infty\big\}.
\end{align*}
We shall omit the subscripts when $k=1$.
Finally, define $\mathbb{B}_0 \coloneqq \cA_{0+1} = \cA_1$.
\end{nota}

Note that if $\Lambda \in \mathbb{B}_k$ and $p_1,\ldots,p_k,p \in [1,\infty]$ are such that $1/p_1+\cdots+1/p_k \leq 1/p$, then $\Lambda$ extends uniquely to a bounded $k$-linear map $L^{p_1}(\E_1) \times \cdots \times L^{p_k}(\E_k) \to L^p(\E_{k+1})$ with operator norm at most $\vertiii{T}_k$.
We shall abuse notation and write $T$ for this extension as well.

We shall occasionally consider maps with only self-adjoint arguments.
Here are some useful facts about such maps;
we leave the proofs to the reader.

\begin{obs}\label{obs.realklin}
Suppose $\Lambda \colon \cA_{1,\sa} \times \cdots \times \cA_{k,\sa} \to \cA_{k+1}$ is a (real--)$k$-linear map.
\begin{enumerate}[label=(\roman*),font=\normalfont]
    \item There exists a unique complex--$k$-linear extension $\tilde{\Lambda} \colon \cA_1 \times \cdots \times \cA_k \to \cA_{k+1}$ of $\Lambda$.
    \item If $\tilde{\Lambda}$ is as in the previous item, then
    \[
    2^{-k}\big\|\tilde{\Lambda}\big\|_{p_1,\ldots,p_k;p} \leq \sup\{\norm{\Lambda[a_1,\ldots,a_n]}_p : a_i \in \cA_{i,\sa} \text{ and } \norm{a_i}_{p_i} \leq 1 \text{ for } i=1,\ldots,k\} \leq \big\|\tilde{\Lambda}\big\|_{p_1,\ldots,p_k;p}
    \]
    for all $p_1,\ldots,p_k,p \in [1,\infty]$.
\end{enumerate}
We shall often identify $\Lambda$ with its complex--$k$-linear extension $\tilde{\Lambda}$.
\end{obs}

Consequently, we may identify via restriction the set of complex--$k$-linear maps $\cA_1 \times \cdots \times \cA_k \to \cA_{k+1}$ with the set of (real--)$k$-linear maps $\cA_{1,\sa} \times \cdots \times \cA_{k,\sa} \to \cA_{k+1}$, and the norm $\vertiii{\cdot}_k$ on the former is equivalent to the analogous norm on the latter.
We therefore define the space $\mathbb{B}_k(\cA_{1,\sa} \times \cdots \times \cA_{k,\sa};\cA_{k+1})$ to be the complex-linear subspace of $\mathbb{B}_k$ consisting of complex--$k$-linear maps.
In the context of maps with self-adjoint arguments, we shall have occasion to consider the ``multivariate'' situation in which $\cA_1,\ldots,\cA_{k+1}$ are all direct sums of a single $\mathrm{C}^*$-probability space $(\cA,\E)$.
\pagebreak

\begin{nota}\label{nota.Bbkmult}
Fix $d_1,\ldots,d_{k+1} \in \N$, and write $d \coloneqq (d_1,\ldots,d_k)$ and $m \coloneqq d_{k+1}$.
For a $\mathrm{C}^*$-probability space $(\cA,\E)$, we write $\cA^d \coloneqq \cA^{d_1} \times \cdots \times \cA^{d_k}$.
Also, we write
\[
\mathbb{B}_k(\cA^d;\cA^m) = \mathbb{B}_k(\cA^{d_1} \times \cdots \times \cA^{d_k} ;\cA^m) \; \text{ and } \; \mathbb{B}_k(\cA_{\sa}^d;\cA^m) = \mathbb{B}_k(\cA_{\sa}^{d_1} \times \cdots \times \cA_{\sa}^{d_k} ;\cA^m)
\]
for the spaces $\mathbb{B}_k(\cA_1 \times \cdots \times \cA_k;\cA_{k+1})$ and $\mathbb{B}_k(\cA_{1,\sa} \times \cdots \times \cA_{k,\sa};\cA_{k+1})$ with $(\cA_i,\E_i) = (\cA^{\oplus d_i},\E^{\oplus d_i})$ for all $i = 1,\ldots,k+1$.
\end{nota}

\subsection{Trace \texorpdfstring{$\ast$}{}-polynomials}\label{sec.trpoly}

In this section, we set notation for trace $\ast$-polynomials with some linear arguments.
Before getting started, we informally explain the notion of a trace polynomial.
For a rigorous treatment and further history, see \cite[\S3.1]{JLS2022}.
(See also \cite{Cebron2013,DHK2013,Kemp2016,Kemp2017}, whence the term ``trace polynomial'' originates.)
Throughout this section, ``$\C$-algebra'' is short for ``unital associative $\C$-algebra,'' and all subalgebras are unital.
Let $\C\la x_1,\ldots,x_n \ra$ be the $\C$-algebra of noncommutative polynomials in the indeterminates $(x_1,\ldots,x_n)$.
The $\C$-algebra $\TrP(x_1,\ldots,x_n)$ of trace polynomials in $(x_1,\ldots,x_n)$ is a superalgebra of $\C\la x_1,\ldots,x_n \ra$ with a complex-linear ``abstract trace'' operation $\tr \colon \TrP(x_1,\ldots,x_n) \to Z(\TrP(x_1,\ldots,x_n))$ such that $\tr(1) = 1$, $\tr(PQ) = \tr(QP)$, and $\tr(\tr(P)Q) = \tr(P)\tr(Q)$ for all $P,Q \in \TrP(x_1,\ldots,x_n)$.
Here, $Z(A) \subseteq A$ is the center of the $\C$-algebra $A$.
Moreover,
\[
\TrP(x_1,\ldots,x_n) = \spn\{\tr(P_1)\cdots \tr(P_{\ell}) \,P_0 : \ell \in \N, P_0,P_1,\ldots,P_{\ell} \in \C\la x_1,\ldots,x_n\ra\}.
\]
For example, $P(x_1,x_2,x_3) = \tr(x_1x_2x_1)x_3x_2^2x_3 - \tr(x_1)\tr(x_2)x_3 \in \TrP(x_1,x_2,x_3)$, and sine $\tr$ is ``tracial,'' $P(x_1,x_2,x_3) = \tr(x_1^2x_2)x_3x_2^2x_3 - \tr(x_1)\tr(x_2)x_3$ as well.

\begin{nota}[Trace ($\ast$-)polynomials]\label{nota.TrPoly}
Fix $n \in \N$, and write $\x = (x_1,\ldots,x_n)$.
\begin{enumerate}[label=(\roman*),font=\normalfont]
    \item Write $\C\la \x \ra = \C\la x_1,\ldots,x_n\ra$ for the $\C$-algebra of noncommutative polynomials in $n$ indeterminates, and write $\TrP(\x) = \TrP(x_1,\ldots,x_n)$ for the $\C$-algebra of trace polynomials in $n$ indeterminates.
    \item Write $\C^*\la \x \ra = \C^*\la x_1,\ldots,x_n \ra$ for $\ast$-algebra of noncommutative $\ast$-polynomials in $n$ indeterminates, and write $\TrP^*(\x) = \TrP^*(x_1,\ldots,x_n)$ for the $\ast$-algebra of trace $\ast$-polynomials in $n$ indeterminates.
\end{enumerate}
To be clear, $\C^*\la \x \ra$ is just the space $\C\la x_1,y_1,\ldots,x_n,y_n \ra$ with the unique $\ast$-operation determined by $x_i^* = y_i$ ($i  = 1,\ldots,n$).
Similarly, $\TrP^*(\x)$ is just the space $\TrP(x_1,y_1,\ldots,x_n,y_n)$ with the unique $\ast$-operation that commutes with the abstract trace $\tr$ and agrees with the $\ast$-operation on $\C^*\la \x \ra \subseteq \TrP^*(\x)$.
\end{nota}

Now, fix $k \in \N$ as well.
We say that $P \in \TrP^*(x_1,\ldots,x_n,y_1,\ldots,y_k)$ is \textbf{real--$\boldsymbol{k}$-linear in $\boldsymbol{(y_1,\ldots,y_k)}$} if $P$ can be written as a complex-linear combination trace $\ast$-polynomials of the form $\tr(P_1)\cdots \tr(P_{\ell}) P_0$, where $P_0,\ldots,P_{\ell}$ are $\ast$-monomials in $(x_1,\ldots,x_n,y_1,\ldots,y_k)$ such that for each $j \in \{1,\ldots,k\}$, either
\begin{itemize}
    \item $y_j$ appears precisely once in $P_0\cdots P_{\ell}$, and $y_j^*$ does not appear in $P_0\cdots P_{\ell}$; or
    \item $y_j^*$ appears precisely once in $P_0\cdots P_{\ell}$, and $y_j$ does not appear in $P_0\cdots P_{\ell}$.
\end{itemize}
We say that $P \in \TrP^*(x_1,\ldots,x_n,y_1,\ldots,y_k)$ is \textbf{complex--$\boldsymbol{k}$-linear in $\boldsymbol{(y_1,\ldots,y_k)}$} if $P$ can be written as a complex-linear combination of trace $\ast$-polynomials of the form $\tr(P_1)\cdots \tr(P_{\ell}) P_0$, where $P_0,\ldots,P_{\ell}$ are $\ast$-monomials in $(x_1,\ldots,x_n,y_1,\ldots,y_k)$ such that for each $j \in \{1,\ldots,k\}$, $y_j$ appears precisely once in $P_0\cdots P_{\ell}$, and $y_j^*$ does not appear in $P_0\cdots P_{\ell}$.

\begin{ex}
The trace $\ast$-polynomial
\[
\tr(x_1y_1^*)x_3^*y_2x_3x_2 - \tr(y_2^*)x_2^7y_1x_3^5 + 4y_1x_3^*y_2 + iy_2^*y_1^*x_2 \in \TrP^*(x_1,x_2,x_3,y_1,y_2)
\]
is real-bilinear in $(y_1,y_2)$, and the trace $\ast$-polynomial
\[
\tr(x_1y_1)x_3^*y_2x_3x_2 - \tr(y_2)x_2^7y_1x_3^5 + 4y_1x_3^*y_2 + iy_2y_1x_2 \in \TrP^*(x_1,x_2,x_3,y_1,y_2)
\]
is complex-bilinear in $(y_1,y_2)$.
\end{ex}

Now, fix $d = (d_1,\ldots,d_k) \in \N^k$, and write $|d| = d_1+\cdots+d_k$ as usual.
Also, write $\x = (x_1,\ldots,x_n)$ and $\y_j = (y_{j,1},\ldots,y_{j,d_j})$ ($j = 1,\ldots,k$) so that
\[
\TrP^*(\x,\y_1,\ldots,\y_k) = \TrP^*(x_1,\ldots,x_n,y_{1,1},\ldots,y_{1,d_1},\ldots,y_{k,1},\ldots,y_{k,d_k})\pagebreak
\]
is the space of trace $\ast$-polynomials in $n+|d|$ indeterminates.
A trace $\ast$-polynomial $P \in \TrP^*(\x,\y_1,\ldots,\y_k)$ is said to be \textbf{real--$\boldsymbol{k}$-linear} (resp., \textbf{complex--$\boldsymbol{k}$-linear}) \textbf{in $\boldsymbol{(\y_1,\ldots,\y_k)}$} if $P$ can be written as a sum
\[
P(\x,\y_1,\ldots,\y_k) = \sum_{i_1=1}^{d_1}\cdots\sum_{i_k=1}^{d_k}P_{i_1,\ldots,i_k}(\x,y_{1,i_1},\ldots,y_{k,i_k}),
\]
where $P_{i_1,\ldots,i_k}(\x,y_{1,i_1},\ldots,y_{k,i_k}) \in \TrP^*(\x, y_{1,i_1},\ldots,y_{k,i_k})$ is real--$k$-linear (resp., complex--$k$-linear) in the indeterminates $(y_{1,i_1},\ldots,y_{k,i_k})$.

\begin{nota}\label{nota.lintrpoly}
For $n,k \in \N$ and $d = (d_1,\ldots,d_k) \in \N^k$, write
\[
\TrP_{n,k,d}^* = \TrP^*(\x)[\y_1,\ldots,\y_k] \subseteq \TrP^*(\x,\y_1,\ldots,\y_k)
\]
for the space of trace $\ast$-polynomials in $n+|d|$ indeterminates (as above) that are real--$k$-linear in $(\y_1,\ldots,\y_k)$.
For $P \in \TrP_{n,k,d}^*$, we shall often write
\[
P(\x)[\y_1,\ldots,\y_k] \coloneqq P(\x,\y_1,\ldots,\y_k).
\]
Also, to cover the $k=0$ case, write $\TrP_n^* = \TrP_{n,0,\emptyset}^* \coloneqq \TrP^*(\x)$.
Finally, write
\[
\TrP_{n,k,d}^{*,\C} = \TrP_{\C}^*(\x)[\y_1,\ldots,\y_k]
\]
for the set of $P \in \TrP_{n,k,d}^*$ that are complex--$k$-linear in $(\y_1,\ldots,\y_k)$.
\end{nota}

\begin{obs}\label{obs.trpolyprop}
If $n,k \in \N$ and $d \in \N^k$, then
\[
\TrP_{n,k,d}^* = \TrP^*(\x)[\y_1,\ldots,\y_k] \subseteq \TrP^*(\x,\y_1,\ldots,\y_k)
\]
is a complex-linear subspace that is closed under the $\ast$-operation.
Also,
\[
\TrP_{\C}^*(\x)[\y_1,\ldots,\y_k] \subseteq \TrP^*(\x)[\y_1,\ldots,\y_k]
\]
is a complex-linear subspace.
\end{obs}

Finally, we discuss evaluations.

\begin{nota}\label{nota.BCloc}
If $\cV$ and $\cW$ are normed vector spaces, then we write $BC_{\loc}(\cV;\cW)$ for the space of continuous maps $F \colon \cV \to \cW$ such that $\sup\{\norm{F(v)}_{\cW} : \norm{v}_{\cV} \leq R\} < \infty$ for all $R > 0$.
We endow this space with the topology of uniform convergence on bounded sets.
When $\cW$ is a Banach space, $BC_{\loc}(\cV;\cW)$ is a Fr\'{e}chet space (\cite[Prop.\ 4.1.4]{NikitopoulosNCk}).
\end{nota}

If $(\cA,\E)$ is a $\mathrm{C}^*$-probability space, then we may define the evaluation map 
\[
\ev_{\mathsmaller{(\cA,\E)}}^n \colon \TrP^*(x_1,\ldots,x_n) \to BC_{\loc}(\cA^n;\cA)
\]
as the unique unital $\ast$-homomorphism such that
\begin{align*}
    \big(\ev_{\mathsmaller{(\cA,\E)}}^nx_i\big)(\a) & = a_i \qquad (i =1,\ldots,n, \; \a = (a_1,\ldots,a_n) \in \cA^n), \\
    \ev_{\mathsmaller{(\cA,\E)}}^n\tr(P) & = \E \circ \ev_{\mathsmaller{(\cA,\E)}}^nP \qquad  (P \in \C^*\la x_1,\ldots,x_n \ra).
\end{align*}
(In the second line, we view $\C$ as the $\ast$-subalgebra $\C1$ of $\cA$.)
Now, fix $k \in \N_0$, $d = (d_1,\ldots,d_k) \in \N^k$, and $m \in \N$.
Observe that if $P = (P_1,\ldots,P_m) \in (\TrP_{n,k,d}^*)^m$ and $\a \in \cA^n$, then the map
\[
\cA^d \ni (\b_1,\ldots,\b_k) \mapsto \Big(\big(\ev_{\mathsmaller{(\cA,\E)}}^{n+|d|}P_1\big)(\a,\b_1,\ldots,\b_k),\ldots,\big(\ev_{\mathsmaller{(\cA,\E)}}^{n+|d|}P_m\big)(\a,\b_1,\ldots,\b_k)\Big) \in \cA^m
\]
belongs to $\mathbb{B}_k(\cA^d;\cA^m)$.
(To be clear, in the $k=0$ case, this should be interpreted as the statement that $((\ev_{\mathsmaller{(\cA,\E)}}^nP_1)(\a),\ldots,(\ev_{\mathsmaller{(\cA,\E)}}^nP_m)(\a)) \in \cA^m$.)
We write $P(\a) \in \mathbb{B}_k(\cA^d;\cA^m)$ for this map.
Moreover, the map
\[
\cA^n \ni \a \mapsto P(\a) \in \mathbb{B}_k(\cA^d;\cA^m)
\]
belongs to $BC_{\loc}(\cA^n;\mathbb{B}_k(\cA^d;\cA^m))$.
Note also that if $P \in (\TrP_{n,k,d}^{*,\C})^m$ and $\a \in \cA^n$, then $P(\a) \colon \cA^d \to \cA^m$ is complex--$k$-linear.

\begin{nota}[Evaluations]\label{nota.eval}
If $n,m \in \N$, $k \in \N_0$, and $d \in \N^k$, then we write
\[
\ev_{\mathsmaller{(\cA,\E)}}^{n,m,k,d} \colon (\TrP_{n,k,d}^*)^m \to BC_{\loc}(\cA^n;\mathbb{B}_k(\cA^d;\cA^m))
\]
for the map $P \mapsto (\a \mapsto P(\a))$ described in the previous paragraph.
Also, when $k=0$, we omit $k=0$ and $d = \emptyset$ from the notation.
Finally, for $P \in (\TrP_{n,k,d}^*)^m$, we write $P_{\mathsmaller{(\cA,\E)}} \coloneqq \ev_{\mathsmaller{(\cA,\E)}}^{n,m,k,d}P$.
\end{nota}

\section{Noncommutative processes}\label{sec.ncprocesses}

\subsection{Filtrations and adaptedness}\label{sec.filtradap}

In this section, we introduce notions of adaptedness that will be important for our noncommutative stochastic integral development.
We first recall the notion of a filtration of a $\mathrm{C}^*$-probability space.

\begin{defi}[Filtration]\label{def.filtr}
A \textbf{filtration} of a $\mathrm{C}^*$-probability space $(\cA,\E)$ is a family $(\cA_t)_{t \geq 0}$ of $\mathrm{C}^*$-subalgebras of $\cA$ such that $\cA_s \subseteq \cA_t$ whenever $0 \leq s \leq t$.
In this case, $(\cA,(\cA_t)_{t \geq 0},\E)$ is called a \textbf{filtered $\boldsymbol{\mathrm{C}^*}$-probability space}.
If $\cA_t \subseteq \cA$ is conditionable for all $t \geq 0$, then $(\cA,(\cA_t)_{t \geq 0},\E)$ is called \textbf{conditionable}.
If $(\cA,\E)$ is a $\mathrm{W}^*$-probability space and $\cA_t$ is a $\mathrm{W}^*$-subalgebra for all $t \geq 0$, then $(\cA,(\cA_t)_{t \geq 0},\E)$ is called a \textbf{filtered $\boldsymbol{\mathrm{W}^*}$-probability space}.
\end{defi}

In what follows, we shall work with arbitrary filtered $\mathrm{C}^*$-probability spaces.
Occasionally, something extra can be said when the filtered $\mathrm{C}^*$-probability spaces under consideration are all conditionable.
To avoid repeating cumbersome phrases in these situations, we institute the following shorthand.

\begin{conv}\label{conv.cond}
``In the  conditionable case'' is short for ``when all filtered $\mathrm{C}^*$-probability spaces in question are conditionable.''
Using this shorthand, we shall write
\[
[1,\infty\ra \coloneqq \begin{cases}
[1,\infty] & \text{in the conditionable case,} \\
[1,\infty) & \text{otherwise},
\end{cases}
\]
to exclude unwanted indices.
\end{conv}

Filtrations of $\mathrm{C}^*$-probability spaces induce ``filtrations'' of the spaces of bounded multilinear maps from Section \ref{sec.multilin}.
We now define and study these induced filtrations.

\begin{defi}[Induced filtrations, adaptedness]\label{def.adapt}
For the remainder of this section, fix $k \in \N$ and, for each $i = 1,\ldots, k+1$, a filtered $\mathrm{C}^*$-probability space $(\cA_i,(\cA_{i,t})_{t \geq 0},\E_i)$.
Also, let $t \geq 0$.
\begin{enumerate}[label=(\roman*),font=\normalfont]
    \item If $p_1,\ldots,p_k,p \in [1,\infty\ra$, then we define $\cF_{k,t}^{p_1,\ldots,p_k;p} = \cF_{k,t}^{p_1,\ldots,p_k;p}(\E_1, \ldots, \E_k;\E_{k+1})$ to be the set of all $\Lambda \in B_k^{p_1,\ldots,p_k;p}$ such that for all $u \geq t$ and $(a_1,\ldots,a_k) \in L^{p_1}(\cA_{1,u},\E_1) \times \cdots \times L^{p_k}(\cA_{k,u},\E_k)$,
    \[
    \E_{k+1}\big[\Lambda[a_1,\ldots,a_{i-1},b,a_{i+1},\ldots,a_k] \mid \cA_{k+1,u}\big] = \Lambda\big[a_1,\ldots,a_{i-1},\E_i[b \mid \cA_{i,u}] ,a_{i+1},\ldots,a_k\big]
    \]
    for all $i =1,\ldots,k$ and $b \in L^{p_i}(\cA_i,\E_i)$.
    Also, write $\cF_t^{p_1;p} \coloneqq \cF_{1,t}^{p_1;p}$.
    Finally, we say that a $k$-linear process $\Lambda \colon \R_+ \to B_k^{p_1,\ldots,p_k;p}$ is \textbf{adapted} if $\Lambda(s) \in \cF_{k,s}^{p_1,\ldots,p_k;p}$ for all $s \geq 0$.\label{item.klinadap1}
    \item Define $\cF_{k,t} = \cF_{k,t}(\E_1,\ldots,\E_k;\E_{k+1})$ to be the set of all $\Lambda \in \mathbb{B}_k$ such that $\Lambda(\cA_{1,u} \times \cdots \times \cA_{k,u}) \subseteq \cA_{k+1,u}$ for all $u \geq t$ and $\Lambda \in \cF_{k,t}^{p_1,\ldots,p_k;p}$ for some (equivalently, for all) $p_1,\ldots,p_k,p \in [1,\infty\ra$ satisfying $1/p_1 + \cdots + 1/p_k \leq 1/p$.
    Also, write $\cF_t \coloneqq \cF_{1,t}$.
    Finally, we say that a $k$-linear process $\Lambda \colon \R_+ \to \mathbb{B}_k$ is \textbf{adapted} if $\Lambda(s) \in \cF_{k,s}$ for all $s \geq 0$.\label{item.klinadap2}
    \item Fix $p \in [1,\infty]$.
    A process $X \colon \R_+ \to L^p(\cA_1,\E_1)$ is \textbf{adapted} if $X(s) \in L^p(\cA_{1,s},\E_1)$ for all $s \geq 0$.
    We write $C_a(\R_+;L^p(\E_1))$ for the complex Fr\'{e}chet space of $L^p$-continuous, adapted processes $\R_+ \to L^p(\E_1)$ equipped with the topology of uniform $L^p$-convergence on compact sets.\label{item.Lpadap}
\end{enumerate}
\end{defi}

\begin{rem}
In the conditionable case, $\cF_{k,t}$ may be defined simply as $\mathbb{B}_k \cap \cF_{k,t}^{\infty,\ldots,\infty;\infty}$.
Also, the spaces $\cF_{k,t}^{p_1,\ldots,p_k;p} = \cF_{k,t}^{p_1,\ldots,p_k;p}(\E_1, \ldots, \E_k;\E_{k+1})$ and $\cF_{k,t} = \cF_{k,t}(\E_1,\ldots,\E_k;\E_{k+1})$ clearly depend on the underlying filtrations of $\cA_1,\ldots,\cA_{k+1}$, as will other objects we introduce later.
Since we shall not have occasion to consider multiple filtrations on the same $\mathrm{C}^*$-probability space, we have chosen not to introduce cumbersome notation that obviates the dependence on the filtrations.
\end{rem}

\begin{obs}\label{obs.Ftpq}
Let $p,q,r,p_1,\ldots,p_k \in [1,\infty\ra$ and $s,t \geq 0$.
\begin{enumerate}[label=(\roman*),font=\normalfont]
    \item If $s \leq t$, then $\cF_{k,s}^{p_1,\ldots,p_k;p} \subseteq \cF_{k,t}^{p_1,\ldots,p_k;p}$ and $\cF_{k,s} \subseteq \cF_{k,t}$.\label{item.grow}
    \item If $\Lambda \in \cF_{k,t}^{p_1,\ldots,p_k;p}$, $u \geq t$, and $\a \in L^{p_1}(\cA_{1,u},\E_1) \times \cdots \times L^{p_k}(\cA_{k,u},\E_k)$, then $\Lambda[\a] \in L^p(\cA_{k+1,u},\E_{k+1})$.\label{item.preservefilt}
    \item $\cF_{k,t}^{p_1,\ldots,p_k;p} \subseteq B_k^{p_1,\ldots,p_k;p}$ is a $\norm{\cdot}_{p_1,\ldots,p_k;p}$-closed, complex-linear subspace, and $\cF_{k,t} \subseteq \mathbb{B}_k$ is a $\vertiii{\cdot}_k$-closed, complex-linear subspace.\label{item.WOTclosed}
\end{enumerate}
\end{obs}

In view of this observation, $(\cF_{k,t}^{p_1,\ldots,p_k;p})_{t \geq 0}$ should be considered as a filtration of $B_k^{p_1,\ldots,p_k;p}$ induced by the filtrations $(\cA_{1,t})_{t \geq 0},\ldots,(\cA_{k+1,t})_{t \geq 0}$.
Similarly $(\cF_{k,t})_{t \geq 0}$ should be viewed as a filtration of $\mathbb{B}_k$ induced by the filtrations $(\cA_{1,t})_{t \geq 0},\ldots,(\cA_{k+1,t})_{t \geq 0}$.
Next, we introduce a special class of adapted multilinear processes called trace $k$-processes, inspired in part by the biprocesses studied by Biane--Speicher \cite{BS1998} that relate to the $k=1$ case.
As we shall see throughout, trace $k$-processes appear in essentially all examples of interest.

\begin{nota}\label{nota.tens}
For the rest of this section, write $(\cA,(\cA_t)_{t \geq 0},\E) \coloneqq (\cA_1,(\cA_{1,t})_{t \geq 0}, \E_1)$, and fix $k \in \N$.
All tensor products below are over $\C$.
\begin{enumerate}[label=(\roman*),font=\normalfont]
    \item Write $\mathbb{B}_k(\cA) \coloneqq \mathbb{B}_k(\cA^k;\cA) = \mathbb{B}_k(\cA^{(1,\ldots,1)};\cA)$ and $\mathbb{B}(\cA) \coloneqq \mathbb{B}_1(\cA)$.\label{item.BbkA}
    \item Let $m \in \N$ and $d \in \N^k$.
    If $t \geq 0$, then we define
    \begin{align*}
        \cT_{m,k,d,t}^0 & \coloneqq \big\{P(\a) : n \in \N, \; P \in (\TrP_{n,k,d}^*)^m, \text{ and } \a \in \cA_t^n\big\} \subseteq \mathbb{B}_k(\cA^d;\cA^m) \, \text{ and} \\
        \cT_{m,k,d,t} & \coloneqq \overline{\cT_{m,k,d,t}^0} \subseteq \mathbb{B}_k(\cA^d;\cA^m) \text{ (closure with respect to $\vertiii{\cdot}_k$)}.
    \end{align*}
    Also, we write $\cT_{k,t}^0 \coloneqq \cT_{1,k,(1,\ldots,1),t}^0 \subseteq \mathbb{B}_k(\cA)$, $\cT_{k,t} \coloneqq \cT_{1,k,(1,\ldots,1),t} \subseteq \mathbb{B}_k(\cA)$, $\cT_t^0 \coloneqq \cT_{1,t}^0$, and $\cT_t \coloneqq \cT_{1,t}$.
    Finally, we define the spaces $\cT_{m,k,d,t}^{\C,0}$, $\cT_{m,k,d,t}^{\C}$, $\cT_{k,t}^{\C,0}$, $\cT_{k,t}^{\C}$, $\cT_t^{\C,0}$, and $\cT_t^{\C}$ similarly using $(\TrP_{n,k,d}^{*,\C})^m$ in place of $(\TrP_{n,k,d}^*)^m$.\label{item.Tfilt}
    \item Write $\#_k \colon \cA^{\otimes(k+1)} \to \mathbb{B}_k(\cA)$ for the complex-linear map determined by
    \[
    \#_k(a_1 \otimes \cdots \otimes a_{k+1})[b_1,\ldots,b_k] = a_1b_1\cdots a_kb_ka_{k+1} \qquad (a_i,b_j \in \cA).
    \]
    Also, write
    \[
    u \sh_k b \coloneqq \#_k(u)[b] \qquad \big(u \in \cA^{\otimes(k+1)}, \;b \in \cA^k\big).
    \]
    When $k=1$, we shall omit the subscript, i.e., $\# = \#_1$.\label{item.hash}
    \item Write $\#_k^{\E} \colon \cA^{\otimes(k+1)} \to \mathbb{B}_k(\cA)$ for the complex-linear map determined by
    \[
    \#_k^{\E}(a_1 \otimes \cdots \otimes a_{k+1})[b_1,\ldots,b_k] = \E[a_1b_1\cdots a_kb_k]a_{k+1} \qquad (a_i,b_j \in \cA).
    \]
    Also, write
    \[
    u \sh_k^{\mathsmaller{\E}} b \coloneqq \#_k^{\E}(u)[b] \qquad \big(u \in \cA^{\otimes(k+1)}, \;b \in \cA^k\big).
    \]
    When $k=1$, we shall omit the subscript, i.e., $\#^{\E} = \#_1^{\E}$.\label{item.Ehash}
\end{enumerate}
\end{nota}

For example, $H \in \cT_t^0$ if and only if there exist $a_i,b_i,c_i,d_i,e_i,f_i,g_i,h_i \in \cA_t$ such that
\[
Hx = \sum_{i=1}^n (a_ixb_i + c_ix^*d_i + \E[e_ix]\,f_i + \E[g_ix^*]\,h_i) \qquad (x \in \cA).
\]
Equivalently, there exist $u_1,u_2,v_1,v_2 \in \cA_t \otimes \cA_t$ such that
\[
Hx = u_1\sh x + u_2 \sh x^* +  v_1\sh^{\mathsmaller{\E}} x + v_2\sh^{\mathsmaller{\E}}x^* \qquad (x \in \cA).
\]
Moreover, $H \in \cT_t^{0,\C}$ if and only if we can take $u_2 = v_2 = 0$ above.

\begin{defi}[Trace $k$-processes]\label{def.trkpr}
A $k$-linear process $\Lambda \colon \R_+ \to \mathbb{B}_k(\cA^d;\cA^m)$ is called a (\textbf{multivariate}) \textbf{trace $\boldsymbol{k}$-process} if $\Lambda(t) \in \cT_{m,k,d,t}$ for all $t \geq 0$.
A trace $1$-process is also called a \textbf{trace biprocess}, and a trace $2$-process is also called a \textbf{trace triprocess}.
A (multivariate) trace $k$-process $\Lambda$ is called \textbf{complex--$\boldsymbol{k}$-linear} if $\Lambda(t) \in \cT_{m,k,d,t}^{\C}$ for all $t \geq 0$.
\end{defi}

We now show that trace $k$-processes are adapted.
\pagebreak

\begin{lem}\label{lem.T0inF}
If $t \geq 0$ and $m,d \in \N$, then $\cT_{m,1,d,t}^0 \subseteq \cF_t(\E^{\oplus d};\E^{\oplus m}) \subseteq \mathbb{B}(\cA^d;\cA^m)$.
\end{lem}

\begin{proof}
By the definitions of $(\TrP_{n,1,d}^*)^m$ and $\cF_t(\E^{\oplus d};\E^{\oplus m})$, it suffices to treat the $m=d=1$ case, i.e., to prove $\cT_t^0 \subseteq \cF_t = \cF_t(\E;\E)$.
To this end, let $a,b,c,d \in \cA_t$.
If $u \geq t$, $x \in L^1(\E)$, and $\e \in \{1,\ast\}$,~then
\begin{align*}
    \E\big[\big(\#(a \otimes b) + \#^{\E}(c \otimes d)\big)x^{\e} \mid \cA_u\big] & = \E[ax^{\e}b \mid \cA_u] + \E[cx^{\e}] \,\E[d \mid \cA_u] \\
    & = a\,\E[x \mid \cA_u]^{\e}\,b + \E[\E[cx^{\e} \mid \cA_u]]\,d \\
    & = (a \otimes b) \sh \E[x \mid \cA_u]^{\e} + \E\big[c\,\E[x \mid \cA_u]^{\e}\big] \,d \\
    & = \big(\#(a \otimes b) + \#^{\E}(c \otimes d)\big)\E[x \mid \cA_u]^{\e}.
\end{align*}
Also, if $x \in \cA_u$, then
\[
\big(\#(a \otimes b) + \#^{\E}(c \otimes d)\big)x^{\e} = ax^{\e}b + \E[cx^{\e}]\,d \in \cA_u.
\]
It follows that $\cT_t^0 \subseteq \cF_t$.
\end{proof}

\begin{prop}\label{prop.TkinFk}
If $0 \leq s \leq t$, $k,m \in \N$, and $d = (d_1,\ldots,d_k) \in \N^k$, then $\cT_{m,k,d,s}^0 \subseteq \cT_{m,k,d,t}^0$, and $\cT_{m,k,d,t}^0$ is a complex-linear subspace of $\cF_{k,t}(\E^{\oplus d_1},\ldots,\E^{\oplus d_k};\E^{\oplus m})$.
Consequently, $\cT_{m,k,d,s} \subseteq \cT_{m,k,d,t}$; $\cT_{m,k,d,t}$ is a closed, complex-linear subspace of $\cF_{k,t}(\E^{\oplus d_1},\ldots,\E^{\oplus d_k};\E^{\oplus m})$; and trace $k$-processes are adapted.
\end{prop}

\begin{proof}
The only nontrivial assertion of the proposition is that
\[
\cT_{m,k,d,t}^0 \subseteq \cF_{k,t}(\E^{\oplus d_1},\ldots,\E^{\oplus d_k};\E^{\oplus m}).
\]
To prove this, we make the key observation that if $\Lambda \in \cT_{m,k,d,t}^0$, $(\b_1,\ldots,\b_k) \in \cA_t^d$, and $i =1,\ldots,k$, then
\[
\cA^{d_i} \ni \b \mapsto \Lambda[\b_1,\ldots,\b_{i-1},\b,\b_{i+1},\ldots,\b_k] \in \cA^m
\]
belongs to $\cT_{m,1,d_i,t}^0$.
Thus, by definition of our induced filtrations, the desired containment follows from Lemma \ref{lem.T0inF}.
This completes the proof.
\end{proof}

\begin{ex}\label{ex.trpoly}
If $A_1,\ldots,A_8 \colon \R_+ \to \cA$ are adapted, then the process
\[
\R_+ \ni t \mapsto H(t) \coloneqq (x \mapsto A_1(t)xA_2(t) + A_3(t)x^*A_4(t) + \E[A_5(t)x]\,A_6(t) + \E[A_7(t)x^*]\,A_8(t)) \in \mathbb{B}(\cA)
\]
is a trace biprocess.
Indeed, $H(t) \in \cT_t^0$ for all $t \geq 0$.
By Proposition \ref{prop.TkinFk}, $H$ is adapted.

More generally, if $n,m \in \N$, $d = (d_1,\ldots,d_k) \in \N^k$, $P \in (\TrP_{n,k,d}^*)^m$, and $X \colon \R_+ \to \cA^n$ is adapted, then the process $\R_+ \ni t \mapsto \Lambda(t) \coloneqq P(X(t)) \in \mathbb{B}_k(\cA^d;\cA^m)$ is a multivariate trace $k$-process and thus is adapted.
This example, which we shall continue to study throughout the development, is one of the primary motivators for our definitions of the induced filtrations and adaptedness.
The other primary motivators are the proofs of Lemma \ref{lem.elemint}\ref{item.martEint} and Theorem \ref{thm.Itocont} below.
\end{ex}

\subsection{Decomposable processes}\label{sec.NCsemi}

We now define and give examples of FV processes, martingales, and our noncommutative analog of a semimartingale, a decomposable process.

\begin{defi}\label{def.processes}
Let $p \in [1,\infty]$ and $X \colon \R_+ \to L^p(\E)$ be a process.
\begin{enumerate}[label=(\roman*),font=\normalfont]
    \item $X$ is \textbf{$\boldsymbol{L^p}$-finite variation} (\textbf{$\boldsymbol{L^p}$-FV}) if it is adapted and has locally bounded variation with respect to $\norm{\cdot}_p$.
    Write $\FV^p = \FV_{\mathsmaller{\cA}}^p \subseteq C_a(\R_+;L^p(\E))$ for the complex Fr\'echet space of continuous $L^p$-FV processes with the topology induced by the collection $\{X \mapsto \norm{X(0)}_p + V_{L^p}(X : [0,t]) : t \geq 0\}$ of seminorms.\label{item.FV}
    \item $X$ is an \textbf{$\boldsymbol{L^p}$-martingale} if it is adapted and satisfies the martingale property, i.e.,
    \[
    \E[X(t) \mid \cA_s] = X(s) \qquad (0 \leq s \leq t).
    \]
    We shall omit the prefix ``$L^p$-'' when $p=\infty$.
    Write $\M^p = \M_{\mathsmaller{\cA}}^p$ for the space of continuous $L^p$-martingales with the topology of locally uniform convergence.
    (Note that $\M^p$ is a closed, complex-linear subspace of $C_a(\R_+;L^p(\E))$.)
    Also, write $\tbM^p = \tbM_{\mathsmaller{\cA}}^p$ for the closure of $\M^{\infty}$ in $\M^p$.\label{item.mart}
    \item Write\hspace{-0.17mm} $r\hspace{-0.17mm} \coloneqq\hspace{-0.17mm} p\hspace{-0.17mm} \wedge\hspace{-0.17mm} q\hspace{-0.17mm} =\hspace{-0.17mm} \min\{p,\hspace{-0.17mm}q\}$.\hspace{-0.17mm}
    An\hspace{-0.17mm} $L^r$-process\hspace{-0.17mm} $Y\hspace{-0.17mm} \colon\hspace{-0.17mm} \R_+\hspace{-0.17mm} \to\hspace{-0.17mm} L^r(\E)$\hspace{-0.17mm} is\hspace{-0.17mm} a\hspace{-0.17mm} (\textbf{continuous})\hspace{-0.17mm} \textbf{$\boldsymbol{(L^p,\hspace{-0.17mm}L^q)}$-decomposable process} if $Y = Y(0) + M + A$ for some $M \in \M^p$ and $A \in \FV^q$ such that $M(0) = A(0) = 0$.
    (Note that $Y \in C_a(\R_+;L^r(\E))$ in this case.)
    If $p=q$, then we shorten $(L^p,L^q)$ to $L^p$.\label{item.decomp}
\end{enumerate}
\end{defi}
\pagebreak

\begin{rem}[Norm of a martingale increases]\label{rem.norminc}
If $p \in [1,\infty]$, $M$ is an $L^p$-martingale, and $0 \leq s \leq t$, then $\norm{M(s)}_p = \norm{\E[M(t) \mid \cA_s]}_p \leq \norm{M(t)}_p$ by the martingale property, Proposition \ref{prop.condexp}\ref{item.condexpLpbd}, and the comments following Example \ref{ex.randommatricescond}.
In other words,
\begin{equation}
    \sup_{0 \leq s \leq t}\norm{M(s)}_p = \norm{M(t)}_p \qquad (t \geq 0).\label{eq.norminc}
\end{equation}
In particular, a sequence $(M_n)_{n \in \N}$ of $L^p$-martingales converges in $L^p(\E)$ uniformly on compact sets if and only if it converges pointwise in $L^p(\E)$.
Also, a process $M \colon \R_+ \to L^p(\E)$ belongs to $\tbM^p$ if and only if there exists a sequence $(M_n)_{n \in \N}$ of $L^{\infty}$-continuous martingales $\R_+ \to \cA$ such that for all $t \geq 0$, $M_n(t) \to M(t)$ in $L^p(\E)$ as $n \to \infty$.
We caution that $\tbM^p \subsetneq \M^p$ in general;
for example, we argue in Remark \ref{rem.Poisson} that the compensated free Poisson process belongs to $\M^2 \setminus \tbM^2$.
\end{rem}

\begin{ex}[Scalar FV processes]\label{ex.scalarFV}
If $g \colon \R_+ \to \C$ has locally bounded variation, then $A(t) \coloneqq g(t)1$ defines an $L^{\infty}$-FV process.
If, in addition, $g$ is continuous, then $A \in \FV^{\infty}$.
Also, if $X \in C_a(\R_+;L^p(\E))$, then the (Riemann--)Stieltjes integral process $\into X(t)\,\d g(t)$ is $L^p$-FV and continuous if $g$ is continuous.
\end{ex}

\begin{ex}[Classical matrix processes]\label{ex.clmart}
Let $(\cA_n,\tau_n)$ be as in Example \ref{ex.randommatrices}, and suppose $(\sF_t)_{t \geq 0}$ is a classical filtration of $\sF$.
If $(\cA_{n,t})_{t \geq 0} \coloneqq (L^{\infty}(\Om,\sF_t,P;\MnC))_{t \geq 0}$, then $(\cA_n,(\cA_{n,t})_{t \geq 0},\tau_n)$ is a filtered $\mathrm{W}^*$-probability space.
If $X \colon \R_+ \times \Om \to \MnC$ is a classical adapted $L^p$-process, then the induced noncommutative process
\[
\R_+ \ni t \mapsto \tilde{X}(t) \coloneqq X(t,\cdot) \in L^p(\Om,\sF,P;\MnC) = L^p(\cA_n,\tau_n)
\]
is adapted.
By Example \ref{ex.randommatricescond}, if $X$ is a classical $L^p$-martingale, then $\tilde{X}$ is a noncommutative $L^p$-martingale.
If, in addition, $p < \infty$ and $X$ is continuous---i.e.,  $P$-almost surely, the paths of $X$ are continuous---then $\tilde{X} \in \M_{\tau_n}^p$ by the dominated convergence theorem and Doob's maximal inequality (\cite[Thm.\ II.1.7]{RY1999}).
Often, this observation can be upgraded, e.g., with Brownian motion.
\end{ex}

\begin{thm}\label{thm.matBMtilde}
Retain the setup of Example \ref{ex.clmart}, and suppose that $(\sF_t)_{t \geq 0}$ satisfies the usual conditions (page \pageref{p.usual}).
If $X \colon \R_+ \times \Om \to \MnC_{\sa}$ is an $n \times n$ Hermitian Brownian motion (with respect to $\ip{\cdot,\cdot}_{L^2(\tr_n)}$) and $1 \leq p < \infty$, then the noncommutative martingale $\R_+ \ni t \mapsto X(t,\cdot) \in L^p(\cA_n,\tau_n)$ belongs to $\tbM_{\tau_n}^p$.
\end{thm}

\begin{proof}
Let $\cV$ be a finite-dimensional normed vector space and $Y \colon \R_+ \times \Om \to \cV$ be an adapted stochastic process.
Let $\alpha \in (0,1]$, and suppose that for every $\om \in \Om$, the path $\R_+ \ni t \mapsto Y(t,\om) \in \cV$ is locally $\alpha$-H\"{o}lder continuous.
Now, if
\[
C(t,\om) \coloneqq \sup_{0 \leq r < s \leq t} \frac{\norm{Y(r,\om) - Y(s,\om)}_{\cV}}{|r-s|^{\alpha}} = \sup_{0 \leq r < s < t} \frac{\norm{Y(r,\om) - Y(s,\om)}_{\cV}}{|r-s|^{\alpha}} \qquad ((t,\om) \in \R_+ \times \Om)
\]
with the interpretation $C(0,\cdot) \equiv 0$, then $C \colon \R_+ \times \Om \to \R_+$ is an adapted, left-continuous stochastic process.
Now, for $r > 0$, define
\[
\tau_r(\om) \coloneqq \inf\{t \geq 0 : \norm{Y(t,\om)}_{\cV} > r\} \wedge \inf\{t \geq 0 : C(t,\om) > r\} \in [0,\infty] \qquad (\om \in \Om).
\]
Since $C$ and $\norm{Y}_{\cV}$ are left-continuous and adapted, $(r,\infty)$ is open, and the filtration is right-continuous, $\tau_r$ is a stopping time.
Moreover, it is easy to see that $\tau_r \nearrow \infty$ pointwise as $r \to \infty$.
Most importantly, the stopped process $Y^{\tau_r}(t,\om) \coloneqq Y(t \wedge \tau_r(\om),\om)$ satisfies
\[
\norm{Y^{\tau_r}(t,\om)}_{\cV} \leq r \; \text{ and } \; \norm{Y^{\tau_r}(s,\om) - Y^{\tau_r}(t,\om)}_{\cV} \leq r|s-t|^{\alpha} \qquad (s,t \geq 0, \; \om \in \Om).
\]
In particular, the map $\R_+ \ni t \mapsto Y^{\tau_r}(t,\cdot) \in L^{\infty}(\Om,\sF,P;\cV)$ is continuous with respect to the $L^{\infty}$ norm.
If, in addition, $Y$ is a classical martingale, then so is $Y^{\tau_r}$.
Finally, if, in addition, $p \in (1,\infty)$ and $Y$ is a classical $L^p$-martingale, then
\[
\E_P\Big[\sup_{0 \leq s \leq t} \norm{Y^{\tau_r}(s,\cdot) - Y(s,\cdot)}_{\cV}^p\Big]^{\frac{1}{p}} \leq \frac{p}{p-1} \E_P\big[\norm{Y^{\tau_r}(t,\cdot) - Y(t,\cdot)}_{\cV}^p\big]^{\frac{1}{p}} \xrightarrow{r \to \infty} 0
\]
by Doob's maximal inequality (\cite[Thm.\ II.1.7]{RY1999}) and the dominated convergence theorem.

Now, apply the previous paragraph with $\cV = \MnC_{\sa}$ (endowed with the operator norm) and $Y = X$, an $n \times n$ Hermitian Brownian motion.
By Kolmogorov's continuity theorem, if $\alpha \in (0,1/2)$, then (a modification of) $X$ has locally $\alpha$-H\"{o}lder continuous paths.
Also, for any $p < \infty$, $X$ is a classical $L^p$-martingale.
By the previous paragraph, if $r > 0$, then the noncommutative martingale $\R_+ \ni t \mapsto X^{\tau_r}(t,\cdot) \in \cA_n$ is $L^{\infty}$-continuous and $X^{\tau_r} \to X$ in $C_a(\R_+;L^p(\tau_n))$ as $r \to \infty$.
We conclude that $X$, viewed as an noncommutative $L^p$-martingale, belongs to $\tbM_{\tau_n}^p$, as claimed.
\end{proof}

\begin{ex}[Free/stationary increments]\label{ex.freeinc}
An adapted process $X \colon \R_+ \to \cA$ has (\textbf{$\boldsymbol{\ast}$-})\textbf{free increments} (\textbf{with respect to $\boldsymbol{(\cA_t)_{t \geq 0}}$}) if $X(t) - X(s)$ is ($\ast$-)free from $\cA_s$ whenever $0 \leq s < t$.
It has \textbf{stationary increments} if the $\ast$-distribution of $X(t) - X(s)$ depends only on $t-s$ whenever $0 \leq s < t$.

First, suppose that an adapted process $M \colon \R_+ \to \cA$ has constant expectation, i.e., $\E[M(t)] = \E[M(0)]$ for all $t \geq 0$, and free increments.
If $0 \leq s \leq t$, then
\begin{align*}
    \E[M(t) \mid \cA_s] & = \E[M(t)-M(s) \mid \cA_s] + \E[M(s) \mid \cA_s] \\
    & = \E[M(t)-M(s)] + M(s) = M(s).
\end{align*}
Thus, $M$ is a martingale.
In particular, if $M$ is also $L^p$-continuous, then $M \in \M^p$.

Next, suppose that an adapted process $X \colon \R_+ \to \cA$ has $\ast$-free and stationary increments.
By the previous paragraph, $M \coloneqq X - \E[X]$ is a martingale.
By \cite[Lem.\ 1(1)]{Anshelevich2002} applied to the real and imaginary parts of $X$, $X$ is $L^p$-continuous whenever $p < \infty$.
We claim that $A \coloneqq \E[X]$ is $L^{\infty}$-FV.
Indeed, let $t > 0$ and $n \in \N$.
Let $\pi_n$ be the partition of $[0,t]$ with $n+1$ evenly spaces points so that $\Delta s = t/n =\vcentcolon \e_n$ whenever $0 < s \in \pi_n$.
If $0 < s \in \pi_n$, then
\begin{align*}
    \E[X(t) - X(0)] & = \sum_{r \in \pi_n} \E[\Delta_r X] = \sum_{0 < r \in \pi_n}\E[X(\e_n) - X(0)] \\
    & = n \, \E[X(\e_n) - X(0)] = n \, \E[\Delta_sX] = n\, \Delta_s\E[X]
\end{align*}
by stationarity (twice).
Thus,
\[
\sum_{s \in \pi_n} |\Delta_s\E[X]| = |\E[X(t) - X(0)]|.
\]
Since $n \in \N$ was arbitrary and $\E[X]$ is continuous, we get that
\[
V_{L^{\infty}}(\E[X] : [0,t]) = |\E[X(t) - X(0)]| < \infty,
\]
as desired.
We conclude that if $p < \infty$, then
\[
X = M+A = X(0) + (M-M(0)) + (A-A(0))
\]
is $(L^p,L^{\infty})$-decomposable.

Important examples of processes with free and stationary increments are the free Brownian motions.
An adapted process $X \colon \R_+ \to \cA$ is a (\textbf{semi})\textbf{circular} (\textbf{free}) \textbf{Brownian motion} if $X(0) = 0$, $X$ has $\ast$-free increments, and $X(t) - X(s) \in \cA_t$ is a (semi)circular element of variance $t-s$ whenever $0 \leq s < t$.
By the first paragraph, (semi)circular Brownian motions are martingales.
Furthermore, (semi)circular Brownian motions are $L^{\infty}$-continuous.
Indeed, if $X \colon \R_+ \to \cA_{\sa}$ is a semicircular Brownian motion, then
\begin{equation}
    \norm{X(t)-X(s)} = 2\sqrt{|t-s|} \qquad (s,t \geq 0); \label{eq.fBMnorm}
\end{equation}
and if $Z \colon \R_+ \to \cA$ is a circular Brownian motion, then $Z = 2^{-1/2}(X + i Y)$, where $X,Y \colon \R_+ \to \cA_{\sa}$ are two (freely independent) semicircular Brownian motions, so $Z$ is $L^{\infty}$-continuous as well.
\end{ex}

More generally, there are $q$-Brownian motions and other $q$-Gaussian martingales.

\begin{ex}[$q$-Gaussian processes]\label{ex.qGauss}
The $q$-Brownian motions ($-1 \leq q \leq 1$) are a family of noncommutative stochastic processes interpolating between the fermionic ($q=-1$), semicircular ($q=0$), and classical ($q=1$) Brownian motions;
see \cite{BKS1997,DonatiMartinS2003} for their definition and basic properties.
If $-1 \leq q < 1$, then the $q$-Brownian motion is an $L^{\infty}$-continuous martingale and, thus, an $L^{\infty}$-decomposable process.
Moreover, if $H_n^{(q)}$ is the $n^{\text{th}}$ $q$-Hermite polynomial (\cite[Def.\ 1.9]{BKS1997}) and $X \colon \R_+ \to \cA$ is the $q$-Brownian motion ($q \neq 1$), then the process $M_n \colon \R_+ \to \cA$ defined by $M_n(0) \coloneqq 0$ and $M_n(t) = t^{n/2}H_n^{(q)}\big(t^{-1/2}X(t)\big)$ for all $t > 0$ is an $L^{\infty}$-continuous martingale and, thus, an $L^{\infty}$-decomposable process.
\pagebreak

The full force of the framework of $q$-Gaussian processes introduced by Bo\.{z}ejko, K\"{u}mmerer, and Speicher in \cite{BKS1997} enables the generalization of the previous paragraph's statements.
Fix $q \in [-1,1)$ and a $q$-Gaussian process $X \colon \R_+ \to \cA$ with covariance function $c \colon \R_+ \times \R_+ \to \R$;
see \cite[Def.\ 3.3]{BKS1997}.
(For the $q$-Brownian motion, $c(s,t) = \min\{s,t\}$.)
By \cite[Prop.\ 3.13]{BKS1997}, $X$ is a martingale if and only if $c(s,t) = c(s,s)$ whenever $0 \leq s \leq t$.
If $c$ is continuous, then $X$ is $L^{\infty}$-continuous.
By \cite[Cor.\ 4.7]{BKS1997}, if the covariance factors as $c(s,t) = g(s)\,f(t)$ for $0 \leq s \leq t$, where $f(s), g(s) > 0$ for all $s > 0$, then the process $M_n$ defined by
\[
M_n(0) \coloneqq 0 \; \text{ and } \; M_n(t) \coloneqq \Bigg(\frac{g(t)}{f(t)}\Bigg)^{\frac{n}{2}}H_n^{(q)}\Bigg(\frac{X(t)}{c(t,t)^{\frac{1}{2}}}\Bigg) \qquad (t > 0)
\]
is a martingale.
\end{ex}

There are also elementary ways to construct new martingales from old.

\begin{nota}\label{nota.stop}
If $a,b \in \R \cup \{\pm \infty\}$, then $a \wedge b \coloneqq \min\{a,b\}$ and $a \vee b \coloneqq \max\{a,b\}$.
If $S$ is a set, $X \colon \R_+ \to S$ is a function, and $t \geq 0$, then $X^t \coloneqq X(\cdot \wedge t) \colon \R_+ \to S$.
\end{nota}

\begin{ex}\label{ex.elem}
If $M$ is an $L^p$-martingale and $t \geq 0$, then the stopped process $M^t$ is also an $L^p$-martingale.
Also, if $p < \infty$, $a \in L^p(\E)$,  and $M_a(t) \coloneqq \E[a\mid \cA_t]$ for all $t \geq 0$, then $M_a$ is an $L^p$-martingale, called Doob's martingale, by the tower property.
(We may include $p=\infty$ in the conditionable case.)
Finally, $n$-tuples of martingales are martingales.
More precisely, let $(\cA_1,(\cA_{1,t})_{t \geq 0},\E_1),\ldots,(\cA_n,(\cA_{n,t})_{t \geq 0},\E_n)$ be $\mathrm{C}^*$-probability spaces, and suppose $(\cA,(\cA_t)_{t \geq 0},\E)$ is the direct sum of these spaces.
If $X_i \colon \R_+ \to L^p(\E_i)$ is adapted for each $i =1,\ldots,n$, then the process $X \coloneqq (X_1,\ldots,X_n) \colon \R_+ \to L^p(\E)$ is adapted as well.
If $X_i$ is an $L^p$-martingale for each $i=1,\ldots,n$, then $X$ is an $L^p$-martingale.
\end{ex}

We end this section by studying decomposable processes more seriously.
We begin by showing that, similar to the classical case, the decomposition $X = X(0) + M + A$ of a decomposable process is unique under mild assumptions.
This is done, again as in the classical case, by establishing a version of the statement, ``A continuous martingale of locally bounded variation is constant.''

\begin{lem}\label{lem.M(t)-M(s)}
If $M \colon \R_+ \to L^2(\E)$ is an $L^2$-martingale, then
\[
\norm{M(t)-M(s)}_2^2 = \norm{M(t)}_2^2 - \norm{M(s)}_2^2 \qquad (0 \leq s \leq t).
\]
\end{lem}

\begin{proof}
Suppose $0 \leq s \leq t$.
By basic properties of conditional expectation, the adaptedness of $M$, and the martingale property,
\begin{align*}
    \norm{M(t)-M(s)}_2^2 & = \norm{M(t)}_2^2 + \norm{M(s)}_2^2 - \ip{M(t),M(s)}_2 - \ip{M(s),M(t)}_2 \\
    & = \norm{M(t)}_2^2 + \norm{M(s)}_2^2 - \ip{\E[M(t) \mid \cA_s], M(s)}_2 - \ip{M(s), \E[M(t) \mid \cA_s] }_2 \\
    & = \norm{M(t)}_2^2 - \norm{M(s)}_2^2,
\end{align*}
as desired.
\end{proof}

\begin{prop}[Continuous FV martingales are constant]\label{prop.FVmart}
If $1/p+1/q \leq 1$ and $M \colon \R_+ \to L^{p \vee q}(\E)$ is an $L^p$-continuous, $L^q$-FV martingale, then $M \equiv M(0)$.
\end{prop}

\begin{proof}
Fix $t \geq 0$ and a partition $\pi$ of $[0,t]$.
Since $1/p+1/q \leq 1$, we have that $p \vee q \geq 2$.
In particular, $M$ is an $L^2$-martingale.
By Lemma \ref{lem.M(t)-M(s)} (twice) and noncommutative H\"{o}lder's inequality,
\begin{align*}
    \norm{M(t) - M(0)}_2^2 & = \norm{M(t)}_2^2 - \norm{M(0)}_2^2 = \sum_{s \in \pi} \big(\norm{M(s)}_2^2 - \norm{M(s_-)}_2^2\big) \\
    & = \sum_{s \in \pi} \norm{M(s) - M(s_-)}_2^2 = \sum_{s \in \pi} \norm{\Delta_sM^* \, \Delta_sM}_1 \\
    & \leq \max_{r \in \pi}\norm{\Delta_rM^*}_p \sum_{s \in \pi}\norm{\Delta_sM}_q \leq \max_{r \in \pi}\norm{\Delta_rM}_p \,V(M : [0,t]) \xrightarrow[\pi \in \cP_{[0,t]}]{|\pi| \to 0} 0.
\end{align*}
Thus, $M(t) = M(0)$.
\end{proof}

\begin{cor}[Uniqueness of decompositions]\label{cor.uniquedecomp}
If $1/p+1/q \leq 1$ and $X \colon \R_+ \to L^{p \wedge q}(\E)$ is an $(L^p,L^q)$-decomposable process, then the decomposition $X = X(0) + M + A$ as in Definition \ref{def.processes}\ref{item.decomp} is unique.
In this case, we call $X^{\mathrm{m}} \coloneqq M$ the \textbf{martingale part} of $X$ and $X^{\mathrm{fv}} \coloneqq A$ the \textbf{FV part} of $X$.
When confusion is possible, we shall use the terms ``$(L^p,L^q)$-martingale part'' and ``$(L^p,L^q)$-FV part.''
\end{cor}

\begin{proof}
If $M,N \in \M^p$, $A,B \in \FV^q$, $A(0) = B(0) = M(0) = N(0) = 0$, and $M + A = X - X(0) = N + B$, then $Y \coloneqq M-N = B - A$ is an $L^p$-continuous, $L^q$-FV martingale.
By Proposition \ref{prop.FVmart}, $Y \equiv Y(0) = 0$.
In particular, $M = N$ and $A = B$, as desired.
\end{proof}

\begin{rem}[Adjoints]\label{rem.adj}
Note that if $M \in \M^p$, then $M^* \in \M^p$;
and if $A \in \FV^q$, then $A^* \in \FV^q$.
In particular, if $X = X(0) + M + A$ is $(L^p,L^q)$-decomposable, then so is $X^* = X(0)^* + M^* + A^*$.
Also, by Corollary \ref{cor.uniquedecomp}, if $1/p + 1/q \leq 1$, then $X^* = X$ if and only if $M^* = M$ and $A^*=A$.
\end{rem}

Here is an example demonstrating that the restriction $1/p + 1/q \leq 1$ is not artificial.

\begin{ex}[Poisson process]\label{ex.Poisson}
Let $(\Om,\sF,(\sF_t)_{t \geq 0},P)$ be a filtered probability space, and suppose $(\cA,(\cA_t)_{t \geq 0},\E) = (L^{\infty}(\Om,\sF,P),(L^{\infty}(\Om,\sF_t,P))_{t \geq 0},\E_P)$.
Fix $\lambda > 0$, and let $X \colon \R_+ \times \Om \to \R$ be a Poisson process with rate $\lambda$, i.e., $X$ is adapted, $X(0,\cdot) = 0$ almost surely, and $X(t,\cdot) - X(s,\cdot)$ is a Poisson random variable with mean $\lambda(t-s)$ that is (classically) $P$-independent of $\sF_s$ whenever $0 \leq s < t$.
We view $X$ as a noncommutative adapted process $\tilde{X} \colon \R_+ \to L^1(\E)$ as in Example \ref{ex.clmart}.
Observe that
\[
\big\|\tilde{X}(t) - \tilde{X}(s)\big\|_1 = \lambda|t-s| \qquad (s,t \geq 0),
\]
so $\tilde{X} \in \FV^1$.
Thus, $\tilde{X}$ is an $(L^{\infty},L^1)$-decomposable process with $(L^{\infty},L^1)$-martingale part $0$ and $(L^{\infty},L^1)$-FV part $\tilde{X}$.
However,
\[
\tilde{X}(t) = \tilde{X}(t) - \E\big[\tilde{X}(t)\big] + \E\big[\tilde{X}(t)\big] = \tilde{X}(t) - \lambda t + \lambda t \qquad  (t > 0)
\]
as well.
Since the compensated Poisson process $M \coloneqq (X(t,\cdot) - \lambda t)_{t \geq 0}$ is a martingale, we conclude that $\tilde{X}$ is an $(L^1,L^{\infty})$-decomposable process with $(L^1,L^{\infty})$-martingale part $M$ and $(L^1,L^{\infty})$-FV part $(\lambda t)_{t \geq 0}$.
The preceding $(L^{\infty},L^1)$- and $(L^1,L^{\infty})$-decompositions of $\tilde{X}$ are both valid $(L^1,L^1)$-decompositions of $\tilde{X}$.
Thus, as an $L^1$-decomposable process, $\tilde{X}$ has two \textit{different} decompositions.
\end{ex}

\begin{rem}\label{rem.cont}
The Poisson process is a prototypical example of a right-continuous process with no continuous modification.
Nevertheless, as a noncommutative process, the Poisson process is $L^1$-continuous;
in fact, it is $L^p$-continuous for all $p < \infty$.
This demonstrates that $L^p$-continuity (with $p < \infty$) is a rather weak form of continuity.
\end{rem}

\section{Stochastic integrals}\label{sec.stochint}

\subsection{Stieltjes integrals}\label{sec.vint}

In this section, we conduct a limited discussion of vector-valued Stieltjes integrals.
At this time, the reader should review Notations \ref{nota.part} and \ref{nota.nota}.
For the remainder of this section, fix $a \in \R$ and $b \in \R \cup \{\infty\}$ such that $a < b$, and write $I \coloneqq [a,b] \cap \R$.

\begin{nota}[Augmented partitions]\label{nota.augpart}
Write $\cP_I^*$ for the set of augmented partitions of $I$, i.e.,
\[
\cP_I^* \coloneqq \{\Pi^* = (\Pi, \ast) : \Pi \in \cP_I \text{ and } \ast \colon \Pi \to I \text{ is such that } t_* \coloneqq \ast(t) \in [t_-,t] \text{ for all } t \in \Pi\}.
\]
Also, if $\Pi^* = (\Pi, \ast) \in \cP_I^*$, then
\[
F^{\Pi^*} = F^{(\Pi,\ast)} \coloneqq 1_{\{a\}} F(a) + \sum_{s \in \Pi} 1_{(s_-,s]}F(s_*) \colon I \to \cV.
\]
If $t_* = \ast(t)$ is always the left endpoint $t_-$ of $[t_-,t]$, then $F^{\Pi} \coloneqq F^{(\Pi,\ast)}$.
\end{nota}

The sets $\cP_I$ and $\cP_I^*$ are frequently directed by refinement.
In this paper, we shall instead direct them by mesh $|\cdot|$, i.e., $\Pi \leq \Pi'$ and $(\Pi,\ast) \leq (\Pi',\ast')$ whenever $|\Pi| \geq |\Pi'|$.
Here is an elementary clarification of what convergence of nets directed by partition mesh means.
We leave the proof to the reader.
\pagebreak

\begin{fact}\label{fact.partlim}
Let $\mathcal{X}$ be a topological space and $x \colon \cP_I \to \mathcal{X}$ be a net.
For $y \in \mathcal{X}$, $\lim_{\Pi \in \cP_I} x(\Pi) = y$ holds if and only if for all open neighborhoods $\cU$ of $y$, there exists a $\delta > 0$ such that $|\Pi| < \delta$ implies $x(\Pi) \in \cU$; which happens if and only if for every sequence $(\Pi_n)_{n \in \N}$ in $\cP_I$ such that $|\Pi_n| \to 0$ as $n \to \infty$, we have $\lim_{n \to \infty}x(\Pi_n) = y$.
In this case, we say that $x(\Pi) \to y$ as $|\Pi| \to 0$.
One can similarly characterize convergence of nets $x^* \colon \cP_I^* \to \mathcal{X}$.
\end{fact}

We now turn to the construction of Riemann--Stieltjes integrals.

\begin{nota}[Elementary integral]\label{nota.elemint}
Write
\begin{align*}
    \sE_I & \coloneqq \{\{a\}\} \cup \{(s,t] : s,t \in I, \, s < t\} \; \text{ and} \\
    \sA_I & \coloneqq \{\text{finite disjoint unions of members of } \sE_I\}.
\end{align*}
Now, let $\cV$ be a vector space and $F \colon I \to \cV$ be a function.
Write $\mu_F^0 \colon \sA_I \to \cV$ for the unique finitely additive function such that $\mu_F^0(\{a\}) = 0$ and $\mu_F^0((s,t]) = F(t) - F(s)$ for all $s,t \in I$ such that $s \leq t$.
Finally, suppose $\cV$ is normed, $\cW$ is another normed vector space, and $S \colon I \to B(\cV;\cW)$ is an $\sA_I$-simple function, i.e.,
\[
S = 1_{\{a\}} T_0 + \sum_{i=1}^n 1_{(s_i,t_i]} T_i
\]
for some $T_0,T_1,\ldots,T_n \in B(\cV;\cW)$ and $s_1,t_1,\ldots,s_n,t_n \in I$.
Then we write
\[
\int_I S\big[\d\mu_F^0\big] = \int_I S(t)\big[\mu_F^0(\d t)\big] \coloneqq \sum_{T \in S(I)} T\mu_F^0(\{t \in I : S(t) = T\}) = \sum_{i=1}^n T_i(F(t_i) - F(s_i))
\]
for the integral of $S$ with respect to $\mu_F^0$ (or ``against $F$'').
In addition, we write
\[
\int_s^t S[\d\mu_F^0] = \int_s^t S(r)[\mu_F^0(\d r)] \coloneqq \int_I (1_{(s,t]}S)[\d\mu_F^0] \qquad (s,t \in I, \; s \leq t).
\]
\end{nota}

Observe that if $F \colon I \to \cV$ is (left-/right-)continuous and $S \colon I \to B(\cV;\cW)$ is $\sA_I$-simple, then the function $\int_a^{\boldsymbol{\cdot}} S[\d\mu_F^0] \colon I \to \cW$ is (left-/right-)continuous.
We also note for later use that if $\cX$ is another normed vector space, $T \colon I \to B(\cW;\cX)$ is $\sA_I$-simple, and $G \coloneqq \int_a^{\boldsymbol{\cdot}} S[\d\mu_F^0]$, then
\begin{equation}
    \int_I T\big[\d\mu_G^0\big] = \int_I TS\big[\d\mu_F^0\big].\label{eq.algsubform}
\end{equation}
We leave it to the reader to verify this.

\begin{lem}[Approximation by step functions]\label{lem.HPi}
Let $\cV$ be a Hausdorff topological vector space and $H \colon I \to \cV$ be a function.
\begin{enumerate}[label=(\roman*),font=\normalfont]
    \item If $H$ has left limits, i.e., $H_-(t) \coloneqq \lim_{s \nearrow t}H(s)$ exists for all $t \in I$ (with the convention $H_-(a) \coloneqq H(a)$), then $H^{\Pi} \to H_-$ pointwise as $|\Pi| \to 0$.
    In particular, if $H$ is left-continuous, then $H^{\Pi} \to H$ pointwise as $|\Pi| \to 0$.
    Also, if $\cV$ is normed, then $\big\|H^{\Pi}(t)\big\|_{\cV} \leq \sup\{\norm{H(s)}_{\cV} : 0 \leq s < t\}$.\label{item.HPipw}
    \item If $H$ is continuous and $\cV$ is normed, then $H^{\Pi^*} \to H$ uniformly on compact subsets of $I$ as $|\Pi| \to 0$.\label{item.HPistar}
    \end{enumerate}
\end{lem}

\begin{proof}
As the reader may easily verify, it suffices to treat the $b < \infty$ case.
Write $\iota \colon I \to \R$ for the inclusion.
Observe that if $\Pi^* = (\Pi,\ast) \in \cP_I^*$ and $t \in (a,b]$, then $H^{\Pi^*} = H \circ \iota^{\Pi^*}$, $0 < t - \iota^{\Pi}(t) \leq |\Pi|$, and $\big|t - \iota^{\Pi^*}(t)\big| \leq |\Pi|$.
Using these observations, the desired results follow easily from the definitions and, in the case of \ref{item.HPistar}, the fact that continuous functions $I \to \cV$ are uniformly continuous.
\end{proof}

\begin{prop}[Riemann--Stieltjes integrals of continuous functions]\label{prop.contRSint}
Suppose that $\cV$ is a normed vector space and $\cW$ is a Banach space.
If $F \colon I \to \cV$ has locally bounded variation, $H \colon I \to B(\cV;\cW)$ is continuous, and $c \in I$, then $H|_{[a,c]}$ is Riemann--Stieltjes $F|_{[a,c]}$-integrable, and
\[
\sum_{t \in \Pi} H(t_*)[F(t \wedge \cdot) - F(t_- \wedge \cdot)] \xrightarrow[\Pi^* \in \cP_I^*]{|\Pi| \to 0} \int_a^{\boldsymbol{\cdot}} H(t)[\d F(t)]
\]
uniformly on compact subsets of $I$.
\end{prop}

\begin{proof}
As the reader may easily verify, it suffices to treat the $b < \infty$ case.
If $\cU$ is a Banach space, write $\ell^{\infty}(I;\cU)$ for the Banach space of bounded functions $I \to \cU$ with the uniform norm.\label{page.ell}
Also, write $\mathbb{S}_I$ for the set of $\sA_I$-simple functions $I \to B(\cV;\cW)$.
If $S \in \mathbb{S}_I$, then there exists a $\Pi \in \cP_I$ such that $S = 1_{\{a\}} S(a) + \sum_{t \in \Pi} 1_{(t_-,t]} S(t)$.
Consequently, if $G \colon I \to \cV$ is any function, then
\begin{align*}
    \bigg\|\int_I S\big[\d\mu_G^0\big] \bigg\|_{\cW} & = \Bigg\| \sum_{t \in \Pi} S(t)[\Delta_tG] \Bigg\|_{\cW} \leq \sum_{t \in \Pi} \norm{S(t)}_{\cV \to \cW} \norm{\Delta_tG}_{\cV} \\
    & \leq \max_{t \in \Pi}\norm{S(t)}_{\cV \to \cW}\sum_{s \in \Pi} \norm{\Delta_sG}_{\cV} \leq V(G : I) \, \norm{S}_{\ell^{\infty}(I;B(\cV;\cW))}.
\end{align*}
It follows that
\[
\sup_{c \in I}\bigg\|\int_a^c S\big[\d\mu_F^0\big]\bigg\|_{\cW} \leq V(F : I) \, \norm{S}_{\ell^{\infty}(I;B(\cV;\cW))}.
\]
In particular, the integral map $\mathbb{S}_I \ni S \mapsto \int_a^{\cdot}S[\d\mu_F^0] \in \ell^{\infty}(I;\cW)$ extends uniquely to a bounded linear map $J \colon \overline{\mathbb{S}_I} \to \ell^{\infty}(I;\cW)$, where $\overline{\mathbb{S}_I}$ is the closure of $\mathbb{S}_I$ in $\ell^{\infty}(I;B(\cV;\cW))$.
Now, if $H \colon I \to B(\cV;\cW)$ is continuous and $\Pi^* \in \cP_I^*$, then $H^{\Pi^*} \in \mathbb{S}_I$, and by Lemma \ref{lem.HPi}\ref{item.HPistar}, $H^{\Pi^*} \to H$ in $\ell^{\infty}(I;B(\cV;\cW))$ as $|\Pi| \to 0$.
Thus, $H \in \overline{\mathbb{S}_I}$, and by the continuity of $J$,
\[
\sum_{t \in \Pi} H(t_*)[F(t \wedge \cdot) - F(t_- \wedge \cdot)] = \int_a^{\cdot} H^{\Pi^*}\big[\d\mu_F^0\big] = J\big[ H^{\Pi^*}\big] \xrightarrow[\Pi^* \in \cP_I^*]{|\Pi| \to 0} J[H]
\]
uniformly.
This completes the proof.
\end{proof}

While the above fact about Riemann--Stieltjes integrals suffices for most situations in practice, general considerations necessitate an understanding of Lebesgue--Stieltjes integrals as well.

\begin{lem}\label{lem.elemintbd}
Let $\cV,\cW$ be normed vector spaces, and suppose $F \colon I \to \cV$ is a right-continuous function of locally bounded variation.
There exists a unique measure $\nu_F \colon \cB_I \to [0,\infty]$ such that $\nu_F(\{a\}) = 0$ and $\nu_F((s,t]) = V(F : [s,t])$ whenever $s,t \in I$ and $s < t$.
Write $\norm{\d F(t)}_{\cV} \coloneqq \nu_F(\d t)$.
If $\mathbb{S}_I$ is the set of $\sA_I$-simple functions $I \to B(\cV;\cW)$, then
\[
\bigg\|\int_I S\big[\d\mu_F^0\big]\bigg\|_{\cW} \leq \int_I \norm{S(t)}_{\cV \to \cW} \, \norm{\d F(t)}_{\cV} \qquad (S \in \mathbb{S}_I).
\]
\end{lem}

\begin{proof}
If $T_F(t) \coloneqq V(F : [a,t])$ for all $t \in I$, then $T_F \colon I \to \R_+$ is right-continuous (continuous if $F$ is), and $T_F(t) - T_F(s) = V(F : [s,t])$ whenever $s,t \in I$ and $s < t$.
The existence and uniqueness of $\nu_F$ follow.
Now, by definition of $\sA_I$, if $A \in \sA_I$, then there are disjoint $A_1,\ldots,A_n \in \sE_I$ such that $A = \bigcup_{i=1}^n A_i$.
Therefore,
\[
\norm{\mu_F^0(A)}_{\cV} \leq \sum_{i=1}^n \norm{\mu_F^0(A_i)}_{\cV} \leq \sum_{i=1}^n \nu_F(A_i) = \nu_F(A).
\]
It follows that if $S \in \mathbb{S}_I$, then
\begin{align*}
    \bigg\|\int_I S\big[\d\mu_F^0\big] \bigg\|_{\cW} & \leq \sum_{T \in S(I)} \norm{T}_{\cV \to \cW}\big\|\mu_F^0(\{t \in I : S(t) = T\})\big\|_{\cV} \\
    & \leq \sum_{T \in S(I)} \norm{T}_{\cV \to \cW}\nu_F(\{t \in I : S(t) = T\}) = \int_I \|S(t)\|_{\cV \to \cW} \, \norm{\d F(t)}_{\cV},
\end{align*}
as desired.
\end{proof}

\begin{thm}[Construction of Lebesgue--Stieltjes integral]\label{thm.LSint}
Suppose $\cV$ is a normed vector space, $\cW$ is a Banach space, and $F \colon I \to \cV$ is a right-continuous function of locally bounded variation.
The integral map $\mathbb{S}_I \ni S \mapsto \int_I S[\d\mu_F^0] \in \cW$ extends uniquely to a bounded linear map $I_F^{\cW} \colon L^1(I,\nu_F;B(\cV;\cW)) \to \cW$ with operator norm at most $1$.
If $H \in L^1(I,\nu_F;B(\cV;\cW))$, then the vector
\[
\int_I H(t)[\d F(t)] = \int_I H[\d F] \coloneqq I_F^{\cW}(H) \in \cW\pagebreak
\]
is the (\textbf{Lebesgue--})\textbf{Stieltjes integral} of $H$ against $F$.
If $H \in L_{\loc}^1(I,\nu_F;B(\cV;\cW))$, then we shall write
\[
\int_s^t H(r)[\d F(r)] = \int_s^t H[\d F] \coloneqq \int_I (1_{(s,t]}H)[\d F] \qquad (s,t \in I, \; s \leq t).
\]
\end{thm}

\begin{proof}
By Lemma \ref{lem.elemintbd} and the completeness of $\cW$, it suffices to show $\mathbb{S}_I$ is dense in $\cX \coloneqq L^1(I,\nu_F;B(\cV;\cW))$.
First, note that if $H \in \cX$, then $1_{[a,c]}H \to H$ in $\cX$ as $c \nearrow b$.
Therefore, we may and do assume $b < \infty$.
Now, since simple functions are dense in $\cX$, it suffices to show that if $E \in \cB_I$ and $T \in B(\cV;\cW)$, then there exists a sequence $(A_n)_{n \in \N}$ in $\sA_I$ such that $1_{A_n}\,T \to 1_E\,T$ in $\cX$ as $n \to \infty$.
Since $\sA_I$ generates $\cB_I$ as a $\sigma$-algebra and $\nu_F(I) < \infty$, if $n \in \N$, then there exists an $A_n \in \sA_I$ such that $\nu_F(A_n \Delta E) < 1/n$.
The sequence $(A_n)_{n \in \N}$ does the trick.
\end{proof}

As an easy consequence of Theorem \ref{thm.LSint} and the definitions, Lebesgue--Stieltjes integral processes have locally bounded variation.

\begin{cor}\label{cor.LSintbdvar}
In the setting of Theorem \ref{thm.LSint}, if $H \in L_{\loc}^1(I,\nu_F;B(\cV;\cW))$ and
\[
G(t) \coloneqq \int_a^t H(s)[\d F(s)] \qquad (t \in I),
\]
then $G$ is right-continuous (continuous if $F$ is), and
\[
V(G : [s,t]) \leq \int_{(s,t]} \norm{H(r)}_{\cV \to \cW} \, \norm{\d F(r)}_{\cV} < \infty \qquad (s,t \in I, \; s < t).
\]
In particular, $G$ has locally bounded variation. \qed
\end{cor}

\subsection{Integration against \texorpdfstring{$L^2$}{}-decomposable processes}\label{sec.L2NCstochint}

We now work toward a definition of stochastic integrals against $L^2$-decomposable processes.
To this end, we introduce a class of ``elementary predictable processes'' and establish some basic properties of ``stochastic integrals'' of such processes.

\begin{defi}[Elementary predictable process]\label{def.EP}
For the duration of this and the following section, fix two filtered $\mathrm{C}^*$-probability spaces $(\cA,(\cA_t)_{t \geq 0},\E = \E_{\mathsmaller{\cA}})$ and $(\cB,(\cB_t)_{t \geq 0},\E_{\mathsmaller{\cB}})$.
Let $p,q \in [1,\infty\ra$ (Convention \ref{conv.cond}).
An \textbf{elementary predictable $\boldsymbol{(L^p;L^q)}$-process} is a map $H \colon \R_+ \to B_1^{p;q} = B(L^p(\E_{\mathsmaller{\cA}});L^q(\E_{\mathsmaller{\cB}}))$ such that
\begin{equation}
    H = 1_{\{0\}}H_0 + \sum_{i=1}^k 1_{(s_i,t_i]}H_i \label{eq.arbdecomp}
\end{equation}
for some times $t_i \geq s_i \geq 0$ and elements $H_0 \in \cF_0^{p;q}$, $H_i \in \cF_{s_i}^{p;q}$ ($i=1,\ldots,k$).
We write $\EP^{p;q} = \EP^{p;q}(\E_{\mathsmaller{\cA}};\E_{\mathsmaller{\cB}})$ for the set of elementary predictable $(L^p;L^q)$-processes.
If $H \colon \R_+ \to \mathbb{B} = \mathbb{B}(\cA;\cB)$ has a decomposition as in \eqref{eq.arbdecomp} such that $H_0 \in \cF_0$ and $H_i \in \cF_{s_i}$ for all $i=1,\ldots,k$, then we write $H \in \EP = \EP(\E_{\mathsmaller{\cA}};\E_{\mathsmaller{\cB}})$.
\end{defi}

\begin{obs}\label{obs.EP}
Let $p,q \in [1,\infty\ra$.
\begin{enumerate}[label=(\roman*),font=\normalfont]
    \item $\EP^{p;q}$ and $\EP$ are complex vector spaces, and $H \in \EP \Rightarrow H \in \EP^{p;p}$.\label{item.EPvecsp}
    \item If $H \in \EP^{p;q}$ and $0 \leq s \leq t$, then $1_{(s,t]}H \in \EP^{p;q}$ and $1_{\{0\} \cup (s,t]}H \in \EP^{p;q}$.
    If $H \in \EP$ and $0 \leq s \leq t$, then $1_{(s,t]}H \in \EP$ and $1_{\{0\} \cup (s,t]}H \in \EP$.\label{item.(s,t]}
    \item If $H \in \EP^{p;q}$, then $H \colon \R_+ \to B_1^{p;q}$ is an $\sA_{\R_+}$-simple function.
    If $H \in \EP$, then $H \colon \R_+ \to \mathbb{B}$ is an $\sA_{\R_+}$-simple function.\label{item.simple}
    \item If $H \in \EP^{p;q}$, then $H$ is adapted, is $\norm{\cdot}_{p;q}$--left-continuous, and has locally bounded variation with respect to $\norm{\cdot}_{p;q}$.
    If $H \in \EP$, then $H$ is adapted, is $\vertiii{\cdot}$--left-continuous, and has locally bounded variation with respect to $\vertiii{\cdot}$.\label{item.EPadapLC}
    \end{enumerate}
\end{obs}

By Observation \ref{obs.EP}\ref{item.simple}, the following definition makes sense.
\pagebreak

\begin{nota}[Elementary stochastic integral]\label{nota.EPint}
Let $p,q \in [1,\infty\ra$ and $X \colon \R_+ \to L^p(\E_{\mathsmaller{\cA}})$ be an arbitrary process.
If $H \in \EP^{p;q}$, then we define
\begin{align*}
    \int_0^{\infty} H(t)[\d X(t)] & = \int_0^{\infty} H[\d X] \coloneqq \int_{\R_+} H\big[\d\mu_X^0\big] \in L^q(\E_{\mathsmaller{\cB}}) \; \text{ and} \\
    \int_s^t H(r)[\d X(r)] & = \int_s^t H[\d X] \coloneqq \int_0^{\infty} (1_{(s,t]}H)[\d X] \qquad (0 \leq s \leq t).
\end{align*}
We set similar notation for arbitrary $X \colon \R_+ \to \cA$ and $H \in \EP$.
\end{nota}

\begin{lem}[Properties of elementary stochastic integral]\label{lem.elemint}
Fix $p,q \in [1,\infty\ra$ (resp., $p=q=\infty$), a process $X \colon \R_+ \to L^p(\E_{\mathsmaller{\cA}})$, and $H \in \EP^{p;q}$ (resp., $H \in \EP$).
\begin{enumerate}[label=(\roman*),font=\normalfont]
    \item If $X$ is adapted, then $\into H[\d X]$ is adapted.\label{item.adapEint}
    \item If $X = M$ is an $L^p$-martingale, then $\into H[\d M]$ is an $L^q$-martingale.\label{item.martEint}
    \item If $(\cC,(\cC_t)_{t \geq 0},\E_{\mathsmaller{\cC}})$ is another $\mathrm{C}^*$-probability space, $r \in [1,\infty\ra$ (resp., $r=\infty$), and $Y \coloneqq \into H[\d X]$, then
    \[
    \into K[\d Y] = \into KH[\d X]
    \]
    for all $K \in \mathrm{EP}^{q;r}(\E_{\mathsmaller{\cB}};\E_{\mathsmaller{\cC}})$ (resp., $K \in \mathrm{EP}(\E_{\mathsmaller{\cB}};\E_{\mathsmaller{\cC}})$).\label{item.subform}
\end{enumerate}
\end{lem}

\begin{proof}
The third item follows from \eqref{eq.algsubform}.
By linearity, it suffices to verify the claims of the first two items when $H = 1_{(u,v]}H_u$, where $0 \leq u < v$ and $H_u \in \cF_u^{p;q}$ (resp., $H_u \in \sF_u$).
(The case $H = 1_{\{0\}}H_0$ is obvious.)
In this case,
\begin{equation}
    \int_0^t H[\d X] = H_u[X(v \wedge t) - X(u \wedge t)] \qquad (t \geq 0). \label{eq.intform}
\end{equation}
By this formula, \ref{item.adapEint} follows from Observation \ref{obs.Ftpq}\ref{item.preservefilt}.

To prove \ref{item.martEint}, suppose $0 \leq s \leq t$, and write $\E_{\mathsmaller{\cA}}^s \coloneqq \E_{\mathsmaller{\cA}}[\,\cdot \mid \cA_s]$ and $\E_{\mathsmaller{\cB}}^s \coloneqq \E_{\mathsmaller{\cB}}[\,\cdot \mid \cB_s]$.
We consider two cases:
$s < u$ and $s \geq u$.
If $s < u$, then
\begin{align*}
    \E_{\mathsmaller{\cB}}^sH_u[M(v \wedge t) - M(u \wedge t)] & = \E_{\mathsmaller{\cB}}^s\E_{\mathsmaller{\cB}}^uH_u[M(v \wedge t) - M(u \wedge t)] \tag{Tower property} \\
    & = \E_{\mathsmaller{\cB}}^sH_u\E_{\mathsmaller{\cA}}^u[M(v \wedge t) - M(u \wedge t)] \tag{$H_u \in \cF_u^{p;q}$  (resp., $H_u \in \cF_u$)} \\
    & = \E_{\mathsmaller{\cB}}^sH_u[M(u \wedge t)-M(u \wedge t)] = 0 \tag{$M^t$ is a martingale} \\
    & = H_u[M(v \wedge s) - M(u \wedge s)]. \tag{$s < u < v$}
\end{align*}
Now, if $s \geq u$, then
\begin{align*}
    \E_{\mathsmaller{\cB}}^sH_u[M(v \wedge t) - M(u \wedge t)] & = H_u\E_{\mathsmaller{\cA}}^s[M(v \wedge t) - M(u \wedge t)] \tag{$H_u \in \cF_u^{p;q}$ (resp., $H_u \in \cF_u$)} \\
    & = H_u[M(v \wedge s) - M(u \wedge s)]. \tag{$M^v$ and $M^u$ are martingales}
\end{align*}
In either case, we conclude from \eqref{eq.intform} that $\E_{\mathsmaller{\cB}}^s\big[\int_0^t H[\d M]\big] = \int_0^sH[\d M]$.
\end{proof}

Next, we extend the definition of this ``elementary'' integral to a much larger space of integrands when $M$ is a (right-)continuous $L^2$-martingale by modifying the classical proof of the It\^{o} isometry.

\begin{lem}\label{lem.fakeDoleans}
If $M \colon \R_+ \to L^2(\E)$ is a right-continuous $L^2$-martingale, then there exists a unique measure $\kappa_M$ on $(\R_+,\cB_{\R_+})$ such that $\kappa_M(\{0\}) = 0$ and $\kappa_M((s,t]) = \norm{M(t)-M(s)}_2^2$ whenever $0 \leq s \leq t$.
\end{lem}

\begin{proof}
Define $F_M \colon \R_+ \to \R_+$ by $F_M(t) \coloneqq \norm{M(t)}_2^2$.
By the right-continuity of $M$, $F_M$ is right-continuous.
By \eqref{eq.norminc}, $F_M$ is non-decreasing.
Therefore, there exists a unique Borel measure $\kappa_M$ on $\R_+$ such that $\kappa_M(\{0\}) = 0$ and
\[
\kappa_M((s,t]) = F_M(t)-F_M(s) = \norm{M(t)}_2^2 - \norm{M(s)}_2^2 = \norm{M(t)-M(s)}_2^2 \qquad (0 \leq s \leq t)
\]
by Lemma \ref{lem.M(t)-M(s)}.
\end{proof}

\begin{thm}[Noncommutative It\^{o} contraction]\label{thm.Itocont}
If $M \colon \R_+ \to L^2(\E_{\mathsmaller{\cA}})$ is a right-continuous $L^2$-martingale and $H \in \EP^{2;2}$, then
\[
\bigg\|\int_0^{\infty} H[\d M] \bigg\|_2 \leq \Bigg(\int_{(0,\infty)} \norm{H}_{2;2}^2\,\d\kappa_M\Bigg)^{\frac{1}{2}}.
\]
\end{thm}

\begin{proof}
Let $H \in \EP^{2;2}$.
It is easy to see that there exist times $0 = t_0 < \cdots < t_k < \infty$ and elements $H_0 \in \cF_0^{2;2}$, $H_i \in \cF_{t_{i-1}}^{2;2}$ ($i=1,\ldots,k$) such that
\[
H = 1_{\{0\}}H_0 + \sum_{i=1}^k 1_{(t_{i-1},t_i]}H_i.
\]
Now, writing $\Delta_i M \coloneqq M(t_i) - M(t_{i-1})$,
\begin{align*}
    \bigg\|\int_0^{\infty} H[\d M]\bigg\|_2^2 & = \Bigg\la \sum_{i=1}^k H_i[\Delta_i M], \sum_{j=1}^k H_j[\Delta_j M]\Bigg\ra_2 \\
    & = \sum_{i=1}^k\norm{H_i[\Delta_i M]}_2^2 + \sum_{i \neq j} \ip{H_i[\Delta_i M], H_j[\Delta_j M]}_2.
\end{align*}
We claim that the second term above vanishes.
Of course, it suffices to show that if $1 \leq i < j \leq k$, then $\ip{H_i[\Delta_i M], H_j[\Delta_j M]}_2 = 0$.
To this end, note that $t_{i-1} < t_i \leq t_{j-1} < t_j$ in this case, which yields $H_i[\Delta_i M] \in L^2(\cB_{t_i},\E_{\mathsmaller{\cB}}) \subseteq L^2(\cB_{t_{j-1}},\E_{\mathsmaller{\cB}})$.
Therefore, by definition of $\cF_{t_{j-1}}^{2;2}$ and the martingale property of~$M$,
\begin{align*}
   \ip{H_i[\Delta_i M], H_j[\Delta_j M]}_2 & = \ip{H_i[\Delta_i M], \E_{\mathsmaller{\cB}}[H_j[\Delta_j M]\mid \cB_{t_{j-1}}]}_2 \\
    & = \ip{H_i[\Delta_i M], H_j\E_{\mathsmaller{\cA}}[\Delta_j M\mid \cA_{t_{j-1}}]}_2 = 0,
\end{align*}
as claimed.
We conclude that
\begin{align*}
    \bigg\|\int_0^{\infty} H[\d M]\bigg\|_2^2 & = \sum_{i=1}^k\norm{H_i[\Delta_i M]}_2^2 \leq \sum_{i=1}^k\norm{H_i}_{2;2}^2\norm{\Delta_i M}_2^2 \\
    & = \sum_{i=1}^k\norm{H_i}_{2;2}^2\kappa_M((t_{i-1},t_i]) = \int_{(0,\infty)}\norm{H}_{2;2}^2\,\d\kappa_M,
\end{align*}
as desired.
\end{proof}

\begin{rem}\label{rem.Dol}
Note that $\kappa_M$ corresponds to the time marginal of the Dol\'{e}ans measure from the classical case (Section \ref{sec.phil}).
Since this time marginal is all we are able to construct in the noncommutative case, we get the contraction in Theorem \ref{thm.Itocont} instead of the isometry from the classical case.
However, once we study noncommutative quadratic variation, we shall see that the It\^{o} isometry in the form of \eqref{eq.clII} does have a noncommutative analog (Corollary \ref{cor.NCII}).
\end{rem}

As we hinted earlier, Theorem \ref{thm.Itocont} allows us to extend the stochastic integral against $M$.
Actually, we shall use this development to extend the elementary stochastic integral against $L^2$-decomposable processes.

\begin{nota}\label{nota.kappa}
If $X = X(0) + M + A \colon \R_+ \to L^2(\E)$ is $L^2$-decomposable, then we define
\[
\kappa_X(\d t) \coloneqq \kappa_M(\d t) + \norm{\d A(t)}_2,
\]
where $\kappa_A(\d t) = \norm{\d A(t)}_2$ is from Lemma \ref{lem.elemintbd}, and $\kappa_M$ is from Lemma \ref{lem.fakeDoleans}.
\end{nota}

\begin{cor}\label{cor.L2bd}
If $X = X(0) + M + A  \colon \R_+ \to L^2(\E_{\mathsmaller{\cA}})$ is $L^2$-decomposable and $H \in \EP^{2;2}$, then
\[
\bigg\|\int_0^{\infty} H[\d X] \bigg\|_2 \leq \Bigg(\int_{(0,\infty)} \norm{H}_{2;2}^2\,\d\kappa_M\Bigg)^{\frac{1}{2}} + \int_{(0,\infty)} \norm{H}_{2;2} \, \d\kappa_A.
\]
\end{cor}

\begin{proof}
Combine Theorems \ref{thm.LSint} and \ref{thm.Itocont}.
\end{proof}

Note that each $H \in \EP^{2;2}$ is a compactly supported simple map $\R_+ \to B_1^{2;2}$.
In particular, if $X$ is $L^2$-decomposable, then we may consider the equivalence class of $H$ in $L_{\loc}^1(\R_+,\kappa_X;B_1^{2;2})$.
Also, note that $\kappa_M \ll \kappa_X$ and $\kappa_A \ll \kappa_X$.
Therefore, the following definition makes sense.
\pagebreak

\begin{defi}[Stochastically integrable processes]\label{def.integrands}
Let $X = X(0) + M + A \colon \R_+ \to L^2(\E_{\mathsmaller{\cA}})$ be an $L^2$-decomposable process.
Write $\cL_X^{\cB} = \cL_X$ for the set of $H \in L_{\loc}^1(\R_+,\kappa_X;B_1^{2;2})$ such that
\[
\norm{H}_{X,\cB,t} = \norm{H}_{X,t} \coloneqq \Bigg(\int_{(0,t]} \norm{H(s)}_{2;2}^2 \, \kappa_M(\d s)\Bigg)^{\frac{1}{2}} + \int_{(0,t]} \norm{H(s)}_{2;2} \, \kappa_A(\d s) < \infty \qquad (t \geq 0).
\]
The set $\cL_X$ is a complex-linear subspace of $L_{\loc}^1(\R_+,\kappa_X;B_1^{2;2})$.
We endow $\cL_X$ with the topology induced by the collection $\{\norm{\cdot}_{X,t} : t \geq 0\}$ of seminorms, which makes it into a complex Fr\'{e}chet space.
Also, write $\mathcal{EP}^{2;2} \subseteq \cL_X$ (resp., $\mathcal{EP}$) for the set of $\kappa_X$-a.e.\ equivalence classes of members of $\mathrm{EP}^{2;2}$ (resp., $\EP$).
Finally, define $\cI^{\cB}(X) = \cI(X)$ (resp., $\tilde{\cI}^{\cB}(X) = \tilde{\cI}(X)$) to be the closure of $\mathcal{EP}^{2;2}$ (resp., $\mathcal{EP}$) in $\cL_X$.
The members of $\cI(X)$ are called \textbf{stochastically $\boldsymbol{X}$-integrable processes}. 
\end{defi}

Note that if $H,K \in \EP^{2;2}$ and $H=K$ $\kappa_X$-a.e., then $H=K$ $\kappa_M$-a.e.\ and $\kappa_A$-a.e.
Therefore, by Corollary \ref{cor.L2bd}, $\into H[\d X] = \into K[\d X]$.
In particular, stochastic integration against $X$ is well defined as a complex-linear map $\mathcal{EP}^{2;2} \to C_a(\R_+; L^2(\E_{\mathsmaller{\cB}}))$.
Also, observe that if $H \in \cI(X)$ and $0 \leq s \leq t$, then $1_{(s,t]}H \in \cI(X)$.

\begin{thm}[Extension of stochastic integral]\label{thm.stochint}
Let $X = X(0) + M + A \colon \R_+ \to L^2(\E_{\mathsmaller{\cA}})$ be an $L^2$-decomposable process.
\begin{enumerate}[label=(\roman*),font=\normalfont]
    \item The stochastic integral map $\mathcal{EP}^{2;2} \ni H \mapsto \into H [\d X] \in C_a(\R_+;L^2(\E_{\mathsmaller{\cB}}))$ extends uniquely to a continuous linear map $I_X^{\cB} = I_X \colon \cI(X) \to C_a(\R_+;L^2(\E_{\mathsmaller{\cB}}))$.
    Moreover, if $H \in \cI(X)$, then\label{item.stochintcont}
    \[
    \norm{I_X(H)(t)}_2 \leq \norm{H}_{X,t} \qquad (t \geq 0).
    \]
    \item If $0 \leq s \leq t$ and $H \in \cI(X)$, then\label{item.stochintinc}
    \[
    I_X(H)(t) - I_X(H)(s) = I_X(1_{(s,t]}H)(t) = I_X(1_{\{0\} \cup (s,t]}H)(t).
    \]
    \item If $H \in \cI(X)$, then $H \in \cI(M)$, $H \in \cI(A) \subseteq L_{\loc}^1(\R_+,\nu_A;B_1^{2;2})$, $I_M(H) \in \M_{\mathsmaller{\cB}}^2$, and
    \[
    I_X(H) = I_M(H) + \into H(t)[\d A(t)].
    \]
    (To be clear, $\into H(t)[\d A(t)]$ is a Lebesgue--Stieltjes integral.)
    In other words, $I_X(H)$ is $L^2$-decomposable with $I_X(H)^{\mathrm{m}} = I_M(H)$ and $I_X(H)^{\mathrm{fv}} = \into H(t)[\d A(t)] = I_A(H)$ (Corollary \ref{cor.uniquedecomp}).\label{item.stochintsemimart}
\end{enumerate}
Henceforth, if $H \in \cI(X)$, then we shall write
\begin{align*}
    \into H[\d X] & = \into H(t)[\d X(t)] \coloneqq I_X(H) \; \text{ and} \\
    \int_s^t H[\d X] & = \int_s^t H(r)[\d X(r)] \coloneqq \int_0^t (1_{(s,t]} H)[\d X] \qquad (0 \leq s \leq t).
\end{align*}
\end{thm}

\begin{proof}
This follows straightforwardly from Lemma \ref{lem.elemint}, Corollary \ref{cor.L2bd}, Theorem \ref{thm.LSint}, and Corollary \ref{cor.LSintbdvar}.
\end{proof}

\begin{rem}\label{rem.RC1}
A similar result holds if we only assume that $M$ and $A$ are right-continuous.
In this case, however, the seminorms and spaces in Definition \ref{def.integrands} depend on the decomposition $X = X(0) + M + A$, not just on $X$ itself.
While it would certainly be worthwhile to study noncommutative processes that are only right-continuous, all the examples of present interest to us are continuous.
We therefore restrict our general development to the continuous case.
\end{rem}

We end this section by showing why we introduced the space $\tilde{\cI}(X)$.

\begin{defi}\label{def.tildesemi}
Let $p \geq 2$.
An $L^p$-decomposable process $X \colon \R_+ \to L^p(\E)$ is called \textbf{$\boldsymbol{\tilde{L}^p}$-decomposable} if $X^{\mathrm{m}} \in \tbM^p$, i.e., the martingale part of $X$ is locally uniformly $L^p$-approximable by $L^{\infty}$-continuous martingales.
\end{defi}

\begin{prop}\label{prop.stochintLtilde}
If $X$ is $\tilde{L}^2$-decomposable and $H \in \tilde{\cI}(X)$, then $\into H[\d X]$ is $\tilde{L}^2$-decomposable.
\end{prop}

\begin{proof}
Let $M \coloneqq X^{\mathrm{m}}$.
By Theorem \ref{thm.stochint}\ref{item.stochintsemimart}, $\into H[\d M]$ is the martingale part of $\into H[\d X]$, so the goal is to show that $\into H[\d M] \in \tbM_{\mathsmaller{\cB}}^2$ whenever $H \in \tilde{\cI}(X)$.
First, suppose $H \in \EP$, and fix a sequence $(M_n)_{n \in \N}$ in $\M_{\mathsmaller{\cA}}^{\infty}$ converging in $\M_{\mathsmaller{\cA}}^2$ to $M$.
If $n \in \N$, then $\into H[\d M_n] \in \M_{\mathsmaller{\cB}}^{\infty}$ by Lemma \ref{lem.elemint}.
Also, if $H$ is decomposed as in \eqref{eq.arbdecomp}, then
\[
\into H[\d M_n] = \sum_{i=1}^k H_i[M_n^{t_i} - M_n^{s_i}] \; \text{ and } \; \into H[\d M] = \sum_{i=1}^k H_i[M^{t_i} - M^{s_i}],
\]
which makes clear that $\into H[\d M_n] \to \into H[\d M]$ in $\M_{\mathsmaller{\cB}}^2$ as $n \to \infty$.
In particular, $\into H[\d M] \in \tbM_{\mathsmaller{\cB}}^2$.
For general $H \in \tilde{\cI}(X)$, let $(H_n)_{n \in \N}$ be a sequence in $\mathcal{EP}$ converging to $H$ in $\cL_X$.
By construction of the stochastic integral, $\into H_n[\d M] \to \into H[\d M]$ in $\M_{\mathsmaller{\cB}}^2$ as $n \to \infty$.
Since we already know that $\into H_n[\d M] \in \tbM_{\mathsmaller{\cB}}^2$ for all $n \in \N$, we conclude that $\into H[\d M] \in \tbM_{\mathsmaller{\cB}}^2$.
\end{proof}

\subsection{Tools to calculate stochastic integrals}\label{sec.RS1}

In this section, we prove two additional facts that help identify or calculate certain stochastic integrals in practice:
a substitution formula (Theorem \ref{thm.subform}) and an expression for certain stochastic integrals as limits of left-endpoint Riemann--Stieltjes sums (Proposition \ref{prop.LERSapprox}).
To begin, note that if $X$ is $L^2$-decomposable, $H \in \cI(X)$, and $Y_H \coloneqq \into H[\d X]$, then $Y_H$ is $L^2$-decomposable by Theorem \ref{thm.stochint}\ref{item.stochintsemimart}.
In particular, we can consider integrals $\into K[\d Y_H]$.
The following result says that $\into K[\d Y_H] = \into KH[\d X]$ frequently holds, as ``$\d Y_H = H[\d X]$'' suggests.

\begin{thm}[Substitution formula]\label{thm.subform}
Fix an $L^2$-decomposable process $X = X(0) + M + A \colon \R_+ \to L^2(\E_{\mathsmaller{\cA}})$, another $\mathrm{C}^*$-probability space $(\cC,(\cC_t)_{t \geq 0},\E_{\mathsmaller{\cC}})$, a stochastically $X$-integrable process $H \in \cI(X)$, and a strongly measurable map $K \colon \R_+ \to B(L^2(\E_{\mathsmaller{\cB}});L^2(\E_{\mathsmaller{\cC}}))$.
Also, write $Y_H \coloneqq \into H[\d X]$.
If there exists a sequence $(K_n)_{n \in \N}$ in $\mathrm{EP}^{2;2}(\E_{\mathsmaller{\cB}};\E_{\mathsmaller{\cC}})$ such that
\[
\int_{(0,t]} \norm{K_n-K}_{2;2}^2\norm{H}_{2;2}^2\,\d\kappa_M + \int_{(0,t]} \norm{K_n-K}_{2;2}\norm{H}_{2;2}\,\d\kappa_A \xrightarrow{n \to \infty} 0 \qquad (t \geq 0),
\]
then $K \in \cI^{\cC}(Y_H)$, $KH \in \cI^{\cC}(X)$, and
\[
\into K[\d Y_H] = \into KH[\d X].
\]
\end{thm}

\begin{proof}
We begin with an observation.
If $H \in \cI(X)$ and $0 \leq s \leq t$, then
\[
\kappa_{I_M(H)}((s,t])= \norm{I_M(H)(t) - I_M(H)(s)}_2^2 = \bigg\|\int_s^t H[\d M]\bigg\|_2^2 \leq \int_{(s,t]} \norm{H}_{2;2}^2 \,\d\kappa_M 
\]
by Theorem \ref{thm.stochint}\ref{item.stochintcont}--\ref{item.stochintinc}.
Also, by Corollary \ref{cor.LSintbdvar},
\[
\kappa_{I_A(H)}((s,t]) = V_{L^2(\E_{\mathsmaller{\cB}})}(I_A(H) : [s,t]) \leq \int_{(s,t]} \norm{H}_{2;2}\,\d\kappa_A.
\]
It follows from the monotone class theorem that $\kappa_{I_M(H)}(E) \leq \int_E \norm{H}_{2;2}^2 \,\d\kappa_M$ and $\kappa_{I_A(H)}(E) \leq \int_E \norm{H}_{2;2} \,\d\kappa_A$ for all $E \in \cB_{\R_+}$.

Now, suppose $H \in \mathrm{EP}^{2;2}(\E_{\mathsmaller{\cA}};\E_{\mathsmaller{\cB}})$ and $L \in \mathrm{EP}^{2;2}(\E_{\mathsmaller{\cB}};\E_{\mathsmaller{\cC}})$.
Also, decompose $L$ in the usual way as $L = 1_{\{0\}} L_0 + \sum_{i=1}^k 1_{(s_i,t_i]}L_i$.
By definition and Lemma \ref{lem.elemint}\ref{item.subform},
\begin{equation}
    \sum_{i=1}^kL_i[Y_H(t_i \wedge t) - Y_H(s_i \wedge t)] = \int_0^t L[\d Y_H] = \int_0^t LH[\d X] \qquad (t \geq 0). \label{eq.subform}
\end{equation}
By an elementary limiting argument using Theorem \ref{thm.stochint}\ref{item.stochintcont} (separately for the left-hand side and the right-hand side), \eqref{eq.subform} extends to all $H \in \cI^{\cB}(X)$.

Finally, fix $H \in \cI^{\cB}(X)$, and let $K$ be as in the statement.
First, note that the hypothesis implies that if $t \geq 0$, then $\norm{K_nH - KH}_{X,\cC,t} \to 0$ as $n \to \infty$.
In particular, $KH \in \cI^{\cC}(X)$, and $I_X^{\cC}(K_nH) \to I_X^{\cC}(KH)$ in $C_a(\R_+;L^2(\E_{\mathsmaller{\cC}}))$ as $n \to \infty$.
Next, we use the observation in the first paragraph to see that if $t \geq 0$, then
\[
\norm{K_n - K}_{Y_H,\cC,t} \leq \Bigg(\int_{(0,t]} \norm{K_n-K}_{2;2}^2\norm{H}_{2;2}^2\,\d\kappa_M\Bigg)^{\frac{1}{2}} + \int_{(0,t]} \norm{K_n-K}_{2;2}\norm{H}_{2;2}\,\d\kappa_A \xrightarrow{n \to \infty} 0.
\]
In particular, $K \in \cI^{\cC}(Y_H)$, and $I_{Y_H}^{\cC}(K_n) \to I_{Y_H}^{\cC}(K)$ in $C_a(\R_+;L^2(\E_{\mathsmaller{\cC}}))$ as $n \to \infty$.
Since we already know that $I_{Y_H}^{\cC}(K_n) = I_X^{\cC}(K_nH)$ for all $n \in \N$, this completes the proof.
\end{proof}

\begin{ex}[LCLB and LLLB]\label{ex.LLLB}
If $K \colon \R_+ \to B(L^2(\E_{\mathsmaller{\cB}});L^2(\E_{\mathsmaller{\cC}}))$ is adapted and \textbf{LCLB}, i.e., $K$ is \textbf{l}eft-\textbf{c}ontinuous and \textbf{l}ocally \textbf{b}ounded, then $K$ satisfies the hypotheses of Theorem \ref{thm.subform}.
More generally, we claim that if $K$ is \textbf{LLLB}, i.e., $K$ has \textbf{l}eft \textbf{l}imits and is \textbf{l}ocally \textbf{b}ounded, then $K_-$ satisfies the hypotheses of Theorem \ref{thm.subform}.
(Recall from Notation \ref{nota.nota}\ref{item.l/rlim} that $K_-$ is the left limit function $K_-(t) \coloneqq K(t-)$.)
Indeed, fix a sequence $(\Pi_n)_{n \in \N}$ of partitions of $\R_+$ such that $|\Pi_n| \to 0$ as $n \to \infty$, and define
\[
K_n \coloneqq 1_{[0,n]}K^{\Pi_n} = \sum_{s \in \Pi_n} 1_{(s_- \wedge n,s \wedge n]} K(s_-) \qquad (n \in \N).
\]
Since $K$ is adapted, $K_n \in \mathrm{EP}^{2;2}(\E_{\mathsmaller{\cB}};\E_{\mathsmaller{\cC}})$.
Now, for $t \geq 0$, let $\rho$ be a finite Borel measure on $[0,t]$.
Since $K$ is LLLB, Lemma \ref{lem.HPi}\ref{item.HPipw} and the dominated convergence theorem imply that if $p \in [1,\infty)$, then $K_n \to K_-$ in $L^p([0,t],\rho;B(L^2(\E_{\mathsmaller{\cB}});L^2(\E_{\mathsmaller{\cC}})))$ as $n \to \infty$.
Applying this with the measures $\rho(\d s) = \norm{H(s)}_{2;2}^2\,\kappa_M(\d s)$ and $\rho(\d s) = \norm{H(s)}_{2;2} \, \kappa_A(\d s)$, we conclude that
\[
\int_{(0,t]} \norm{K_n-K_-}_{2;2}^2\norm{H}_{2;2}^2\,\d\kappa_M + \int_{(0,t]} \norm{K_n-K_-}_{2;2}\norm{H}_{2;2}\,\d\kappa_A \xrightarrow{n \to \infty} 0,
\]
as claimed.
\end{ex}

Next comes our result on left-endpoint Riemann--Stieltjes approximations of stochastic integrals.

\begin{nota}\label{nota.bLplim}
We shall write $\mathbb{L}^p\text{-}\lim$ to indicate that a given limit of functions $\R_+ \to L^p$ is uniform on compact subsets of $\R_+$.
\end{nota}

\begin{prop}\label{prop.LERSapprox}
Let $X \colon \R_+ \to L^2(\E_{\mathsmaller{\cA}})$ be an $L^2$-decomposable process.
If $H \colon \R_+ \to B_1^{2;2}$ is adapted and LLLB, then $H_- \in \cI(X)$, and
\[
\into H_-[\d X] = \mathbb{L}^2\text{-}\lim_{\Pi \in \cP_{\R_+}}\sum_{t \in \Pi} H(t_-)[X(t \wedge \cdot) - X(t_- \wedge \cdot)]. \numberthis\label{eq.lRSform}
\]
If, in addition, $H(t) \in \cF_t$ for all $t \geq 0$, then $H_- \in \tilde{\cI}(X)$.
\end{prop}

\begin{proof}
Let $H \colon \R_+ \to B_1^{2;2}$ be adapted and LLLB.
If $\Pi$ is a partition of $\R_+$ and $t \geq 0$, then
\begin{align*}
    1_{[0,t]}H^{\Pi} & = 1_{\{0\}} H(0) + \sum_{s \in \Pi} 1_{(s_- \wedge t,s \wedge t]} H(s_-) \\
    & = 1_{\{0\}}H(0) + \sum_{s \in \Pi : s_- < t} 1_{(s_-,s \wedge t]} H(s_-) \in \mathcal{EP}^{2;2} \subseteq \cI(X) \numberthis\label{eq.HPit}
\end{align*}
because $H$ is adapted.
(If $H(t) \in \cF_t$ for all $t \geq 0$, then $1_{[0,t]}H^{\Pi} \in \mathcal{EP}$.)
Taking $t \to \infty$, we conclude that $H^{\Pi} \in \cI(X)$ (resp., $H^{\Pi} \in \tilde{\cI}(X)$).

Next, we show that $H^{\Pi} \to H_-$ in $\cL_X$ as $|\Pi| \to 0$, from which it follows that $H_- \in \cI(X)$ (resp., $H_- \in \tilde{\cI}(X)$) and, by continuity of the stochastic integral map, that
\[
\into H_-[\d X] = \mathbb{L}^2\text{-}\lim_{\Pi \in \cP_{\R_+}}\into H^{\Pi}[\d X].
\]
To this end, note that if $t \geq 0$ and $\pi \coloneqq (\Pi \cap [0,t]) \cup \{t\} \in \cP_{[0,t]}$, then
\[
H^{\Pi}|_{[0,t]} = (H|_{[0,t]})^{\pi}
\]
by inspecting \eqref{eq.HPit}.
Therefore, $\norm{H^{\Pi} - H_-}_{X,t} \to 0$ as $|\Pi| \to 0$ by two applications of Lemma \ref{lem.HPi}\ref{item.HPipw} and the dominated convergence theorem.
Since $t \geq 0$ was arbitrary, we conclude that $H^{\Pi} \to H_-$ in $\cL_X$ as $|\Pi| \to 0$, as claimed.

Finally, we compute $\into H^{\Pi}[\d X]$.
If $t \geq 0$, then the first paragraph and the definition of integrals of elementary predictable processes yield
\[
\int_0^t H^{\Pi}[\d X] = \sum_{s \in \Pi : s_- < t} H(s_-)[X(s \wedge t) - X(s_-)] = \sum_{s \in \Pi} H(s_-)[X(s \wedge t) - X(s_- \wedge t)].
\]
This completes the proof.
\end{proof}

In particular, if $H \colon \R_+ \to \mathbb{B}$ is a $\norm{\cdot}_{2;2}$-LCLB trace biprocess, then $H \in \tilde{\cI}(X)$, and we can compute its stochastic integral against $X$ as a limit of left-endpoint Riemann--Stieltjes sums.
Here is a very common example of this kind.

\begin{ex}\label{ex.elemtrbipinteg}
In this example, we assume $(\cA,(\cA_t)_{t \geq 0},\E=\E_{\mathsmaller{\cA}}) = (\cB,(\cB_t)_{t \geq 0},\E_{\mathsmaller{\cB}})$.
By Example \ref{ex.trpoly}, if $A_1,\ldots,A_8 \colon \R_+ \to \cA$ are adapted and $L^{\infty}$-LCLB (resp., $L^{\infty}$-continuous), then the process
\[
\R_+ \ni t \mapsto H(t) \coloneqq (x \mapsto A_1(t)xA_2(t) + A_3(t)x^*A_4(t) + \E[A_5(t)x]\,A_6(t) + \E[A_7(t)x^*]\,A_8(t)) \in \mathbb{B}(\cA)
\]
is adapted and $\vertiii{\cdot}$-LCLB (resp., $\vertiii{\cdot}$-continuous).
By Proposition \ref{prop.LERSapprox}, if $X \colon \R_+ \to L^2(\E)$ is an $L^2$-decomposable process, then $H$ is stochastically $X$-integrable; and if $t \geq 0$, then
\begin{align*}
    \int_0^t H(s)[\d X(s)] & =\vcentcolon \int_0^t A_1(s)\,\d X(s)\,A_2(s) + \int_0^t A_3(s)\,\d X^*(s)\,A_4(s) \\
    & \hspace{10mm} + \int_0^t\E[A_5(s)\,\d X(s)]\,A_6(s) + \int_0^t\E[A_7(s)\,\d X^*(s)]\,A_8(s) \\
    & = L^2\text{-}\lim_{\pi \in \cP_{[0,t]}} \sum_{s \in \pi} \big(A_1(s_-)\,\Delta_sX \,A_2(s_-) + A_3(s_-)\,\Delta_sX \,A_4(s_-) \\
    & \hspace{30mm} + \E[A_5(s_-)\,\Delta_sX]\,A_6(s_-) + \E[A_7(s_-)\,\Delta_s X^*]\,A_8(s_-)\big).
\end{align*}
Interestingly, if $X = M \in \M^2$, then the last two terms vanish.
Indeed, if $\e \in \{1,\ast\}$, then
\begin{align*}
    \int_0^t \E[A_1(s)\,\d M^{\e}(s)]\,A_2(s) & = L^2\text{-}\lim_{\pi \in \cP_{[0,t]}}\sum_{s \in \pi}\E[A_1(s_-) \,\Delta_sM^{\e}] \,A_2(s_-) \\
    & = L^2\text{-}\lim_{\pi \in \cP_{[0,t]}}\sum_{s \in \pi}\E[\E[A_1(s_-) \,\Delta_sM^{\e} \mid \cA_{s_-} ]]\,A_2(s_-) \\
    & = L^2\text{-}\lim_{\pi \in \cP_{[0,t]}}\sum_{s \in \pi}\E[A_1(s_-) \,\E[\Delta_sM \mid \cA_{s_-}]^{\e}] \,A_2(s_-) = 0
\end{align*}
because $A_1$ is adapted and $M$ is a martingale.

More generally, if $n,m,d \in \N$, $P \in (\TrP_{n,1,d}^*)^m$, and $\mathbf{X} \colon \R_+ \to \cA^n$ is adapted and $L^{\infty}$-LCLB (resp., $L^{\infty}$-continuous), then the linear process $\R_+ \ni t \mapsto P(\mathbf{X}(t)) \in \mathbb{B}(\cA^d;\cA^m)$ is a $\vertiii{\cdot}$-LCLB (resp., $\vertiii{\cdot}$-continuous) multivariate trace biprocess.
In this case, if $\mathbf{Y} \colon \R_+ \to L^2(\E^{\oplus d})$ is $L^2$-decomposable and $t \geq 0$, then
\[
\int_0^t P(\mathbf{X}(s),\d\mathbf{Y}(s)) = L^2\text{-}\lim_{\pi \in \cP_{[0,t]}} \sum_{s \in \pi} P(\mathbf{X}(s_-),\Delta_s\mathbf{Y}) \in L^2(\E^{\oplus m}) = L^2(\E)^{\oplus m}
\]
by Proposition \ref{prop.LERSapprox}.
\end{ex}

\section{Quadratic covariation}\label{sec.QC}

For the duration of this section, fix three filtered $\mathrm{C}^*$-probability spaces $(\cA,(\cA_t)_{t \geq 0},\E = \E_{\mathsmaller{\cA}})$, $(\cB,(\cB_t)_{t \geq 0},\E_{\mathsmaller{\cB}})$, and $(\cC,(\cC_t)_{t \geq 0},\E_{\mathsmaller{\cC}})$.
\pagebreak

\subsection{It\^{o} product rule and reduction to martingales}\label{sec.NCIPR}

In this section, we prove a kind of noncommutative It\^{o} product rule (Theorem \ref{thm.NCIPR}) and reduce the task of constructing quadratic covariation integrals $\int_0^t \Lambda(s)[\d X(s), \d Y(s)]$ for pairs $(X,Y)$ of decomposable processes to the task of constructing these integrals when $X$ and $Y$ are martingales.
To begin, we set notation for quadratic Riemann--Stieltjes sums.

\begin{nota}[Quadratic Riemann--Stieltjes sums]\label{nota.qRS}
Fix $p,q,r \in [1,\infty]$ and processes $X \colon \R_+ \to L^p(\E_{\mathsmaller{\cA}})$ and $Y \colon \R_+ \to L^q(\E_{\mathsmaller{\cB}})$.
If $\Pi$ is a partition of $\R_+$ and $\Lambda \colon \R_+ \to B_2^{p,q;r} = B_2(L^p(\E_{\mathsmaller{\cA}}) \times L^q(\E_{\mathsmaller{\cB}});L^r(\E_{\mathsmaller{\cC}}))$ is a bilinear process, then we write
\[
\mathrm{RS}_{\Pi}^{X,Y}(\Lambda)(t) \coloneqq \sum_{s \in \Pi}\Lambda(s_-)\big[\Delta_sX^t, \Delta_sY^t\big] \in L^r(\E_{\mathsmaller{\cC}}) \qquad (t \geq 0).
\]
Recall that $X^t = X(\cdot \wedge t)$ and $Y^t = Y(\cdot \wedge t)$.
\end{nota}

Loosely speaking, $\int_0^t \Lambda(s)[\d X(s), \d Y(s)]$ will be defined as $\lim_{\Pi \in \cP_{\R_+}}\mathrm{RS}_{\Pi}^{X,Y}(\Lambda)(t)$.
Essential to making this work is the following result, which should be viewed as a generalization of the ``free It\^{o} product rule'' (\cite[Thm.\ 4.2.1]{BS1998} or \cite[Thm.\ 3.2.5]{NikitopoulosIto}).
In the theorem below, is $\int_0^t \d\Lambda(s)[X(s),Y(s)]$ the Riemann--Stieltjes integral of the integrand $t \mapsto (T \mapsto T[X(t),Y(t)])$ against the integrator $t \mapsto \Lambda(t)$.

\begin{nota}\label{nota.Q}
Write $\mathrm{Q}_0 = \mathrm{Q}_0(\cA \times \cB;\cC)$ for the set of adapted, bilinear processes $\Lambda \colon \R_+ \to \mathbb{B}_2 = \mathbb{B}_2(\cA \times \cB; \cC)$ that are left-continuous and have locally bounded variation with respect to $\vertiii{\cdot}_2$ (Notation \ref{nota.bddlin}).
Note that such processes are $\vertiii{\cdot}_2$-LCLB.
\end{nota}

\begin{thm}[Noncommutative It\^{o} product rule]\label{thm.NCIPR}
If $X \colon \R_+ \to \cA$ and $Y \colon \R_+ \to \cB$ are $L^{\infty}$-decomposable processes and $\Lambda \in \mathrm{Q}_0$, then
\begin{equation}
\begin{split}
    \mathbb{L}^2\text{-}\lim_{\Pi \in \cP_{\R_+}} \mathrm{RS}_{\Pi}^{X,Y}(\Lambda) = \Lambda[X&, Y] - \Lambda(0)[X(0), Y(0)] - \into \d\Lambda(t)[X(t), Y(t)] \\
    &  - \into \Lambda(t)[\d X(t), Y(t)] - \into \Lambda(t)[X(t), \d Y(t)].
\end{split}\label{eq.QC1}
\end{equation}
Review Notation \ref{nota.bLplim} for the meaning of $\mathbb{L}^2$ above.
\end{thm}

\begin{proof}
If $t \geq 0$ and $\pi$ is a partition of $[0,t]$, then
\begin{align*}
    \delta_t & \coloneqq \Lambda(t)[X(t), Y(t)] - \Lambda(0)[X(0), Y(0)] = \sum_{s \in \pi}\big(\Lambda(s)[X(s), Y(s)] - \Lambda(s_-)[X(s_-), Y(s_-)]\big) \\
    & = \sum_{s \in \pi}\big((\Lambda(s_-) + \Delta_s\Lambda)[X(s_-) + \Delta_sX, Y(s_-) + \Delta_sY] - \Lambda(s_-)[X(s_-), Y(s_-)]\big) \\
    & = \sum_{s \in \pi}\big(\Lambda(s_-)[\Delta_sX, Y(s_-)] + \Lambda(s_-)[X(s_-), \Delta_sY] + \Lambda(s_-)[\Delta_sX, \Delta_sY] \\
    & \hspace{13mm} + \Delta_s\Lambda[X(s_-), Y(s_-)] + \Delta_s\Lambda[\Delta_sX, Y(s_-)] + \Delta_s\Lambda[X(s_-), \Delta_sY] + \Delta_s\Lambda[\Delta_sX,\Delta_sY]\big).
\end{align*}
Now, let $\Pi$ be a partition of $\R_+$.
Applying the above with $\pi = (\Pi \cap [0,t]) \cup \{t\}$ yields
\begin{align*}
    \mathrm{RS}_{\Pi}^{X,Y}(\Lambda)(t) & = \Lambda(t)[X(t), Y(t)] - \Lambda(0)[X(0), Y(0)] - \sum_{s \in \Pi}\Delta_s\Lambda^t[X(s_-), Y(s_-)] \\
    & \hspace{12mm} - \sum_{s \in \Pi} \Lambda(s_-)\big[\Delta_sX^t, Y(s_-)\big] - \sum_{s \in \Pi} \Lambda(s_-)\big[X(s_-), \Delta_sY^t\big] \\
    & \hspace{12mm} - \underbrace{\sum_{s \in \Pi} \big(\Delta_s\Lambda^t\big[\Delta_sX^t,\Delta_sY^t\big] + \Delta_s\Lambda^t\big[\Delta_sX^t, Y(s_-)\big] + \Delta_s\Lambda^t\big[X(s_-), \Delta_sY^t\big]\big).}_{=\vcentcolon \,\e_{_{\Pi}}(t)}
\end{align*}
Let us now investigate each term above.

Fix $p,q,r \in [1,\infty]$ with $1/p+1/q \leq 1/r$.
First, note that if $t \geq 0$, then
\begin{align*}
    \sup_{0 \leq s \leq t}\norm{\e_{\Pi}(s)}_r & \leq \bigg( \sup_{u,v \leq t : |u-v| \leq |\Pi|}\norm{X(u) - X(v)}_p \sup_{u,v \leq t : |u-v| \leq |\Pi|}\norm{Y(u) - Y(v)}_q \\
    & \hspace{15mm} + \sup_{u,v \leq t : |u-v| \leq |\Pi|}\norm{X(u) - X(v)}_p \sup_{0 \leq s \leq t}\norm{Y(s)}_q \\
    & \hspace{15mm} + \sup_{0 \leq s \leq t}\norm{X(s)}_p \sup_{u,v \leq t : |u-v| \leq |\Pi|}\norm{Y(u) - Y(v)}_q\bigg)\, V_{B_2^{p,q;r}}(\Lambda : [0,t]) \xrightarrow[\Pi \in \cP_{\R_+}]{|\Pi| \to 0} 0
\end{align*}
because $X$ is $L^p$-continuous, $Y$ is $L^q$-continuous, and $\Lambda$ has locally bounded variation with respect to $\norm{\cdot}_{p,q;r} \leq \vertiii{\cdot}_2$.
Second, for the same reasons, Proposition \ref{prop.contRSint} says that
\[
\mathbb{L}^r\text{-}\lim_{\Pi \in \cP_{\R_+}}\sum_{t \in \Pi}(\Lambda(t \wedge \cdot) - \Lambda(t_- \wedge \cdot))[X(t_-), Y(t_-)] = \into \d\Lambda(t)[X(t), Y(t)].
\]
Finally, $\Lambda \colon \R_+ \to \mathbb{B}_2$ is adapted and $\norm{\cdot}_{2,\infty;2} \leq \vertiii{\cdot}_2$-LCLB.
In addition, $Y$ is adapted and $L^{\infty}$-continuous.
Thus, the linear process
\[
\R_+ \ni t \mapsto H(t) \coloneqq \Lambda(t)[\cdot, Y(t)] \in B(L^2(\E_{\mathsmaller{\cA}});L^2(\E_{\mathsmaller{\cC}}))
\]
is adapted and $\norm{\cdot}_{2;2}$-LCLB.
We then get from Proposition \ref{prop.LERSapprox} that $H \in \cI(X)$ and
\[
\mathbb{L}^2\text{-}\lim_{\Pi \in \cP_{\R_+}}\sum_{t \in \Pi}\Lambda(t_-)[X(t \wedge \cdot) - X(t_- \wedge \cdot), Y(t_-)] = \into H(t)[\d X(t)] = \into \Lambda(t)[\d X(t),Y(t)].
\]
Similarly, the linear process $\R_+ \ni t \mapsto \Lambda(t)[X(t),\cdot] \in B(L^2(\E_{\mathsmaller{\cB}});L^2(\E_{\mathsmaller{\cC}}))$ belongs to $\cI(Y)$, and
\[
\mathbb{L}^2\text{-}\lim_{\Pi \in \cP_{\R_+}}\sum_{t \in \Pi}\Lambda(t_-)[X(t_-), Y(t \wedge \cdot) - Y(t_- \wedge \cdot)] = \into \Lambda(t)[X(t),\d Y(t)].
\]
Putting it all together, we obtain \eqref{eq.QC1}.
\end{proof}

The most important takeaway from the result above is that $\mathrm{RS}_{\Pi}^{X,Y}(\Lambda)$ does, in fact, have a limit as $|\Pi| \to 0$ when $X$, $Y$, and $\Lambda$ are sufficiently nice.
However, the assumptions on $\Lambda$ in this noncommutative It\^{o} product rule are too strong for most applications.
Namely, in order to prove It\^{o}'s formula, we need to weaken the bounded variation assumption on $\Lambda$ substantially.
To begin this process, we show that we only need to treat the case when the decomposable processes of interest are martingales.

\begin{prop}\label{prop.QCofFV}
Fix $p,q,r \in [1,\infty]$ and processes $X \colon \R_+ \to L^p(\E_{\mathsmaller{\cA}})$ and $Y \colon \R_+ \to L^q(\E_{\mathsmaller{\cB}})$.
Suppose $\Lambda \colon \R_+ \to B_2^{p,q;r}$ is locally $\norm{\cdot}_{p,q;r}$-bounded.
If 
\begin{enumerate}[label=(\roman*),font=\normalfont]
    \item $X$ is $L^p$-continuous and $Y$ is $L^q$-FV, or
    \item $X$ is $L^p$-FV and $Y$ is $L^q$-continuous,
\end{enumerate}
then
\[
\mathbb{L}^r\text{-}\lim_{\Pi^* \in \cP_{\R_+}^*}\sum_{t \in \Pi} \Lambda(t_*)[X(t \wedge \cdot) - X(t_- \wedge \cdot), Y(t \wedge \cdot) - Y(t_- \wedge \cdot)] = 0.
\]
In particular, $\mathbb{L}^r\text{-}\lim_{\Pi \in \cP_{\R_+}}\mathrm{RS}_{\Pi}^{X,Y}(\Lambda) = 0$.
\end{prop}

\begin{proof}
Let $t \geq 0$, and define $C_t \coloneqq \sup\{\norm{\Lambda(s)}_{p,q;r} : 0 \leq s \leq t\}< \infty$.
If $\Pi^* \in \cP_{\R_+}^*$, then
\begin{align*}
    \Bigg\|\sum_{s \in \Pi} \Lambda(s_*)\big[\Delta_sX^t, \Delta_sY^t\big]\Bigg\|_r & \leq \sum_{s \in \Pi}\big\|\Lambda(s_*)\big[\Delta_sX^t, \Delta_sY^t\big]\big\|_r \leq C_t\sum_{s \in \Pi}\big\|\Delta_sX^t\big\|_p \big\|\Delta_sY^t\big\|_q \\
    & \leq C_t \min\Big\{V_{L^p(\E_{\mathsmaller{\cA}})}(X : [0,t])\sup_{u,v \leq t : |u-v| \leq |\Pi|} \norm{Y(u) - Y(v)}_q, \\
    & \hspace{22.5mm} V_{L^q(\E_{\mathsmaller{\cB}})}(Y : [0,t])\sup_{u,v \leq t : |u-v| \leq |\Pi|} \norm{X(u) - X(v)}_p\Big\}.
\end{align*}
In either case, the result follows.
\end{proof}

\begin{cor}\label{cor.reducetomart}
Let $p,q,r \in [1,\infty]$.
If $X \in C(\R_+;L^p(\E_{\mathsmaller{\cA}}))$, $Y \in C(\R_+;L^q(\E_{\mathsmaller{\cB}}))$, $A \in \mathbb{FV}_{\mathsmaller{\cA}}^p$, $B \in \mathbb{FV}_{\mathsmaller{\cB}}^q$, and $\Lambda \colon \R_+ \to B_2^{p,q;r}$ is locally $\norm{\cdot}_{p,q;r}$-bounded, then
\[
\mathbb{L}^r\text{-}\lim_{\Pi \in \cP_{\R_+}}\Big(\mathrm{RS}_{\Pi}^{X+A,Y+B}(\Lambda) - \mathrm{RS}_{\Pi}^{X,Y}(\Lambda)\Big) = 0.
\]
\end{cor}

\begin{proof}
If $\Pi$ is a partition of $\R_+$, then
\[
\mathrm{RS}_{\Pi}^{X+A,Y+B}(\Lambda) = \mathrm{RS}_{\Pi}^{X,Y}(\Lambda) + \mathrm{RS}_{\Pi}^{X,B}(\Lambda) + \mathrm{RS}_{\Pi}^{A,Y+B}(\Lambda).
\]
By Proposition \ref{prop.QCofFV},
\[
\mathbb{L}^r\text{-}\lim_{\Pi \in \cP_{\R_+}}\mathrm{RS}_{\Pi}^{X,B}(\Lambda) = \mathbb{L}^r\text{-}\lim_{\Pi \in \cP_{\R_+}}\mathrm{RS}_{\Pi}^{A,Y+B}(\Lambda) = 0.
\]
The result follows.
\end{proof}

\subsection{Construction of quadratic covariation}\label{sec.QVextend}

We now use the results of the previous section to start building a general definition of quadratic covariation of a pair of $\tilde{L}^2$-decomposable processes (Definition \ref{def.tildesemi}).

\begin{lem}\label{lem.cheapbd}
If $M \colon \R_+ \to L^2(\E_{\mathsmaller{\cA}})$ and $N \colon \R_+ \to L^2(\E_{\mathsmaller{\cB}})$ are $L^2$-martingales and $\Lambda \colon \R_+ \to B_2^{2,2;1}$ is arbitrary, then
\[
\sup_{0 \leq s \leq t}\Big\|\mathrm{RS}_{\Pi}^{M,N}(\Lambda)(s)\Big\|_1 \leq \sup_{0 \leq s \leq t}\norm{\Lambda(s)}_{2,2;1} \norm{M(t) - M(0)}_2 \norm{N(t) - N(0)}_2 \qquad (\Pi \in \cP_{\R_+}, \; t \geq 0).
\]
\end{lem}

\begin{proof}
Write $C_t \coloneqq \sup\{\norm{\Lambda(s)}_{2,2;1} : 0 \leq s \leq t\}$.
If $\Pi$ is a partition of $\R_+$, then
\begin{align*}
    \Big\|\mathrm{RS}_{\Pi}^{M,N}(\Lambda)(t)\Big\|_1 & \leq \sum_{s \in \Pi} \norm{\Lambda(s_-)}_{2,2;1} \norm{\Delta_sM^t}_2 \norm{\Delta_sN^t}_2 \\
    & \leq C_t \Bigg(\sum_{s \in \Pi} \norm{\Delta_sM^t}_2^2\Bigg)^{\frac{1}{2}} \Bigg(\sum_{s \in \Pi} \norm{\Delta_sN^t}_2^2\Bigg)^{\frac{1}{2}} \\
    & = C_t \norm{M(t) - M(0)}_2 \norm{N(t) - N(0)}_2
\end{align*}
by the Cauchy--Schwarz inequality and Lemma \ref{lem.M(t)-M(s)}.
Thus,
\begin{align*}
    \sup_{0 \leq s \leq t}\Big\|\mathrm{RS}_{\Pi}^{M,N}(\Lambda)(s)\Big\|_1 & \leq C_t\sup_{0 \leq s \leq t} \norm{M(s) - M(0)}_2 \sup_{0 \leq s \leq t}\norm{N(s) - N(0)}_2 \\
    & = C_t\norm{M(t) - M(0)}_2\norm{N(t) - N(0)}_2
\end{align*}
by \eqref{eq.norminc}.
\end{proof}

\begin{thm}[QC construction I]\label{thm.QC1}
Write $\mathrm{Q} = \mathrm{Q}(\cA \times \cB ; \cC)$ for the closure of $\mathrm{Q}_0$ in the complex Fr\'{e}chet space $\ell_{\loc}^{\infty}(\R_+;B_2^{2,2;1})$ of locally $\norm{\cdot}_{2,2;1}$-bounded maps $\R_+ \to B_2^{2,2;1}$ with the topology of uniform convergence on compact sets.
\begin{enumerate}[label=(\roman*),font=\normalfont]
    \item The trilinear map
    \[
    \M_{\mathsmaller{\cA}}^{\infty} \times \M_{\mathsmaller{\cB}}^{\infty} \times \mathrm{Q}_0 \ni (M,N,\Lambda) \mapsto \mathbb{L}^1\text{-}\lim_{\Pi \in \cP_{\R_+}}\mathrm{RS}_{\Pi}^{M,N}(\Lambda) \in C_a(\R_+;L^1(\E_{\mathsmaller{\cC}}))
    \]
    extends uniquely to a continuous trilinear map $C \colon \tbM_{\mathsmaller{\cA}}^2 \times \tbM_{\mathsmaller{\cB}}^2 \times \mathrm{Q} \to C_a(\R_+;L^1(\E_{\mathsmaller{\cC}}))$ such that for all $(M,N,\Lambda) \in \tbM_{\mathsmaller{\cA}}^2 \times \tbM_{\mathsmaller{\cB}}^2 \times \mathrm{Q}$ and $t \geq 0$,\label{item.cheapQC}
    \begin{equation}
        \sup_{0 \leq s \leq t} \norm{C[M,N,\Lambda](s)}_1 \leq \sup_{0 \leq s \leq t} \norm{\Lambda(s)}_{2,2;1} \norm{M(t) - M(0)}_2 \norm{N(t) - N(0)}_2. \label{eq.Qbd}
    \end{equation}
    \item If $X \colon \R_+ \to L^2(\E_{\mathsmaller{\cA}})$ and $Y \colon \R_+ \to L^2(\E_{\mathsmaller{\cB}})$ are $\tilde{L}^2$-decomposable and $\Lambda \in \mathrm{Q}$ (e.g., if $\Lambda \colon \R_+ \to \mathbb{B}_2$ is adapted and $\norm{\cdot}_{2,2;1}$-continuous), then\label{item.cheapRSsumQC}
    \[
    \mathbb{L}^1\text{-}\lim_{\Pi \in \cP_{\R_+}}\mathrm{RS}_{\Pi}^{X,Y}(\Lambda) = C\big[X^{\mathrm{m}},Y^{\mathrm{m}},\Lambda\big].
    \]
    Recall that $X^{\mathrm{m}}$ is the martingale part of $X$ and $Y^{\mathrm{m}}$ is the martingale part of $Y$.
\end{enumerate}
\end{thm}
\pagebreak

\begin{proof}
We take each item in turn.

\ref{item.cheapQC} If $(M,N,\Lambda) \in \M_{\mathsmaller{\cA}}^{\infty} \times \M_{\mathsmaller{\cB}}^{\infty} \times \mathrm{Q}_0$ and $t \geq 0$, then\vspace{-0.35mm}
\[
\sup_{0 \leq s \leq t} \bigg\|L^1\text{-}\lim_{\Pi \in \cP_{\R_+}} \mathrm{RS}_{\Pi}^{M,N}(\Lambda)(s)\bigg\|_1 \leq \sup_{0 \leq s \leq t} \norm{\Lambda(s)}_{2,2;1} \norm{M(t) - M(0)}_2 \norm{N(t) - N(0)}_2\vspace{-0.35mm}
\]
by Lemma \ref{lem.cheapbd}.
(Recall that the existence of this limit is guaranteed by Theorem \ref{thm.NCIPR}.)
From this bound, the claims of this item follow from the completeness of $C_a(\R_+;L^1(\E_{\mathsmaller{\cC}}))$ and elementary limiting arguments.

\ref{item.cheapRSsumQC} Let $M \coloneqq X^{\mathrm{m}}$ and $N \coloneqq Y^{\mathrm{m}}$.
By Corollary \ref{cor.reducetomart},
\[
\mathbb{L}^1\text{-}\lim_{\Pi \in \cP_{\R_+}}\big(\mathrm{RS}_{\Pi}^{X,Y}(\Lambda) - \mathrm{RS}_{\Pi}^{M,N}(\Lambda)\big) = 0,\vspace{-0.35mm}
\]
so it suffices to prove that
\begin{equation}
    \mathbb{L}^1\text{-}\lim_{\Pi \in \cP_{\R_+}}\mathrm{RS}_{\Pi}^{M,N}(\Lambda) = C[M,N,\Lambda]. \vspace{-0.35mm}\label{eq.Qgoal}
\end{equation}
To this end, fix a partition $\Pi$ of $\R_+$, a sequence $(M_n)_{n \in \N}$ in $\M_{\mathsmaller{\cA}}^{\infty}$ converging in $\M_{\mathsmaller{\cA}}^2$ to $M$, a sequence $(N_n)_{n \in \N}$ in $\M_{\mathsmaller{\cB}}^{\infty}$ converging in $\M_{\mathsmaller{\cB}}^2$ to $N$, and a sequence $(\Lambda_n)_{n \in \N}$ in $\mathrm{Q}_0$ converging in $\ell_{\loc}^{\infty}(\R_+ ; B_2^{2,2;1})$ to $\Lambda$.
Then\vspace{-0.35mm}
\begin{align*}
    \mathrm{RS}_{\Pi}^{M,N}(\Lambda) - C[\Lambda,M,N] & = \mathrm{RS}_{\Pi}^{M,N}(\Lambda) - \mathrm{RS}_{\Pi}^{M_n,N_n}(\Lambda_n)\vspace{-0.35mm} \\
    & \hspace{5mm} + \mathrm{RS}_{\Pi}^{M_n,N_n}(\Lambda_n) - C[M_n,N_n,\Lambda_n]\vspace{-0.35mm} \\
    & \hspace{5mm} + C[M_n,N_n,\Lambda_n] - C[M,N,\Lambda]\vspace{-0.35mm}
\end{align*}
for all $n \in \N$.
Now, let $t \geq 0$ and $\e > 0$.
By the trilinearity of $\mathrm{RS}_{\Pi}^{U,V}(\Lambda)$ in $(U,V,\Lambda)$, Lemma \ref{lem.cheapbd}, the trilinearity of $C$, and \eqref{eq.Qbd}, there exists an $m \in \N$ such that $n \geq m$ implies\vspace{-0.35mm}
\begin{align*}
    \sup_{\Pi_0 \in \cP_{\R_+}}\sup_{0 \leq s \leq t}\Big\|\mathrm{RS}_{\Pi_0}^{M,N}(\Lambda)(s) - \mathrm{RS}_{\Pi_0}^{M_n,N_n}(\Lambda_n)(s)\Big\|_1 & < \frac{\e}{3} \; \text{ and\vspace{-0.35mm}} \\
    \sup_{0 \leq s \leq t}\norm{C[M_n,N_n,\Lambda_n](s) - C[M,N,\Lambda](s)}_1 & < \frac{\e}{3}.\vspace{-0.35mm}
\end{align*}
For this fixed $m$, there exists a $\delta > 0$ such that $|\Pi| < \delta$ implies\vspace{-0.35mm}
\[
\sup_{0 \leq s \leq t}\Big\|\mathrm{RS}_{\Pi}^{M_m,N_m}(\Lambda_m)(s) - C[M_m,N_m,\Lambda_m](s)\Big\|_1 < \frac{\e}{3}.\vspace{-0.35mm}
\]
Putting it all together, we obtain that $|\Pi| < \delta$ implies\vspace{-0.35mm}
\[
\sup_{0 \leq s \leq t}\Big\|\mathrm{RS}_{\Pi}^{M,N}(\Lambda)(s) - C[M,N,\Lambda](s)\Big\|_1 < \frac{\e}{3} + \frac{\e}{3} + \frac{\e}{3} = \e.\vspace{-0.35mm}
\]
Since $\e > 0$ and $t \geq 0$ were arbitrary, this completes the proof of \eqref{eq.Qgoal}.

To complete the proof of this item, we explain the parenthetical in the statement.
If $\Lambda \colon \R_+ \to \mathbb{B}_2$ is adapted and $\Pi$ is a partition of $\R_+$, then $\Lambda^{\Pi}  = 1_{\{0\}} \Lambda(0) + \sum_{t \in \Pi} 1_{(t_-,t]} \Lambda(t_-) \in \mathrm{Q}_0$.
If $\Lambda$ is also $\norm{\cdot}_{2,2;1}$-continuous, then $\Lambda^{\Pi} \to \Lambda$ in $\ell_{\loc}^{\infty}(\R_+;B_2^{2,2;1})$ as $|\Pi| \to 0$ by Lemma \ref{lem.HPi}\ref{item.HPistar}.
Thus, $\Lambda \in \mathrm{Q}$, as desired.
\end{proof}

Using Theorem \ref{thm.QC1}, we make the following preliminary definition.

\begin{defi}[Quadratic covariation I]\label{def.QC1}
If $X \colon \R_+ \to L^2(\E_{\mathsmaller{\cA}})$ and $Y \colon \R_+ \to L^2(\E_{\mathsmaller{\cB}})$ are $\tilde{L}^2$-decomposable and $\Lambda \in \mathrm{Q}$, then we define\vspace{-0.55mm}
\begin{align*}
    \llbracket X, Y \rrbracket^{\Lambda} & = \into \Lambda(t)[\d X(t), \d Y(t)] = \into \Lambda[\d X, \d Y] \\
    & \coloneqq \mathbb{L}^1\text{-}\lim_{\Pi \in \cP_{\R_+}}\mathrm{RS}_{\Pi}^{X,Y}(\Lambda) = C\big[X^{\mathrm{m}},Y^{\mathrm{m}},\Lambda\big] \in C_a(\R_+;L^1(\E_{\mathsmaller{\cC}}))
\end{align*}
to be the \textbf{$\boldsymbol{\Lambda}$-quadratic covariation of $\boldsymbol{X}$ and $\boldsymbol{Y}$}.
\end{defi}

Observe that the map $\mathrm{Q} \ni \Lambda \mapsto \llbracket X, Y \rrbracket^{\Lambda} \in C_a(\R_+;L^1(\E_{\mathsmaller{\cC}}))$ is complex linear.
Also, note that if $X$ and $Y$ are $L^{\infty}$-decomposable and $\Lambda \in \mathrm{Q}_0$, then\vspace{-0.35mm}
\[
\llbracket X, Y \rrbracket^{\Lambda} = \Lambda[X,Y] - \Lambda(0)[X(0), Y(0)] - \into \d\Lambda(t)[X(t), Y(t)]  - \into \Lambda(t)[\d X(t), Y(t)] - \into \Lambda(t)[X(t), \d Y(t)]\vspace{-0.35mm}
\]
by Theorem \ref{thm.NCIPR}.
Now, here is a very common example, similar to Example \ref{ex.elemtrbipinteg}, which we shall upgrade in the next section (Example \ref{ex.elemtrtripQCupgrade}) and expand on in special cases in Section \ref{sec.QCexs}.
\pagebreak

\begin{ex}\label{ex.elemtrtripQC}
In this example, we assume $(\cA,(\cA_t)_{t \geq 0},\E = \E_{\mathsmaller{\cA}}) = (\cB,(\cB_t)_{t \geq 0},\E_{\mathsmaller{\cB}}) = (\cC,(\cC_t)_{t \geq 0}, \E_{\mathsmaller{\cC}})$.
Fix, for each $i = 1,\ldots,15$, an adapted process $A_i \colon \R_+ \to \cA$.
Define\vspace{-0.2mm}
\begin{align*}
    \Lambda(t)[x,y] & \coloneqq A_1(t)xA_2(t)yA_3(t) + \E[A_4(t)xA_5(t)y]A_6(t) + \E[A_7(t)x]\E[A_8(t)y]A_9(t) \\
    & \hspace{15mm} + \E[A_{10}(t)x]A_{11}(t)yA_{12}(t) + A_{13}(t)xA_{14}(t)\E[A_{15}(t)y]\vspace{-0.2mm}
\end{align*}
for all $t \geq 0$ and $x,y \in \cA$.
If $A_1,\ldots,A_{15} \colon \R_+ \to \cA$ are adapted, then $\Lambda$ is a (complex-bilinear) trace triprocess.
In particular, by Proposition \ref{prop.TkinFk}, $\Lambda \colon \R_+ \to \mathbb{B}_2(\cA)$ is adapted.
Moreover, if, in addition, $A_1,\ldots,A_{15}$ are $L^{\infty}$-continuous, then $\Lambda$ is $\vertiii{\cdot}_2$-continuous.
In this case, Theorem \ref{thm.QC1}\ref{item.cheapRSsumQC} says that $\Lambda \in \mathrm{Q}$ and\vspace{-0.2mm}
\[
\int_0^t\Lambda(s)[\d X(s), \d Y(s)] = L^1\text{-}\lim_{\pi \in \cP_{[0,t]}}\sum_{s \in \pi}\Lambda(s_-)[\Delta_sX, \Delta_sY]\vspace{-0.2mm}
\]
for all $\tilde{L}^2$-decomposable processes $X,Y \colon \R_+ \to L^2(\E)$.
Also, since $\llbracket X, Y \rrbracket^{\Lambda} = \llbracket X^{\mathrm{m}}, Y^{\mathrm{m}} \rrbracket^{\Lambda}$, Theorem \ref{thm.QC1}\ref{item.cheapRSsumQC} implies that\vspace{-0.2mm}
\begin{align*}
    \int_0^t \E[A_7(s)\,\d X(s)]\E[A_8(s)\,\d Y(s)]A_9(s) & = \int_0^t \E[A_{10}(s)\,\d X(s)]A_{11}(s) \,\d Y(s) \, A_{12}(s) \\
    & = \int_0^t A_{13}(s)\,\d X(s) \, A_{14}(s)\E[A_{15}(s)\,\d Y(s)] = 0\vspace{-0.2mm}
\end{align*}
by a calculation similar to the one in Example \ref{ex.elemtrbipinteg}.

More generally, if $n,m,d_1,d_2 \in \N$, $P \in (\TrP_{n,2,(d_1,d_2)}^*)^m$, and $\mathbf{X} \colon \R_+ \to \cA^n$ is adapted, then the bilinear process $\R_+ \ni t \mapsto P(\mathbf{X}(t)) \in \mathbb{B}_2(\cA^{d_1} \times \cA^{d_2};\cA^m)$ is a multivariate trace triprocess.
If, in addition, $\mathbf{X}$ is $L^{\infty}$-continuous, then $\Lambda$ is $\vertiii{\cdot}_2$-continuous.
In this case, if $\mathbf{Y} \colon \R_+ \to L^2(\E^{\oplus d_1})$ and $\mathbf{Z} \colon \R_+ \to L^2(\E^{\oplus d_2})$ are $\tilde{L}^2$-decomposable, then\vspace{-0.2mm}
\[
\int_0^tP(\mathbf{X}(s),\d\mathbf{Y}(s),\d\mathbf{Z}(s)) = L^1\text{-}\lim_{\pi \in \cP_{[0,t]}}\sum_{s \in \pi}P(\mathbf{X}(s_-),\Delta_s\mathbf{Y},\Delta_s\mathbf{Z}) \in L^1(\E^{\oplus m}) = L^1(\E)^{\oplus m}\vspace{-0.2mm}
\]
by Theorem \ref{thm.QC1}\ref{item.cheapRSsumQC}.
\end{ex}

In particular, the ``cheap'' quadratic covariation in Definition \ref{def.QC1} already allows us to consider interesting examples.
It is also sufficient to support interesting applications, e.g., continuous-time noncommutative Burkholder--Davis--Gundy inequalities for $p \in [2,\infty)$ and It\^{o}'s formula.
However, for general considerations, it is desirable to extend the definition of $\llbracket X, Y \rrbracket^{\Lambda}$ to a larger class of bilinear processes $\Lambda$.

\begin{lem}\label{lem.QCL1bd}
Let $X \colon \R_+ \to L^2(\E_{\mathsmaller{\cA}})$ and $Y \colon \R_+ \to L^2(\E_{\mathsmaller{\cB}})$ be $\tilde{L}^2$-decomposable, $\Lambda \in \mathrm{Q}$, and $t \geq 0$.
Also, write $M \coloneqq X^{\mathrm{m}}$, $N \coloneqq Y^{\mathrm{m}}$, and $\kappa_{M,N} \coloneqq (\kappa_M+\kappa_N)/2$ (Lemma \ref{lem.fakeDoleans}).\vspace{-0.2mm}
\begin{enumerate}[label=(\roman*),font=\normalfont]
    \item $\displaystyle\Bigg\|\sum_{s \in \pi} \Lambda(s_*)[\Delta_sM, \Delta_sN]\Bigg\|_1 \leq \int_0^t \norm{\Lambda}_{2,2;1}^{\pi^*}\,\d\kappa_{M,N}$ for all $\pi^* \in \cP_{[0,t]}^*$.\label{item.RSbd}\vspace{-0.2mm}
    \item $\displaystyle\big\|\llbracket X, Y \rrbracket^{\Lambda}(t)\big\|_1 \leq \int_0^t\norm{\Lambda}_{2,2;1}\,\d\kappa_{M,N}$.\label{item.QCbd}\vspace{-0.2mm}
\end{enumerate}
Note that since $M$ and $N$ are $L^2$-continuous, the measure $\kappa_{M,N}$ is non-atomic.
In particular, the notation $\int_s^t \boldsymbol{\cdot} \, \d\kappa_{M,N}$ is not ambiguous.
\end{lem}
\begin{proof}
We take both items in turn.

\ref{item.RSbd} We have\vspace{-0.2mm}
\begin{align*}
    \Bigg\|\sum_{s \in \pi}\Lambda(s_*)[\Delta_sM, \Delta_sN]\Bigg\|_1 & \leq \sum_{s \in \Pi}\norm{\Lambda(s_*)[\Delta_sM, \Delta_sN]}_1 \leq \sum_{s \in \pi}\norm{\Lambda(s_*)}_{2,2;1}\norm{\Delta_sM}_2\norm{\Delta_sN}_2 \\
    & \leq \frac{1}{2}\sum_{s \in \pi}\norm{\Lambda(s_*)}_{2,2;1}\big(\norm{\Delta_sM}_2^2 + \norm{\Delta_sN}_2^2\big) \\
    & = \sum_{s \in \pi}\norm{\Lambda(s_*)}_{2,2;1}\kappa_{M,N}((s_-,s]) = \int_0^t\norm{\Lambda}_{2,2;1}^{\pi^*} \,\d\kappa_{M,N},
\end{align*}
as claimed.
\pagebreak

\ref{item.QCbd} Note that $\Lambda$ is $\norm{\cdot}_{2,2;1}$-LCLB, so Lemma \ref{lem.HPi}\ref{item.HPipw} and the dominated convergence theorem yield that $\norm{\Lambda}_{2,2;1}^{\pi} \to \norm{\Lambda}_{2,2;1}$ in $L^1([0,t],\kappa_{M,N})$ as $|\pi| \to 0$.
Thus, by Theorem \ref{thm.QC1}\ref{item.cheapRSsumQC} and the first item,
\begin{align*}
    \big\|\llbracket X,Y \rrbracket^{\Lambda}(t)\big\|_1 & = \lim_{\pi \in \cP_{[0,t]}}\Bigg\|\sum_{s \in \pi}\Lambda(s_-)[\Delta_sM, \Delta_sN]\Bigg\|_1 \\
    & \leq \lim_{\pi \in \cP_{[0,t]}}\int_0^t\norm{\Lambda}_{2,2;1}^{\pi} \,\d\kappa_{M,N} = \int_0^t\norm{\Lambda}_{2,2;1} \,\d\kappa_{M,N},
\end{align*}
as claimed.
\end{proof}

\begin{nota}\label{nota.QXY}
Write $\mathcal{Q}$ for the set of equivalences classes in $L_{\loc}^1(\R_+,\kappa_{M,N};B_2^{2,2;1})$ of members of $\mathrm{Q}$ and $\mathcal{Q}(X,Y)$ for the closure of $\mathcal{Q}$ in $L_{\loc}^1(\R_+,\kappa_{M,N};B_2^{2,2;1})$.
\end{nota}

By Lemma \ref{lem.QCL1bd}\ref{item.QCbd} and the linearity of the map $\mathrm{Q} \ni \Lambda \mapsto \llbracket X, Y \rrbracket^{\Lambda} \in C_a(\R_+;L^1(\E_{\mathsmaller{\cC}}))$, if $\Lambda_1, \Lambda_2 \in \mathrm{Q}$ and $\Lambda_1 = \Lambda_2$ $\kappa_{M,N}$-a.e., then
\[
\int_0^t \Lambda_1[\d X, \d Y] = \int_0^t \Lambda_2[\d X, \d Y] = \int_0^t (1_{(0,T]}\Lambda_2)[\d X, \d Y] \qquad (T \geq t \geq 0).
\]
In particular, $\llbracket X, Y \rrbracket^{\Lambda} \in C_a(\R_+;L^1(\E_{\mathsmaller{\cC}}))$ is well defined for $\Lambda \in \mathcal{Q}$.

\begin{thm}[QC construction II]\label{thm.QC2}
Let $X \colon \R_+ \to L^2(\E_{\mathsmaller{\cA}})$ and $Y \colon \R_+ \to L^2(\E_{\mathsmaller{\cB}})$ be $\tilde{L}^2$-decomposable processes.
The map $\mathcal{Q} \ni \Lambda \mapsto \llbracket X, Y \rrbracket^{\Lambda} \in C_a(\R_+;L^1(\E_{\mathsmaller{\cC}}))$ extends uniquely to a continuous complex-linear map $\llbracket X, Y \rrbracket = \llbracket X, Y \rrbracket_{\cC} \colon \mathcal{Q}(X,Y) \to C_a(\R_+;L^1(\E_{\mathsmaller{\cC}}))$.
Let $\Lambda \in \mathcal{Q}(X,Y)$, $M \coloneqq X^{\mathrm{m}}$, and $N \coloneqq Y^{\mathrm{m}}$.
The map $\llbracket X, Y \rrbracket$ satisfies the following properties.
\begin{enumerate}[label=(\roman*),font=\normalfont]
    \item $\displaystyle\big\|\llbracket X, Y \rrbracket(\Lambda)(t)\big\|_1 \leq \int_0^t \norm{\Lambda}_{2,2;1}\,\d\kappa_{M,N}$ for all $t \geq 0$.\label{item.QCbd2}
    \item $\llbracket X, Y \rrbracket(\Lambda)(t) - \llbracket X, Y \rrbracket(\Lambda)(s) = \llbracket X, Y \rrbracket(1_{(s,t]}\Lambda)(t)$ whenever $0 \leq s < t$.\label{item.QCinc}
    \item $\llbracket X, Y \rrbracket(\Lambda)$ is $L^1$-FV.\label{item.QCFV}
    \item $\mathcal{Q}(X,Y) = \mathcal{Q}(M,N)$, and $\llbracket X, Y \rrbracket = \llbracket M, N \rrbracket$.\label{item.QCmart}
\end{enumerate}
\end{thm}

\begin{proof}
By Lemma \ref{lem.QCL1bd}\ref{item.QCbd}, the map $\mathcal{Q} \ni \Lambda \mapsto \llbracket X, Y \rrbracket^{\Lambda} \in C_a(\R_+;L^1(\E_{\mathsmaller{\cC}}))$ is (complex-linear and) continuous.
The existence and uniqueness of the continuous complex-linear extension $\llbracket X, Y \rrbracket$ from $\mathcal{Q}(X,Y)$ to $C_a(\R_+;L^1(\E_{\mathsmaller{\cC}}))$ then follows from the completeness of $C_a(\R_+;L^1(\E_{\mathsmaller{\cC}}))$.
Item \ref{item.QCbd2} then follows from an elementary limiting argument.
We take each of the remaining items in turn.

\ref{item.QCinc} By another elementary limiting argument using the first item, it suffices to prove the desired identity assuming $\Lambda \in \mathrm{Q}$.
To this end, let $\pi$ be a partition of $[0,t]$ such that $s \in \pi$.
Also, write $s_+$ for the member of $\pi$ to the right of $s$, i.e., $s_+ = \min\{r \in \pi : s < r\} \in \pi$. 
Then
\begin{align*}
    \sum_{r \in \pi} \Lambda(r_-)[\Delta_rX, \Delta_rY] & = \sum_{r \in \pi \cap [0,s]} \Lambda(r_-)[\Delta_rX, \Delta_rY] + \sum_{r \in \pi \cap (s,t]}  \Lambda(r_-)[\Delta_rX, \Delta_rY] \\
    & = \sum_{r \in \pi \cap [0,s]} \Lambda(r_-)[\Delta_rX, \Delta_rY] + \sum_{r \in \pi}  (1_{(s,t]}\Lambda)(r_-)[\Delta_rX, \Delta_rY] + \Lambda(s)[\Delta_{s_+}X, \Delta_{s_+}Y].
\end{align*}
If $|\pi| \to 0$, then $s_+ \searrow s$.
Thus, by the continuity of $X$ and $Y$, $\Lambda(s)[\Delta_{s_+}X, \Delta_{s_+}Y] \to 0$ in $L^1(\E_{\mathsmaller{\cC}})$ as $|\pi| \to 0$.
It therefore follows, by taking $|\pi| \to 0$, from Theorem \ref{thm.QC1} and Definition \ref{def.QC1} that
\[
\int_0^t \Lambda[\d X, \d Y] = \int_0^s \Lambda[\d X, \d Y] + \int_0^t (1_{(s,t]} \Lambda)[\d X, \d Y],
\]
as desired.

\ref{item.QCFV} Let $s,t \geq 0$ be such that $s < t$.
By the previous two items,
\[
\big\|\llbracket X, Y\rrbracket^{\Lambda}(t) - \llbracket X, Y \rrbracket^{\Lambda}(s)\big\|_1 = \Bigg\|\int_0^t (1_{(s,t]}\Lambda)[\d X, \d Y]\Bigg\|_1 \leq \int_s^t \norm{\Lambda}_{2,2;1} \, \d\kappa_{M,N}.\pagebreak
\]
In particular,
\begin{align*}
    V_{L^1(\E_{\mathsmaller{\cC}})}\Big(\llbracket X, Y \rrbracket^{\Lambda} : [0,t]\Big) & = \sup_{\pi \in \cP_{[0,t]}} \sum_{r \in \pi} \big\|\Delta_r\llbracket X, Y \rrbracket^{\Lambda}\big\|_1 \\
    & \leq \sup_{\pi \in \cP_{[0,t]}} \sum_{r \in \pi} \int_{r_-}^r \norm{\Lambda}_{2,2;1} \, \d\kappa_{M,N} \\
    & = \int_0^t \norm{\Lambda}_{2,2;1} \, \d\kappa_{M,N} < \infty,
\end{align*}
so this item is proven.

\ref{item.QCmart} This item is clear from the definitions.
\end{proof}

Finally, we arrive at the general definition.

\begin{defi}[Quadratic covariation II]\label{def.QC2}
For $\Lambda \in \mathcal{Q}(X,Y)$, we define
\[
\llbracket X, Y \rrbracket^{\Lambda} = \into \Lambda[\d X, \d Y] = \into \Lambda(t)[\d X(t), \d Y(t)] \coloneqq \llbracket X, Y \rrbracket(\Lambda) \in C_a(\R_+;L^1(\E_{\mathsmaller{\cC}}))
\]
to be the \textbf{$\boldsymbol{\Lambda}$-quadratic covariation of $\boldsymbol{X}$ and $\boldsymbol{Y}$}, where $\llbracket X, Y \rrbracket$ is as in Theorem \ref{thm.QC2}.
\end{defi}

\subsection{Tools to calculate quadratic covariations}\label{sec.RS2}

In this section, we prove three additional facts about (certain) quadratic covariations that help with their explicit calculation in practice:
a formula for quadratic covariations of stochastic integrals (Theorem \ref{thm.QCSI}) and two expressions for $\into \Lambda[\d X, \d Y]$ as a limit of left-endpoint quadratic Riemann--Stieltjes sums (Proposition \ref{prop.LEQRSapprox} and Theorem \ref{thm.NCcondQC}).

Let $X = X(0) + M + A \colon \R_+ \to L^2(\E_{\mathsmaller{\cA}})$ and $Y = Y(0) + N + B \colon \R_+ \to L^2(\E_{\mathsmaller{\cB}})$ be $\tilde{L}^2$-decomposable processes for the duration of this section.
By Proposition \ref{prop.stochintLtilde}, if $H \in \tilde{\cI}(X)$, $K \in \tilde{\cI}(Y)$, $U \coloneqq \into H[\d X]$, and $V \coloneqq \into K[\d Y]$, then $U$ and $V$ are $\tilde{L}^2$-decomposable.
Therefore, the construction from the previous section enables us to consider quadratic covariation integrals $\into \Lambda[\d U, \d V]$.
The next result shows that it is frequently true that the bilinear process $\Lambda[H,K] \coloneqq (t \mapsto \Lambda(t)[H(t)[\cdot],K(t)[\cdot]])$ belongs to $\mathcal{Q}(X,Y)$ and $\into \Lambda[\d U, \d V] = \into \Lambda[H[\d X],K[\d Y]]$ holds, as ``$\d U = H[\d X]$'' and ``$\d V = K[\d Y]$'' suggest.

\begin{thm}[QC of stochastic integrals]\label{thm.QCSI}
Fix two more filtered $\mathrm{C}^*$-probability spaces $(\cD,(\cD_t)_{t \geq 0},\E_{\mathsmaller{\cD}})$ and $(\cE,(\cE_t)_{t \geq 0},\E_{\mathsmaller{\cE}})$. 
Suppose $H \colon \R_+ \to B(L^2(\E_{\mathsmaller{\cA}});L^2(\E_{\mathsmaller{\cD}}))$ and $K \colon \R_+ \to B(L^2(\E_{\mathsmaller{\cB}});L^2(\E_{\mathsmaller{\cE}}))$ are strongly measurable maps and that there exist sequences $(H_n)_{n \in \N}$ and $(K_n)_{n \in \N}$ in $\EP(\E_{\mathsmaller{\cA}};\E_{\mathsmaller{\cD}})$ and $\EP(\E_{\mathsmaller{\cB}};\E_{\mathsmaller{\cE}})$, respectively, such that for all $t \geq 0$,
\begin{equation}
    \int_0^t (\norm{H - H_n}_{2;2}^2 + \norm{K - K_n}_{2;2}^2)\,\d\kappa_{M,N} + \int_0^t\norm{H - H_n}_{2;2} \,\d\kappa_A + \int_0^t \norm{K - K_n}_{2;2}\,\d\kappa_B \xrightarrow{n \to \infty} 0. \label{eq.QCSI1}
\end{equation}
(In this case, we have $H \in \tilde{\cI}^{\cD}(X)$ and $K \in \tilde{\cI}^{\cE}(Y)$.)
Now, write $U \coloneqq \into H[\d X]$ and $V \coloneqq \into K[\d Y]$.
If $\Lambda \colon \R_+ \to B_2(L^2(\E_{\mathsmaller{\cD}}) \times L^2(\E_{\mathsmaller{\cE}});L^1(\E_{\mathsmaller{\cC}}))$ is strongly measurable and there exists a sequence $(\Lambda_n)_{n \in \N}$ in $\mathrm{Q}(\cD \times \cE;\cC)$ such that for all $t \geq 0$,
\begin{equation}
    \int_0^t \norm{\Lambda - \Lambda_n}_{2,2;1} (\norm{H}_{2;2}^2 + \norm{K}_{2;2}^2)\,\d\kappa_{M,N} \xrightarrow{n \to \infty} 0, \label{eq.QCSI2}
\end{equation}
then $\Lambda \in \mathcal{Q}(U,V)$, $\Lambda[H,K] \in \mathcal{Q}(X,Y)$, and $\llbracket U,V \rrbracket^{\Lambda} = \llbracket X,Y \rrbracket^{\Lambda[H,K]}$, i.e.,
\begin{equation}
    \into \Lambda(t)[\d U(t), \d V(t)] = \into \Lambda(t)\big[H(t)[\d X(t)],K(t)[\d Y(t)]\big]. \label{eq.QCSI3}
\end{equation}
\end{thm}

\begin{proof}
By Theorems \ref{thm.stochint}\ref{item.stochintsemimart} and \ref{thm.QC2}\ref{item.QCmart}, it suffices to assume $(X,Y) = (M,N)$ so that $U = \into H[\d M]$ and $V = \into K[\d N]$.
Recall from the proof of Theorem \ref{thm.subform} that
\[
\kappa_U(E) \leq \int_E \norm{H}_{2;2}^2\,\d\kappa_M \; \text{ and } \; \kappa_V(E) \leq \int_E \norm{K}_{2;2}^2\,\d\kappa_N \qquad (E \in \cB_{\R_+}).\pagebreak
\]
Thus,
\[
\kappa_{U,V}(E) \leq \frac{1}{2}\Bigg(\int_E \norm{H}_{2;2}^2\,\d\kappa_M + \int_E \norm{K}_{2;2}^2\,\d\kappa_N\Bigg) \leq \int_E(\norm{H}_{2;2}^2 + \norm{K}_{2;2}^2)\,\d\kappa_{M,N} \qquad (E \in \cB_{\R_+}).
\]
From this and \eqref{eq.QCSI2}, it easily follows that if $\Lambda$ is as in the statement, then $\Lambda \in \mathcal{Q}(U,V)$.

Proving that $\Lambda[H,K] \in \mathcal{Q}(M,N)$ takes a bit more work.
If $n,m \in \N$, then $\Lambda_n[H_m,K_m] \in \mathrm{Q}(\cA \times \cB ; \cC)$, as we encourage the reader to check.
Also, writing $C_t^n \coloneqq \sup\{\norm{\Lambda_n(s)}_{2,2;1} : 0 \leq s \leq t\}$,
\begin{align*}
    \int_0^t \norm{\Lambda_n[H_m,K_m] - \Lambda_n[H,K]}_{2,2;1} &\,\d\kappa_{M,N} \leq \int_0^t \norm{\Lambda_n[H_m - H,K_m]}_{2,2;1}\,\d\kappa_{M,N} \\
    & \hspace{20mm} + \int_0^t \norm{\Lambda_n[H,K_m - K]}_{2,2;1}\,\d\kappa_{M,N} \\
    & \leq C_t^n\int_0^t \norm{H_m - H}_{2;2} \norm{K_m}_{2;2} \,\d\kappa_{M,N} \\
    & \hspace{6.5mm} + C_t^n\int_0^t \norm{H}_{2;2} \norm{K_m - K}_{2;2}\,\d\kappa_{M,N} \\
    & \leq C_t^n\Bigg(\int_0^t \norm{H_m - H}_{2;2}^2 \,\d\kappa_{M,N}\Bigg)^{\frac{1}{2}}\Bigg(\int_0^t \norm{K_m}_{2;2}^2 \,\d\kappa_{M,N}\Bigg)^{\frac{1}{2}} \\
    & \hspace{6.5mm} + C_t^n\Bigg(\int_0^t \norm{H}_{2;2}^2\,\d\kappa_{M,N}\Bigg)^{\frac{1}{2}}\Bigg(\int_0^t \norm{K_m - K}_{2;2}^2\,\d\kappa_{M,N}\Bigg)^{\frac{1}{2}} \xrightarrow{m \to \infty} 0
\end{align*}
by the Cauchy--Schwarz inequality and \eqref{eq.QCSI1}.
We conclude that $\Lambda_n[H,K] \in \mathcal{Q}(M,N)$.
Since, by \eqref{eq.QCSI2}, we also have that
\begin{align*}
    \int_0^t \norm{\Lambda_n[H,K] - \Lambda[H,K]}_{2,2;1}\,\d\kappa_{M,N} & \leq \int_0^t \norm{\Lambda_n - \Lambda}_{2,2;1} \norm{H}_{2;2} \norm{K}_{2;2} \, \d\kappa_{M,N} \\
    & \leq \frac{1}{2}\int_0^t \norm{\Lambda_n - \Lambda}_{2,2;1} (\norm{H}_{2;2}^2 + \norm{K}_{2;2}^2) \, \d\kappa_{M,N} \xrightarrow{n \to \infty} 0,
\end{align*}
we conclude that $\Lambda[H,K] \in \mathcal{Q}(M,N)$.

To prove \eqref{eq.QCSI3}, we do some reductions.
As the reader may verify by making use of $(\Lambda_n)_{n \in \N}$, it suffices to prove the formula when $\Lambda \in \mathrm{Q}(\cD \times \cE ; \cC)$.
Next, we argue that it also suffices to prove the formula when $H \in \EP(\E_{\mathsmaller{\cA}};\E_{\mathsmaller{\cD}})$ and $K \in \EP(\E_{\mathsmaller{\cB}};\E_{\mathsmaller{\cE}})$.
Indeed, write $U_n \coloneqq \into H_n[\d M]$ and $V_n \coloneqq \into K_n[\d N]$.
Suppose $\into \Lambda[\d U_n, \d V_n] = \into \Lambda[H_n[\d M],K_n[\d N]]$ for all $n \in \N$.
Since $(U_n,V_n,\Lambda) \to (U,V,\Lambda)$  in $\tbM_{\mathsmaller{\cD}}^2 \times \tbM_{\mathsmaller{\cE}}^2 \times \mathrm{Q}(\cD \times \cE ; \cC)$ as $n \to \infty$ by \eqref{eq.QCSI1}, Theorem \ref{thm.QC1} ensures that $\into \Lambda[\d U_n, \d V_n] \to \into \Lambda[\d U, \d V]$ in $C_a(\R_+;L^1(\E_{\mathsmaller{\cC}}))$ as $n \to \infty$.
Also, as we showed in the previous paragraph, $\Lambda[H_n,K_n] \to \Lambda[H,K]$ in $\mathcal{Q}(M,N)$ as $n \to \infty$ so that $\into \Lambda[H_n[\d M],K_n[\d N]] \to \into \Lambda[H[\d M],K[\d N]]$ in $C_a(\R_+;L^1(\E_{\mathsmaller{\cC}}))$ as $n \to \infty$.
Thus, \eqref{eq.QCSI3} holds.

Therefore, it remains to prove \eqref{eq.QCSI3} when $\Lambda \in \mathrm{Q}(\cD \times \cE ; \cC)$, $H \in \EP(\E_{\mathsmaller{\cA}};\E_{\mathsmaller{\cD}})$, and $K \in \EP(\E_{\mathsmaller{\cB}};\E_{\mathsmaller{\cE}})$.
To this end, recall that $\Lambda[H,K] \in \mathrm{Q}(\cA \times \cB ; \cC)$ in this case.
Now, let $\Pi$ be a partition of $\R_+$.
Observe that if $s \in \Pi$ and $t \geq 0$, then
\[
H(s_-)\big[\Delta_sM^t\big] = \int_{s_- \wedge t}^{s \wedge t} H^{\Pi}[\d M] = \Delta_s I_M\big(H^{\Pi}\big)^t \; \text{ and } \; K(s_-)\big[\Delta_sN^t\big] = \Delta_s I_M\big(K^{\Pi}\big)^t.
\]
Thus,
\[
\mathrm{RS}_{\Pi}^{M,N}(\Lambda[H,K]) = \mathrm{RS}_{\Pi}^{I_M(H^{\Pi}), I_M(K^{\Pi})}(\Lambda).
\]
Therefore,
\begin{align*}
    \mathrm{RS}_{\Pi}^{U,V}(\Lambda) - \mathrm{RS}_{\Pi}^{M,N}(\Lambda[H,K]) & = \mathrm{RS}_{\Pi}^{I_M(H), I_M(K)}(\Lambda) - \mathrm{RS}_{\Pi}^{I_M(H^{\Pi}), I_M(K^{\Pi})}(\Lambda) \\
    & = \mathrm{RS}_{\Pi}^{I_M(H) - I_M(H^{\Pi}), I_M(K)}(\Lambda) + \mathrm{RS}_{\Pi}^{I_M(H^{\Pi}), I_M(K) - I_M(K^{\Pi})}(\Lambda).
\end{align*}
By Lemma \ref{lem.cheapbd} and (the proof of) Proposition \ref{prop.LERSapprox}---recall that $H$ and $K$ are $\vertiii{\cdot}$-LCLB---we conclude that if $C_t \coloneqq \sup\{\norm{\Lambda(s)}_{2,2;1} : 0 \leq s \leq t\}$, then
\begin{align*}
    \sup_{0 \leq s \leq t}\Big\|\mathrm{RS}_{\Pi}^{U,V}(\Lambda)(s) - \mathrm{RS}_{\Pi}^{M,N}(\Lambda[H,K])(s)\Big\|_1 & \leq C_t\big\|I_M\big(H - H^{\Pi}\big)(t)\big\|_2 \norm{I_M(K)(t)}_2 \\
    & \hspace{8mm} + C_t\big\|I_M\big(H^{\Pi}\big)(t)\big\|_2\big\|I_M\big(K - K^{\Pi}\big)(t)\big\|_2 \xrightarrow[\Pi \in \cP_{\R_+}]{|\Pi| \to 0} 0.
\end{align*}
Since
\[
\mathbb{L}^1\text{-}\lim_{\Pi \in \cP_{\R_+}}\Big( \mathrm{RS}_{\Pi}^{U,V}(\Lambda) - \mathrm{RS}_{\Pi}^{M,N}(\Lambda[H,K])\Big) = \into \Lambda[\d U, \d V] - \into \Lambda[H[\d M],K[\d N]]
\]
as well, this completes the proof.
\end{proof}

\begin{ex}\label{ex.LLLB2}
By arguments like those in Example \ref{ex.LLLB}, if $J \colon \R_+ \to \mathbb{B}(\cA;\cD)$ and $L \colon \R_+ \to \mathbb{B}(\cB;\cE)$ are adapted and $\norm{\cdot}_{2;2}$-LLLB, and $\Xi \colon \R_+ \to \mathbb{B}_2(\cD \times \cE ; \cC)$ is adapted and $\norm{\cdot}_{2,2;1}$-LLLB, then the triple $(H,K,\Lambda) \coloneqq (J_-,L_-,\Xi_-)$ satisfies the hypotheses of Theorem \ref{thm.QCSI}.
Also, note that if $M=N$ or $M=N^*$, then the first hypothesis, namely, \eqref{eq.QCSI1}, is merely the requirement that $H \in \tilde{\cI}^{\cD}(X)$ and $K \in \tilde{\cI}^{\cE}(Y)$.
\end{ex}

Next, we prove another result on the convergence of left-endpoint quadratic Riemann--Stieltjes sums to quadratic covariations.

\begin{prop}\label{prop.LEQRSapprox}
If $\Lambda \colon \R_+ \to \mathbb{B}_2$ is adapted and $\norm{\cdot}_{2,2;1}$-LLLB, then $\Lambda_- \in \mathcal{Q}(X,Y)$, and
\[
\mathbb{L}^1\text{-}\lim_{\Pi \in \cP_{\R_+}}\mathrm{RS}_{\Pi}^{X,Y}(\Lambda) = \into \Lambda_-[\d X, \d Y] = \into \Lambda(t-)[\d X(t), \d Y(t)].
\]
If $\Lambda$ is continuous with respect to $\norm{\cdot}_{2,2;1}$, then we may use any evaluation points (not just the left endpoints) in the quadratic Riemann--Stieltjes sums.
\end{prop}

\begin{proof}
Let $\Pi$ be a partition of $\R_+$.
We have seen already that if $\Lambda \colon \R_+ \to \mathbb{B}_2$ is adapted, then $\Lambda^{\Pi} \in \mathrm{Q}_0$.
Now, if $\Lambda$ is also $\norm{\cdot}_{2,2;1}$-LLLB, then Lemma \ref{lem.HPi}\ref{item.HPipw} and the dominated convergence theorem yield that $\Lambda^{\Pi} \to \Lambda_-$ in $L_{\loc}^1(\R_+,\kappa_{M,N};B_2^{2,2;1})$ as $|\Pi| \to 0$.
Thus, $\Lambda_- \in \mathcal{Q}(X,Y)$.

For the second claim, as usual, it suffices to treat the case $(X,Y) = (M,N)$.
To this end, suppose $\Xi \in \mathrm{Q}$ and $t \geq s \geq 0$.
By Lemma \ref{lem.QCL1bd}\ref{item.RSbd} and Theorem \ref{thm.QC2}\ref{item.QCbd2},
\begin{align*}
    \e_{\Pi}(s) & \coloneqq \Big\|\mathrm{RS}_{\Pi}^{M,N}(\Lambda)(s) - \llbracket M, N \rrbracket^{\Lambda_-}(s)\Big\|_1 \\
    & \leq  \Big\|\mathrm{RS}_{\Pi}^{M,N}(\Lambda - \Xi)(s)\Big\|_1 + \Big\|\mathrm{RS}_{\Pi}^{M,N}(\Xi)(s) - \llbracket M, N \rrbracket^{\Xi}(s)\Big\|_1  + \big\|\llbracket M, N \rrbracket^{\Xi - \Lambda_-}(s)\big\|_1 \\
    & \leq \int_0^s\big(\norm{\Lambda - \Xi}_{2,2;1}^{\Pi} + \norm{\Lambda_- - \Xi}_{2,2;1}\big)\,\d\kappa_{M,N} + \Big\|\mathrm{RS}_{\Pi}^{M,N}(\Xi)(s) - \llbracket M, N \rrbracket^{\Xi}(s)\Big\|_1.
\end{align*}
Now, let $\e > 0$.
Since $\Lambda_- \in \mathcal{Q}(M,N)$, we can choose a $\Xi \in \mathrm{Q}$ so that
\[
\int_0^t\norm{\Lambda_- - \Xi}_{2,2;1}\,\d\kappa_{M,N} < \frac{\e}{3}.
\]
Next, since $\norm{\Lambda - \Xi}_{2,2;1}$ is LLLB, $\norm{\Lambda - \Xi}_{2,2;1}^{\Pi} \to (\norm{\Lambda-\Xi}_{2,2;1})_- = \norm{\Lambda_- - \Xi_-}_{2,2;1} = \norm{\Lambda_- - \Xi}_{2,2;1}$ in $L_{\loc}^1(\R_+,\kappa_{M,N})$ as $|\Pi| \to 0$.
In particular, there exists a $\delta > 0$ such that $|\Pi| < \delta$ implies
\[
\int_0^t \norm{\Lambda - \Xi}_{2,2;1}^{\Pi} \, \d\kappa_{M,N} < \frac{\e}{3}.
\]
Finally, by Theorem \ref{thm.QC1}\ref{item.cheapRSsumQC}, we may shrink $\delta$ so that $|\Pi| < \delta$ also implies
\[
\sup_{0 \leq s \leq t}\Big\|\mathrm{RS}_{\Pi}^{M,N}(\Xi)(s) - \llbracket M, N \rrbracket^{\Xi}(s)\Big\|_1 < \frac{\e}{3}.
\]
Putting it all together, we conclude that if $|\Pi| < \delta$, then
\[
\sup_{0 \leq s \leq t}\e_{\Pi}(s) = \sup_{0 \leq s \leq t} \Big\|\mathrm{RS}_{\Pi}^{M,N}(\Lambda)(s) - \llbracket M, N \rrbracket^{\Lambda_-}(s)\Big\|_1 < \frac{\e}{3} + \frac{\e}{3} + \frac{\e}{3} = \e.
\]
This completes the proof of the second claim.
\pagebreak

For the third claim, note that if $\Lambda$ is $\norm{\cdot}_{2,2;1}$-continuous, $t \geq 0$, and $\Pi^* \in \cP_{\R_+}^*$, then
\begin{align*}
    \sup_{0 \leq s \leq t}\delta_{\Pi^*}(s) & \coloneqq \sup_{0 \leq s \leq t }\Bigg\| \sum_{r \in \Pi} \Lambda(r_*)[\Delta_rM^s, \Delta_rN^s] - \mathrm{RS}_{\Pi}^{M,N}(\Lambda)(s)\Bigg\|_1 \\
    & = \sup_{0 \leq s \leq t }\Bigg\|\sum_{r \in \Pi} (\Lambda(r_*) - \Lambda(r_-))[\Delta_rM^s, \Delta_rN^s]\Bigg\|_1 \\
    & \leq \sup_{0 \leq s \leq t}\sum_{r \in \Pi} \norm{\Lambda(r_*) - \Lambda(r_-)}_{2,2;1} \norm{\Delta_rM^s}_2 \norm{\Delta_rN^s}_2 \\
    & \leq \sup_{u,v \leq t : |u-v| \leq |\Pi|}\norm{\Lambda(u)-\Lambda(v)}_{2,2;1}\sup_{0 \leq s \leq t}\sum_{r \in \Pi}\frac{\norm{\Delta_rM^s}_2^2 + \norm{\Delta_rN^s}_2^2}{2} \\
    & = \kappa_{M,N}((0,t])\sup_{u,v \leq t : |u-v| \leq |\Pi|}\norm{\Lambda(u)-\Lambda(v)}_{2,2;1} \xrightarrow[\Pi^* \in \cP_{\R_+}^*]{|\Pi| \to 0} 0
\end{align*}
by the $\norm{\cdot}_{2,2;1}$-continuity of $\Lambda$.
This completes the proof.
\end{proof}

Already, the above result can be useful for calculating quadratic covariations, but the next result is the true to key to most calculations of interest, as we see in the next section.

\begin{thm}\label{thm.NCcondQC}
If $\Lambda$ is as in Proposition \ref{prop.LEQRSapprox}, then
\[
\mathbb{L}^1\text{-}\lim_{\Pi \in \cP_{\R_+}} \sum_{t \in \Pi} \E_{\mathsmaller{\cC}}[\Lambda(t_-)[X(t \wedge \cdot) - X(t_- \wedge \cdot), Y(t \wedge \cdot) - Y(t_- \wedge \cdot)] \mid \cC_{t_-}] = \into \Lambda(t-)[\d X(t), \d Y(t)].
\]
\end{thm}

\begin{proof}
Define
\[
\tilde{\Lambda}(t)[x,y] \coloneqq \E_{\mathsmaller{\cC}}[\Lambda(t)[x,y] \mid \cC_t] \in L^1(\E_{\mathsmaller{\cC}}) \qquad (t \geq 0, \; x \in L^2(\E_{\mathsmaller{\cA}}), \; y \in L^2(\E_{\mathsmaller{\cB}})).
\]
In this notation, our goal is to prove
\[
\mathbb{L}^1\text{-}\lim_{\Pi \in \cP_{\R_+}}\mathrm{RS}_{\Pi}^{X,Y}\big(\tilde{\Lambda}\big) = \into \Lambda_-[\d X, \d Y].
\]
By Proposition \ref{prop.LEQRSapprox}, this is equivalent to
\[
\mathbb{L}^1\text{-}\lim_{\Pi \in \cP_{\R_+}}\Big( \mathrm{RS}_{\Pi}^{X,Y}(\Lambda) - \mathrm{RS}_{\Pi}^{X,Y}\big(\tilde{\Lambda}\big)\Big) = \mathbb{L}^1\text{-}\lim_{\Pi \in \cP_{\R_+}}\mathrm{RS}_{\Pi}^{X,Y}\big(\Lambda - \tilde{\Lambda}\big) = 0.
\]
We first prove that
\[
\mathbb{L}^1\text{-}\lim_{\Pi \in \cP_{\R_+}} \mathrm{RS}_{\Pi}^{M,N}(\Lambda - \tilde{\Lambda}) = 0 \qquad \big(M \in \M_{\mathsmaller{\cA}}^{\infty}, \; N \in \M_{\mathsmaller{\cB}}^{\infty}, \; \Lambda \in \mathrm{Q}_0\big).
\]
To this end, let $M \in \M_{\mathsmaller{\cA}}^{\infty}$, $N \in \M_{\mathsmaller{\cB}}^{\infty}$, $\Lambda \in \mathrm{Q}_0$, and $t \geq 0$.
Also, write $\Xi \coloneqq \Lambda - \tilde{\Lambda}$.
Observe the following.
\begin{enumerate}[leftmargin=2\parindent]
\itemsep0em
    \item If $u \geq t$ and $(x,y) \in (\cA_u \times L^2(\cB_u,\E_{\mathsmaller{\cB}})) \cup (L^2(\cA_u, \E_{\mathsmaller{\cA}}) \times \cB_u)$, then $\Xi(t)[x,y] \in L^2(\cC_u, \E_{\mathsmaller{\cC}})$.
    \item $\norm{\Xi(t)}_{2,\infty;2} \vee \norm{\Xi(t)}_{\infty,2;2} \leq 2 \vertiii{\Lambda(t)}_2$.
    \item If $(x,y) \in L^2(\E_{\mathsmaller{\cA}}) \times L^2(\E_{\mathsmaller{\cB}})$, then $\E_{\mathsmaller{\cC}}[\Xi(t)[x,y] \mid \cC_t] = 0$.
\end{enumerate}
With these in mind, let $\Pi$ be a partition of $\R_+$, and define
\[
C_{r,s}(t) \coloneqq \E_{\mathsmaller{\cC}}[\Xi(r_-)[\Delta_r M^t, \Delta_r N^t]^* \Xi(s_-)[\Delta_s M^t, \Delta_s N^t]] \qquad (r,s \in \Pi).
\]
We claim that if $r \neq s$, then $C_{r,s}(t) = 0$.
Indeed, if $r < s$, in which case $r \leq s_-$, then
\begin{align*}
    C_{r,s}(t) & = \E_{\mathsmaller{\cC}}\big[\E_{\mathsmaller{\cC}}\big[\Xi(r_-)[\Delta_r M^t, \Delta_r N^t]^* \Xi(s_-)[\Delta_s M^t, \Delta_s N^t] \mid \cC_{s_-}\big]\big] \\
    & = \E_{\mathsmaller{\cC}}\big[\Xi(r_-)[\Delta_r M^t, \Delta_r N^t]^*\E_{\mathsmaller{\cC}}\big[ \Xi(s_-)[\Delta_s M^t, \Delta_s N^t] \mid \cC_{s_-}\big]\big] = 0.
\end{align*}
If $s < r$, then $C_{r,s}(t) = \overline{C_{s,r}(t)} = 0$ as well.
Writing $K_t \coloneqq \sup\{\vertiii{\Lambda(s)}_2 : 0 \leq s \leq t\} < \infty$, it follows that
\begin{align*}
    \Big\|\mathrm{RS}_{\Pi}^{M,N}(\Xi)(t)\Big\|_2^2 & = \Bigg\|\sum_{s \in \Pi} \Xi(s_-)[\Delta_s M^t, \Delta_s N^t]\Bigg\|_2^2 \\
    & = \sum_{r,s \in \Pi} \E_{\mathsmaller{\cC}}\big[\Xi(r_-)[\Delta_r M^t, \Delta_r N^t]^* \Xi(s_-)[\Delta_s M^t, \Delta_s N^t]\big] \\
    & = \sum_{s \in \Pi} \big\|\Xi(s_-)[\Delta_s M^t, \Delta_s N^t]\big\|_2^2 \leq \sum_{s \in \Pi} \norm{\Xi(s_-)}_{2,\infty;2}^2 \norm{\Delta_sM^t}_2^2 \norm{\Delta_sN^t}_{\infty}^2 \\
    & \leq 4K_t^2 \max_{s \in \Pi}\norm{\Delta_sN^t}_{\infty}^2\sum_{r \in \Pi}\norm{\Delta_rM^t}_2^2 = 4K_t^2\max_{s \in \Pi}\norm{\Delta_sN^t}_{\infty}^2 \norm{M(t) - M(0)}_2^2.
\end{align*}
Thus,
\[
\sup_{0 \leq s \leq t}\Big\|\mathrm{RS}_{\Pi}^{M,N}(\Xi)(s)\Big\|_2 \leq 2K_t\sup_{r,s \leq t : |r-s| \leq |\Pi|}\norm{N(r) - N(s)}_{\infty} \norm{M(t) - M(0)}_2 \xrightarrow[\Pi \in \cP_{\R_+}]{|\Pi| \to 0} 0
\]
by the $L^{\infty}$-continuity of $N$.

To complete the proof, we conduct one more ``$\frac{\e}{3}$ argument.''
Let $(M,N) \in \tbM_{\mathsmaller{\cA}}^2 \times \tbM_{\mathsmaller{\cB}}^2$ and $\Lambda$ be as in the statement.
Also, fix a sequence $(M_n)_{n \in \N}$ in $\M_{\mathsmaller{\cA}}^{\infty}$ converging to $M$ in $\M_{\mathsmaller{\cA}}^2$ and a sequence $(N_n)_{n \in \N}$ in $\M_{\mathsmaller{\cB}}^{\infty}$ converging to $N$ in $\M_{\mathsmaller{\cB}}^2$.
Finally, let $(\Pi_n)_{n \in \N}$ be a sequence of partitions of $\R_+$ such that $|\Pi_n| \to 0$ as $n \to \infty$.
If $\Lambda_n \coloneqq \Lambda^{\Pi_n} \in \mathrm{Q}_0$ and $\Xi_n \coloneqq \Lambda_n - \widetilde{\Lambda}_n$, then $\norm{\Xi - \Xi_n}_{2;2,1} \leq 2\norm{\Lambda - \Lambda_n}_{2,2;1}$, and
\[
\mathrm{RS}_{\Pi}^{M,N}(\Xi) = \mathrm{RS}_{\Pi}^{M,N}(\Xi) - \mathrm{RS}_{\Pi}^{M_m,N_m}(\Xi) + \mathrm{RS}_{\Pi}^{M_m,N_m}(\Xi - \Xi_n) + \mathrm{RS}_{\Pi}^{M_m,N_m}(\Xi_n) \qquad (n,m \in \N).
\]
Let $\e > 0$ and $t \geq 0$.
By Lemma \ref{lem.cheapbd}, there exists an $m \in \N$ such that
\[
\sup_{\Pi \in \cP_{\R_+}}\sup_{0 \leq s \leq t}\Big\|\mathrm{RS}_{\Pi}^{M,N}(\Xi)(s) - \mathrm{RS}_{\Pi}^{M_m,N_m}(\Xi)(s)\Big\|_1 < \frac{\e}{3}.
\]
For this fixed $m$, note that
\[
\sup_{0 \leq s \leq t}\Big\|\mathrm{RS}_{\Pi}^{M_m,N_m}(\Xi - \Xi_n)(s)\Big\|_1 \leq \int_0^t \norm{\Xi - \Xi_n}_{2,2;1}^{\Pi} \,d \kappa_{M_m,N_m} \leq 2\int_0^t \norm{\Lambda - \Lambda_n}_{2,2;1}^{\Pi} \,\d\kappa_{M_m,N_m}.
\]
Since $\Lambda$ is $\norm{\cdot}_{2,2;1}$-LLLB, $\Lambda_n \to \Lambda_-$ in $L_{\loc}^1(\R_+,\kappa_{M_m,N_m};B_2^{2,2;1})$ as $n \to \infty$, so there exists an $n \in \N$ such~that
\[
\int_0^t \norm{\Lambda_- - \Lambda_n}_{2,2;1} \,\d\kappa_{M_m,N_m} < \frac{\e}{6}.
\]
Similarly, since $\norm{\Lambda - \Lambda_n}_{2,2;1}$ is LLLB, $\norm{\Lambda - \Lambda_n}_{2,2;1}^{\Pi} \to (\norm{\Lambda - \Lambda_n}_{2,2;1})_- = \norm{\Lambda_- - (\Lambda_n)_-}_{2,2;1} = \norm{\Lambda_- - \Lambda_n}_{2,2;1}$ in $L_{\loc}^1(\R_+,\kappa_{M_m,N_m};B_2^{2,2;1})$ as $|\Pi| \to 0$.
Therefore, there exists a $\delta > 0$ such that $|\Pi| < \delta$ implies
\[
\int_0^t \norm{\Lambda - \Lambda_n}_{2,2;1}^{\Pi} \,\d\kappa_{M_m,N_m} < \frac{\e}{6}.
\]
Thus, for these fixed $n$ and $m$, $|\Pi| < \delta$ implies
\[
\sup_{0 \leq s \leq t}\Big\|\mathrm{RS}_{\Pi}^{M_m,N_m}(\Xi - \Xi_n)(s)\Big\|_1 < \frac{\e}{3}.
\]
Finally, by the previous paragraph, we can shrink $\delta$ so that $|\Pi| < \delta$ also implies
\[
\sup_{0 \leq s \leq t}\Big\|\mathrm{RS}_{\Pi}^{M_m,N_m}(\Xi_n)(s)\Big\|_1 < \frac{\e}{3}.
\]
Putting it all together, we conclude that $|\Pi| < \delta$ implies
\[
\sup_{0 \leq s \leq t}\Big\|\mathrm{RS}_{\Pi}^{M,N}(\Xi)(s)\Big\|_1 < \frac{\e}{3} + \frac{\e}{3} + \frac{\e}{3} = \e.
\]
This completes the proof.
\end{proof}

\begin{rem}[Application and interpretation]\label{rem.Poisson}
This remark may be safely skipped on a first reading.
Let $Z \colon \R_+ \to \cA_{\sa}$ be a free Poisson process with rate $\lambda > 0$, i.e., $Z(0) = 0$, $Z$ has free increments, and the distribution of $Z(t) - Z(s)$ is free Poisson\footnote{See \cite[Def.\ 12.12]{NS2006} for the definition of the free Poisson distribution, but beware of the typo:
In (12.14), $\lambda \tilde{\nu}$ should be $\tilde{\nu}$.} with jump size one and rate $\lambda(t-s)$ whenever $0 \leq s < t$.
Now, let $p \in [1,\infty)$, and define
\[
\mathring{Z}(t) \coloneqq Z(t) - \E[Z(t)] = Z(t) - \lambda t \qquad (t \geq  0)
\]
to be the compensated free Poisson process with rate $\lambda$.
By Example \ref{ex.freeinc}, the process $\mathring{Z}$ is a martingale.
By \cite[Lem.\ 1(1)]{Anshelevich2002}, $Z$ is $L^p$-continuous.
Thus, $Z(t) = \mathring{Z}(t) + \lambda t$ is $(L^p,L^{\infty})$-decomposable.
We claim that $\mathring{Z} \not\in \tbM^2$, i.e., $Z$ is not $\tilde{L}^2$-decomposable.
Indeed, \cite[Cor.\ 4]{Anshelevich2000} says that
\[
L^{\infty}\text{-}\lim_{\pi \in \cP_{[0,t]}}\sum_{s \in \pi} (\Delta_s Z)^2 = Z(t) \qquad (t \geq 0).
\]
On the other hand,
\[
\sum_{s \in \pi} \E[(\Delta_s Z)^2 \mid \cA_{s_-}] = \sum_{s \in \pi}\E[(\Delta_s Z)^2] = \sum_{s \in \pi}(1+\lambda \Delta s)\lambda\Delta s \xrightarrow[\pi \in \cP_{[0,t]}]{|\pi| \to 0} \lambda t \qquad (t \geq 0)
\]
in $\cA$ by Lemma \ref{lem.condexpcalc}\ref{item.freeinc} below and the fact that $\E[(Z(t) - Z(s))^2] = (1+\lambda(t-s))\lambda(t-s)$ whenever $0 \leq s < t$.
Since $Z \neq (\lambda t)_{t \geq 0}$, it follows from Theorem \ref{thm.NCcondQC} that $Z$ cannot be $\tilde{L}^2$-decomposable, as claimed.

What is going on conceptually in the previous paragraph is rather subtle and requires further comments on classical stochastic analysis to explain.
Let $(\Om,\sF,(\sF_t)_{t \geq 0},P)$ be a filtered probability space satisfying the usual conditions, and let $U$ and $V$ be classical semimartingales with jumps, i.e., $U = U_0 + M_U+A_U$ and $V=V_0 + M_V+A_V$ for some RCLL (right-continuous with left limits) local martingales $M_U,M_V$ and RCLL FV processes $A_U,A_V$ with $(M_U)_0 = (M_V)_0 = (A_U)_0 = (A_V)_0 = 0$ almost surely.
As in the continuous case,
\[
L^0\text{-}\lim_{\pi \in \cP_{[0,t]}}\sum_{s \in \pi} \Delta_sU \, \Delta_sV = U_tV_t - U_0V_0 - \int_0^t U_{s-}\,\d V_s - \int_0^t V_{s-} \,\d U_s \qquad (t \geq 0),
\]
and $[U,V] \coloneqq UV - U_0V_0 - \into U_{s-}\,\d V_s - \into V_{s-} \,\d U_s$ is an RCLL FV process called the quadratic covariation of $U$ and $V$ (\cite[\S{II.6}]{Protter2005}).
If the variation process of $[U,V]$ is locally integrable, then there is another kind of quadratic variation.
Indeed, in this case, there exists a unique-up-to-indistinguishability predictable RCLL FV process $C$ with $C_0 = 0$ such that $[U,V] - C$ is an RCLL local martingale (\cite[\S{III.5}]{Protter2005}).
We write $\la U,V \ra\hspace{-0.25mm} \coloneqq\hspace{-0.25mm} C$ and call $\la U,V \ra$ the \textbf{predictable} or \textbf{conditional quadratic covariation} of $U$ and $V$.
In special situations (see, e.g., \cite[\S{VI.31}]{RW2}), one may compute $\la U, V \ra_t$ as an appropriate limit of $\sum_{s \in \pi} \E_P[\Delta_s U \, \Delta_sV \mid \sF_{s_-}]$ as $|\pi| \to 0$.
Also, if $U$ and $V$ are continuous semimartingales, then $[U,V] = \la U, V \ra$.
Accordingly, we should conceptualize Theorem \ref{thm.NCcondQC} as the statement that ``for $\tilde{L}^2$-decomposable processes, the noncommutative quadratic covariation agrees with the noncommutative predictable quadratic covariation,'' and we should interpret the condition that $(M,N) \in \tbM_{\mathsmaller{\cA}}^2 \times \tbM_{\mathsmaller{\cB}}^2$ as a kind of continuity.
Finally, a prototypical example of when $[U,V]$ differs from $\la U,V \ra$ is when $U=V=$ a classical Poisson process.
Therefore, it is reasonable to expect the same phenomenon---and, consequently, the same lack of ``continuity''---from the free Poisson process.
This is what we witness in the previous paragraph.
\end{rem}

The previous two results allow us to upgrade Example \ref{ex.elemtrtripQC}.

\begin{ex}\label{ex.elemtrtripQCupgrade}
By Proposition \ref{prop.LEQRSapprox}, all the statements in Example \ref{ex.elemtrtripQC} remain true when ``continuous'' is replaced with ``LCLB.''
In this case, we also have the limiting expressions
\begin{align*}
    \int_0^t\Lambda(s)[\d X(s), \d Y(s)] & = L^1\text{-}\lim_{\pi \in \cP_{[0,t]}}\sum_{s \in \pi}\E[\Lambda(s_-)[\Delta_sX, \Delta_sY] \mid \cC_{s_-}] \, \text{ and} \\
    \int_0^t P(\mathbf{X}(s),\d\mathbf{Y}(s),\d\mathbf{Z}(s)) & = L^1\text{-}\lim_{\pi \in \cP_{[0,t]}}\sum_{s \in \pi}\E^{\oplus m}[P(\mathbf{X}(s_-),\Delta_s\mathbf{Y},\Delta_s\mathbf{Z}) \mid \cA_{s_-}^m]
\end{align*}
by Theorem \ref{thm.NCcondQC}.
\end{ex}

As mentioned above, in many cases of interest, Theorem \ref{thm.NCcondQC} also allows us to find explicit formulas for the remaining uncalculated quadratic covariations in the previous example (e.g., $\into A(t)\,\d X(t)\,B(t)\,\d Y(t)\,C(t)$).
We undertake some such calculations in Section \ref{sec.QCexs}.

\subsection{Application: Burkholder--Davis--Gundy inequalities}\label{sec.BDG}

In this section, we use our theory of quadratic covariation to extend the discrete-time noncommutative (NC) Burkholder--Davis--Gundy (BDG) inequalities of Pisier--Xu \cite{PX1997} to the present continuous-time setting when $p \geq 2$.
As a consequence, we obtain $L^p$-norm estimates on stochastic integrals.
We begin by recalling the discrete-time NC BDG inequalities.

\begin{thm}[Discrete-time NC BDG inequalities \cite{PX1997}]\label{thm.discBDG}
There exist increasing families $(\alpha_p)_{p \geq 2}$ and $(\beta_p)_{p \geq 2}$ of strictly positive constants such that the following holds.
If $p \in [2,\infty)$, $(\cM,\tau)$ is a $\mathrm{C}^*$-probability space, $N \in \N$, $(\cM_n)_{n = 0}^N$ is a finite filtration  of $\cM$, and $x \colon \{0,\ldots,N\} \to L^p(\tau)$ is a discrete-time $L^p$-martingale, then
\[
\alpha_p^{-1}\norm{x}_{\cH^p(\cM)} \leq \max_{0 \leq n \leq N}\norm{x_n}_p = \norm{x_N}_p \leq \beta_p\norm{x}_{\cH^p(\cM)},
\]
where
\[
\norm{x}_{\cH^p(\cM)} \coloneqq \max\Bigg\{\Bigg\|x_0^*x_0 + \sum_{n=1}^N (x_n - x_{n-1})^*(x_n - x_{n-1})\Bigg\|_{\frac{p}{2}}^{\frac{1}{2}}, \Bigg\|x_0x_0^* + \sum_{n=1}^N (x_n - x_{n-1})(x_n - x_{n-1})^*\Bigg\|_{\frac{p}{2}}^{\frac{1}{2}}\Bigg\}.
\]
Furthermore,
\[
\norm{x}_{\cH^2(\cM)}^2 = \tau\bigg[x_0^*x_0 + \sum_{n=1}^N (x_n - x_{n-1})^*(x_n - x_{n-1})\bigg] = \tau\bigg[x_0x_0^* + \sum_{n=1}^N (x_n - x_{n-1})(x_n - x_{n-1})^*\bigg].
\]
\end{thm}

\begin{rem}\label{rem.BDG}
In truth, the result quoted from \cite{PX1997} is stated and proven only in the $\mathrm{W}^*$ setting.
However, it is easy to see from the development in Appendix \ref{sec.CstarLp} that the result in the $\mathrm{W}^*$ setting actually implies the result stated above in the $\mathrm{C}^*$ setting.
Also, \cite{PX1997} contains a similar result for $p \in (1,2)$ with a different norm.
At this time, we are unable to adapt this regime to our continuous-time setting.
\end{rem}

Aside from our theory of quadratic covariation, the key to transferring the result above over to the continuous-time setting is a short list of basic facts about noncommutative $L^p$ convergence.

\begin{lem}\label{lem.basicLp}
Suppose $1 \leq p < q < \infty$.
\begin{enumerate}[label=(\roman*),font=\normalfont]
    \item Let $(a_n)_{n \in \N}$ be a sequence in $L^p(\E)$ and $a \in L^1(\E)$.
    If $\liminf_{n \to \infty} \norm{a_n}_p < \infty$ and $a_n \to a$ in $L^1(\E)$ as $n \to \infty$, then $a \in L^p(\E)$, and $\norm{a}_p \leq \liminf_{n \to \infty} \norm{a_n}_p$.\label{item.Fatou}
    \item Let $(a_n)_{n \in \N}$ be a sequence in $L^q(\E)$ and $a \in L^1(\E)$.
    If $\sup\{\norm{a_n}_q : n \in \N\} < \infty$ and $a_n \to a$ in $L^1(\E)$ as $n \to \infty$, then $a \in L^q(\E)$, and $a_n \to a$ in $L^p(\E)$ as $n \to \infty$.\label{item.Lrbdd}
    \item If $a \in L^q(\E)$, then $\lim_{p \nearrow q}\norm{a}_p = \norm{a}_q$.\label{item.limofpnorm}
\end{enumerate}
\end{lem}

\begin{proof}
We take each item in turn.

\ref{item.Fatou}
If $p=1$, then this is obvious, so assume $p > 1$.
If $b \in \cA$, then
\[
|\E[ab]| = \lim_{n \to \infty}|\E[a_nb]| \leq \liminf_{n \to \infty}(\norm{a_n}_p\norm{b}_{p'}) = \norm{b}_{p'}\liminf_{n \to \infty}\norm{a_n}_p,
\]
where $1/p+1/p' = 1$.
Since $1 < p,p' < \infty$, duality for noncommutative $L^p$ spaces tells us that $a \in L^p(\E)$ and $\norm{a}_p \leq \liminf_{n \to \infty} \norm{a_n}_p$, as desired.

\ref{item.Lrbdd} First, observe $a \in L^q(\E)$ and $\norm{a}_q \leq \liminf_{n \to \infty}\norm{a_n}_q \leq \sup\{\norm{a_n}_q : n \in \N\} < \infty$ by the first item.
For the second claim, the case $p=1$ is obvious, so we assume $p > 1$.
If $b \in \cA$ and $M > 0$, then
\begin{align*}
    \norm{b}_p^p & = \E[|b|^p] = \int_{\R_+} x^p \,\mu_{|b|}(\d x) = \int_{[0,M]} x^p \,\mu_{|b|}(\d x) + \int_{(M,\infty)} x^p \,\mu_{|b|}(\d x) \\
    & \leq M^{p-1}\int_{[0,M]} x \,\mu_{|b|}(\d x) + \frac{1}{M^{q-p}}\int_{(M,\infty)} x^q \,\mu_{|b|}(\d x) \\
    & \leq M^{p-1}\int_{\R_+} x \,\mu_{|b|}(\d x) + \frac{1}{M^{q-p}}\int_{\R_+} x^q \,\mu_{|b|}(\d x) = M^{p-1}\norm{b}_1 + \frac{1}{M^{q-p}}\norm{b}_q^q.
\end{align*}
By density, this inequality extends to all $b \in L^q(\E)$.
Letting $C \coloneqq \sup\{\norm{a_n}_q : n \in \N\}$, this gives
\[
\norm{a_n-a_m}_p^p \leq M^{p-1}\norm{a_n-a_m}_1 + \frac{2C^q}{M^{q-p}} \qquad (n,m \in \N).\pagebreak
\]
Now, let $\e > 0$, and choose an $M > 0$ such that $2C^q/M^{q-p} < \e^p/2$ and an $N \in \N$ such that $m,n \geq N$ implies $\norm{a_n-a_m}_1 < \e^p/(2M^{p-1})$.
Then $m,n \geq N$ also implies $\norm{a_n - a_m}_p < \e$.
We have just shown that $(a_n)_{n \in \N}$ is Cauchy and therefore convergent in $L^p(\E)$.
By the uniqueness of $L^p$ limits, we conclude that $a_n \to a$ in $L^p(\E)$ as $n \to \infty$, as desired.

\ref{item.limofpnorm} First, suppose $a \in \cA$ is non-zero.
(The $a=0$ case is obvious.)
Note that
\begin{align*}
    \norm{a}_p^p & = \int_{[0,1]} x^p \, \mu_{|a|}(\d x) + \int_{(1,\infty)} x^p \, \mu_{|a|}(\d x) \xrightarrow{p \nearrow q} \int_{[0,1]} x^q \, \mu_{|a|}(\d x) + \int_{(1,\infty)} x^q \, \mu_{|a|}(\d x) = \norm{a}_q^q
\end{align*}
by the dominated convergence theorem for the first term and the monotone convergence theorem for the second.
Now, the function $f \colon (0,\infty) \times [0,\infty) \to \R$ defined by $f(x,p) \coloneqq x^p$ is jointly continuous.
Therefore,
\[
\lim_{p \nearrow q}\norm{a}_p = \lim_{p \nearrow q}f\big(p^{-1},\norm{a}_p^p\big) = f\big(q^{-1},\norm{a}_q^q\big) = \norm{a}_q,
\]
as desired.
An elementary ``$\frac{\e}{2}$ argument'' then extends this identity to all $a \in L^q(\E)$.
\end{proof}

\begin{thm}[Continuous-time NC BDG Inequalities]\label{thm.BDG}
Let $(\alpha_p)_{p \geq 2}$ and $(\beta_p)_{p \geq 2}$ be as in Theorem \ref{thm.discBDG}.
Also, let $M \colon \R_+ \to L^p(\E)$ be an $L^p$-martingale and $t \geq 0$.
Suppose there exists a sequence $P = (\pi_n)_{n \in \N}$ in $\cP_{[0,t]}$ such that $|\pi_n| \to 0$ as $n \to \infty$ and the limits
\begin{align*}
    [M^*,M]_t^P & \coloneqq L^1\text{-}\lim_{n \to \infty}\sum_{s \in \pi_n} \Delta_sM^* \, \Delta_sM \in L^1(\cA_t,\E) \; \text{ and} \\
    [M,M^*]_t^P & \coloneqq L^1\text{-}\lim_{n \to \infty}\sum_{s \in \pi_n} \Delta_sM \, \Delta_sM^* \in L^1(\cA_t,\E)
\end{align*}
exist.
Then $[M^*,M]_t^P, [M,M^*]_t^P \in L^{p/2}(\cA_t,\E)$, and
\[
\alpha_p^{-1}\norm{M}_{\cH_t^p(\cA)} \leq \norm{M(t)}_p = \sup_{0 \leq s \leq t} \norm{M(s)}_p \leq \beta_p\norm{M}_{\cH_t^p(\cA)},
\]
where
\[
\norm{M}_{\cH_t^p(\cA)} \coloneqq \max\bigg\{\Big\|M(0)^*M(0) + [M^*,M]_t^P\Big\|_{\frac{p}{2}}^{\frac{1}{2}}, \Big\|M(0)M(0)^* + [M,M^*]_t^P\Big\|_{\frac{p}{2}}^{\frac{1}{2}}\bigg\}.
\]
Furthermore, $\norm{M(t)}_2^2 = \E\big[M^*(0)M(0) + [M^*,M]_t^P\big] = \E\big[M(0)M^*(0) + [M,M^*]_t^P\big]$.
\end{thm}

\begin{proof}
We leave the $p=2$ case, i.e., the last sentence of the statement, to the reader and assume $p > 2$ throughout the proof.
If $\pi \in \cP_{[0,t]}$, then the discrete-time process $\pi \ni s \mapsto M_{\pi}(s) \coloneqq M(s) \in L^p(\E)$ is an $L^p$-martingale with respect to the filtration $(\cA_s)_{s \in \pi}$.
By the discrete-time NC BDG inequalities,
\begin{equation}
    \alpha_p^{-1}\norm{M_{\pi}}_{\mathcal{H}^p(\cA)} \leq \max_{s \in \pi}\norm{M(s)}_p  \leq \beta_p\norm{M_{\pi}}_{\mathcal{H}^p(\cA)}. \label{eq.partBDG}
\end{equation}
Since $t \in \pi \subseteq [0,t]$,
\[
\max_{s \in \pi}\norm{M(s)}_p = \norm{M(t)}_p = \sup_{0 \leq s \leq t} \norm{M(s)}_p
\]
by \eqref{eq.norminc}.
Since the right-hand side is independent of $\pi$, we conclude from the first inequality in \eqref{eq.partBDG} and the definition of $\norm{\cdot}_{\cH^p(\cA)}$ that the sequences
\[
\Bigg(M(0)^*M(0) + \sum_{s \in \pi_n}\Delta_sM^* \, \Delta_sM\Bigg)_{n \in \N} \; \text{ and } \; \Bigg(M(0)M(0)^* + \sum_{s \in \pi_n}\Delta_sM \, \Delta_sM^*\Bigg)_{n \in \N}
\]
are bounded in $L^{p/2}(\cA_t,\E)$.
Since $M(0)^*M(0),M(0)M(0)^* \in L^{p/2}(\cA_t,\E)$, the sequences
\[
\Bigg(\sum_{s \in \pi_n}\Delta_sM^* \, \Delta_sM\Bigg)_{n \in \N} \; \text{ and } \; \Bigg(\sum_{s \in \pi_n}\Delta_sM \, \Delta_sM^*\Bigg)_{n \in \N}
\]
are bounded in $L^{p/2}(\cA_t,\E)$.
Since the left sequence converges in $L^1(\cA_t,\E)$ to $[M^*,M]_t^P$ and the right sequence converges in $L^1(\cA_t,\E)$ to $[M,M^*]_t^P$, Lemma \ref{lem.basicLp}\ref{item.Lrbdd} yields that
\[
\big[M^{\e_1},M^{\e_2}\big]_t^P \in L^{\frac{p}{2}}(\cA_t,\E) \; \text{ and } \; \big[M^{\e_1},M^{\e_2}\big]_t^P = L^{\frac{q}{2}}\text{-}\lim_{n \to \infty}\sum_{s \in \pi_n}\Delta_s M^{\e_1} \, \Delta_s M^{\e_2}\pagebreak
\]
whenever $(\e_1,\e_2) \in \{(\ast,1),(1,\ast)\}$ and $q \in [2,p)$.
Therefore, two applications of Lemma \ref{lem.basicLp}\ref{item.Fatou} and the first inequality in \eqref{eq.partBDG} yield
\[
\alpha_p^{-1}\norm{M}_{\cH_t^p(\cA)} \leq \alpha_p^{-1}\liminf_{n \to \infty}\norm{M_{\pi_n}}_{\cH_t^p(\cA)} \leq \norm{M(t)}_p,
\]
as required.
Finally, let $q \in (2,p)$.
If $\pi \in \cP_{[0,t]}$, then \eqref{eq.partBDG} says
\[
\norm{M(t)}_q \leq \beta_q \norm{M_{\pi}}_{\cH^q(\cA)} \leq \beta_p \norm{M_{\pi}}_{\cH^q(\cA)}
\]
because $(\beta_r)_{r \geq 2}$ is increasing.
Therefore, taking $\pi = \pi_n$ and $n \to \infty$, we obtain
\[
\norm{M(t)}_q \leq \beta_p \norm{M}_{\cH_t^q(\cA)}.
\]
Three applications of Lemma \ref{lem.basicLp}\ref{item.limofpnorm} then yield
\[
\norm{M(t)}_p = \lim_{q \nearrow p}\norm{M(t)}_q \leq \beta_p\lim_{q \nearrow p}\norm{M}_{\cH_t^q(\cA)} = \beta_p\norm{M}_{\cH_t^p(\cA)},
\]
as desired.
This completes the proof.
\end{proof}

\begin{ex}\label{ex.M2tilde}
By Theorem \ref{thm.QC1}, if $p \geq 2$ and $M \colon \R_+ \to L^p(\E)$ is an $L^p$-martingale such that $M \in \tbM^2$, e.g., if $M \in \tbM^p$, then $M$ satisfies the hypotheses of Theorem \ref{thm.BDG} with any sequence $P = (\pi_n)_{n \in \N}$ in $\cP_{[0,t]}$ such that $|\pi_n| \to 0$ as $n \to \infty$, in which case $[M^*,M]_t^P = \int_0^t \d M^*(s)\,\d M(s)$ and $[M,M^*]_t^P = \int_0^t \d M(s)\,\d M^*(s)$.
Thus, Theorem \ref{thm.NCBDG} is proven.
\end{ex}

Since we have computed quadratic variations of stochastic integrals in Theorem \ref{thm.QCSI}, we get the following equivalence of $L^p$ norms of stochastic integrals.

\begin{thm}\label{thm.LpnormSI}
Let $2 \leq p < \infty$ and $\alpha_p,\beta_p \in (0,\infty)$ be the constants from Theorem \ref{thm.discBDG}.
Suppose $M \in \tbM^2$ and $H \in \tilde{\cI}(M)$ are such that $\int_0^t H[\d M] \in L^p(\cB_t,\E_{\mathsmaller{\cB}})$ for all $t \geq 0$.
If $t \geq 0$, then
\[
\int_0^t H(s)[\d M(s)]^*H[\d M(s)], \, \int_0^t H(s)[\d M(s)]H[\d M(s)]^* \in L^{\frac{p}{2}}(\cB_t,\E_{\mathsmaller{\cB}}),
\]
and
\[
\alpha_p^{-1}\norm{H}_{\cH_{t,M}^p(\cB)} \leq \Bigg\|\int_0^t H(s)[\d M(s)]\Bigg\|_p \leq \beta_p\norm{H}_{\cH_{t,M}^p(\cB)},
\]
where
\[
\norm{H}_{\cH_{t,M}^p(\cB)} \coloneqq \max\Bigg\{\Bigg\|\int_0^t H(s)[\d M(s)]^*H[\d M(s)]\Bigg\|_{\frac{p}{2}}^{\frac{1}{2}}, \Bigg\|\int_0^t H(s)[\d M(s)]H[\d M(s)]^*\Bigg\|_{\frac{p}{2}}^{\frac{1}{2}}\Bigg\}.
\]
\end{thm}

\begin{proof}
Define $N \coloneqq \into H[\d M]$.
Since $N \in \tbM_{\mathsmaller{\cB}}^2$ by Proposition \ref{prop.stochintLtilde} and $N$ is an $L^p$-martingale by assumption, the continuous-time NC BDG inequalities give
\[
\alpha_p^{-1}\norm{N}_{\cH_t^p(\cB)} \leq \norm{N(t)}_p \leq \beta_p\norm{N}_{\cH_t^p(\cB)}.
\]
Now, since it is easy to see that $\big(\into H[\d M]\big)^* = \into H[\d M]^*$, Theorem \ref{thm.QCSI} says
\begin{align*}
    \into \d N_t^*\,\d N_t & = \into H(t)[\d M(t)]^*H(t)[\d M(t)] \; \text{ and} \\
    \into \d N_t\,\d N_t^* &= \into H(t)[\d M(t)]H(t)[\d M(t)]^*.
\end{align*}
Thus, $\norm{N}_{\cH_t^p(\cB)} = \norm{H}_{\cH_{t,M}^p(\cB)}$, which completes the proof.
\end{proof}

\begin{ex}
If $H \in \EP$ and $M \colon \R_+ \to L^p(\E)$ is an $L^p$-martingale belonging to $\tbM^2$, then $\into H[\d M]$ satisfies the hypotheses of Theorem \ref{thm.LpnormSI}.
\end{ex}
\pagebreak

Finally, recall that our development of the stochastic integral used an ``It\^{o} contraction'' to give a bound on the $L^2$ norm of the stochastic integral.
This was enough to construct the integral of a large class of integrands, but it left a conceptual gap.
Specifically, we were left with no It\^{o} isometry.
The $p=2$ case of Theorem \ref{thm.LpnormSI} fills this gap by providing us a noncommutative analog of \eqref{eq.clII}.

\begin{cor}[Noncommutative It\^{o} isometry]\label{cor.NCII}
If $M \in \tbM^2$ and $H \in \tilde{\cI}(M)$, then
\[
\Bigg\|\int_0^t H(s)[\d M(s)]\Bigg\|_2^2 = \E\Bigg[ \int_0^t H(s)[\d M(s)]^*\,H(s)[\d M(s)] \Bigg] \qquad (t \geq 0).
\]
\end{cor}

\begin{proof}
This follows from (the proof of) Theorem \ref{thm.LpnormSI} and the last sentence in Theorem \ref{thm.BDG}.
\end{proof}

\subsection{Examples}\label{sec.QCexs}

In this section, we demonstrate how Theorem \ref{thm.NCcondQC} can be used to compute quadratic covariations.
The key method of our examples will be to turn knowledge of ``noncommutative conditional covariances'' like $\E[(M(t) - M(s)) a (N(t) - N(s)) \mid \cA_s]$ into formulas for $\into \Lambda(t)[\d M(t), \d N(t)]$ when $\Lambda$ is a trace triprocess.
Specifically, when $M$ and $N$ have some kind of independent increments condition, it often happens that
\[
\frac{\E[(M(t) - M(s))a (N(t) - N(s)) \mid \cA_s]}{\E[(M(t) - M(s))(N(t) - N(s))]}
\]
is independent of $s$ and $t$ in an appropriate sense.
Our goal is to show that the latter property makes it possible to compute $\into \Lambda(t)[\d M(t), \d N(t)]$ explicitly for many trace triprocesses $\Lambda$.

\begin{lem}\label{lem.complexDoleans}
If $M,N \colon \R_+ \to L^2(\E)$ are right-continuous $L^2$-martingales, then there exists a unique complex Borel measure $\mu_{M,N}$ on $\R_+$ such that $\mu_{M,N}(\{0\}) = 0$ and
\[
\mu_{M,N}((s,t]) = \E[(M(t) - M(s))(N(t) - N(s))] \qquad (0 \leq s \leq t).
\]
Of course, $\mu_{M,N} = \mu_{N,M}$ and $\mu_{M,M^*} = \kappa_M = \kappa_{M^*}$ (Lemma \ref{lem.fakeDoleans}).
\end{lem}

\begin{proof}
By the polarization identity, if $0 \leq s \leq t$, then
\begin{align*}
    \E[(M(t) - M(s))(N(t) - N(s))] & = \ip{ N(t) - N(s), M^*(t) - M^*(s) }_2 \\
    & = \frac{1}{4}\sum_{k=0}^3 i^k\norm{N(t) - N(s) + i^k(M^*(t) - M^*(s))}_2^2 \\
    & = \frac{1}{4}\sum_{k=0}^3 i^k\kappa_{N + i^kM^*}((s,t]).
\end{align*}
This takes care of existence, and uniqueness is standard.
\end{proof}

In what follows, $\potimes$ is the Banach space projective tensor product over $\C$;
see \cite[\S2.2]{NikitopoulosIto} for a concise review.
Observe that the maps $\#_k,\#_k^{\E} \colon \cA^{\otimes(k+1)} \to \mathbb{B}_k(\cA)$ (Notation \ref{nota.tens}) extend uniquely to bounded complex-linear maps $\cA^{\potimes(k+1)} \to \mathbb{B}_k(\cA)$, which we notate the same way.
Now, for $k \in \N$, a map $U \colon \R_+ \to \cA^{\potimes k}$ is called \textbf{adapted} if $U(t) \in \cA_t^{\potimes k} \subseteq \cA^{\potimes k}$ for all $t \geq 0$.

\begin{lem}\label{lem.tensadap}
Let $k \in \N$.
If $t \geq 0$ and $u \in \cA_t^{\potimes(k+1)}$, then $\#_k(u) \in \cT_{k,t}^{\C}$, and $\#_k^{\E}(u) \in \cT_{k,t}^{\C}$.
In particular, if $U \colon \R_+ \to \cA^{\potimes(k+1)}$ is adapted, then $\#_k(U)$ and $\#_k^{\E}(U)$ are complex--$k$-linear trace $k$-processes.
\end{lem}

\begin{proof}
Define $\Xi \coloneqq \#_k(u)$.
Since $u \in \cA_t^{\potimes (k+1)}$, there exist sequences $(a_n^1)_{n \in \N},\ldots,(a_n^{k+1})_{n \in \N}$ in $\cA_t$ such~that
\[
\sum_{n = 1}^{\infty} \big\|a_n^1\big\| \cdots \big\|a_n^{k+1}\big\| < \infty \; \text{ and } \; u = \sum_{n = 1}^{\infty} a_n^1 \otimes \cdots \otimes a_n^{k+1}.
\]
If $u_N \coloneqq \sum_{n=1}^N a_n^1 \otimes \cdots \otimes a_n^{k+1} \in \cA_t^{\otimes (k+1)}$ and $\Xi_N \coloneqq \#_k(u_N) \in \mathbb{B}_k(\cA)$ for all $N \in \N$, then $\Xi_N \in \cT_{k,t}^{\C,0}$, and $\vertiii{\Xi - \Xi_N}_k \leq \norm{u - u_N}_{\cA_t^{\potimes (k+1)}} \to 0$ as $N \to \infty$.
Thus, $\Xi \in \cT_{k,t}^{\C}$, as desired.
The argument for $\Xi = \#_k^{\E}(u)$ is similar, so we leave it to the reader.
\end{proof}

\begin{nota}\label{nota.magic}
For a bounded complex-linear map $\Gamma \colon \cA \to L^1(\E)$, write $\cM_{\Gamma} \colon \cA \potimes \cA \potimes \cA \to L^1(\E)$ and $\cM_{\Gamma}^{\E} \colon \cA \potimes \cA \potimes \cA \to \cA$ for the bounded complex-linear maps determined respectively by
\[
\cM_{\Gamma}(a \otimes b \otimes c) = a \,\Gamma(b)\,c \; \text{ and } \; \cM_{\Gamma}^{\E}(a \otimes b \otimes c) = \E[a\,\Gamma(b)]\,c \qquad (a,b,c \in \cA).
\]
\end{nota}

\begin{thm}\label{thm.magicQC}
Let $M,N \in \tbM^2$, and assume that there is a bounded complex-linear map $\Gamma \colon \cA \to L^1(\E)$ such that
\begin{equation}
    \E[(M(t) - M(s))a(N(t) - N(s)) \mid \cA_s] = \Gamma(a) \, \mu_{M,N}((s,t]) \qquad (0 \leq s < t, \; a \in \cA_s). \label{eq.Gammahyp}
\end{equation}
If $U \colon \R_+ \to \cA^{\potimes 3}$ is adapted and $\norm{\cdot}_{\cA^{\potimes 3}}$-LCLB, then\label{item.tensQC}
\begin{align*}
    \into U\sh_2^{\mathsmaller{(\E)}}[\d M(t), \d N(t)] & = \into \mathcal{M}_{\Gamma}^{(\E)}(U(t)) \, \mu_{M,N}(\d t).
\end{align*}
To be clear, the above is shorthand for two identities:
one with $\sh_2$ on the left-hand side and $\cM_{\Gamma}$ on the right-hand side, and one with $\sh_2^{\mathsmaller{\E}}$ on the left-hand side and $\mathcal{M}_{\Gamma}^{\E}$ on the right-hand side.
\end{thm}

\begin{proof}
First, observe that \eqref{eq.Gammahyp} implies
\begin{equation}
     \E\big[u \sh_2^{\mathsmaller{(\E)}}[M(t) - M(s), N(t) - N(s)] \mid \cA_s\big] = \cM_{\Gamma}^{(\E)}(u) \,\mu_{M,N}((s,t]) \qquad (s < t, \; u \in \cA_s^{\potimes 3}). \label{eq.EGamma}
\end{equation}
Now, if $U \colon \R_+ \to \cA^{\potimes 3}$ is adapted and $\norm{\cdot}_{\cA^{\potimes 3}}$-LCLB, then $\#_2^{(\E)}(U) \colon \R_+ \to \mathbb{B}_k(\cA)$ is a $\vertiii{\cdot}_2$-LCLB trace triprocess by Lemma \ref{lem.tensadap}.
Consequently, if $t \geq 0$, then
\begin{align*}
    \int_0^t U(s)\sh_2^{\mathsmaller{(\E)}}[\d M(s), \d N(s)] & = L^1\text{-}\lim_{\pi \in \cP_{[0,t]}}\sum_{s \in \pi} \E\big[U(s_-)\sh_2^{\mathsmaller{(\E)}}[\Delta_sM, \Delta_sN] \mid \cA_{s_-}\big] \tag{Thm.\ \ref{thm.NCcondQC}}\\
    & = L^1\text{-}\lim_{\pi \in \cP_{[0,t]}}\sum_{s \in \pi} \mathcal{M}_{\Gamma}^{(\E)}(U(s_-)) \,\mu_{M,N}((s_-,s]) \tag{Eq.\ \eqref{eq.EGamma}} \\
    & = L^1\text{-}\lim_{\pi \in \cP_{[0,t]}}\int_0^t \mathcal{M}_{\Gamma}^{(\E)}(U)^{\pi} \,\d\mu_{M,N} = \int_0^t \mathcal{M}_{\Gamma}^{(\E)}(U)\,\d\mu_{M,N}. \tag{Lem.\ \ref{lem.HPi}\ref{item.HPipw}, DCT}
\end{align*}
Note that Lemma \ref{lem.HPi}\ref{item.HPipw} applies (after writing $\mu_{M,N}$ as a complex-linear combination of locally finite positive measures) in the last line because $\mathcal{M}_{\Gamma}^{(\E)}(U) \colon \R_+ \to L^1(\E)$ is LCLB.
\end{proof}

Next, we make an observation that will allow us in certain situations to upgrade the formulas in Theorem \ref{thm.magicQC} to formulas for $\into \Lambda(t)[\d M(t), \d N(t)]$ when $\Lambda$ is a more general trace triprocess.
Though the statement is somewhat technical, the result below is simple in spirit:
Certain ``trace terms'' vanish when one plugs martingale increments into them.

\begin{lem}\label{lem.T2s}
Let $s \geq 0$ and $\Xi \in \cT_{2,s}^0$.
For each $\e = (\e_1,\e_2) \in \mathcal{S} \coloneqq \{1,\ast\}^2$, there exist $u_{\e}^1,u_{\e}^2,v_{\e}^1,v_{\e}^2 \in \cA_s^{\otimes 3}$ such that for all $L^2$-martingales $M,N \colon \R_+ \to L^2(\E)$ and all $t \geq s$,
\begin{align*}
    \Xi[M(&t) - M(s), N(t) - N(s)] \\
    & = \sum_{\e \in \mathcal{S}} \Big( u_{\e}^1\sh_2[(M(t) - M(s))^{\e_1}, (N(t) - N(s))^{\e_2}] + u_{\e}^2\sh_2[(N(t) - N(s))^{\e_1}, (M(t) - M(s))^{\e_2}] \\
    & \hspace{12.5mm} + v_{\e}^1\sh_2^{\mathsmaller{\E}}[(M(t) - M(s))^{\e_1}, (N(t) - N(s))^{\e_2}] + v_{\e}^2\sh_2^{\mathsmaller{\E}}[(N(t) - N(s))^{\e_1}, (M(t) - M(s))^{\e_2}]\Big).
\end{align*}
Moreover, if $\Xi \in \cT_{2,s}^{\C,0}$, then we may take $u_{\e}^1 = u_{\e}^2 = v_{\e}^1 = v_{\e}^2 = 0$ for $\e \in \{(1,\ast),(\ast,1),(\ast,\ast)\}$.
\end{lem}

\begin{proof}
For each $i =1,\ldots,72$, let $a_i \in \cA$.
Now, define
\begin{align*}
    \Xi[x,y] & \coloneqq a_1xa_2ya_3 + a_4x^*a_5ya_6 + a_7xa_8y^*a_9 + a_{10}x^*a_{11}y^*a_{12} \\
    & \hspace{7.5mm} + a_{13}ya_{14}xa_{15} + a_{16}y^*a_{17}xa_{18} + a_{19}ya_{20}x^*a_{21} + a_{22}y^*a_{23}x^*a_{24} \\
    & \hspace{7.5mm} + \E[a_{25}x]a_{26}ya_{27} + \E[a_{28}x^*]a_{29}ya_{30} + \E[a_{31}x]a_{32}y^*a_{33} + \E[a_{34}x^*]a_{35}y^*a_{36} \\
    & \hspace{7.5mm} + \E[a_{37}y]a_{38}xa_{39} + \E[a_{40}y^*]a_{41}xa_{42} + \E[a_{43}y]a_{44}x^*a_{45} + \E[a_{46}y^*]a_{47}x^*a_{48} \\
    & \hspace{7.5mm} + \E[a_{49}x]\E[a_{50}y]a_{51} + \E[a_{52}x^*]\E[a_{53}y]a_{54} + \E[a_{55}x]\E[a_{56}y^*]a_{57} + \E[a_{58}x^*]\E[a_{59}y^*]a_{60} \\
    & \hspace{7.5mm} + \E[a_{61}xa_{62}y]a_{63} + \E[a_{64}x^*a_{65}y]a_{66} + \E[a_{67}xa_{68}y^*]a_{69} + \E[a_{70}x^*a_{71}y^*]a_{72}.
\end{align*}
Then $\cT_{2,s}^0 = \spn\{\Xi$ as above with $a_1,\ldots,a_{72} \in \cA_s\}$ by definition of trace $\ast$-polynomials and traciality.
Therefore, it suffices to prove the lemma for $\Xi$ as above with $a_1,\ldots,a_{72} \in \cA_s$.
In this case, we define
\begin{align*}
    \tilde{\Xi}[x,y] & \coloneqq a_1xa_2ya_3 + a_4x^*a_5ya_6 + a_7xa_8y^*a_9 + a_{10}x^*a_{11}y^*a_{12} \\
    & \hspace{7.5mm} + a_{13}ya_{14}xa_{15} + a_{16}y^*a_{17}xa_{18} + a_{19}ya_{20}x^*a_{21} + a_{22}y^*a_{23}x^*a_{24} \\
    & \hspace{7.5mm} + \E[a_{61}xa_{62}y]a_{63} + \E[a_{64}x^*a_{65}y]a_{66} + \E[a_{67}xa_{68}y^*]a_{69} + \E[a_{70}x^*a_{71}y^*]a_{72}.
\end{align*}
If $M,N \colon \R_+ \to L^2(\E)$ are $L^2$-martingales, then
\[
\Xi[M(t) - M(s), N(t) - N(s)] = \tilde{\Xi}[M(t) - M(s), N(t) - N(s)] \qquad (t \geq s)
\]
by the martingale property.
More explicitly, if $a \in \cA_s$, $\e \in \{1,\ast\}$, and $0 \leq s < t$, then
\[
\E[a(M(t) - M(s))^{\e}] =\E[\E[a(M(t) - M(s))^{\e} \mid \cA_s]] = \E[a \, \E[M(t) - M(s) \mid \cA_s]^{\e}] = 0.
\]
Thus, all the terms in $\Xi$ with $a_{25},\ldots,a_{60}$ vanish when one plugs in $(x,y) = (M(t) - M(s),N(t) - N(s))$.
Unraveling the notation, we see that we have achieved our goal.
The final sentence follows from almost the same proof, except that one leaves out any term with $x^*$ or $y^*$ in it.
\end{proof}

\begin{prop}\label{prop.QCoffree}
Let $M,N \in \tbM^2$.
\begin{enumerate}[label=(\roman*),font=\normalfont]
    \item Suppose
    \begin{equation}
        \E[(M(t) - M(s))a(N(t) - N(s)) \mid \cA_s] = 0  \qquad (0 \leq s < t, \; a \in \cA_s).  \label{eq.freeinchyp1}
    \end{equation}
    If $\Lambda \colon \R_+ \to \mathbb{B}_2(\cA)$ is a $\norm{\cdot}_{2,2;1}$-LCLB complex-bilinear trace triprocess, then\label{item.QCoffreeC}
    \[
    \into \Lambda(t)[\d M(t), \d N(t)] \equiv 0.
    \]
    \item Suppose
    \begin{equation}
        \E[(M(t) - M(s))^{\e}a(N(t) - N(s)) \mid \cA_s] = 0 \qquad (0 \leq s < t, \; a \in \cA_s, \; \e \in \{1,\ast\}). \label{eq.freeinchyp2}
    \end{equation}
    If $\Lambda \colon \R_+ \to \mathbb{B}_2(\cA)$ is a $\norm{\cdot}_{2,2;1}$-LCLB trace triprocess, then\label{item.QCoffreeR}
    \[
    \into \Lambda(t)[\d M(t), \d N(t)] \equiv 0.
     \]
\end{enumerate}
\end{prop}

\begin{proof}
One can prove this result from Theorem \ref{thm.magicQC} and Lemma \ref{lem.T2s}, but we present a proof from scratch, i.e., with no reference to topological tensor products. 
Also, we present only the proof of \ref{item.QCoffreeR} since the proof of \ref{item.QCoffreeC} is similar and easier.

Suppose $0 \leq s < t$, $a,b,c \in \cA_s$, and $\e,\e_1,\e_2 \in \{1,\ast\}$.
First, we claim that
\[
\E[b(M(t) - M(s))^{\e_1}a(N(t) - N(s))^{\e_2} c\mid \cA_s] = \E[b(N(t) - N(s))^{\e_1}a(M(t) - M(s))^{\e_2}c \mid \cA_s] = 0.
\]
Since the conditional expectation $\E[\cdot \mid \cA_s]$ is an $\cA_s$-$\cA_s$ bimodule map, it suffices to treat the $b=c=1$ case.
To this end, note that if $d \in \cA_s$, then
\begin{align*}
    \E[(N(t) - N(s)) a (M(t) - M(s))^{\e} d] & = \E[a (M(t) - M(s))^{\e} d(N(t) - N(s))] \\
    & = \E[a \, \E[(M(t) - M(s))^{\e} d(N(t) - N(s)) \mid \cA_s]] = 0
\end{align*}
by \eqref{eq.freeinchyp2}.
Thus, $\E[(N(t) - N(s)) a (M(t) - M(s))^{\e} \mid \cA_s] = 0$.
Also,
\begin{align*}
    \E[(N(t) - N(s))^* a (M(t) - M(s))^{\e} \mid \cA_s] & = \E[(N(t) - N(s))^* (a^*)^* (M(t) - M(s))^{\e} \mid \cA_s] \\
    & = \E[((M(t) - M(s))^{\e})^*a^*(N(t) - N(s)) \mid \cA_s]^* = 0
\end{align*}
by \eqref{eq.freeinchyp2}.
This covers the cases $(\e_1,\e_2) \in \{(1,\ast),(1,1),(\ast,\ast),(\ast,1)\}$.
The cases $(\e_1,\e_2) \in \{(\ast,1),(1,1)\}$ are precisely the hypothesis \eqref{eq.freeinchyp2}, so the claim is proven.
Next, since $\E = \E \circ \E[\cdot \mid \cA_s]$, this implies
\[
\E[b(M(t) - M(s))^{\e_1}a(N(t) - N(s))^{\e_2} c] = \E[b(N(t) - N(s))^{\e_1}a(M(t) - M(s))^{\e_2}c] = 0
\]
as well.
Putting all this together with Lemma \ref{lem.T2s}, we conclude that
\begin{equation}
    \E[\Xi[M(t) - M(s), N(t) - N(s)] \mid \cA_s] = 0 \qquad (\Xi \in \cT_{2,s}^0).\label{eq.Om}\pagebreak
\end{equation}
By density, \eqref{eq.Om} holds for all $\Xi \in \cT_{2,s}$ as well.
To complete the proof, we apply Theorem \ref{thm.NCcondQC} (and Proposition \ref{prop.TkinFk}):
If $\Lambda \colon \R_+ \to \mathbb{B}_2(\cA)$ is a $\norm{\cdot}_{2,2;1}$-LCLB trace triprocess, then
\[
\int_0^t\Lambda(s)[\d M(s), \d N(s)] = L^1\text{-}\lim_{\pi \in \cP_{[0,t]}} \sum_{s \in \pi} \E[\Lambda(s_-)[\Delta_sM, \Delta_s N] \mid \cA_{s_-}] = 0
\]
by what we just proved.
\end{proof}

Before diving into specific examples, we prove one more general result.

\begin{thm}\label{thm.newmagicQC}
Let $M, N \in \tbM^2$, and suppose there exists an $r > 0$ such that
\begin{equation}
    \mu_{M,N}((t,t+r]) = \E[(M(t+r) - M(t))(N(t+r) - N(t))] \neq 0 \qquad (t \geq 0). \label{eq.nondeg}
\end{equation}
In this case, write
\[
\E_{M,N}[\Lambda](t) \coloneqq \frac{\E[\Lambda(t)[M(t+r) - M(t), N(t+r) - N(t)] \mid \cA_t]}{\E[(M(t+r) - M(r))(N(t+r) - N(t))]} \in L^1(\E) \qquad (t \geq 0)
\]
whenever $\Lambda \colon \R_+ \to B_2^{2,2;1}$.
\begin{enumerate}[label=(\roman*),font=\normalfont]
    \item Suppose the pairs $(M,N)$ and $(N,M)$ both satisfy the hypotheses of Theorem \ref{thm.magicQC} (with possibly different maps $\Gamma$).
    If $\Lambda \colon \R_+ \to \mathbb{B}_2(\cA)$ is a $\norm{\cdot}_{2,2;1}$-LCLB complex-bilinear trace triprocess, then $\E_{M,N}[\Lambda] \in L_{\loc}^1(\R_+,\mu_{M,N};L^1(\E))$, and\label{item.newmagicC}
    \begin{equation}
    \into \Lambda(t)[\d M(t), \d N(t)] = \into \E_{M,N}[\Lambda](t) \, \mu_{M,N}(\d t).\label{eq.EMNQC}
    \end{equation}
    \item Suppose that $(M^*,N^*) = (M,N)$ and the pair $(M,N)$ satisfies the hypotheses of Theorem \ref{thm.magicQC}.
    If $\Lambda \colon \R_+ \to \mathbb{B}_2(\cA)$ is a $\norm{\cdot}_{2,2;1}$-LCLB trace triprocess, then $\E_{M,N}[\Lambda] \in L_{\loc}^1(\R_+,\mu_{M,N};L^1(\E))$, and \eqref{eq.EMNQC} holds.\label{item.newmagicR}
\end{enumerate}
\end{thm}

\begin{rem}\label{rem.M=N*}
If $N=M^*$, then $\mu_{M,N} = \mu_{M,M^*} = \kappa_M$.
Moreover,
\[
\E_{M,M^*}[\Lambda](t) = \E[\Lambda(t)[e(t), e(t)^*] \mid \cA_t],
\]
where $e(t) \coloneqq \norm{M(t+r) - M(t)}_2^{-1}(M(t+r) - M(t))$.
Therefore, when $N=M^*$ in Theorem \ref{thm.newmagicQC}, the formulas read $\into \Lambda(t)[\d M(t), \d M^*(t)] = \into \E[\Lambda(t)[e(t),e(t)^*] \mid \cA_t] \, \kappa_M(\d t)$.
\end{rem}

\begin{proof}
Once again, it is possible to prove this result using Theorem \ref{thm.magicQC} and Lemma \ref{lem.T2s}, but we shall present a proof that does not rely on any topological tensor products.
Also, we present only the proof of \ref{item.newmagicR} since the proof of \ref{item.newmagicC} is similar and easier.

We begin with some technical observations.
Specifically, \eqref{eq.Gammahyp} implies $(M,N) = (M^*,N^*)$ satisfies the following invariance and continuity properties:
If $r \geq 0$, $0 \leq s < t$, $0 \leq u < v$, $s \leq u$, and $\Xi \in \cT_{2,s} \subseteq \cT_{2,u}$,~then
\begin{align*}
    \begin{split}
    \mu_{M,N}((u,v]) \, &\E[\Xi[M(t) - M(s) , N(t) - N(s)] \mid \cA_s] \\
    & = \mu_{M,N}((s,t]) \, \E[\Xi[M(v) - M(u), N(v) - N(u)] \mid \cA_u] \, \text{ and}
    \end{split}\numberthis\label{eq.inv} \\
    [s,\infty) \ni t \mapsto \,& g_{\Xi}(t) \coloneqq \E[\Xi[M(t +r) - M(t), M(t+r) - M(t)] \mid \cA_t] \in L^1(\E) \text{ is continuous.} \numberthis\label{eq.cont}
\end{align*}
To prove \eqref{eq.inv} and \eqref{eq.cont}, it suffices, by an easy limiting argument, to treat the case $\Xi \in \cT_{2,s}^0$.
For such $\Xi$, one can use Lemma \ref{lem.T2s}, (the purely algebraic version of) \eqref{eq.EGamma}, and $(M^*,N^*) = (M,N)$ to prove \eqref{eq.inv} and \eqref{eq.cont}.
We leave the details to the reader.

We now begin in earnest.
Suppose, in addition, that there is some $r > 0$ such that \eqref{eq.nondeg} holds.
Then we may divide by $\mu_{M,N}((s,s+r])$ in \eqref{eq.inv} with $(u,v) = (s,s+r)$ to see that
\begin{equation}
\begin{split}
    \E[\Xi[M(t) - M&(s), N(t) - N(s)] \mid \cA_s] \\
    & = \frac{\E[\Xi[M(s+r) - M(s), N(s+r) - N(s)] \mid \cA_s]}{\mu_{M,N}((s,s+r])} \, \mu_{M,N}((s,t]) \quad (s < t, \; \Xi \in \cT_{2,s}). 
\end{split}\label{eq.inv2}
\end{equation}
Now, let $\Lambda \colon \R_+ \to \mathbb{B}_2(\cA)$ be a $\norm{\cdot}_{2,2;1}$-LCLB trace triprocess and $\Pi \in \cP_{\R_+}$.
By examining $\E_{M,N}[\Lambda^{\Pi}]$ on each interval $(t_-,t]$ with $t \in \Pi$, \eqref{eq.cont} implies that $\E_{M,N}[\Lambda^{\Pi}] \colon \R_+ \to L^1(\E)$ is LCLB.
Consequently, if $t \geq 0$ and $\Om \coloneqq \Lambda^{\Pi}$, then\vspace{-0.2mm}
\begin{align*}
    \int_0^t \Om(s)[\d M(s), \d N(s)] & = L^1\text{-}\lim_{\pi \in \cP_{[0,t]}}\sum_{s \in \pi}\E[\Om(s_-)[\Delta_sM, \Delta_sN] \mid \cA_{s_-}] \tag{Thm.\ \ref{thm.NCcondQC}} \\
    & = L^1\text{-}\lim_{\pi \in \cP_{[0,t]}}\sum_{s \in \pi}\E_{M,N}[\Om](s_-) \,\mu_{M,N}((s_-,s]) \tag{Eq.\ \eqref{eq.inv2}} \\
    & = L^1\text{-}\lim_{\pi \in \cP_{[0,t]}} \int_0^t \E_{M,N}[\Om]^{\pi} \,\d\mu_{M,N} = \int_0^t \E_{M,N}[\Om] \,\d\mu_{M,N}. \tag{Lem.\ \ref{lem.HPi}\ref{item.HPipw}}
\end{align*}
Finally, since $\Lambda$ is $\norm{\cdot}_{2,2;1}$-LCLB, $\Lambda^{\Pi} \to \Lambda$ pointwise and in $L_{\loc}^1(\R_+,\kappa_{M,N};B_2^{2,2;1})$ as $|\Pi| \to 0$ by Lemma \ref{lem.HPi}\ref{item.HPipw} and the dominated convergence theorem.
Therefore,\vspace{-0.2mm}
\[
\mathbb{L}^1\text{-}\lim_{\Pi \in \cP_{\R_+}}\into \Lambda^{\Pi}[\d M, \d N] = \into \Lambda[\d M, \d N],\vspace{-0.2mm}
\]
and $\E_{M,N}[\Lambda^{\Pi}] \to \E_{M,N}[\Lambda]$ pointwise as $|\Pi| \to 0$.
Finally, observe that if\vspace{-0.2mm}
\[
C_t \coloneqq \sup_{0 \leq s \leq t} \frac{\norm{M(s+r) - M(s)}_2\norm{N(s+r) - N(s)}_2}{|\E[(M(s+r) - M(s))(N(s+r) - N(s))]|} < \infty,\vspace{-0.2mm}
\]
then\vspace{-0.2mm}
\[
\sup_{\Pi \in \cP_{\R_+}}\sup_{0 \leq s \leq t}\big\|\E_{M,N}\big[\Lambda^{\Pi}\big](s)\big\|_1 \leq C_t\sup_{\Pi \in \cP_{\R_+}}\sup_{0 \leq s \leq t}\big\|\Lambda^{\Pi}(s)\big\|_{2,2;1} \leq C_t\sup_{0 \leq s \leq t}\norm{\Lambda(s)}_{2,2;1} < \infty.\vspace{-0.2mm}
\]
Thus, by the dominated convergence theorem (and Fact \ref{fact.partlim}), $\E_{M,N}[\Lambda] \in L_{\loc}^1(\R_+,\mu_{M,N};L^1(\E))$, and $\E_{M,N}[\Lambda^{\Pi}] \to \E_{M,N}[\Lambda]$ in $L_{\loc}^1(\R_+,\mu_{M,N};L^1(\E))$ as $|\Pi| \to 0$, from which it follows that\vspace{-0.2mm}
\[
\mathbb{L}^1\text{-}\lim_{\Pi \in \cP_{\R_+}}\into \E_{M,N}\big[\Lambda^{\Pi}\big] \,\d\mu_{M,N} = \into \E_{M,N}[\Lambda] \,\d\mu_{M,N}.\vspace{-0.2mm}
\]
In the end, we finally get $\into \Lambda[\d M, \d N] = \into \E_{M,N}[\Lambda]\,\d\mu_{M,N}$, as desired.
\end{proof}

We now give several examples of pairs $(M,N)$ satisfying the hypotheses of the results above.

\begin{lem}\label{lem.condexpcalc}
Suppose $0 \leq s < t$ and $n \in \N$.
\begin{enumerate}[label=(\roman*),font=\normalfont]
    \item If $x,y \in \cA$ are centered and $\{x,y\}$ is free from $\cA_s$, then $\E[xay \mid \cA_s] = \E[a] \,  \E[xy]$ for all $a \in \cA_s$.\label{item.freeinc}
    \item Suppose that $(\cA_n,(\cA_{n,t})_{t \geq 0},\tau_n)$ is as in Example \ref{ex.clmart} and $x,y \in L^2(\tau_n)$ are such that $(x,y)$ is classically $P$-independent of $\sF_s$.
    If $\cE \subseteq \MnC$ is an $\la \cdot, \cdot \ra_{L^2(\tr_n)}$-orthogonal basis for $\MnC$ and $z_e \coloneqq \tr_n(e^*z)/\tr_n(e^*e)$ for all $z \in \MnC$ and $e \in \cE$, then\label{item.classinc}\vspace{-0.2mm}
    \[
    \tau_n[xay \mid \cA_{n,s}] = \sum_{e,f \in \cE} \E_P[x_ey_f]\,eaf = \sum_{e \in \cE} \E_P[xy_e] \, ae \qquad (a \in \cA_{n,s}).\vspace{-0.2mm}
    \]
\end{enumerate}
\end{lem}

\begin{proof}
The first item is an easy exercise in using the definition of free independence, so we leave it to the reader.
For the second, since $z = \sum_{e \in \cE} z_e\,e$ for all $z \in \MnC$, if $a \in \cA_{n,s} = L^{\infty}(\Om,\sF_s,P;\MnC)$, then\vspace{-0.2mm}
\begin{align*}
    \tau_n[xay \mid \cA_{n,s}] & = \E_P[xay \mid \sF_s] = \sum_{e \in \cE} \E_P[xay_ee \mid \sF_s] = \sum_{e \in \cE} \E_P[xy_e \mid \sF_s] \, ae \vspace{-0.2mm}\\
    & = \sum_{e \in \cE} \E_P[xy_e]\,ae = \sum_{e,f \in \cE} \E_P[x_eey_f] \, af = \sum_{e,f \in \cE} \E_P[x_ey_f] \, eaf,\vspace{-0.2mm}
\end{align*}
where we used the independence assumption in the fourth equality.
\end{proof}

We end this section with some examples.
A process $X = (X_1,\ldots,X_n) \colon \R_+ \to \cA^n$ is said to have \textbf{jointly} (\textbf{$\boldsymbol{\ast}$-})\textbf{free increments} if $X$ is adapted and $0 \leq s < t$ implies that $\{X_i(t) - X_i(s) : 1 \leq i \leq n\}$ is ($\ast$-)free from $\cA_s$.
For instance, a process $Y \colon \R_+ \to \cA$ has $\ast$-free increments if and only if $(Y,Y^*)$ has jointly free increments.
A process $X = (X_1,\ldots,X_n) \colon \R_+ \to \cA^n$ is called an \textbf{$\boldsymbol{n}$-dimensional} (\textbf{semi})\textbf{circular} \textbf{Brownian motion} if $X(0) = 0$, $X$ has jointly $\ast$-free increments, and $(X_1(t) - X_1(s),\ldots,X_n(t) - X_n(s))$ is a $\ast$-free family of (semi)circular elements each with variance $t-s$ whenever $0 \leq s < t$.
\pagebreak

\begin{ex}[Free examples]\label{ex.freeQC}
Fix $M,N \in C_a(\R_+;\cA)$ with constant expectation, and assume that $(M,N)$ has jointly free increments.
By Example \ref{ex.freeinc}, $M,N \in \M^{\infty}$.
By Lemma \ref{lem.condexpcalc}\ref{item.freeinc}, $(M,N)$ and $(N,M)$ both satisfy the hypotheses of Theorem \ref{thm.magicQC} with $\Gamma = \E$.
Assume also that there is some $r > 0$ such that \eqref{eq.nondeg} holds.
Then Theorem \ref{thm.newmagicQC}\ref{item.newmagicC} applies to $(M,N)$.
If, in addition, $(M^*,N^*) = (M,N)$, then Theorem \ref{thm.newmagicQC}\ref{item.newmagicR} applies to $(M,N)$.

If $\E[(M(t) - M(s))(N(t) - N(s))] = 0$, then \eqref{eq.freeinchyp1} holds by Lemma \ref{lem.condexpcalc}\ref{item.freeinc}.
Therefore, Proposition \ref{prop.QCoffree}\ref{item.QCoffreeC} says that $\into \Lambda(t)[\d M(t), \d N(t)] \equiv 0$ whenever $\Lambda \colon \R_+ \to \mathbb{B}_2(\cA)$ is a $\norm{\cdot}_{2,2;1}$-LCLB complex-bilinear trace triprocess.
If, in addition, $(M^*,N)$ has jointly free increments and $\E[(M(t)-M(s))^*(N(t) - N(s))] = 0$, then \eqref{eq.freeinchyp2} holds by Lemma \ref{lem.condexpcalc}\ref{item.freeinc}.
Therefore, Proposition \ref{prop.QCoffree}\ref{item.QCoffreeR} says that $\into \Lambda(t)[\d M(t), \d N(t)] \equiv 0$ whenever $\Lambda \colon \R_+ \to \mathbb{B}_2(\cA)$ is any $\norm{\cdot}_{2,2;1}$-LCLB trace triprocess.
\end{ex}

Suppose $\Lambda \colon \R_+ \to \mathbb{B}_2(\cA)$ is a $\norm{\cdot}_{2,2;1}$-LCLB trace triprocess, $X = (X_1,\ldots,X_n) \colon \R_+ \to \cA^n$ is an $n$-dimensional (semi)circular Brownian motion, and $H_1,K_1,\ldots,H_n,K_n \colon \R_+ \to \mathbb{B}(\cA)$ are $\norm{\cdot}_{2;2}$-LCLB trace biprocesses.
By Example \ref{ex.freeQC} and Theorem \ref{thm.QCSI}, if $U \coloneqq \sum_{i=1}^n\into H_i[\d X_i]$ and $V \coloneqq \sum_{i=1}^n \into K_i[\d X_i]$, then
\begin{align*}
    \into \Lambda(t)[\d U(t), \d V(t)] & = \sum_{i,j=1}^n\into \Lambda(t)[H_i(t)[\d X_i(t)], K_j(t)[\d X_j(t)]] \\
    & = \sum_{i=1}^n \into \Lambda(t)[H_i(t)[\d X_i(t)], K_i(t)[\d X_i(t)]],
\end{align*}
and $\into \Lambda(t)[H_i(t)[\d X_i(t)], K_i(t)[\d X_i(t)]]$ may often be computed with Theorems \ref{thm.magicQC} or \ref{thm.newmagicQC}.
This provides a natural way to compute (certain) quadratic covariations of It\^{o} processes driven by multidimensional free Brownian motions.
Moreover, such computations agree with the free It\^{o} product rule in \cite{NikitopoulosIto}.

\begin{ex}[Classical examples]\label{ex.clQC}
Let $n \in \N$ and $(\cA_n,(\cA_{n,t})_{t \geq 0}, \tau_n)$ be as in Example \ref{ex.clmart}.
Suppose two classical adapted stochastic processes $M,N \colon \R_+ \times \Om \to \MnC$ have constant $P$-expectation and jointly $P$-independent increments, i.e., $(M(t,\cdot) - M(s,\cdot), N(t,\cdot) - N(s,\cdot))$ is $P$-independent of $\sF_s$ whenever $0 \leq s < t$.
Then $M$ and $N$ are $L^2$-martingales.
If, in addition, $(M(t,\cdot))_{t \geq 0}$ and $(N(t,\cdot))_{t \geq 0}$ are $P$-independent, then \eqref{eq.freeinchyp2} holds with $(\E,\cA_s) = (\tau_n,\cA_{n,s})$ by Lemma \ref{lem.condexpcalc}\ref{item.classinc}.
Consequently, if we also know that $M,N \in \tbM_{\tau_n}^2$ (as is the case, by Theorem \ref{thm.matBMtilde}, when $M$ and $N$ are Hermitian Brownian motions and $(\sF_t)_{t \geq 0}$ satisfies the usual conditions), then Proposition \ref{prop.QCoffree}\ref{item.QCoffreeR} says that $\into \Lambda(t)[\d M(t), \d N(t)] \equiv 0$ whenever $\Lambda \colon \R_+ \to \mathbb{B}_2(\cA_n)$ is a $\norm{\cdot}_{2,2;1}$-LCLB trace triprocess.

Now, write
\[
\la a,b \ra_n \coloneqq n \Tr_n(b^*a) = n^2\tr_n(b^*a) \qquad (a,b \in \MnC),
\]
and let $X \colon \R_+ \times \Om \to \MnC_{\sa}$ be an $\la \cdot, \cdot \ra_n$-Brownian motion.
In other words, if $\cE \subseteq \MnC_{\sa}$ is an $\la \cdot, \cdot \ra_n$-orthonormal basis of the real inner product space $(\MnC_{\sa},\la \cdot, \cdot \ra_n)$, then $(\la X,e \ra_n)_{e \in \cE}$ is a Brownian motion in $\R^{n^2}$.
Since such an $\cE$ is an $\la \cdot, \cdot \ra_{L^2(\tr_n)}$-orthogonal basis for the complex inner product space $(\MnC,\la \cdot, \cdot \ra_{L^2(\tr_n)})$, Lemma \ref{lem.condexpcalc}\ref{item.classinc} says that if $0 \leq s < t$ and $a \in \cA_{n,s}$, then
\begin{align*}
    \tau_n[(X(t) - X(s))&  a(X(t) - X(s))\mid \cA_{n,s}] = \sum_{e,f \in \cE} \E_P[\la X(t) - X(s), e\ra_n\la X(t) - X(s), f\ra_n] \, eaf \\
    & = \sum_{e \in \cE} \E_P[\la X(t) - X(s), e\ra_n^2] \, eae = (t-s)\sum_{e \in \cE} eae  = \kappa_X((s,t]) \sum_{e \in \cE}eae.
\end{align*}
Above, we have considered $\R_+ \ni t \mapsto X(t) \coloneqq X(t,\cdot) \in L^2(\tau_n)$ as a noncommutative $L^2$-martingale in the usual way.
Now, by the ``magic formulas'' (\cite[\S3.1]{DHK2013}),
\[
\sum_{e \in \cE} eae = \tr_n(a).
\]
This gives
\[
\tau_n[(X(t) - X(s))a(X(t) - X(s)) \mid \cA_{n,s}] = \tr_n(a) \, \kappa_X((s,t]) \qquad (0 \leq s < t, \; a \in \cA_{n,s}).
\]
Thus, \eqref{eq.Gammahyp} holds with $M=N=X$ and $\Gamma = \tr_n$ (defined as a map $\cA_n \to \cA_n$).
If, in addition, $(\sF_t)_{t \geq 0}$ satisfies the usual conditions so that $X \in \tbM_{\tau_n}^2$, then Theorems \ref{thm.magicQC} and \ref{thm.newmagicQC} apply to $M=N=X$.
\end{ex}

\begin{ex}[$q$-Brownian motion]\label{ex.qBM}
Let $q \in [-1,1)$, and write $\Gamma_q(q)$ for the second quantization of the contraction $q = q \id_{L^2(\R_+)} \in B(L^2(\R_+))$;
see \cite[Thm.\ 2.11]{BKS1997}.
By \cite[Thm.\ 3.1]{DonatiMartinS2003}, if $X \colon \R_+ \to \cA_{\sa}$ is a $q$-Brownian motion, then
\[
\E[(X(t) - X(s))a(X(t) - X(s)) \mid \cA_s] = (t-s)\,\Gamma_q(q)a = \Gamma_q(q)a \, \kappa_X((s,t]) \qquad (0 \leq s < t, \;a \in \cA_s).
\]
Thus, \eqref{eq.Gammahyp} holds with $M=N=X$ and $\Gamma = \Gamma_q(q)$, and Theorems \ref{thm.magicQC} and \ref{thm.newmagicQC} apply to $M=N=X$.
\end{ex}

\begin{rem}
The ``noncommutative conditional variance'' formulas in the above examples form our primary motivation for the key hypothesis \eqref{eq.Gammahyp} in Theorem \ref{thm.magicQC}.
\end{rem}

Similar to the comments made after Example \ref{ex.freeQC}, Example \ref{ex.qBM} and Theorem \ref{thm.QCSI} combine to give a natural way to compute quadratic covariations of It\^{o} processes driven by $q$-Brownian motion, and these computations agree with the known ``It\^{o} product rule'' for $q$-Brownian stochastic integrals (\cite[Thm.\ 3.2]{DonatiMartinS2003}---see also \cite[Prop.\ 4.4]{DS2018}).

\section{It\texorpdfstring{\^{o}}{}'s formula}\label{sec.NCIF}

Retain the filtered $\mathrm{C}^*$-probability spaces $(\cA,(\cA_t)_{t \geq 0},\E = \E_{\mathsmaller{\cA}})$, $(\cB,(\cB_t)_{t \geq 0},\E_{\mathsmaller{\cB}})$, and $(\cC,(\cC_t)_{t \geq 0},\E_{\mathsmaller{\cC}})$ from Section \ref{sec.QC}.
Also, write $\cA_{\beta}$ for a fixed element of $\{\cA,\cA_{\sa}\}$ and $\cB_{\gamma}$ for a fixed element of $\{\cB,\cB_{\sa}\}$.
Henceforth, we assume the reader is familiar with (higher-order) Fr\'{e}chet derivatives; 
see \cite[Ch.\ 1]{HJ2014} for the relevant background.
If $\cV,\cW$ are normed vector spaces, $\cU \subseteq \cV$ is an open set, and $F \colon \cU \to \cV$ is $k$-times Fr\'{e}chet differentiable, then we shall write $D^kF \colon \cU \to B_k(\cV^k;\cW)$ for the $k^{\text{th}}$ Fr\'{e}chet derivative of $F$, i.e., $D^kF(p)[v_1,\ldots,v_k] = \partial_{v_k}\cdots \partial_{v_1}F(p)$ for all $p \in \cU$ and $v_1,\ldots,v_k \in \cV$.

\subsection{Adapted \texorpdfstring{$C^k$}{} maps}\label{sec.adapCk}

In this section, we define the class of functions to which our noncommutative It\^{o}'s formula will apply. 
We also provide elementary examples of such functions.
In later sections, we provide some more sophisticated examples drawing on work from \cite{JLS2022,NikitopoulosNCk}.

\begin{defi}[$C^{k,\ell}$ map]\label{def.Ckl}
Let $\cV,\cW,\cZ$ be real normed vector spaces, $\cU \subseteq \cV \times \cW$ be an open set, and $k,\ell \in \N_0$.
A map $F \colon \cU \to \cZ$ is called $\boldsymbol{C^{k,\ell}}$, written $F \in C^{k,\ell}(\cU;\cZ)$, if for every $(v,w) \in \cU$, there is exists a radius $r > 0$ such that 
\begin{enumerate}[label=(\roman*),font=\normalfont]
    \item $B_r(v) \times B_r(w) \subseteq \cU$;
    \item for all $(x,y) \in B_r(v) \times B_r(w)$, $F(x,\cdot) \in C^{\ell}(B_r(w);\cZ)$ and $F(\cdot,y) \in C^k(B_r(v);\cZ)$; and
    \item for all $i=0,\ldots,k$ and $j=0,\ldots,\ell$, the maps
    \begin{align*}
    & \cU \ni (x,y) \mapsto D_1^iF(x,y) \coloneqq D^i(F(\cdot,y))(x) \in B_i(\cV^i;\cZ) \, \text{ and} \\
    & \cU \ni (x,y) \mapsto D_2^jF(x,y) \coloneqq D^j(F(x,\cdot))(y) \in B_j(\cW^j;\cZ)
    \end{align*}
    are continuous.
\end{enumerate}
By convention, the zeroth derivative of a function is the function itself.
\end{defi}

For the definition below, recall that a (real--)$k$-linear map $T \colon \cA_{\sa}^k \to \cB$ is always identified with its complex--$k$-linear extension $\cA^k \to \cB$ (Observation \ref{obs.realklin}).

\begin{defi}[Adapted $C^{k,\ell}$ map]\label{def.adapCkl}
Let $k,\ell \in \N_0$ and $\cU \subseteq \cA_{\beta} \times \cB_{\gamma}$ be an open set.
A map $F \colon \cU \to \cC$ is called \textbf{adapted $\boldsymbol{C^{k,\ell}}$}, written $F \in C_a^{k,\ell}(\cU;\cC)$, if
\begin{enumerate}[label=(\roman*),font=\normalfont]
    \item $F \in C^{k,\ell}(\cU;\cC)$ when we consider $\cA_{\beta}$, $\cB_{\gamma}$, and $\cC$ as real Banach spaces;
    \item if $i=0,\ldots,k$ and $(a,b) \in \cU$, then $D_1^iF(a,b) \in \mathbb{B}_i(\cA^i;\cC)$, and $D_1^iF \colon U \to \mathbb{B}_i(\cA^i;\cC)$ is continuous with respect to $\vertiii{\cdot}_i$;
    \item if $j=0,\ldots,\ell$ and $(a,b) \in \cU$, then $D_2^jF(a,b) \in \mathbb{B}_j(\cB^j;\cC)$, and $D_2^jF \colon \cU \to \mathbb{B}_j(\cB^j;\cC)$ is continuous with respect to $\vertiii{\cdot}_j$;\pagebreak
    \item if $i=0,\ldots,k$, $t \geq 0$, and $(a,b) \in (\cA_t \times \cB_t) \cap U$, then $D_1^iF(a,b)$ belongs to $\cF_{i,t}(\E_{\mathsmaller{\cA}},\ldots,\E_{\mathsmaller{\cA}};\E_{\mathsmaller{\cC}})$; and
    \item if $j=0,\ldots,\ell$, $t \geq 0$, and $(a,b) \in (\cA_t \times \cB_t) \cap U$, then $D_2^jF(a,b)$ belongs to $\cF_{j,t}(\E_{\mathsmaller{\cB}},\ldots,\E_{\mathsmaller{\cB}};\E_{\mathsmaller{\cC}})$.
\end{enumerate}
If $\cV \subseteq \cB_{\gamma}$ is an open subset, a map $G \colon \cV \to \cC$ is called \textbf{adapted $\boldsymbol{C^k}$}, written $G \in C_a^k(\cV;\cC)$, if the map $\R \times \cV \ni (t,b) \mapsto G(b) \in \cC$ is adapted $C^{m,k}$ for some (equivalently, all) $m \in \N_0$.
As one might expect, we also write $C_a^{\infty}(\cV;\cC) \coloneqq \bigcap_{k \in \N}C_a^k(\cV;\cC)$.
\end{defi}

\begin{ex}[Inversion map]\label{ex.inv}
If $\cU = \mathrm{GL}(\cA) \coloneqq \{$invertible elements of $\cA\} \subseteq \cA$ and $F(g) \coloneqq g^{-1}$ for all $g \in \cU$, then $F \in C^{\infty}(\cU;\cA)$, and
\[
D^kF(g)[b_1,\ldots,b_k] = (-1)^k\sum_{\pi \in S_k}g^{-1}b_{\pi(1)}\cdots g^{-1}b_{\pi(k)}g^{-1} \qquad (g \in \cU, \; b_1,\ldots,b_k \in \cA),
\]
where $S_k$ is the symmetric group on $k$ letters.
Thus, by Proposition \ref{prop.TkinFk}, $F \in C_a^{\infty}(\cU;\cA)$.
\end{ex}

\begin{nota}[Noncommutative derivative]\label{nota.ncder}
For $p(\lambda) = \sum_{i=0}^n c_i \lambda^i \in \C[\lambda]$ and $k \in \N$, define
\[
\partial_{\mathsmaller{\otimes}}^kp(\a) \coloneqq k!\sum_{i=0}^nc_i\sum_{\delta \in \N_0^{k+1} : |\delta|=i-k}a_1^{\delta_1} \otimes \cdots \otimes a_{k+1}^{\delta_{k+1}} \in \cA^{\otimes(k+1)} \qquad \big(\a = (a_1,\ldots,a_{k+1}) \in \cA^{k+1}\big),
\]
where $|\delta| = \delta_1 + \cdots + \delta_{k+1}$ for $\delta = (\delta_1,\ldots,\delta_{k+1}) \in \N_0^{k+1}$, and empty sums are defined to be zero.
Also, write $\partial_{\mathsmaller{\otimes}} \coloneqq \partial_{\mathsmaller{\otimes}}^1$ and $\partial_{\mathsmaller{\otimes}}^kp(a) \coloneqq \partial_{\mathsmaller{\otimes}}^kp(a,\ldots,a)$ for all $a \in \cA$.
\end{nota}

\begin{ex}[Polynomials]\label{ex.polyadapCk}
If $p \in \C[\lambda]$ and $p_{\mathsmaller{\cA}} \colon \cA \to \cA$ is the map $a \mapsto p(a)$, then $p_{\mathsmaller{\cA}} \in C^{\infty}(\cA;\cA)$, and
\[
D^kp_{\mathsmaller{\cA}}(a)[b_1,\ldots,b_k] = \frac{1}{k!}\sum_{\pi \in S_k}\partial_{\mathsmaller{\otimes}}^kp(a)\sh_k[b_{\pi(1)},\ldots,b_{\pi(k)}] \qquad (a,b_1,\ldots,b_k \in \cA);
\]
see \cite[Prop.\ 4.3.1]{NikitopoulosNCk}.
Therefore, by Proposition \ref{prop.TkinFk}, $p_{\mathsmaller{\cA}} \in C_a^{\infty}(\cA;\cA)$.
Later, we shall see more generally that if $P \in (\TrP_n^*)^m$ is an $m$-tuple of trace $\ast$-polynomials in $n$ indeterminates, then $P_{\mathsmaller{(\cA,\E)}} \in C_a^{\infty}(\cA^n;\cA^m)$ (Example \ref{ex.trpolytrC} and Theorem \ref{thm.trCkadapCk}).
\end{ex}

\begin{defi}[Wiener space]\label{def.Wk}
Write $M(\R,\cB_{\R})$ for the space of complex Borel measures on $\R$.
For $\mu \in M(\R,\cB_{\R})$, write $\mu_{(0)} \coloneqq |\mu|(\R)$ for the total variation norm of $\mu$ and $\mu_{(k)} \coloneqq \int_{\R} |\xi|^k\,|\mu|(\d\xi) \in [0,\infty]$ for the ``$k^{\text{th}}$ moment'' of $|\mu|$.
The $\boldsymbol{k^{\text{\textbf{th}}}}$ \textbf{Wiener space} $W_k(\R)$ is the set of functions $f \colon \R \to \C$ such that there exists a (necessarily unique) $\mu \in M(\R,\cB_{\R})$ with $\mu_{(k)} < \infty$ and $f(\lambda) = \int_{\R} e^{i \xi \lambda} \, \mu(\d\xi)$ for all $\lambda \in \R$.
\end{defi}

\begin{ex}[Operator functions]\label{ex.Wiener}
If $f \colon \R \to \C$ is a continuous function, then the map $f_{\mathsmaller{\cA}} \colon \cA_{\sa} \to \cA$ defined via the functional calculus by $a \mapsto f(a)$ is called the \textbf{operator function} associated to $f$.
Using Duhamel's formula, i.e.,
\[
e^a - e^b = \int_0^1 e^{ta}(a-b)e^{(1-t)b} \,\d t \qquad (a,b \in \cA),
\]
it is possible to show that if $f = \int_{\R} e^{i\xi\boldsymbol{\cdot}} \,\mu(\d\xi) \in W_k(\R)$ and $\Sigma_k = \{(s_1,\ldots,s_k) \in \R_+^k : s_1+\cdots+s_k \leq 1\}$, then $f_{\mathsmaller{\cA}} \in C^k(\cA_{\sa};\cA)$, and
\begin{equation}
    D^kf_{\mathsmaller{\cA}}(a)[b_1,\ldots,b_k] = \sum_{\pi \in S_k}\int_{\R}\int_{\Sigma_k} (i\xi)^k e^{is_1\xi a}b_{\pi(1)}\cdots e^{is_k\xi a}b_{\pi(k)} e^{i(1-\sum_{j=1}^ks_j)\xi a} \,\d s_1\cdots \d s_k\,\mu(\d\xi) \label{eq.Wkderform}
\end{equation}
for all $a,b_1,\ldots,b_k \in \cA_{\sa}$;
see \cite[\S5]{ACDS2009} for this kind of argument.
From the derivative formula \eqref{eq.Wkderform}, it can be shown that $f_{\mathsmaller{\cA}} \in C_a^k(\cA_{\sa};\cA)$, as we encourage the reader to ponder.
We shall provide details in a more general context in Remark \ref{rem.Wiener} below.
\end{ex}

All the examples above are actually ``trace $C^k$ maps,'' which we define in Section \ref{sec.trsmoothmaps}.
What we witness concretely in these examples is the general fact that trace $C^k$ maps are adapted $C^k$ (Theorem \ref{thm.trCkadapCk}).

To end this section, we show that classical $C^k$ maps on spaces of matrices give rise to adapted $C^k$ maps on the spaces of random matrices from Example \ref{ex.clmart}.

\begin{lem}\label{lem.clCk}
Fix $n,m,d,\ell \in \N$ and $\MnC_{\beta} \in \{\MnC,\MnC_{\sa}\}$, and let $(\cA_n,(\cA_{n,t})_{t \geq 0},\tau_n)$ be as in Example \ref{ex.clmart}.
If $k \in \N_0$ and $f \in C^k(\MnC_{\beta}^d;\MmC^{\ell})$, then the map $\cA_{n,\beta}^d \ni \mathbf{a} \mapsto f_*(\mathbf{a}) \coloneqq f \circ \mathbf{a} \in \cA_m^{\ell}$ is (Fr\'echet) $C^k$, and
\begin{equation}
    \big(\partial_{\mathbf{b}_k}\cdots\partial_{\mathbf{b}_1}f_*(\mathbf{a}))(\om) = \partial_{\mathbf{b}_k(\om)}\cdots \partial_{\mathbf{b}_1(\om)}f(\mathbf{a}(\om)) \qquad \big(\mathbf{a},\mathbf{b}_i \in \cA_{n,\beta}^d, \text{ a.e.\ } \om \in \Om\big). \label{eq.derivsanity}
\end{equation}
(If $k=0$, then \eqref{eq.derivsanity} should be interpreted as $f_*(\mathbf{a}) = f \circ \mathbf{a}$, i.e., the definition of $f_*$.
Also, $\MnC_{\beta}^d$ and $\MmC^{\ell}$ are viewed as real Banach spaces.)
\end{lem}

\begin{proof}
We proceed by induction on $k$.
For the base case, we just need to prove that if $f$ is continuous, then so is $f_*$.
To this end, suppose $(\mathbf{a}_j)_{j \in \N}$ is a sequence in $\cA_{n,\beta}^d$ converging to $\mathbf{a} \in \cA_{n,\beta}$.
Then there exists an $R > 0$ such that for a.e.\ $\om \in \Om$, $\{\mathbf{a}_j(\om) : j \in \N\} \cup \{\mathbf{a}(\om)\} \subseteq C_R \coloneqq \{\mathbf{b} \in \MnC_{\beta}^d : \norm{\mathbf{b}} \leq R\}$;
here, $\norm{\cdot} = \norm{\cdot}_{L^{\infty}(\tr_n^{\oplus d})}$.
Since $C_R$ is compact and $f$ is continuous, $f|_{C_R}$ is uniformly continuous.
It follows that $f_*(\mathbf{a}_j) = f \circ \mathbf{a}_j \to f \circ \mathbf{a} = f_*(\mathbf{a})$ in $\cA_m^{\ell}$ as $j \to \infty$.
Thus, $f_*$ is continuous.

For the induction step, suppose we know the desired conclusions for $C^{k-1}$ functions with $k \geq 1$.
If $f \in C^k(\MnC_{\beta}^d;\MmC^{\ell})$ and $\mathbf{a},\mathbf{b}_1,\ldots,\mathbf{b}_k \in \cA_{n,\beta}^d$, then for a.e.\ $\om \in \Om$,
\begin{align*}
    \delta_{\om}(\mathbf{b}_k) & \coloneqq (\partial_{\mathbf{b}_{k-1}}\cdots \partial_{\mathbf{b}_1}f_*(\mathbf{a}+\mathbf{b}_k))(\om) - (\partial_{\mathbf{b}_{k-1}}\cdots \partial_{\mathbf{b}_1}f_*(\mathbf{a}))(\om) \\
    & = \partial_{\mathbf{b}_{k-1}(\om)}\cdots \partial_{\mathbf{b}_1(\om)}f(\mathbf{a}(\om)+\mathbf{b}_k(\om)) - \partial_{\mathbf{b}_{k-1}(\om)}\cdots \partial_{\mathbf{b}_1(\om)}f(\mathbf{a}(\om)) \\
    & = \int_0^1 \partial_{\mathbf{b}_k(\om)}\cdots \partial_{\mathbf{b}_1(\om)}f(\mathbf{a}(\om) + t\mathbf{b}_k(\om))\,\d t
\end{align*}
by the induction hypothesis and the fundamental theorem of calculus.
It follows that
\begin{align*}
    \e_{\om}(\mathbf{b}_k) & \coloneqq \delta_{\om}(\mathbf{b}_k) - \partial_{\mathbf{b}_k(\om)}\cdots \partial_{\mathbf{b}_1(\om)}f(\mathbf{a}(\om)) \\
    & = \int_0^1 \big(D^kf(\mathbf{a}(\om) + t\mathbf{b}_k(\om)) - D^kf(\mathbf{a}(\om))\big)[\mathbf{b}_1(\om),\ldots,\mathbf{b}_k(\om)]\,\d t.
\end{align*}
Writing $B_k \coloneqq B_k((\MnC_{\beta}^d)^k;\MmC^{\ell}) = B_k(L^{\infty}(\tr_{n,d})^k;L^{\infty}(\tr_{m,\ell}))$, this gives
\[
\norm{\e_{\om}(\mathbf{b}_k)} \leq \norm{\mathbf{b}_1(\om)}\cdots\norm{\mathbf{b}_k(\om)}\sup_{0 \leq t \leq 1}\norm{D^kf(\mathbf{a}(\om) + t\mathbf{b}_k(\om)) - D^kf(\mathbf{a}(\om))}_{B_k}.
\]
Using this estimate and the fact that $\norm{\a}_{L^{\infty}(\tau_n^{\oplus d})} = P\text{-}\esssup\{\norm{\a(\om)} : \om \in \Om\}$ for all $\a \in \cA_n^d$, we may appeal to the continuity of $D^kf$ and the compactness of $C_R$ as in the previous paragraph to conclude that $f_* \in C^k(\cA_{n,\beta}^d;\cA_m^{\ell})$ and \eqref{eq.derivsanity} holds.
\end{proof}

\begin{prop}[Classical functions]\label{prop.clCk}
If $k \in \N$ and $f \in C^k(\MnC_{\beta}^d;\MmC^{\ell})$, then $f_* \in C_a^k(\cA_{n,\beta}^d;\cA_m^{\ell})$.
\end{prop}

\begin{proof}
Let $B_k$ be as in the proof of Lemma \ref{lem.clCk}, and write $\tau_{n,d} \coloneqq \tau_n^{\oplus d}$, etc.\ for direct sum traces.
We begin by arguing that a $P$-essentially bounded $B_k$-valued random variable determines an element of $\mathbb{B}_k((\cA_{n,\beta}^d)^k ; \cA_m^{\ell})$ in the obvious way.
To this end, first observe that
\[
\norm{\a}_{L^p(\tr_{n,d})} \leq \norm{\a}_{L^{\infty}(\tr_{n,d})} \leq (nd)^\frac1p\norm{\a}_{L^p(\tr_{n,d})} \qquad \big(\a \in \MnC^d,\;p \in [1,\infty)\big).
\]
Consequently, if $T \colon (\MnC_{\beta}^d)^k \to \MmC^{\ell}$ is a $k$-linear map and $p,p_1,\ldots,p_k \in [1,\infty]$, then
\begin{equation}
    \norm{T}_{B_k(L^{p_1}(\tr_{n,d}) \times \cdots \times L^{p_k}(\tr_{n,d});L^p(\tr_{m,\ell}))} \leq (nd)^{\frac{1}{p_1}+\cdots+\frac{1}{p_k}}\norm{T}_{B_k}.\label{eq.Bkbd}
\end{equation}
Now, suppose $L \colon \Om \to B_k$ is a $P$-essentially bounded $B_k$-valued random variable, and let $\b_1,\ldots,\b_k \in \cA_{n,\beta}^d$.
If $p,p_1,\ldots,p_k \in [1,\infty]$ satisfy $1/p_1+\cdots+1/p_k = 1/p$, then
\begin{align*}
    \norm{L[\mathbf{b}_1,\ldots,\mathbf{b}_k]}_{L^p(\tau_{m,\ell})} & = \E_P\big[\norm{L[\mathbf{b}_1,\ldots,\mathbf{b}_k]}_{L^p(\tr_{m,\ell})}^p\big]^{\frac{1}{p}} \\
    & \leq (nd)^{\frac{1}{p_1}+\cdots+\frac{1}{p_k}} P\text{-}\esssup_{\om \in \Om} \norm{L(\om)}_{B_k} \E_P\big[\norm{\mathbf{b}_1}_{L^{p_1}(\tr_{n,d})}^p \cdots \norm{\mathbf{b}_k}_{L^{p_k}(\tr_{n,d})}^p\big]^{\frac{1}{p}} \\
    & \leq (nd)^{\frac1p}P\text{-}\esssup_{\om \in \Om} \norm{L(\om)}_{B_k} \prod_{i=1}^k\E_P\big[\norm{\mathbf{b}_i}_{L^{p_i}(\tr_{n,d})}^{p_i}\big]^{\frac{1}{p_i}} \\
    & = (nd)^{\frac1p}P\text{-}\esssup_{\om \in \Om} \norm{L(\om)}_{B_k} \norm{\mathbf{b}_1}_{L^{p_1}(\tau_{n,d})}\cdots \norm{\mathbf{b}_k}_{L^{p_k}(\tau_{n,d})} \displaybreak
\end{align*}
by \eqref{eq.Bkbd} and H\"older's inequality (with obvious adjustments for infinite indices).
Consequently, if we define $\Lambda \colon (\cA_{n,\beta}^d)^k \to \cA_m^{\ell}$ by $(\b_1,\ldots,\b_k) \mapsto L[\b_1,\ldots,\b_k]$, then $\Lambda \in \mathbb{B}_k((\cA_{n,\beta}^d)^k ; \cA_m^{\ell})$, and
\[
\vertiii{\Lambda}_k \leq nd\,P\text{-}\esssup_{\om \in \Om} \norm{L(\om)}_{B_k}.
\]
With this in hand, it is easy to see Lemma \ref{lem.clCk} implies that if $f$ is a $C^k$ function as in the statement, then $D^kf_* \colon \cA_{n,\beta}^d \to \mathbb{B}_k((\cA_{n,\beta}^d)^k ; \cA_m^{\ell})$ is continuous.

Next, suppose $f$ is a $C^k$ function as in the statement.
If $t \geq 0$ and $\mathbf{a} \in (\cA_{n,t}^d)_{\beta}$, then $\mathbf{a}$ is $\sF_t$-measurable as a random variable $\Om \to \MnC_{\beta}^d$.
Thus, the random variable $f \circ \mathbf{a} = f_*(\mathbf{a})$ is $\sF_t$-measurable, i.e., belongs to $\cA_{m,t}^{\ell}$.
In particular, $f_* \in C_a^0(\cA_{n,\beta}^d;\cA_m^{\ell})$.
It remains to show that if $u \geq t \geq 0$, $\mathbf{a} \in (\cA_{n,t}^d)_{\beta}$, $i =1,\ldots,k$, and $\mathbf{b}_1,\ldots,\mathbf{b}_k \in (\cA_{n,u}^d)_{\beta}$, then
\[
\tau_{m,\ell}\big[D^kf_*(\mathbf{a})[\mathbf{b}_1,\ldots,\mathbf{b}_{i-1},\mathbf{c},\mathbf{b}_{i+1},\ldots,\mathbf{b}_k] \mid \cA_{m,u}^{\ell}\big] = D^kf_*(\a)\big[\mathbf{b}_1,\ldots,\mathbf{b}_{i-1},\tau_{n,d}\big[\mathbf{c} \mid \cA_{n,u}^d\big],\mathbf{b}_{i+1},\ldots,\mathbf{b}_k\big]
\]
for all $\mathbf{c} \in \cA_{n,\beta}^d$.
Since $\mathbf{a},\mathbf{b}_1,\ldots,\mathbf{b}_k$ are all $\sF_u$-measurable, formula \eqref{eq.derivsanity} in Lemma \ref{lem.clCk} guarantees that there exists a $P$-essentially bounded $\sF_u$-measurable random variable $T = T(\mathbf{a},\mathbf{b}_1,\ldots,\b_{i-1},\b_{i+1},\ldots,\mathbf{b}_k)$ with values in $B(\MnC_{\beta}^d;\MmC^{\ell})$ such that
\[
(D^kf_*(\mathbf{a})[\mathbf{b}_1,\ldots,\mathbf{b}_{i-1},\mathbf{c},\mathbf{b}_{i+1},\ldots,\mathbf{b}_k])(\om) = T(\om)[\c(\om)] \qquad \big(\mathbf{c} \in \cA_{n,\beta}^d, \text{ a.e.\ } \om \in \Om\big).
\]
The required identity then follows from the fact that $\tau_{m,\ell}[\cdot \mid \cA_{m,u}^{\ell}]$ and $\tau_{n,d}[\cdot \mid \cA_{n,u}^d]$ are given in terms of classical matrix-valued conditional expectations of the form $\E_P[\cdot \mid \sF_u]$.
\end{proof}

\subsection{The formula}\label{sec.form}

In this section, we first state our noncommutative It\^{o}'s formula and discuss some useful special cases.
At the end of the section, we prove the formula.

\begin{thm}[Noncommutative It\^{o}'s formula]\label{thm.NCIFtimedep}
Let $\cU \subseteq \cA_{\beta} \times \cB_{\gamma}$ be an open set.
Suppose $A \in \mathbb{FV}_{\mathsmaller{\cA}}^{\infty}$ and $X \colon \R_+ \to \cB$ is $L^{\infty}$-decomposable.
If $(A(t),X(t)) \in \cU$ for all $t \geq 0$ and $F \in C_a^{1,2}(\cU;\cC)$, then
\[
\d F(A(t), X(t)) = D_aF(A(t),X(t))[\d A(t)] + D_xF(A(t),X(t))[\d X(t)] + \frac{1}{2} D_x^2F(A(t), X(t))[\d X(t), \d X(t)],
\]
where $D_a = D_1$, $D_x = D_2$, and $D_x^2 = D_2^2$ in the notation of Definition \ref{def.Ckl}.
More precisely,
\begin{align*}
    F(A,X) & = F(A(0),X(0)) + \into D_aF(A(t),X(t))[\d A(t)] \\
    & \hspace{7.5mm} + \into D_xF(A(t),X(t))[\d X(t)] + \frac{1}{2} \into D_x^2F(A(t), X(t))[\d X(t), \d X(t)].
\end{align*}
\end{thm}

Recall that we have seen several examples of $L^{\infty}$-decomposable processes.
Indeed, if $q \in [-1,1)$, $n \in \N_0$, and $X \colon \R_+ \to \cA_{\sa}$ is a $q$-Brownian motion, then the process $M_n(t) = t^{n/2}H_n^{(q)}\big(t^{-1/2}X(t)\big)$ belongs to $\M_{\mathsmaller{\cA}}^{\infty}$ (Example \ref{ex.qGauss}).
Also, if $X \colon \R_+ \to \cA$ is any $L^{\infty}$-decomposable process and $H \in \EP$, then $\into H[\d X] \colon \R_+ \to \cB$ is $L^{\infty}$-decomposable.
Using the work of Biane--Speicher \cite{BS1998}, we can also show that many stochastic integrals against semicircular Brownian motion are $L^{\infty}$-continuous martingales and thus are $L^{\infty}$-decomposable.

\begin{prop}\label{prop.babyfrint}
\hspace{-0.25mm}If\hspace{-0.25mm} $X\hspace{-0.25mm} \colon\hspace{-0.25mm} \R_+\hspace{-0.25mm} \to\hspace{-0.25mm} \cA_{\sa}$\hspace{-0.25mm} is\hspace{-0.25mm} a\hspace{-0.25mm} semicircular\hspace{-0.25mm} Brownian\hspace{-0.25mm} motion,\hspace{-0.25mm} $P\hspace{-0.25mm} \in\hspace{-0.25mm} \TrP_{n,1,1}^*$,\hspace{-0.25mm} and\hspace{-0.25mm} $Y_1,\hspace{-0.25mm}\ldots,\hspace{-0.25mm}Y_n\hspace{-0.25mm} \colon\hspace{-0.25mm} \R_+\hspace{-0.25mm} \to\hspace{-0.25mm} \cA$ are adapted and $L^{\infty}$-LCLB, then $\into P(Y_1(t),\ldots,Y_n(t),\d X(t)) \in \M_{\mathsmaller{\cA}}^{\infty}$.
\end{prop}

\begin{proof}
As the reader may verify, it suffices to prove that if $A,B,C,D,E \colon \R_+ \to \cA$ are adapted and $L^{\infty}$-LCLB, then $\into (A(t)\,\d X(t)\,B(t) + \E[C(t)\,\d X(t)\,D(t)]\,E(t)) \in \M_{\mathsmaller{\cA}}^{\infty}$.
Since
\[
\into \E[C(t)\,\d X(t)\,D(t)]\,E(t) = \into \E[D(t)\,C(t)\,\d X(t)]\,E(t) \equiv 0
\]
by Example \ref{ex.elemtrbipinteg}, we need to show $\into A(t)\,\d X(t)\,B(t) \in \M_{\mathsmaller{\cA}}^{\infty}$.
To this end, recall from Proposition \ref{prop.LERSapprox} that
\[
\mathbb{L}^2\text{-}\lim_{\Pi \in \cP_{\R_+}}\sum_{t \in \Pi}A(t_-)\,(X(t \wedge \cdot) - X(t_- \wedge \cdot))\,B(t_-) = \into A(t)\,\d X(t)\,B(t).\pagebreak
\]
Also, by the proof of the aforementioned proposition, if $\Pi,\Pi' \in \cP_{\R_+}$, then
\begin{align*}
    \e_{\Pi,\Pi'} & \coloneqq \sum_{t \in \Pi}A(t_-)\,(X(t \wedge \cdot) - X(t_- \wedge \cdot))\,B(t_-) - \sum_{t \in \Pi'}A(t_-)\,(X(t \wedge \cdot) - X(t_- \wedge \cdot))\,B(t_-) \\
    & = \into \big(A^{\Pi} \otimes B^{\Pi} - A^{\Pi'} \otimes B^{\Pi'}\big) \sh \d X,
\end{align*}
where the latter is an elementary integral.
Consequently, by \cite[Thm.\ 3.2.1]{BS1998}, Lemma \ref{lem.HPi}\ref{item.HPipw}, and the dominated convergence theorem, if $t \geq 0$, then
\begin{align*}
    \sup_{0 \leq s \leq t}\big\|\e_{\Pi,\Pi'}(s)\big\|_{\infty}^2 & \leq 8\int_0^t \big\|A^{\Pi} \otimes B^{\Pi} - A^{\Pi'} \otimes B^{\Pi'}\big\|_{\infty}^2 \,\d s \\
    & \leq 8\int_0^t \big(\big\|A^{\Pi}\big\|_{\infty} \big\|B^{\Pi} - B^{\Pi'}\big\|_{\infty} + \big\|A^{\Pi} - A^{\Pi'}\big\|_{\infty} \big\|B^{\Pi'}\big\|_{\infty}\big)^2 \,\d s \xrightarrow[\Pi,\Pi' \in \cP_{\R_+}]{|\Pi|,|\Pi'| \to 0} 0.
\end{align*}
To be clear, the $L^{\infty}$ norm on the right-hand side of the first inequality above is the (operator) norm on the minimal $\mathrm{C}^*$-tensor product $\cA \otimes_{\min} \cA^{\op}$ (\cite[\S2.2]{NikitopoulosIto}).
Since $\M_{\mathsmaller{\cA}}^{\infty}$ is complete, we conclude that $\big(\sum_{t \in \Pi}A(t_-)\,(X(t \wedge \cdot) - X(t_- \wedge \cdot))\,B(t_-)\big)_{\Pi \in \cP_{\R_+}}$ converges in $\M_{\mathsmaller{\cA}}^{\infty}$.
The result follows.
\end{proof}

\begin{rem}
We note that \cite[Thm.\ 3.2.1]{BS1998} is stated and proved in the $\mathrm{W}^*$ case.
Once again, as in Remark \ref{rem.BDG}, one can deduce the general $\mathrm{C}^*$ case from Appendix \ref{sec.CstarLp}.
\end{rem}

Next, we examine some useful special cases of our noncommutative It\^{o}'s formula.

\begin{ex}[Time-dependent It\^{o}'s formula]\label{ex.timedepIF}
The prototypical example of an FV argument $A$ in Theorem \ref{thm.NCIFtimedep} is $A(t) = t$, which results in the ``noncommutative time-dependent It\^{o}'s formula.''
Fix an open set $\cU \subseteq \cA_{\beta}$ and a map $F \colon \R_+ \times \cU \to \cB$.
Suppose there exist $r > 0$ and $F_r \in C_a^{1,2}((-r,\infty) \times \cU; \cB)$ such that $F_r|_{\R_+ \times \cU} = F$.
(This should be interpreted as the condition $F \in C_a^{1,2}(\R_+ \times \cU; \cB)$.
Also, note that we consider $(-r,\infty)$ to be an open subset of $\C_{\sa} = \R$.)
If $X \colon \R_+ \to \cA$ is $L^{\infty}$-decomposable and $X(t) \in \cU$ for all $t \geq 0$, then Theorem \ref{thm.NCIFtimedep} gives
\[
\d F(t,X(t)) = \partial_tF(t,X(t)) \,\d t + D_xF(t,X(t))[\d X(t)] + \frac{1}{2}D_x^2F(t,X(t))[\d X(t), \d X(t)].
\] 
If $F$ has no time dependence, i.e., if $F \in C_a^2(\cU;\cB)$, then we get
\[
\d F(X(t)) = DF(X(t))[\d X(t)] + \frac{1}{2}D^2F(X(t))[\d X(t), \d X(t)],
\]
which is Theorem \ref{thm.NCIF} from the introduction.
\end{ex}

\begin{ex}[Conjugation]\label{ex.conjug}
Let $a \in \cA_0$ and $X \colon \R_+ \to \GL(\cA) \subseteq \cA$ be $L^{\infty}$-decomposable.
Noncommutative It\^{o}'s formula (in the form of Theorem \ref{thm.NCIF}) applied to the adapted $C^2$ map $\cU = \mathrm{GL}(\cA) \ni g \mapsto gag^{-1} \in \cA$ (similar to Example \ref{ex.inv}) yields
\begin{align*}
    XaX^{-1} & = X(0)aX(0)^{-1} + \into \Big(\d X(t)\,aX(t)^{-1} - X(t)aX(t)^{-1}\,\d X(t)\,X(t)^{-1}\Big) \\
    & \hspace{5mm} + \into \Big(X(t)aX(t)^{-1}\,\d X(t)\,X(t)^{-1}\,\d X(t)\,X(t)^{-1} - \d X(t)\,aX(t)^{-1}\,\d X(t)\,X(t)^{-1}\Big).
\end{align*}
Already, this is an example other noncommutative It\^{o} formulas from the literature cannot directly handle.
\end{ex}

\begin{ex}[Multivariate case]\label{ex.multivarIF}
Fix $n,m \in \N$ and, for each $i =1,\ldots,n$ and $j =1,\ldots,m$, filtered $\mathrm{C}^*$-probability spaces $(\cA_i, (\cA_{i,t})_{t \geq 0}, \E_{\mathsmaller{\cA}_i})$ and $(\cB_j, (\cB_{j,t})_{t \geq 0}, \E_{\mathsmaller{\cB}_j})$.
If we take
\[
\cA = \cA_1 \oplus \cdots \oplus \cA_n \; \text{ and } \; \cB = \cB_1 \oplus \cdots \oplus \cB_m
\]
(with the direct sum filtrations and traces) in Theorem \ref{thm.NCIFtimedep}, then we obtain a multivariate version of noncommutative It\^{o}'s formula.
Specifically, fix an open set $\cU \subseteq \cA_{\beta} \times \cB_{\gamma}$.
Now, suppose that, for each\pagebreak\ $i =1,\ldots,n$ and $j =1,\ldots,m$, $A_i \in \mathbb{FV}_{\mathsmaller{\cA_i}}^{\infty}$ and $X_j \colon \R_+ \to \cB_j$ is $L^{\infty}$-decomposable.
If $F \in C_a^{1,2}(\cU;\cC)$ and the process $(A,X) \coloneqq (A_1,\ldots,A_n,X_1,\ldots,X_m)$ takes values in $\cU$, then
\begin{align*}
    \d F(A(t),X(t)) & = \sum_{i=1}^n D_{a_i}F(A(t),X(t))[\d A_i(t)] + \sum_{j=1}^m D_{x_j}F(A(t),X(t))[\d X_j(t)] \\
    & \hspace{18mm} + \frac{1}{2}\sum_{j,k=1}^m D_{x_k}D_{x_j}F(A(t),X(t))[\d X_j(t), \d X_k(t)],
\end{align*}
where $D_{a_i}$ is the derivative in the $i^{\text{th}}$ variable and $D_{x_j}$ is the derivative in the $(n+j)^{\text{th}}$ variable.
\end{ex}

We now get to work on proving Theorem \ref{thm.NCIFtimedep}.
To begin, we recall Taylor's theorem with integral remainder.
Let $\cV$ and $\cW$ be real Banach spaces, and let $\cU \subseteq \cV$ be a convex open set.
Taylor's theorem (e.g., \cite[Thm.\ 1.107]{HJ2014}) says that if $k \in \N$ and $F \in C^k(\cU;\cW)$, then
\[
F(p+h) - F(p) - \sum_{i=1}^{k-1} \frac{1}{i!}\partial_h^iF(p) = \frac{1}{(k-1)!}\int_0^1 (1-t)^{k-1}\partial_h^kF(p+th)\,\d t
\]
for all $p \in \cU$ and $h \in \cV$ such that $p+h \in \cU$.
We shall freely use this below.

\begin{proof}[Proof of Theorem \ref{thm.NCIFtimedep}]
Let $t \geq 0$.
By a standard Lebesgue number lemma argument using the compactness of $\{(A(s),X(s)) : 0 \leq s \leq t\}$, there exist $\e,\delta_1,\ldots,\delta_n > 0$ and $(a_1,x_1),\ldots,(a_n,x_n) \in \cU$ such~that
\begin{enumerate}[label=(\roman*),font=\normalfont]
    \item $\cU_i \coloneqq B_{\delta_i}(a_i) \times B_{\delta_i}(x_i) \subseteq \cU$ for all $i =1,\ldots,n$; and
    \item if $0 \leq r,s \leq t$ and $|r-s| < \e$, then $(A(r),X(s)) \in \cU_i$ for some $i \in \{1,\ldots,n\}$.
\end{enumerate}
Now, fix a partition $\pi$ of $[0,t]$ such that $|\pi| < \e$.
If $s \in \pi$, then $|s-s_-| \leq |\pi| < \e$, so there exists an $i \in \{1,\ldots,n\}$ such that $(A(r_1),X(r_2)) \in \cU_i = B_{\delta_i}(a_i) \times B_{\delta_i}(x_i) \subseteq \cU$ whenever $r_1,r_2 \in [s_-,s]$.
We may therefore write
\begin{align*}
    F(A(t),X(t)) - F(A(0),X(0)) & = \sum_{s \in \pi}\big(F(A(s),X(s)) - F(A(s_-),X(s_-))\big) \\
    & = \sum_{s \in \pi}\big(F(A(s),X(s)) - F(A(s_-),X(s)) \\
    & \hspace{15mm} + F(A(s_-),X(s)) - F(A(s_-),X(s_-))\big).
\end{align*}
Next, letting $s \in \pi$ and $i$ be as before, we appeal to the convexity of $B_{\delta_i}(a_i)$ and $B_{\delta_i}(x_i)$ to use Taylor's theorem with integral remainder in two ways.
First,
\begin{align*}
    F(A(s)&, X(s)) - F(A(s_-),X(s)) = \int_0^1 D_aF(A(s_-) + r\Delta_sA,X(s))[\Delta_sA]\,\d r \\
    & = D_aF(A(s_-),X(s))[\Delta_sA] + \int_0^1 \big(D_aF(A(s_-) + r\Delta_sA,X(s)) - D_aF(A(s_-),X(s))\big)[\Delta_sA]\,\d r.
\end{align*}
Second,
\begin{align*}
    F(A(s_-)&,X(s)) - F(A(s_-),X(s_-)) \\
    & = D_xF(A(s_-),X(s_-))[\Delta_sX] + \int_0^1(1-r) D_x^2F(A(s_-),X(s_-)+r\Delta_sX)[\Delta_sX,\Delta_sX]\,\d r \\
    & = D_xF(A(s_-),X(s_-))[\Delta_sX] + \frac{1}{2}D_x^2F(A(s_-),X(s_-))[\Delta_sX,\Delta_sX] \\
    & \hspace{6.75mm} + \int_0^1(1-r) \big(D_x^2F(A(s_-),X(s_-)+r\Delta_sX) - D_x^2F(A(s_-),X(s_-))\big)[\Delta_sX,\Delta_sX]\,\d r
\end{align*}
because $\int_0^1(1-r)\,\d r = 1/2$.
\pagebreak

We now identify the terms of interest.
Since the linear process $\R_+ \ni t \mapsto D_xF(A(t),X(t)) \in \mathbb{B}(\cB;\cC)$ is $\norm{\cdot}_{2;2} \leq \vertiii{\cdot}$-continuous and adapted, Proposition \ref{prop.LERSapprox} yields
\begin{equation}
    L^2\text{-}\lim_{\pi \in \cP_{[0,t]}}\sum_{s \in \pi}D_xF(A(s_-),X(s_-))[\Delta_sX] = \int_0^t D_xF(A(s),X(s))[\d X(s)]. \label{eq.IFproof1}
\end{equation}
Since the bilinear process $\R_+ \ni t \mapsto D_x^2F(A(t),X(t)) \in \mathbb{B}_2(\cB^2;\cC)$ is $\norm{\cdot}_{2,2;1} \leq \vertiii{\cdot}_2$-continuous and adapted, Theorem \ref{thm.QC1}\ref{item.cheapRSsumQC} yields
\begin{equation}
    L^1\text{-}\lim_{\pi \in \cP_{[0,t]}}\sum_{s \in \pi}D_x^2F(A(s_-),X(s_-))[\Delta_sX, \Delta_sX] = \int_0^t D_x^2F(A(s),X(s))[\d X(s), \d X(s)]. \label{eq.IFproof2}
\end{equation}
Next, a slight adjustment of the proof of Proposition \ref{prop.contRSint} yields
\begin{equation}
    L^{\infty}\text{-}\lim_{\pi \in \cP_{[0,t]}}\sum_{s \in \pi}D_aF(A(s_-),X(s))[\Delta_sA] = \int_0^t D_aF(A(s),X(s))[\d A(s)] \label{eq.IFproof3}
\end{equation}
because $\R_+ \ni t \mapsto D_aF(A(t),X(t)) \in \mathbb{B}(\cA;\cC)$ is $\norm{\cdot}_{\infty;\infty} \leq \vertiii{\cdot}$-continuous and $A$ is $L^{\infty}$-FV.

Finally, we show that the remaining terms converge to zero.
To this end, define
\begin{align*}
    \e_{\pi} & \coloneqq \sum_{s \in \pi}\int_0^1 \big(D_aF(A(s_-) + r\Delta_sA,X(s)) - D_aF(A(s_-),X(s))\big)[\Delta_sA]\,\d r \, \text{ and} \\
    \delta_{\pi} & \coloneqq \sum_{s \in \pi}\int_0^1(1-r) \big(D_x^2F(A(s_-),X(s_-)+r\Delta_sX) - D_x^2F(A(s_-),X(s_-))\big)[\Delta_sX,\Delta_sX]\,\d r.
\end{align*}
Then
\begin{align*}
     \norm{\e_{\pi}}_{\infty} & \leq \int_0^1\sum_{s \in \pi}\big\|D_aF(A(s_-) + r\Delta_sA,X(s)) - D_aF(A(s_-),X(s))\big\|_{\infty;\infty}\norm{\Delta_sA}_{\infty}\,\d r \\
    & \leq V_{\cA}(A : [0,t])\sup_{(s,r) \in \pi \times [0,1]}\vertiii{D_aF(A(s_-) + r\Delta_sA,X(s)) - D_aF(A(s_-),X(s))} \xrightarrow[\pi \in \cP_{[0,t]}]{|\pi| \to 0} 0 \numberthis\label{eq.IFproof4}
\end{align*}
because $(A,X) \colon \R_+ \to \cU$ and $D_aF \colon \cU \to \mathbb{B}(\cA;\cC)$ are uniformly continuous on compact sets.
Next, decomposing $X$ as $X = X(0) + N + B$, note that
\begin{align*}
    \sum_{s \in \pi} \norm{\Delta_sX}_2^2 & \leq \sum_{s \in \pi} \norm{\Delta_sN}_2^2 + \sum_{s \in \pi} \big(2\norm{\Delta_sN}_2+\norm{\Delta_sB}_2\big)\norm{\Delta_sB}_2 \\
    & \leq \kappa_N((0,t]) + 2\sup_{0 \leq s \leq t}\big(2\norm{N(s)}_2 + \norm{B(s)}_2\big) \, V_{L^2(\E_{\mathsmaller{\cB}})}(B : [0,t]) =\vcentcolon C_t.
\end{align*}
Since $C_t < \infty$, we obtain
\begin{align*}
    \norm{\delta_{\pi}}_1 & \leq \int_0^1(1-r)\sum_{s \in \pi}\big\|D_x^2F(A(s_-),X(s_-)+r\Delta_sX) - D_x^2F(A(s_-),X(s_-)\big\|_{2,2;1}\norm{\Delta_sX}_2^2\,\d r \\
    & \leq \frac{C_t}{2}\sup_{(s,r) \in \pi \times [0,1]}\vertiii{D_x^2F(A(s_-),X(s_-)+r\Delta_sX) - D_x^2F(A(s_-),X(s_-))}_2 \xrightarrow[\pi \in \cP_{[0,t]}]{|\pi| \to 0} 0 \numberthis\label{eq.IFproof5}
\end{align*}
because $(A,X) \colon \R_+ \to \cU$ and $D_x^2F \colon \cU \to \mathbb{B}_2(\cB^2;\cC)$ are uniformly continuous on compact sets.
Putting together \eqref{eq.IFproof1}--\eqref{eq.IFproof5}, we conclude that
\begin{align*}
    F(A(t),X(t)) - F(A(0),X(0)) & = \sum_{s \in \pi}D_aF(A(s_-),X(s))[\Delta_sA] + \sum_{s \in \pi}D_xF(A(s_-),X(s_-))[\Delta_sX] \\
    & \hspace{15mm} + \frac{1}{2}\sum_{s \in \pi} D_x^2F(A(s_-),X(s_-))[\Delta_sX,\Delta_sX] + \e_{\pi} + \delta_{\pi} \\
    & \hspace{-5.25mm}\xrightarrow[\pi \in \cP_{[0,t]}]{|\pi| \to 0} \int_0^tD_aF(A,X)[\d A] + \int_0^t D_xF(A,X)[\d X] + \frac{1}{2}\int_0^tD_x^2F(A,X)[\d X, \d X]
\end{align*}
in $L^1(\E_{\mathsmaller{\cC}})$.
This completes the proof.
\end{proof}

\subsection{Examples: Trace smooth maps}\label{sec.trsmoothmaps}

In this section, we introduce a class of adapted $C^k$ maps large enough that it contains most common examples of interest, including those induced by functional calculus.
The maps we consider are inspired by the tracial noncommutative $C^k$ functions introduced and studied by Jekel--Li--Shlyakhtenko \cite{JLS2022}.
To begin our study, we describe how to differentiate trace $\ast$-polynomials.

\begin{lem}\label{lem.partial}
Let $n \in \N$, and write $P_{i,\e}(x_1,\ldots,x_n) \coloneqq x_i^{\e} \in \TrP^*(x_1,\ldots,x_n) = \TrP^*(\x)$ for $i =1,\ldots,n$ and $\e \in \{1,\ast\}$.
If $i =1,\ldots,n$, then there exists a unique complex-linear map
\[
\partial_{x_i} \colon \TrP^*(\x) \to \TrP^*(\x)[y]
\]
such that for all $j =1,\ldots,n$, $\e \in \{1,\ast\}$, $P \in \C^*\la \x\ra$, and $Q,R \in \TrP^*(\x)$,
\begin{align*}
    (\partial_{x_i}P_{j,\e})(\x,y) & = \delta_{ij}y^{\e}, \\
   (\partial_{x_i}\tr(P))(\x,y) & = \tr((\partial_{x_i}P)(\x,y)), \;\text{ and} \\
    (\partial_{x_i}(QR))(\x,y) & = (\partial_{x_i}Q)(\x,y)\,R(\x) + Q(\x)\,(\partial_{x_i}R)(\x,y).
\end{align*}
\end{lem}

\begin{proof}
By \cite[Lem.\ 3.5]{JLS2022}, there exists a unique complex-linear map $\tilde{\partial}_{x_i} \colon \TrP(\x) \to \TrP^*(\x)[y]$ such that for all $j =1,\ldots,n$, $P \in \C\la \x\ra$, and $Q,R \in \TrP(\x)$,
\begin{align*}
    (\partial_{x_i}P_{j,1})(\x,y) & = \delta_{ij}y, \\
    (\partial_{x_i}\tr(P))(\x,y) & = \tr((\partial_{x_i}P)(\x,y)), \;\text{ and} \\
    (\partial_{x_i}(QR))(\x,y) & = (\partial_{x_i}Q)(\x,y)\,R(\x) + Q(\x)\,(\partial_{x_i}R)(\x,y).
\end{align*}
Recall that $\TrP^*(x_1,\ldots,x_n) = \TrP(x_1,y_1,\ldots,x_n,y_n)$, where $x_j^* = y_j$.
Thus, $\partial_{x_i} \coloneqq \tilde{\partial}_{x_i} + \tilde{\partial}_{y_i} = \tilde{\partial}_{x_i} + \tilde{\partial}_{x_i^*}$ does the job.
For uniqueness, note that if $D \subseteq \TrP^*(\x)$ is a complex-linear subspace containing $1,x_1,x_1^*\ldots,x_n,x_n^*$ that is closed under multiplication and $\tr$ (i.e., $P,Q \in D \Rightarrow PQ,\tr(P) \in D$), then $D = \TrP^*(\x)$.
Taking $D$ to be the set on which two candidates for $\partial_{x_i}$ agree, we see these candidates must agree on all of $\TrP^*(\x)$.
\end{proof}

\begin{ex}
If $P(x_1,x_2,x_3) = x_1x_2x_2^*x_3 + 3i\tr(x_1x_2^*)\,x_2 + x_1^*x_3^2 + 5$, then
\[
(\partial_{x_2}P)(x_1,x_2,x_3,y) = x_1yx_2^*x_3 + x_1x_2y^*x_3 + 3i\tr(x_1y^*)\,x_2 + 3i\tr(x_1x_2^*)\,y.
\]
Procedurally speaking, the trace polynomial $\partial_{x_i}P$ is computed by finding each individual occurrence of $x_i^{\e}$, replacing it with $y^{\e}$, and then adding up the resulting trace polynomials.
\end{ex}

As we shall see, $\partial_{x_i}P$ is related to the calculation of the first derivative of $P_{\mathsmaller{(\cA,\E)}}$.
Next, we describe how the $\partial_{x_i}$ operators can be applied multiple times, which will help us compute higher derivatives of $P_{\mathsmaller{(\cA,\E)}}$.
Let $k \in \N$ and $i_1,\ldots,i_k \in \{1,\ldots,n\}$, and suppose we have defined
\[
Q \coloneqq \partial_{x_{i_k}}\cdots\partial_{x_{i_1}}P \in \TrP^*(\x)[y_1,\ldots,y_k] \subseteq \TrP^*(\x,y_1,\ldots,y_k).
\]
If $i_{k+1} \in \{1,\ldots,n\}$, then we define
\[
\partial_{x_{i_{k+1}}}\cdots\partial_{x_{i_1}}P \coloneqq \partial_{x_{i_{k+1}}}Q \in \TrP^*(\x,y_1,\ldots,y_k)[y_{k+1}].
\]
Above, $\partial_{x_{i_{k+1}}}$ is the operator from Lemma \ref{lem.partial} that maps $\TrP^*(\x,y_1,\ldots,y_k)$ to $\TrP^*(\x,y_1,\ldots,y_k)[y_{k+1}]$.
From this recursive definition, it is easy to see that, in fact,
\[
\partial_{x_{i_{k+1}}}\cdots\partial_{x_{i_1}}P \in \TrP^*(\x)[y_1,\ldots,y_{k+1}] \subseteq \TrP^*(\x,y_1,\ldots,y_k)[y_{k+1}].
\]
With this notation, we can define an algebraic $k^{\text{th}}$ ``total'' derivative of $P$.

\begin{nota}\label{nota.TrPolyder}
If $n,k \in \N$ and $P \in \TrP_n^* = \TrP^*(\x)$, then we write
\[
\big(\partial^kP\big)(\x,\y_1,\ldots,\y_k) \coloneqq \sum_{i_1,\ldots,i_k=1}^n(\partial_{x_{i_k}}\cdots\partial_{x_{i_1}}P)(\x,y_{1,i_1},\ldots,y_{k,i_k}) \in \TrP^*(\x)[\y_1,\ldots,\y_k],
\]
where $\y_i = (y_{i,1},\ldots,y_{i,n})$ ($i=1,\ldots,n$).
Also, if $m \in \N$ and $P = (P_1,\ldots,P_m) \in (\TrP_n^*)^m$, then we write
\[
\partial^kP \coloneqq (\partial^kP_1,\ldots,\partial^kP_m) \in (\TrP_{n,k,(n,\ldots,n)}^*)^m.
\]
(Recall that $\TrP_{n,k,(n,\ldots,n)}^* = \TrP^*(\x)[\y_1,\ldots,\y_k]$.)
Finally, we write $\partial P \coloneqq \partial^1P$.
\end{nota}
\pagebreak

\begin{ex}
If $p(x) = x^n \in \C[x] \subseteq \TrP^*(x)$, then\vspace{-0.275mm}
\[
\big(\partial^kp\big)(x,y_1,\ldots,y_k) = \sum_{\pi \in S_k}\sum_{|\delta| = n-k} x^{\delta_1}y_{\pi(1)}\cdots x^{\delta_k}y_{\pi(k)}x^{\delta_{k+1}} \in \TrP^*(x,y_1,\ldots,y_k).\vspace{-0.275mm}
\]
In particular, if $a,b_1,\ldots,b_k \in \cA$, then\vspace{-0.275mm}
\[
\big(\partial^kp\big)(a,b_1,\ldots,b_k) = \frac{1}{k!}\sum_{\pi \in S_k}\partial_{\mathsmaller{\otimes}}^kp(a)\sh_k[b_{\pi(1)},\ldots,b_{\pi(k)}],\vspace{-0.275mm}
\]
where $\partial_{\mathsmaller{\otimes}}^kp(a) \in \cA^{\otimes(k+1)}$ is the noncommutative derivative from Notation \ref{nota.ncder}.
\end{ex}

We now prove a result that makes rigorous the idea that $\partial^kP$ is the $k^{\text{th}}$ derivative of $P$.
Both the result and its proof are very similar to \cite[Lem.\ 3.7]{JLS2022}.
In fact, it is possible to deduce our result from \cite[Lem.\ 3.7]{JLS2022} by breaking arguments into their real and imaginary parts, but doing so requires a similar level of effort to simply (re-)proving the result from scratch.

\begin{thm}[Higher Fr\'echet derivatives of trace $\ast$-polynomials]\label{thm.TrPolyder}
If $n,m \in \N$ and $P \in (\TrP_n^*)^m$, then $P_{\mathsmaller{(\cA,\E)}} \in C^{\infty}(\cA^n;\cA^m)$ when $\cA^n$ and $\cA^m$ are viewed as real Banach spaces.
Moreover, if $k \in \N$, then\vspace{-0.275mm}
\[
D^kP_{\mathsmaller{(\cA,\E)}}(\a)[\b_1,\ldots,\b_k] = \big(\partial^kP\big)(\a)[\b_1,\ldots,\b_k] \qquad \big(\a,\b_1,\ldots,\b_k \in \cA^n\big).\vspace{-0.275mm}
\]
In particular, $D^kP_{\mathsmaller{(\cA,\E)}}(\a) \in \mathbb{B}_k((\cA^n)^k;\cA^m)$ for all $\a \in \cA^n$, and, as a map from $\cA^n$ to $\mathbb{B}_k((\cA^n)^k;\cA^m)$, $D^kP_{\mathsmaller{(\cA,\E)}} = (\partial^kP)_{\mathsmaller{(\cA,\E)}}$ belongs to $BC_{\loc}(\cA^n;\mathbb{B}_k((\cA^n)^k;\cA^m))$.
\end{thm}

\begin{proof}
It suffices to treat the $m=1$ case.
To this end, let $P \in \TrP_n^*$ and $k \in \N$.
Recall that $(\partial^kP)_{\mathsmaller{(\cA,\E)}}$ belongs to $BC_{\loc}(\cA^n; \mathbb{B}_k((\cA^n)^k ;\cA))$.
In particular, $(\partial^kP)_{\mathsmaller{(\cA,\E)}}$ is continuous as a map from $\cA^n$ to $B_k((\cA^n)^k;\cA)$.
By \cite[Fact 1.73]{HJ2014}, it therefore suffices to prove that if $\a,\b_1,\ldots,\b_k \in \cA^n$, then the directional derivative $\partial_{\b_k}\cdots \partial_{\b_1}P_{\mathsmaller{(\cA,\E)}}(\a)$ exists in $\cA$, and $\partial_{\b_k}\cdots \partial_{\b_1}P_{\mathsmaller{(\cA,\E)}}(\a) = (\partial^kP)(\a)[\b_1,\ldots,\b_k]$.
We shall prove this by induction on $k$.

For the base case, let $D \coloneqq \{Q \in \TrP_n^* : \partial_{\b}Q_{\mathsmaller{(\cA,\E)}}(\a) = (\partial Q)(\a)[\b]$ for all $\a,\b \in \cA^n\}$.
Clearly, $D \subseteq \TrP_n^*$ is a complex-linear subspace containing $1$.
Now, if $j =1,\ldots,n$ and $\e \in \{1,\ast\}$, then\vspace{-0.275mm}
\[
\partial_{\b}(P_{j,\e})_{\mathsmaller{(\cA,\E)}}(\a) = \lim_{t \to 0} \frac{(a_j+tb_j)^{\e} - a_j^{\e}}{t} = b_j^{\e} = \sum_{i=1}^n\delta_{ij}b_i^{\e} = (\partial P_{j,\e})(\a)[\b].\vspace{-0.275mm}
\]
Thus, $x_1,x_1^*,\ldots,x_n,x_n^* \in D$.
Now, if $P,Q \in D$, then the Leibniz rule yields\vspace{-0.275mm}
\begin{align*}
    \partial_{\b}(PQ)_{\mathsmaller{(\cA,\E)}}(\a) & = \partial_{\b}(P_{\mathsmaller{(\cA,\E)}}Q_{\mathsmaller{(\cA,\E)}})(\a) = (\partial_{\b}P_{\mathsmaller{(\cA,\E)}}(\a))\,Q(\a) + P(\a) \, (\partial_{\b}Q_{\mathsmaller{(\cA,\E)}}(\a)) \vspace{-0.275mm}\\
    & = (\partial P)(\a)[\b] \, Q(\a) + P(\a) \, \partial Q(\a)[\b] = \partial (PQ)(\a)[\b]\vspace{-0.275mm}
\end{align*}
so that $PQ \in D$ as well.
Finally, since $\E$ is a linear map, we get\vspace{-0.275mm}
\begin{align*}
    \partial_{\b}\tr(P)_{\mathsmaller{(\cA,\E)}}(\a) & = \partial_{\b}\E[P_{\mathsmaller{(\cA,\E)}}](\a)  = \E[\partial_{\b}P_{\mathsmaller{(\cA,\E)}}(\a)] \vspace{-0.275mm}\\
    & = \E[(\partial P)(\a)[\b]] = \tr(\partial P)(\a)[\b] = (\partial \tr(P))(\a)[\b]\vspace{-0.275mm}
\end{align*}
so that $\tr(P) \in D$.
It follows that $D = \TrP_n^*$.

Finally, suppose we know the claimed formula for $(k-1)$-fold directional derivatives of trace $\ast$-polynomials.
Since $\partial^{k-1}P \in \TrP_{n,k,(n,\ldots,n)}^* \subseteq \TrP_{kn}^*$, we can apply the base case to $Q \coloneqq \partial^{k-1}P$ viewed simply as a member of $\TrP_{kn}^* = \TrP^*(\x,\y_1,\ldots,\y_{k-1})$.
In particular, if $\mathbf{A},\mathbf{B} \in \cA^{kn}$, then\vspace{-0.275mm}
\[
\partial_{\mathbf{B}}Q_{\mathsmaller{(\cA,\E)}}(\mathbf{A}) = (\partial Q)(\mathbf{A})[\mathbf{B}] = \sum_{j=1}^n \Bigg(\big(\partial_{x_j}Q\big)(\mathbf{A})[B_{1,j}] + \sum_{i=1}^{k-1}\big(\partial_{y_{i,j}}Q\big)(\mathbf{A})[B_{i+1,j}]\Bigg),\vspace{-0.275mm}
\]
where $\mathbf{B} = (B_{1,1},\ldots,B_{1,n},\ldots,B_{k,1},\ldots,B_{k,n})$.
Now, applying the induction hypothesis to $P$ and plugging in $\mathbf{A} = (\a,\b_1,\ldots,\b_{k-1}) \in \cA^{kn}$ and $\mathbf{B} = (\b_k,0,\ldots,0) \in \cA^{kn}$ above, we get\vspace{-0.275mm}
\begin{align*}
    \partial_{\b_k}\cdots \partial_{\b_1}P_{\mathsmaller{(\cA,\E)}}(\a) & = \frac{\d}{\d t}\Big|_{t=0}\big(\partial^{k-1}P\big)(\a+t\b_k,\b_1,\ldots,\b_{k-1}) = \partial_{\mathbf{B}}Q_{\mathsmaller{(\cA,\E)}}(\mathbf{A}) \vspace{-0.275mm}\\
    & = \sum_{j=1}^n \big(\partial_{x_j}Q\big)(\a,\b_1,\ldots,\b_{k-1})[b_{k,j}] = \big(\partial^kP\big)(\a)[\b_1,\ldots,\b_k].\vspace{-0.275mm}
\end{align*}
This completes the proof.
\end{proof}\pagebreak

We use this result as a jumping-off point to define a large class of functions $\cA_{\beta}^n \supseteq \cU \to \cA^m$ of multiple noncommuting variables.

\begin{defi}[Trace continuous/smooth maps]\label{def.tracecontCk}
Fix $n,m \in \N$, $k \in \N_0$, and $d \coloneqq (d_1,\ldots,d_k) \in \N^k$.
Also, let $\cU \subseteq \cA_{\beta}^n$ be an open set.
Finally, recall $\mathbb{B}_0(\cA_{\beta}^{\emptyset};\cA^m) \coloneqq \cA^m$ and $\vertiii{\cdot}_0 = \norm{\cdot}$.
\begin{enumerate}[label=(\roman*),font=\normalfont]
    \item Let $\cA_{\gamma} \in \{\cA,\cA_{\sa}\}$.
    Define $C_{\E}(\cU;\mathbb{B}_k(\cA_{\gamma}^d;\cA^m))$ to be the set of $F \colon \cU \to \mathbb{B}_k(\cA_{\gamma}^d;\cA^m)$ such that for all $\a \in \cU$, there exists an $r > 0$ and a sequence $(P_j)_{j \in \N}$ in $(\TrP_{n,k,d}^*)^m$ such that
    \[
    B_r(\a) \coloneqq \big\{\b \in \cA_{\beta}^n : \norm{\a - \b}_{\infty} < r\big\} \subseteq \cU \; \text{ and }  \; \sup_{\b \in B_r(\a)} \vertiii{F(\b) - (P_j)_{\mathsmaller{(\cA,\E)}}(\b)}_k \xrightarrow{j \to \infty} 0.
    \]
    The members of $C_{\E}(\cU;\mathbb{B}_k(\cA_{\gamma}^d;\cA^m))$ are called \textbf{trace continuous} maps from $\cU$ to $\mathbb{B}_k(\cA_{\gamma}^d;\cA^m)$.
    Note that $C_{\E}(\cU;\mathbb{B}_k(\cA_{\gamma}^d;\cA^m)) \subseteq C(\cU;\mathbb{B}_k(\cA_{\gamma}^d;\cA^m))$.\label{item.tracecont}
    \item Define $C_{\E}^k(\cU;\cA^m)$ to be the space of $F \in C^k(\cU;\cA^m)$ such that
    \[
    D^iF \in C_{\E}(\cU;\mathbb{B}_i((\cA_{\beta}^n)^i;\cA^m)) \qquad (i =0,\ldots,k).
    \]
    The members of $C_{\E}^k(\cU;\cA^m)$ are called \textbf{trace $\boldsymbol{C^k}$} maps from $\cU$ to $\cA^m$.\label{item.traceCk}
\end{enumerate}
Also, write $C_{\E}^0(\cU;\mathbb{B}_k(\cA_{\gamma}^d;\cA^m)) \coloneqq C_{\E}(\cU;\mathbb{B}_k(\cA_{\gamma}^d;\cA^m))$ and $C_{\E}^{\infty}(\cU;\cA^m) \coloneqq \bigcap_{k \in \N}C_{\E}^k(\cU;\cA^m)$.
\end{defi}

\begin{ex}[Trace $\ast$-polynomials]\label{ex.trpolytrC}
If $Q \in (\TrP_{n,k,d}^*)^m$, then $Q_{\mathsmaller{(\cA,\E)}} \in C_{\E}(\cA^n;\mathbb{B}_k(\cA^d;\cA^m))$.
Consequently, by Theorem \ref{thm.TrPolyder}, if $P \in (\TrP_n^*)^m$, then $P_{\mathsmaller{(\cA,\E)}} \in C_{\E}^{\infty}(\cA^n;\cA^m)$.
\end{ex}

\begin{ex}[Inversion map]
Using geometric series arguments and the formula from Example \ref{ex.inv}, one can show that if $\cU = \GL(\cA)$ and $F(g) \coloneqq g^{-1}$ for all $g \in \cU$, then $F \in C_{\E}^{\infty}(\cU;\cA)$.
\end{ex}

For the next example, recall that $BC_{\loc}(\cV;\cW)$ is the Fr\'echet space of continuous maps $\cV \to \cW$ that are bounded on bounded sets (Notation \ref{nota.BCloc}).

\begin{ex}\label{ex.JLS}
The closure of $\{P_{\mathsmaller{(\cA,\E)}} : P \in (\TrP_{n,k,d}^*)^m\}$ in $BC_{\loc}(\cA_{\beta}^n;\mathbb{B}_k(\cA_{\gamma}^d;\cA^m))$ is contained in $C_{\E}(\cA_{\beta}^n;\mathbb{B}_k(\cA_{\gamma}^d;\cA^m))$.
In particular, if $F \in C^k(\cA_{\beta}^n;\cA^m)$ is such that for all $i =0,\ldots,k$, the $i^{\text{th}}$ derivative $D^iF$ belongs to the closure of $\{P_{\mathsmaller{(\cA,\E)}} : P \in (\TrP_{n,i,(n,\ldots,n)}^*)^m\}$ in $BC_{\loc}(\cA_{\beta}^n;\mathbb{B}_i((\cA_{\beta}^n)^i;\cA^m))$, then $F \in C_{\E}^k(\cA_{\beta}^n;\cA^m)$.
Consequently, the tracial noncommutative $C^k$ functions introduced and studied in \cite{JLS2022} provide examples of elements of $C_{\E}^k(\cA_{\sa}^n;\cA^m)$.
In the next section, we give examples of this kind that arise from the functional calculus (i.e., operator functions);
see Remark \ref{rem.JLS}.
\end{ex}

Next, we demonstrate why trace continuous/$C^k$ maps are relevant to us.

\begin{lem}\label{lem.trcontadap}
Let $n,m \in \N$ and $\cU \subseteq \cA_{\beta}^n$ be an open set.
\begin{enumerate}[label=(\roman*),font=\normalfont]
    \item If $F \in C_{\E}(\cU;\cA^m)$, $t \geq 0$, and $\a \in \cU \cap \cA_t^n$, then $F(\a) \in \cA_t^m$.\label{item.trcont1}
    \item Let $k \in \N$, $d \in \N^k$, and $\cA_{\gamma} \in \{\cA,\cA_{\sa}\}$.
    If $F \in C_{\E}(\cU;\mathbb{B}_k(\cA_{\gamma}^d;\cA^m))$, $t \geq 0$, and $\a \in \cU \cap \cA_t^n$, then $F(\a) \in \cT_{m,k,d,t}$.
    (Recall that we view $\mathbb{B}_k(\cA_{\sa}^d;\cA^m)$ as a subset of $\mathbb{B}_k(\cA^d;\cA^m)$.)
    In particular, if $X \colon \R_+ \to \cA^n$ is an adapted, $L^{\infty}$-LCLB (resp., continuous) process with values in $\cU$, then $F(X)$ is a $\vertiii{\cdot}_k$-LCLB (resp., continuous) multivariate trace $k$-process.\label{item.trcont2}
\end{enumerate}
\end{lem}

\begin{proof}
We leave the first item to the reader.
For the second, let $t \geq 0$ and $\a \in \cU \cap \cA_t^n$.
If $P \in (\TrP_{n,k,d}^*)^m$, then $P(\a) \in \cT_{m,k,d,t}^0 \subseteq \cT_{m,k,d,t}$ by definition.
Now, if $r > 0$ and $(P_j)_{j \in \N}$ are as in Definition \ref{def.tracecontCk}\ref{item.tracecont}, then $P_j(\a) \to F(\a)$ in $\mathbb{B}_k(\cA_{\gamma}^d;\cA^m) \subseteq \mathbb{B}_k(\cA^d;\cA^m)$ as $j \to \infty$.
Since $\cT_{m,k,d,t} \subseteq \mathbb{B}_k(\cA^d;\cA^m)$ is a $\vertiii{\cdot}_k$-closed set, we conclude that $F(\a) \in \cT_{k,m,d,t}$.
\end{proof}

\begin{thm}[Trace $C^k\Rightarrow$ adapted $C^k$]\label{thm.trCkadapCk}
If $n,m,k \in \N$ and $\cU \subseteq \cA_{\beta}^n$ is an open set, then
\[
C_{\E}^k(\cU;\cA^m) \subseteq C_a^k(\cU;\cA^m).
\]
\end{thm}

\begin{proof}
By definition, if $F \in C_{\E}^k(\cU;\cA^m)$ and $i =1,\ldots,k$, then $D^iF \in C_{\E}(\cU;\mathbb{B}_i((\cA_{\beta}^n)^i;\cA^m))$.
In particular, $D^iF \colon \cU \to \mathbb{B}_i((\cA_{\beta}^n)^i;\cA^m)$ is continuous with respect to $\vertiii{\cdot}_i$.
In addition, if $t \geq 0$ and $\a \in \cU \cap \cA_t^n$, then $D^iF(\a) \in \cT_{m,i,(n,\ldots,n),t} \subseteq \cF_{i,t}(\E^{\oplus n},\ldots,\E^{\oplus n};\E^{\oplus m})$ by Lemma \ref{lem.trcontadap}\ref{item.trcont2} and Proposition \ref{prop.TkinFk}.
Since $F(\a) \in \cA_t^m$ as well by Lemma \ref{lem.trcontadap}\ref{item.trcont1}, we conclude that $F \in C_a^k(\cU;\cA^m)$.
\end{proof}
\pagebreak

In particular, noncommutative It\^{o}'s formula (the multivariate version, Example \ref{ex.multivarIF}) applies to trace $C^2$ maps from $\cU$ to $\cA^m$.

\begin{ex}[Noncommutative It\^{o}'s formula for trace $C^2$ maps]
Let $n,m \in \N$ and $\cU \subseteq \cA_{\beta}^n$ be an open set.
Suppose $X = (X_1,\ldots,X_n) \colon \R_+ \to \cA^n$ is an $n$-tuple of $L^{\infty}$-decomposable processes such that $X(t) \in \cU$ for all $t \geq 0$.
If $F \in C_{\E}^2(\cU;\cA^m)$, then noncommutative It\^{o}'s formula (from Example \ref{ex.multivarIF}) says
\begin{align*}
    \d F(X(t)) & = DF(X(t))[\d X(t)] + \frac{1}{2}D^2F(X(t))[\d X(t), \d X(t)] \\
    & = \sum_{i=1}^n D_{x_i}F(X(t))[\d X_i(t)] + \frac{1}{2}\sum_{i,j=1}^n D_{x_j}D_{x_i}F(X(t))[\d X_i(t), \d X_j(t)]
\end{align*}
In particular, if $P \in (\TrP_n^*)^m$, then
\begin{align*}
    \d P(X(t)) & = (\partial P)(X(t))[\d X(t)] + \frac{1}{2}(\partial^2P)(X(t))[\d X(t), \d X(t)] \\
    & = \sum_{i=1}^n \big(\partial_{x_i}P\big)(X(t))[\d X_i(t)] + \frac{1}{2}\sum_{i,j=1}^n \big(\partial_{x_j}\partial_{x_i}P\big) (X(t))[\d X_i(t), \d X_j(t)]
\end{align*}
by Theorem \ref{thm.TrPolyder}.
Now, write $M_i \coloneqq X_i^{\mathrm{m}}$, and suppose, in addition, that
\begin{enumerate}[leftmargin=2\parindent]
\itemsep0em
    \item $(M_i^*,M_i^*)=(M_i,M_i)$ satisfies the hypotheses of Theorem \ref{thm.newmagicQC} and
    \item $i \neq j \Rightarrow \E[(M_i(t) - M_i(s))a(M_j(t) - M_j(s)) \mid \cA_s] = 0$ whenever $0 \leq s < t$ and $a \in \cA_s$.
\end{enumerate}
This is the case if, e.g., $M = (M_1,\ldots,M_n)$ is an $n$-dimensional semicircular Brownian motion or if $n=1$ and $M_1$ is a $q$-Brownian motion ($-1 \leq q < 1$).
By Lemma \ref{lem.trcontadap}\ref{item.trcont2}, Proposition \ref{prop.QCoffree}, and Theorem \ref{thm.newmagicQC} (plus Remark \ref{rem.M=N*}),
\begin{align*}
    \d F(X(t)) & = \sum_{i=1}^n D_{x_i}F(X(t))[\d X_i(t)] + \frac{1}{2}\sum_{i=1}^n \E\big[D_{x_i}^2F(X(t))[e_i(t),e_i(t)] \mid \cA_t\big] \, \kappa_{M_i}(\d t), \; \text{ and} \\
    \d P(X(t)) & = \sum_{i=1}^n \big(\partial_{x_i}P\big)(X(t))[\d X_i(t)] + \frac{1}{2}\sum_{i=1}^n \E\big[\big(\partial_{x_i}^2P\big)(X(t))[e_i(t),e_i(t)] \mid \cA_t\big] \, \kappa_{M_i}(\d t),
\end{align*}
where $e_i(t) = \norm{M_i(t+r_i) - M_i(t+r_i)}_2^{-1}(M_i(t+r_i) - M_i(t+r_i))$ as in Remark \ref{rem.M=N*}.
\end{ex}

\subsection{Examples: Scalar functions}\label{sec.NCk}

If $f \in C(\R)$, then it is easy to show that $f_{\mathsmaller{\cA}} \in C(\cA_{\sa} ; \cA)$.
(See the beginning of the proof of Theorem \ref{thm.NCkistrCk} below.)
However, it is not generally true that if $k \in \N$ and $f \in C^k(\R)$, then $f_{\mathsmaller{\cA}} \in C^k(\cA_{\sa};\cA)$.
In this section, we show that if $f \colon \R \to \C$ is ``slightly better than $C^k$,'' then operator function $f_{\mathsmaller{\cA}} \colon \cA_{\sa} \to \cA$ associated to $f$ is not only $C^k$ but adapted $C^k$ (actually, trace $C^k$).
The object needed to express $D^kf_{\mathsmaller{\cA}}$ in this case is called a multiple operator integral (MOI).
We begin by reviewing relevant facts about MOIs.
For much more information, see the survey book \cite{ST2019}.

For the duration of this section, fix $m \in \N$ and Polish spaces (i.e., complete separable metric spaces) $\Om_1,\ldots,\Om_m$. Also, write $\Om \coloneqq \Om_1 \times \cdots \times \Om_m$.
We first review the notion of the integral projective tensor product $\ell^{\infty}(\Om_1,\cB_{\Om_1}) \iotimes \cdots \iotimes \ell^{\infty}(\Om_m,\cB_{\Om_m})$, the idea for which is due to Peller \cite{Peller2006}.
Here, if $\Xi$ is a set and $\sG$ is a $\sigma$-algebra on $\Xi$, then $\ell^{\infty}(\Xi,\sG)$ is the space of bounded $\sG$/$\cB_{\C}$-measurable functions $\Xi \to \C$.\label{page.bddmeas}

\begin{defi}[IPTPs]\label{def.babyIPTP}
An $\boldsymbol{\ell^{\infty}}$\textbf{-integral projective decomposition} (IPD) of a function $\varphi \colon \Om \to \C$ is a choice $(\Sigma,\rho,\varphi_1,\ldots,\varphi_m)$ of a $\sigma$-finite measure space $(\Sigma,\sH,\rho)$ and, for each $j =1,\ldots,m$, a product-measurable function $\varphi_j \colon \Om_j \times \Sigma \to \C$ such that $\varphi_j(\cdot,\sigma) \in \ell^{\infty}(\Om_j,\cB_{\Om_j})$ for all $\sigma\in \Sigma$,
\begin{align*}
    & \int_{\Sigma} \|\varphi_1(\cdot,\sigma)\|_{\ell^{\infty}(\Om_1)}\cdots\|\varphi_m(\cdot,\sigma)\|_{\ell^{\infty}(\Om_m)} \, \rho(\d\sigma) < \infty, \; \text{ and} \numberthis\label{eq.intfincond} \\
    & \varphi(\boldsymbol{\om}) = \int_{\Sigma} \varphi_1(\om_1,\sigma) \cdots \varphi_m(\om_m,\sigma) \, \rho(\d\sigma) \; \text{ for all } \; \boldsymbol{\om} \in \Om,
\end{align*}
where $\boldsymbol{\om} = (\om_1,\ldots,\om_m)$.
Also, for any function $\varphi \colon \Om \to \C$, define
\[
\|\varphi\|_{\ell^{\infty}(\Om_1,\cB_{\Om_1}\hspace{-0.1mm}) \iotimes \cdots \iotimes \ell^{\infty}(\Om_m,\cB_{\Om_m}\hspace{-0.1mm})} \hspace{-0.2mm}\coloneqq\hspace{-0.2mm} \inf\hspace{-0.6mm}\Bigg\{\hspace{-0.6mm}\int_{\Sigma} \prod_{j=1}^m\|\varphi_j(\cdot,\sigma)\|_{\ell^{\infty}(\Om_j)}\,\rho(\d\sigma) : (\Sigma,\rho,\varphi_1,\ldots,\varphi_m)  \text{ is an } \ell^{\infty}\text{-IPD of } \varphi\hspace{-0.6mm}\Bigg\},
\]
where $\inf \emptyset \coloneqq \infty$. Finally, we define
\[
\ell^{\infty}(\Om_1,\cB_{\Om_1}) \iotimes \cdots \iotimes \ell^{\infty}(\Om_m,\cB_{\Om_m}) \coloneqq \big\{\varphi \in \ell^{\infty}(\Om,\cB_{\Om}) : \|\varphi\|_{\ell^{\infty}(\Om_1,\cB_{\Om_1}\hspace{-0.1mm}) \iotimes \cdots \iotimes \ell^{\infty}(\Om_m,\cB_{\Om_m}\hspace{-0.1mm})} < \infty\big\}
\]
to be the \textbf{integral projective tensor product of} $\boldsymbol{\ell^{\infty}(\Om_1,\cB_{\Om_1}),\ldots,\ell^{\infty}(\Om_m,\cB_{\Om_m})}$.
\end{defi}

It is not obvious that the integral in \eqref{eq.intfincond} makes sense.
In fact, the function being integrated is not necessarily measurable, but it \textit{is} ``almost measurable,'' i.e., measurable with respect to the $\rho$-completion of $\sH$;
see \cite[Lem.\ 2.2.1]{NikitopoulosNCk} for a proof.
Now, it is easy to see that if $\varphi \colon \Om \to \C$ is a function, then
\[
\|\varphi\|_{\ell^{\infty}(\Om)} \leq  \|\varphi\|_{\ell^{\infty}(\Om_1,\cB_{\Om_1}\hspace{-0.1mm}) \iotimes \cdots \iotimes \ell^{\infty}(\Om_m,\cB_{\Om_m}\hspace{-0.1mm})}.
\]
It is also the case that $\ell^{\infty}(\Om_1,\cB_{\Om_1}) \iotimes \cdots \iotimes \ell^{\infty}(\Om_m,\cB_{\Om_m}) \subseteq \ell^{\infty}(\Om,\cB_{\Om})$ is a unital $\ast$-subalgebra and that $(\ell^{\infty}(\Om_1,\cB_{\Om_1}) \iotimes \cdots \iotimes \ell^{\infty}(\Om_m,\cB_{\Om_m}),\|\cdot\|_{\ell^{\infty}(\Om_1,\cB_{\Om_1}\hspace{-0.1mm}) \iotimes \cdots \iotimes \ell^{\infty}(\Om_m,\cB_{\Om_m}\hspace{-0.1mm})})$ is a unital Banach $\ast$-algebra with respect to pointwise operations;
see \cite[Prop.\ 2.2.3]{NikitopoulosNCk} for proofs of these facts.

Next, we review a special case of the ``separation of variables'' approach to defining multiple operator integrals, developed to various degrees in \cite{Peller2006,ACDS2009,Peller2016,NikitopoulosMOI}.
For the remainder of this section, fix a complex Hilbert space $H$, a von Neumann algebra $\cM \subseteq B_{\C}(H)$, and $k \in \N$.
If $(\Sigma,\sH,\rho)$ is a measure space and $F \colon \Sigma \to \cM$ is a map, we say that $F$ is \textbf{pointwise Pettis integrable} if for every $h_1,h_2 \in H$, $\la F(\cdot)h_1,h_2 \ra \colon \Sigma \to \C$ is $(\sH,\cB_{\C})$-measurable and $\int_{\Sigma} |\la F(\sigma)h_1,h_2 \ra|\,\rho(\d\sigma) < \infty$.
In this case, \cite[Lem.\ 4.2.1]{NikitopoulosNCk} says that there exists a unique $T \in B_{\C}(H)$ such that $\la Th_1,h_2 \ra = \int_{\Sigma} \la F(\sigma)h_1,h_2 \ra\,\rho(\d\sigma)$ for all $h_1,h_2 \in H$;
moreover, $T \in \mathrm{W}^*(F(\sigma) : \sigma \in \Sigma) \subseteq \cM$.
We shall write $\int_{\Sigma} F \,\d\rho = \int_{\Sigma}F(\sigma)\,\rho(\d\sigma) \coloneqq T$ for this operator.

\begin{thm}[Definition of MOIs]\label{thm.babyMOI}
Let $\a = (a_1,\ldots,a_{k+1}) \in \cM_{\sa}^{k+1}$,
\[
\varphi \in \ell^{\infty}\big(\sigma(a_1),\cB_{\sigma(a_1)}\big) \iotimes \cdots \iotimes \ell^{\infty}\big(\sigma(a_{k+1}),\cB_{\sigma(a_{k+1})}\big),
\]
and $(b_1,\ldots,b_k) \in \cM^k$.
\begin{enumerate}[label=(\roman*),font=\normalfont]
    \item If $(\Sigma,\rho,\varphi_1,\ldots,\varphi_{k+1})$ is an $\ell^{\infty}$-IPD of $\varphi$, then the map
    \[
    \Sigma \ni \sigma \mapsto F(\sigma) \coloneqq \varphi_1(a_1,\sigma)\,b_1\cdots \varphi_k(a_k,\sigma)\,b_k\,\varphi_{k+1}(a_{k+1},\sigma) \in \cM
    \]
    is pointwise Pettis integrable, and the pointwise Pettis integral
    \begin{align*}
        \big(I^{\a}\varphi\big)[b_1,\ldots,b_k] & = \int_{\sigma(a_{k+1})}\cdots\int_{\sigma(a_1)} \varphi(\blambda)\,P^{a_1}(\d\lambda_1)\,b_1\cdots P^{a_k}(\d\lambda_k)\,b_k\,P^{a_{k+1}}(\d\lambda_{k+1}) \\
        & \coloneqq \int_{\Sigma} F\,\d\rho \in \cM
    \end{align*}
    is independent of the chosen $\ell^{\infty}$-IPD of $\varphi$.
    In the notation above, $P^a$ represents the projection-valued spectral measure of the operator $a \in \cM_{\sa}$.\label{item.welldef}
    \item The map $\cM^k \ni (b_1,\ldots,b_k) \mapsto (I^{\a}\varphi)[b_1,\ldots,b_k] \in \cM$ is complex $k$-linear and bounded.
    Also, the map
    \[
    \ell^{\infty}\big(\sigma(a_1),\cB_{\sigma(a_1)}\big) \iotimes \cdots \iotimes \ell^{\infty}\big(\sigma(a_{k+1}),\cB_{\sigma(a_{k+1})}\big) \ni \varphi \mapsto I^{\a}\varphi \in B_k(\cM^k;\cM)
    \]
    is complex linear and has operator norm at most one.
    The object $I^{\a}\varphi$ is the \textbf{multiple operator integral} (MOI) of $\varphi$ with respect to $P^{a_1},\ldots,P^{a_{k+1}}$.\label{item.bddlin}
    \item If $(\cM,\E_{\mathsmaller{\cM}})$ is a $\mathrm{W}^*$-probability space, then $I^{\a} \varphi \in \mathbb{B}_k(\cM)$, and\label{item.vertiiibd}
    \[
    \vertiii{I^{\a}\varphi}_k \leq \norm{\varphi}_{\ell^{\infty}(\sigma(a_1),\cB_{\sigma(a_1)}\hspace{-0.1mm}) \iotimes \cdots \iotimes \ell^{\infty}(\sigma(a_{k+1}),\cB_{\sigma(a_{k+1})}\hspace{-0.1mm}))}.
    \]
\end{enumerate}
\end{thm}

\begin{proof}
The first two items are \cite[Thm.\ 4.2.4]{NikitopoulosNCk}.
The final item is a special case of \cite[Prop.\ 4.3.3]{NikitopoulosMOI}.
(See also \cite[Ex.\ 4.1.5]{NikitopoulosMOI}.)
\end{proof}
\pagebreak

\begin{ex}[Algebraic tensor functions]\label{ex.elemtensfunc}
Let $n \in \N$.
For each $j =1,\ldots,m$ and $\ell =1,\ldots,n$, fix a bounded Borel measurable function $\psi_{j,\ell} \colon \Om_j \to \C$.
If
\[
\psi(\boldsymbol{\om}) \coloneqq \sum_{\ell=1}^n\psi_{1,\ell}(\om_1)\cdots\psi_{m,\ell}(\om_m) \qquad (\boldsymbol{\om} \in \Om),
\]
then it is easy to see that $\psi \in \ell^{\infty}(\Om_1,\cB_{\Om_1}) \iotimes \cdots \iotimes \ell^{\infty}(\Om_m,\cB_{\Om_m})$ with
\[
\norm{\psi}_{\ell^{\infty}(\Om_1,\cB_{\Om_1}\hspace{-0.1mm}) \iotimes \cdots \iotimes \ell^{\infty}(\Om_m,\cB_{\Om_m}\hspace{-0.1mm})} \leq \sum_{\ell=1}^n\norm{\psi_{1,\ell}}_{\ell^{\infty}(\Om_1)}\cdots\norm{\psi_{m,\ell}}_{\ell^{\infty}(\Om_m)}.
\]
If $m = k+1$, $\a = (a_1,\ldots,a_{k+1}) \in \cM_{\sa}^{k+1}$, and $\Om_j = \sigma(a_j)$ ($j =1,\ldots,k+1$) as well, then
\[
\big(I^{\a}\psi\big)[b_1,\ldots,b_k] = \sum_{\ell=1}^n \psi_{1,\ell}(a_1)\,b_1 \cdots \psi_{k,\ell}(a_k)\,b_k \,\psi_{k+1,\ell}(a_{k+1}) \qquad \big((b_1,\ldots,b_k) \in \cM^k\big).
\]
This applies when, e.g., $\psi(\lambda_1,\ldots,\lambda_{k+1}) = \sum_{|\delta| \leq d} c_{\delta} \lambda_1^{\delta_1}\cdots\lambda_{k+1}^{\delta_{k+1}} \in \C[\lambda_1,\ldots,\lambda_{k+1}]$.
\end{ex}

Next, we make precise the notion of ``slightly better than $C^k$'' mentioned at the beginning of the section and introduced in \cite{NikitopoulosNCk}.
To begin, we define divided differences, a scalar counterpart to the noncommutative derivatives from Notation \ref{nota.ncder}.

\begin{defi}[Divided differences]\label{def.divdiff}
Let $S \subseteq \C$ and $f \colon S \to \C$ be a function.
Define $f^{[0]} \coloneqq f$ and, for $k \in \N$ and distinct $\lambda_1,\ldots,\lambda_{k+1} \in S$, recursively define
\[
f^{[k]}(\lambda_1,\ldots,\lambda_{k+1}) \coloneqq \frac{f^{[k-1]}(\lambda_1,\ldots,\lambda_k) - f^{[k-1]}(\lambda_1,\ldots,\lambda_{k-1},\lambda_{k+1})}{\lambda_k-\lambda_{k+1}}.
\]
We call $f^{[k]}$ the $\boldsymbol{k^{\textbf{th}}}$ \textbf{divided difference} of $f$.
\end{defi}

By an elementary induction argument,
\[
f^{[k]}(\lambda_1,\ldots,\lambda_{k+1}) = \sum_{i=1}^{k+1} f(\lambda_i) \prod_{j \neq i}(\lambda_i-\lambda_j)^{-1}
\]
for all distinct $\lambda_1,\ldots,\lambda_{k+1} \in S$.
In particular, $f^{[k]}$ is symmetric in its arguments.
Now, we state a useful expression for $f^{[k]}$ when $f \in C^k(\R)$ or when $f \colon \C \to \C$ is entire;
see \cite[Prop.\ 2.1.3(ii)]{NikitopoulosNCk} for a proof.

\begin{prop}\label{prop.divdiff}
Fix $S \subseteq \C$, $f \colon S \to \C$, and $k \in \N$.
In addition, write
\[
\Sigma_k \coloneqq \big\{(s_1,\ldots,s_k) \in \R_+^k : s_1+\cdots+s_k \leq 1\big\}.
\]
If $S = \R$ and $f \in C^k(\R)$ or if $S = \C$ and $f \colon \C \to \C$ is entire, then
\[
f^{[k]}(\lambda_1,\ldots,\lambda_{k+1}) = \int_{\Sigma_k}f^{(k)}\Bigg(\sum_{j=1}^ks_j\lambda_j+\Bigg(1-\sum_{j=1}^k s_j\Bigg)\lambda_{k+1}\Bigg) \,  \d s_1 \cdots \d s_k
\]
for all distinct $\lambda_1,\ldots,\lambda_{k+1}$ belonging to $\R$ or $\C$, respectively.
In particular, if $f \in C^k(\R)$, then $f^{[k]}$ extends uniquely to a (symmetric) continuous function $\R^{k+1} \to \C$;
and if $f \colon \C \to \C$ is entire, then $f^{[k]}$ extends uniquely to a (symmetric) continuous function $\C^{k+1} \to \C$.
We use the same notation for these extensions.
\end{prop}

\begin{ex}[Divided differences of polynomials]\label{ex.divdiffpoly}
Let $p(\lambda) = \sum_{i=0}^n c_i \lambda^i \in \C[\lambda]$, viewed as an entire function $\C \to \C$.
If $\blambda \coloneqq (\lambda_1,\ldots,\lambda_{k+1}) \in \C^{k+1}$ has distinct entries, then
\[
p^{[k]}(\blambda) = \sum_{i=0}^nc_i\sum_{|\delta| = i-k} \blambda^{\delta} = \sum_{i=0}^nc_i\sum_{\delta \in \N_0^{k+1} : |\delta| = i-k}\lambda_1^{\delta_1}\cdots \lambda_{k+1}^{\delta_{k+1}}. \numberthis\label{eq.divdiffpoly}
\]
As is the case with many properties of divided differences, the identity above may be proven by induction on $k$;
see \cite[Ex.\ 2.1.5]{NikitopoulosNCk}.
By continuity, i.e., Proposition \ref{prop.divdiff}, \eqref{eq.divdiffpoly} holds for \textit{all} $\blambda \in \C^{k+1}$.
In particular, $p^{[k]} \in \C[\lambda_1,\ldots,\lambda_{k+1}]$.
\end{ex}

For the next example, recall that $W_k(\R)$ is the $k^{\text{th}}$ Wiener space (Definition \ref{def.Wk}).
\pagebreak

\begin{ex}[Divided differences of $W_k$ functions]\label{ex.divdiffWk}
If $f = \int_{\R} e^{i \xi \boldsymbol{\cdot}} \,\mu(\d\xi) \in W_k(\R)$, then $f \in C^k(\R)$, and $f^{(k)}(\lambda) = \int_{\R} (i\xi)^ke^{i\xi\lambda} \, \mu(\d\xi)$ for all $\lambda \in \R$.
In particular, by Proposition \ref{prop.divdiff},
\[
f^{[k]}(\blambda) = \int_{\Sigma_k}\int_{\R} (i\xi)^k e^{is_1\xi\lambda_1}\cdots e^{is_k\xi\lambda_k}e^{i(1-\sum_{j=1}^ks_j)\xi\lambda_{k+1}}\,\mu(\d\xi) \,\d s_1\cdots \d s_k \numberthis\label{eq.divdiffWk}
\]
for all $\blambda = (\lambda_1,\ldots,\lambda_{k+1}) \in \R^{k+1}$.
\end{ex}

We now finally turn to the definition of the space of functions $\R \to \C$ of interest:
the space $NC^k(\R)$ of noncommutative $C^k$ functions.

\begin{nota}\label{nota.superCk}
Let $r > 0$.
For a function $\varphi \colon \R^{k+1} \to \C$, define
\[
\|\varphi\|_{r,k+1} \coloneqq \big\|\varphi|_{[-r,r]^{k+1}}\big\|_{\ell^{\infty}([-r,r],\cB_{[-r,r]}\hspace{-0.1mm})^{\iotimes(k+1)}} \in [0,\infty].
\]
Now, if $f \in C^k(\R)$, then we define
\[
\|f\|_{\cC^{[k]},r} \coloneqq \sum_{j=0}^k\big\|f^{[j]}\big\|_{r,j+1} \in [0,\infty] \; \text{ and } \; \cC^{[k]}(\R) \coloneqq \big\{ g \in C^k(\R) : \|g\|_{\cC^{[k]},s}  < \infty \text{ for all } s > 0\big\},
\]
where $\|\cdot\|_{r,1} \coloneqq \|\cdot\|_{\ell^{\infty}([-r,r])}$.
\end{nota}

Note that $\cC^{[k]}(\R) \subseteq C^k(\R)$ is a complex-linear subspace and $\{\|\cdot\|_{\cC^{[k]},r} : r > 0\}$ is a collection of seminorms on $\cC^{[k]}(\R)$.
This collection of seminorms makes $\cC^{[k]}(\R)$ into a complex Fr\'{e}chet space---actually, a Fr\'{e}chet $\ast$-algebra.
This is proven as \cite[Prop.\ 3.1.3(iv)]{NikitopoulosNCk}.

\begin{ex}[Polynomials]\label{ex.polysuperCk}
Fix $p \in \C[\lambda]$, viewed as a smooth function $\R \to \C$.
By Example \ref{ex.divdiffpoly}, $p^{[k]} \in \C[\lambda_1,\ldots,\lambda_{k+1}]$ for all $k \in \N$.
Thus, $p \in \bigcap_{k \in \N}\cC^{[k]}(\R)$ by Example \ref{ex.elemtensfunc}.
\end{ex}

\begin{defi}[Noncommutative $C^k$ functions]\label{def.NCk}
If $k \in \N$, then we define $NC^k(\R) \coloneqq \overline{\C[\lambda]} \subseteq \cC^{[k]}(\R)$ to be the space of \textbf{noncommutative $\boldsymbol{C^k}$ functions}.
To be clear, the closure in the previous sentence takes place in the complex Fr\'{e}chet space $\cC^{[k]}(\R)$.
\end{defi}

Since $\C[\lambda] \subseteq \cC^{[k]}(\R)$ is a $\ast$-subalgebra, $NC^k(\R)$ is a Fr\'{e}chet $\ast$-algebra in its own right.
Before giving many examples of noncommutative $C^k$ functions, we demonstrate why $NC^k(\R)$ is of current interest to us.

\begin{lem}\label{lem.CstarMOI}
Suppose $(\cM,\E_{\mathsmaller{\cM}})$ is a $\mathrm{W}^*$-probability space such that $\cM$ contains $\cA$ as a unital $\mathrm{C}^*$-subalgebra and $\E_{\mathsmaller{\cM}}|_{\cA} = \E$.
(Such an $(\cM,\E_{\mathsmaller{\cM}})$ always exists;
see Appendix \ref{sec.CstarLp}.)
\begin{enumerate}[label=(\roman*),font=\normalfont]
    \item If $p \in \C[\lambda]$, then\label{item.polyMOI}
    \[
    \partial_{\mathsmaller{\otimes}}^kp(\a)\sh_k[b_1,\ldots,b_k] = k! \big(I^{\a}p^{[k]}\big)[b_1,\ldots,b_k] \qquad \big(\a \in \cM_{\sa}^{k+1}, \; b_1,\ldots,b_k \in \cM\big).
    \]
    \item If $f \in NC^k(\R)$, $\a = (a_1,\ldots,a_{k+1}) \in \cA_{\sa}^{k+1}$, and $b = (b_1,\ldots,b_k) \in \cA^k$, then
    \[
    \big(I^{\a}f^{[k]}\big)[b] \in \mathrm{C}^*(1,a_1,\ldots,a_{k+1},b_1,\ldots,b_k) \subseteq \cA \subseteq \cM.
    \]
    Moreover, the restricted map $I^{\a}f^{[k]} \colon \cA^k \to \cA$ belongs to $\mathbb{B}_k(\cA)$, and
    \[
    \vertiii{I^{\a}f^{[k]}}_k \leq \big\|f^{[k]}\big\|_{\ell^{\infty}(\sigma(a_1),\cB_{\sigma(a_1)}\hspace{-0.1mm}) \iotimes \cdots \iotimes \ell^{\infty}(\sigma(a_{k+1}),\cB_{\sigma(a_{k+1})}\hspace{-0.1mm}))} \leq \big\|f^{[k]}\big\|_{\norm{\a}_{\infty},k+1}.
    \]
    Finally, the map $\cA_{\sa}^{k+1} \ni \a \mapsto I^{\a}f^{[k]} \in \mathbb{B}_k(\cA)$ belongs to $C_{\E}(\cA_{\sa}^{k+1};\mathbb{B}_k(\cA))$.\label{item.CstarMOI} 
\end{enumerate}
Owing to the second item, we shall use the same MOI notation as in the $\mathrm{W}^*$ case for $I^{\a}f^{[k]} \colon \cA^k \to \cA$ when $f \in NC^k(\R)$ and $\a \in \cA_{\sa}^{k+1}$.
\end{lem}

\begin{proof}
For the first item, combine Examples \ref{ex.elemtensfunc} and \ref{ex.divdiffpoly}.
For the second, let $(p_n)_{n \in \N}$ be a sequence in $\C[\lambda]$ converging to $f$ in $NC^k(\R)$.
By the first item, it is clear that
\[
\big(I^{\a}p_n^{[k]}\big)[b_1,\ldots,b_k] \in \mathrm{C}^*(1,a_1,\ldots,a_{k+1},b_1,\ldots,b_k) \qquad (n \in \N).
\]
By Theorem \ref{thm.babyMOI}\ref{item.bddlin},
\[
\big(I^{\a}p_n^{[k]}\big)[b_1,\ldots,b_k] \xrightarrow{n \to \infty} \big(I^{\a}f^{[k]}\big)[b_1,\ldots,b_k]
\]
in $\cM$ (i.e., in operator norm).
Since $\mathrm{C}^*(1,a_1,\ldots,a_{k+1},b_1,\ldots,b_k) \subseteq \cA \subseteq \cM$ is closed, we conclude that $\big(I^{\a}f^{[k]}\big)[b_1,\ldots,b_k] \in \mathrm{C}^*(1,a_1,\ldots,a_{k+1},b_1,\ldots,b_k)$.
\pagebreak

Next, since $\E_{\mathsmaller{\cM}}|_{\cA} = \E$, we have that $\norm{a}_{L^p(\E_{\mathsmaller{\cM}})} = \norm{a}_{L^p(\E)}$ for all $a \in \cA$.
Thus, by what we proved in the previous paragraph and Theorem \ref{thm.babyMOI}\ref{item.vertiiibd}, if $1/p_1+\cdots+1/p_k=1/p$, then
\begin{align*}
    \big\|\big(I^{\a}f^{[k]}\big)[b_1,\ldots,b_k]\big\|_{L^p(\E)} & = \big\|\big(I^{\a}f^{[k]}\big)[b_1,\ldots,b_k]\big\|_{L^p(\E_{\mathsmaller{\cM}})} \\
    & \leq \big\|f^{[k]}\big\|_{\ell^{\infty}(\sigma(a_1),\cB_{\sigma(a_1)}\hspace{-0.1mm}) \iotimes \cdots \iotimes \ell^{\infty}(\sigma(a_{k+1}),\cB_{\sigma(a_{k+1})}\hspace{-0.1mm}))} \norm{b_1}_{L^{p_1}(\E_{\mathsmaller{\cM}})}\cdots \norm{b_k}_{L^{p_k}(\E_{\mathsmaller{\cM}})} \\
    & = \big\|f^{[k]}\big\|_{\ell^{\infty}(\sigma(a_1),\cB_{\sigma(a_1)}\hspace{-0.1mm}) \iotimes \cdots \iotimes \ell^{\infty}(\sigma(a_{k+1}),\cB_{\sigma(a_{k+1})}\hspace{-0.1mm}))} \norm{b_1}_{L^{p_1}(\E)}\cdots \norm{b_k}_{L^{p_k}(\E)}.
\end{align*}
This gives the claimed $\vertiii{\cdot}_k$-norm bound.

Finally, write $\big(f^{[k]}\big)_{\mathsmaller{\cA}} \colon \cA_{\sa}^{k+1} \to \mathbb{B}_k(\cA)$ for the map $\a \mapsto I^{\a}f^{[k]}$.
It is clear from the first item that
\[
\big(p_n^{[k]}\big)_{\mathsmaller{\cA}} \in \{P_{\mathsmaller{(\cA,\E)}} : P \in \TrP_{k+1,k,(1,\ldots,k)}^*\} \subseteq C_{\E}(\cA_{\sa}^{k+1};\mathbb{B}_k(\cA)) \qquad (n \in \N).
\]
By the bound proven in the previous paragraph, $\big(p_n^{[k]}\big)_{\mathsmaller{\cA}} \to \big(f^{[k]}\big)_{\mathsmaller{\cA}}$ uniformly on bounded sets, i.e., in the topology of $BC_{\loc}(\cA_{\sa}^{k+1};\mathbb{B}_k(\cA))$, as $n \to \infty$.
Since $C_{\E}(\cA_{\sa}^{k+1};\mathbb{B}_k(\cA))$ is closed under uniform convergence on bounded subsets, we conclude that $\big(f^{[k]}\big)_{\mathsmaller{\cA}} \in C_{\E}(\cA_{\sa}^{k+1};\mathbb{B}_k(\cA))$, as desired.
\end{proof}

\begin{thm}[$NC^k \Rightarrow$ trace $C^k$]\label{thm.NCkistrCk}
If $k \in \N$ and $f \in NC^k(\R)$, then $f_{\mathsmaller{\cA}} \in C_{\E}^k(\cA_{\sa};\cA)$, and
\begin{equation}
    D^kf_{\mathsmaller{\cA}}(a)[b_1,\ldots,b_k] = \sum_{\pi \in S_k} \underbrace{\int_{\sigma(a)}\cdots\int_{\sigma(a)}}_{k+1 \, \mathrm{times}} f^{[k]}(\blambda) \, P^a(\d\lambda_1)\,b_{\pi(1)}\cdots P^a(\d\lambda_k)\,b_{\pi(k)}\,P^a(\d\lambda_{k+1}) \label{eq.opfuncderform}
\end{equation}
for all $a,b_1,\ldots,b_k \in \cA_{\sa}$.
\end{thm}

\begin{proof}
First, let $f \in C(\R)$.
By the classical Weierstrass approximation theorem, there exists a sequence $(p_n)_{n \in \N}$ in $\C[\lambda]$ converging uniformly on compact sets to $f$.
For $r > 0$, write $C_r \coloneqq \{a \in \cA_{\sa} : \norm{a} \leq r\}$.
By basic properties of the functional calculus, if $r > 0$, then
\[
\sup_{a \in C_r} \norm{f(a) - p_n(a)} = \sup_{a \in C_r}\norm{f - p_n}_{\ell^{\infty}(\sigma(a))} = \norm{f - p_n}_{\ell^{\infty}([-r,r])} \xrightarrow{n \to \infty} 0.
\]
Thus, $(p_n)_{\mathsmaller{\cA}} \to f_{\mathsmaller{\cA}}$ in $BC_{\loc}(\cA_{\sa};\cA)$ as $n \to \infty$.
Since $p_{\mathsmaller{\cA}} \in C_{\E}^{\infty}(\cA_{\sa};\cA) \subseteq C_{\E}(\cA_{\sa};\cA)$ for all $p \in \C[\lambda]$, we conclude that $f_{\mathsmaller{\cA}} \in C_{\E}(\cA_{\sa};\cA)$.

Next, the fact that $f \in NC^k(\R)$ implies that $f_{\mathsmaller{\cA}} \in C^k(\cA_{\sa};\cA)$ and \eqref{eq.opfuncderform} holds is \cite[Thm.\ 1.2.3]{NikitopoulosNCk}.
(However, the proof of Lemma \ref{lem.CstarMOI} is not far from showing this, as we encourage the reader to ponder.)
To complete the proof, we argue that \eqref{eq.opfuncderform} implies $f_{\mathsmaller{\cA}} \in C_{\E}^k(\cA_{\sa};\cA)$.
Indeed, let $m,n \in \N$, $d \in \N^k$, $\cA_{\gamma} \in \{\cA,\cA_{\sa}\}$, and $F \in C_{\E}(\cA_{\beta}^n;\mathbb{B}_k(\cA_{\gamma}^d;\cA^m))$.
We make two easy observations.
First, the map
\[
\cA_{\beta} \ni a \mapsto F(a,\ldots,a) \in \mathbb{B}_k(\cA_{\gamma}^d;\cA^m)
\]
belongs to $C_{\E}(\cA_{\beta};\mathbb{B}_k(\cA_{\gamma}^d;\cA^m))$.
Next, if $\cV$ and $\cW$ are vector spaces and $T \colon \cV^k \to \cW$ is a $k$-linear map, then we write
\[
\operatorname{Sym} (T)[v_1,\ldots,v_k] \coloneqq \sum_{\pi \in S_k} T[v_{\pi(1)},\ldots,v_{\pi(k)}] \qquad (v_1,\ldots,v_k \in \cV).
\]
The second observation is that if $d_1 = \cdots = d_k$, then $\operatorname{Sym}(F) \in C_{\E}(\cA_{\beta}^n;\mathbb{B}_k(\cA_{\gamma}^d;\cA^m))$.
Combining these two observations, we conclude from Lemma \ref{lem.CstarMOI}\ref{item.CstarMOI} that if $f \in NC^k(\R)$, then the map
\[
\cA_{\sa} \ni a \mapsto \operatorname{Sym}\big(I^{a,\ldots,a}f^{[k]}\big|_{\cA_{\sa}^k}\big) \in \mathbb{B}_k\big(\cA_{\sa}^k;\cA\big)
\]
belongs to $C_{\E}(\cA_{\sa};\mathbb{B}_k(\cA_{\sa}^k;\cA))$.
Since \eqref{eq.opfuncderform} may be rewritten as
\[
D^kf_{\mathsmaller{\cA}}(a) = \operatorname{Sym}\big(I^{a,\ldots,a}f^{[k]}|_{\cA_{\sa}^k}\big) \qquad (a \in \cA_{\sa}),
\]
we see that $f_{\mathsmaller{\cA}} \in C_{\E}^k(\cA_{\sa};\cA)$.
This completes the proof.
\end{proof}

\begin{rem}\label{rem.JLS}
A careful study of the proof yields that if $f \in NC^k(\R)$, then $f_{\mathsmaller{\cA}} \colon \cA_{\sa} \to \cA$ is an example of the functions described in Example \ref{ex.JLS} (with $n=m=1$ and $\cA_{\beta} = \cA_{\sa}$).
\end{rem}

Combining this with work from the last section, we arrive at a general It\^{o}'s formula for noncommutative $C^2$ functions of self-adjoint $L^{\infty}$-decomposable processes.
This generalizes the functional free It\^{o} formula for free It\^{o} processes (\cite[Thm.\ 4.3.4]{NikitopoulosIto}).
\pagebreak

\begin{cor}[It\^{o}'s formula for $NC^2$ functions]
If $X \colon \R_+ \to \cA$ is a self-adjoint $L^{\infty}$-decomposable process and $f \in NC^2(\R)$, then
\begin{align*}
    \d f(X(t)) & = \int_{\sigma(X(t))}\int_{\sigma(X(t))} f^{[1]}(\lambda,\mu) \, P^{X(t)}(\d\lambda) \,\d X(t) \,P^{X(t)}(\d\mu) \\
    & \hspace{7.5mm} + \int_{\sigma(X(t))}\int_{\sigma(X(t))}\int_{\sigma(X(t))} f^{[2]}(\lambda,\mu,\nu) \, P^{X(t)}(\d\lambda) \,\d X(t) \,P^{X(t)}(\d\mu) \,\d X(t) \,P^{X(t)}(\d\nu).
\end{align*}
\end{cor}

\begin{proof}
Combine Theorems \ref{thm.NCIFtimedep}, \ref{thm.trCkadapCk}, and \ref{thm.NCkistrCk}.
\end{proof}

At this point, it is reasonable to wonder whether all this work was worthwhile.
Specifically, one may wonder whether there are nontrivial examples of noncommutative $C^k$ functions.
As promised, here is a result from \cite{NikitopoulosNCk} showing that a function $f \colon \R \to \C$ only has to be ``slightly better than $C^k$'' to be $NC^k$.

\begin{thm}[Nikitopoulos \cite{NikitopoulosNCk}]\label{thm.NCk}
Let $k \in \N$.
Write $\dot{B}_1^{k,\infty}(\R)$ for the homogeneous $(k,\infty,1)$-Besov space (\cite[Def.\ 3.3.1]{NikitopoulosNCk}), $C_{\loc}^{k,\e}(\R)$ for the space of $C^k$ functions whose $k^{\text{th}}$ derivatives are locally $\e$-H\"{o}lder continuous (\cite[Def.\ 3.3.8]{NikitopoulosNCk}), and $W_k(\R)_{\loc}$ for set of functions $f \colon \R \to \C$ such that for all $r > 0$, there exists a $g \in W_k(\R)$ such that $f|_{[-r,r]} = g|_{[-r,r]}$.
\begin{enumerate}[label=(\roman*),font=\normalfont]
    \item $C^{k+1}(\R) \subseteq W_k(\R)_{\loc} \subseteq NC^k(\R)$, and $W_k(\R)$ is dense in $NC^k(\R)$.\label{item.WkinNCk}
    \item $\dot{B}_1^{k,\infty}(\R) \subseteq NC^k(\R)$, and $C_{\loc}^{k,\e}(\R) \subseteq NC^k(\R)$ for all $\e > 0$.\label{item.BesovHolderinNCk}
\end{enumerate}
\end{thm}

\begin{proof}
See \cite[\S3.2]{NikitopoulosNCk} for \ref{item.WkinNCk} and \cite[\S3.3]{NikitopoulosNCk} for \ref{item.BesovHolderinNCk}.
Alternatively, see the end of \cite[\S4.1]{NikitopoulosIto} for a brief summary of all the relevant arguments.
\end{proof}

\begin{rem}\label{rem.Wiener}
By combining Theorems \ref{thm.trCkadapCk}, \ref{thm.NCkistrCk}, and \ref{thm.NCk}\ref{item.WkinNCk} with Example \ref{ex.divdiffWk}, we finally get a full proof---using rather heavy machinery---of the claims made in Example \ref{ex.Wiener}.
Since the direct proof suggested in Example \ref{ex.Wiener} is less complicated, it is reasonable to wonder whether we have gained anything by working with $NC^k(\R)$ instead of $W_k(\R)$ or $W_k(\R)_{\loc}$.
In fact, one gains two things.
First, one gains more functions.
Specifically, \cite[Thm.\ 3.4.1]{NikitopoulosNCk} demonstrates that the containment $W_k(\R)_{\loc} \subseteq NC^k(\R)$ is strict.
Second, one gains computational flexibility, even when $f \in W_k(\R)_{\loc}$.
Specifically, instead of being restricted to working with decompositions as in \eqref{eq.divdiffWk} when computing $D^kf_{\mathsmaller{\cA}}$, one can work with \textit{any} integral projective decomposition of $f^{[k]}$.
\end{rem}

\appendix
\section{\texorpdfstring{$L^p$}{} spaces of \texorpdfstring{$\mathrm{C}^*$}{}-probability spaces}\label{sec.CstarLp}

In this appendix, we show how basic facts about $L^p$ spaces of $\mathrm{W}^*$-probability spaces imply those of $L^p$ spaces of $\mathrm{C}^*$-probability spaces.
We take the $\mathrm{W}^*$ theory for granted;
see \cite{Dixmier1953,daSilva2018} for relevant results.
Let $(\cA,\E)$ be a $\mathrm{C}^*$-probability space, and write $\pi \colon \cA \to B_{\C}(H)$ for the (faithful) GNS representation corresponding to $\E$.
Recall that this means $(H,\ip{\cdot,\cdot})$ is the completion of the complex inner product space $(\cA,\ip{\cdot,\cdot}_{\E})$, where $\ip{a,b}_{\E} \coloneqq \E[b^*a]$, and $\pi(a) \colon H \to H$ is the bounded complex-linear map determined by $\pi(a)b = ab$ ($a,b \in \cA$). 

\begin{prop}\label{prop.reducetovNa}
If $\cM$ is the $\sigma$-WOT closure (equivalently, WOT closure) of $\cA$ in $B_{\C}(H)$ and
\[
\E_{\mathsmaller{\cM}}[A] \coloneqq \ip{A1,1} \qquad (A \in \cM),
\]
then $(\cM,\E_{\mathsmaller{\cM}})$ is a $\mathrm{W}^*$-probability space, and $\E_{\mathsmaller{\cM}} \circ \pi = \E$.
\end{prop}

\begin{proof}
The only nontrivial assertions are that $\E_{\mathsmaller{\cM}}$ is faithful and tracial on $\cM$.
To see that $\E_{\mathsmaller{\cM}}$ is tracial, note that if $a,b \in \cA$, then
\[
\E_{\mathsmaller{\cM}}[\pi(a)\pi(b)] = \E_{\mathsmaller{\cM}}[\pi(ab)] = \E[ab] = \E[ba] = \E_{\mathsmaller{\cM}}[\pi(ba)] = \E_{\mathsmaller{\cM}}[\pi(b)\pi(a)]
\]
by the traciality of $\E$.
In other words, $\E_{\mathsmaller{\cM}}[AB] = \E_{\mathsmaller{\cM}}[BA]$ for all $A,B \in \pi(\cA)$.
Since multiplication $\cM \times \cM \to \cM$ is argumentwise $\sigma$-WOT continuous and $\E_{\mathsmaller{\cM}}$ is normal, we conclude that $\E_{\mathsmaller{\cM}}$ is tracial from the $\sigma$-WOT density of $\pi(\cA)$ in $\cM$.

To see that $\E_{\mathsmaller{\cM}}$ is faithful, it suffices to prove that if $A \in \cM$ and $A1 = 0$, then $A = 0$, i.e., that $1 \in \cA \subseteq H$ is separating for $\cM$.
To this end, suppose $A \in \cM$ and $A1=0$.
Now, let $(a_j)_{j \in J}$ be a net in $\cA$ such that $\pi(a_j) \to A$ in the $\sigma$-WOT.
If $b,c \in \cA$, then $\ip{\pi(a_j)b, c} = \E[c^*a_jb] = \E[(cb^*)^*a_j] = \ip{\pi(a_j)1,cb^*}$ again by the traciality of $\E$.
But then $\ip{Ab,c} = \lim_{j \in J}\ip{\pi(a_j)b,c} = \lim_{j \in J}\ip{\pi(a_j)1,cb^*} = \lim_{j \in J} \ip{A1,cb^*} = 0$.
Since $\cA$ is dense in $H$, we conclude that $\ip{Ah,k} = 0$ for all $h,k \in H$, from which it follows that $A = 0$.
\end{proof}

\begin{cor}\label{cor.CstarLpfromWstarLp}
Let $p \in [1,\infty)$, and write $\norm{a}_p \coloneqq \E[|a|^p]^{1/p}$ and $\norm{a}_{\infty} \coloneqq \norm{a}$ for $a \in \cA$.
\begin{enumerate}[label=(\roman*),font=\normalfont]
    \item $\norm{\cdot}_p$ is a norm on $\cA$, and $|\E[a]| \leq \norm{a}_1$ for all $a \in \cA$.
    \item Noncommutative H\"{o}lder's inequality holds:
    If $p_1,\ldots,p_k,q \in [1,\infty]$ and $1/p_1+\cdots+1/p_k \leq 1/q$, then $\norm{a_1 \cdots a_k}_q \leq \norm{a_1}_{p_1}\cdots\norm{a_k}_{p_k}$ for all $a_1,\ldots,a_k \in \cA$.
\end{enumerate}
We write $L^p(\cA,\E)$ for the completion of $\cA$ with respect to $\norm{\cdot}_p$ and $\tilde{\E} \colon L^1(\cA,\E) \to \C$ for the bounded complex-linear extension of $\E \colon \cA \to \C$.
\end{cor}

\begin{proof}
Since $\pi$ is a $\ast$-homomorphism, if $a \in \cA$, then $|\pi(a)| = \pi(|a|)$.
Also, if $a \in \cA$ is normal ($a^*a=aa^*$) and $f \colon \R \to \C$ is a continuous function, then $\pi(a)$ is normal, and $f(\pi(a)) = \pi(f(a))$.
From this, we obtain the key observation that
\begin{equation}
    \E[|a|^p] = \E_{\mathsmaller{\cM}}[\pi(|a|^p)] = \E_{\mathsmaller{\cM}}[|\pi(|a|)|^p] = \E_{\mathsmaller{\cM}}[|\pi(a)|^p] \qquad (a \in \cA). \label{eq.Lpnormisom}
\end{equation}
Since $\E = \E_{\mathsmaller{\cM}}\circ \pi$ as well and $(\cM,\E_{\mathsmaller{\cM}})$ is a $\mathrm{W}^*$-probability space, the claimed properties follow easily from the corresponding properties of the noncommutative $L^p$ norm $\norm{A}_{L^p(\E_{\mathsmaller{\cM}})} = \E_{\mathsmaller{\cM}}[|A|^p]^{1/p}$ on $\cM$.
\end{proof}

By definition, $L^2(\cA,\E)$ is $H$ as a Banach space, so $L^2(\cA,\E)$ is a Hilbert space.
We write $\ip{\cdot,\cdot}_2 = \ip{\cdot,\cdot}$ for its inner product.
We now prove additional properties of $L^p(\cA,\E)$.

\begin{lem}\label{lem.LpALpM}
Suppose $1 \leq p < q < \infty$.
\begin{enumerate}[label=(\roman*),font=\normalfont]
    \item $\pi \colon \cA \to \cM$ extends to an isometric isomorphism $\iota_p \colon L^p(\cA,\E) \to L^p(\cM,\E_{\mathsmaller{\cM}})$, and $\tilde{\E} = \E_{\mathsmaller{\cM}} \circ \iota_1$.\label{item.LpALpM}
    \item The identity on $\cA$ extends to an injective contraction $\iota_{q,p} \colon L^q(\cA,\E) \to L^p(\cA,\E)$.
    Accordingly, we shall consider $L^q(\cA,\E)$ as a subset of $L^p(\cA,\E)$.\label{item.LpLq}
\end{enumerate}
\end{lem}

\begin{proof}
We take both items in turn.

\ref{item.LpALpM} By \eqref{eq.Lpnormisom}, $\pi$ extends to an isometry $L^p(\cA,\E) \to L^p(\cM,\E_{\mathsmaller{\cM}})$, so all we need to prove is that $\pi(\cA)$ is dense in $L^p(\cM,\E_{\mathsmaller{\cM}})$.
To this end, let $A \in \cM$.
By Kaplansky's density theorem, there is a bounded net $(a_j)_{j \in J}$ in $\cA$ such that $\pi(a_j) \to A$ in the $\sigma$-S$^*$OT ($\sigma$-strong$^*$ operator topology).
We claim that $\pi(a_j) \to A$ in $L^p(\cM,\E_{\mathsmaller{\cM}})$.
Indeed, since multiplication is jointly $\sigma$-S$^*$OT-continuous on bounded sets, if $n \in \N$, then the product $|\pi(a_j) - A|^{2n} = ((\pi(a_j) - A)^*(\pi(a_j) - A))^n$ converges to $0$ in the $\sigma$-S$^*$OT (in particular, in the $\sigma$-WOT).
Since $\E_{\mathsmaller{\cM}}$ is normal, we conclude that $\lim_{j \in J}\norm{\pi(a_j) - A}_{L^{2n}(\E_{\mathsmaller{\cM}})}^{2n} = \lim_{j \in J}\E_{\mathsmaller{\cM}}[|\pi(a_j)-A|^{2n}] = 0$.
Taking $n > p/2$, we get $\lim_{j \in J}\norm{\pi(a_j) - A}_{L^p(\E_{\mathsmaller{\cM}})} \leq \lim_{j \in J}\norm{\pi(a_j) - A}_{L^{2n}(\E_{\mathsmaller{\cM}})} = 0$.
This proves the claim.
Since $\cM$ is dense in $L^p(\cM,\E_{\mathsmaller{\cM}})$ by definition, this completes the proof that $\pi \colon \cA \to \cM$ extends to an isometric isomorphism $\iota_p \colon L^p(\cA,\E) \to L^p(\cM,\E_{\mathsmaller{\cM}})$.
The identity $\tilde{\E} = \E_{\mathsmaller{\cM}} \circ \iota_1$ then follows from the identity $\E = \E_{\mathsmaller{\cM}} \circ \pi$.

\ref{item.LpLq} Write $I_{q,p} \colon L^q(\cM,\E_{\mathsmaller{\cM}}) \to L^p(\cM,\E_{\mathsmaller{\cM}})$ for the inclusion, which we know to be an injective contraction from the $\mathrm{W}^*$ theory.
Then $\iota_{q,p} \coloneqq \iota_p^{-1}\circ I_{q,p}  \circ \iota_q$ is an injective contraction.
Since $\iota_{q,p}$ clearly agrees with $\id_{\cA}$ on $\cA$, we are done.
\end{proof}

\begin{prop}\label{prop.dual}
Let $p,q \in [1,\infty]$ be such that $1/p + 1/q = 1$.
\begin{enumerate}[label=(\roman*),font=\normalfont]
    \item If $a \in \cA$, then $\norm{a}_p = \sup\{|\E[ab]| : b \in \cA, \, \norm{a}_q \leq 1\}$.
    If $1 < p,q < \infty$ as well, then the map $\cA \ni a \mapsto (b \mapsto \tilde{\E}[ab]) \in L^p(\cA,\E)^*$ extends to an isometric isomorphism $L^q(\cA,\E) \to L^p(\cA,\E)^*$.\label{item.dual}
    \item Let us identify $\cA$ with $\pi(\cA)$ so that $\cA \subseteq \cM$.
    The map $\cA \ni a \mapsto (b \mapsto \tilde{\E}[ab]) \in L^1(\cA,\E)$ extends to an isometric isomorphism $\cM \to L^1(\cA,\E)^*$ that is a homeomorphism with respect to the $\sigma$-WOT on $\cM$ and the weak$^*$ topology on $L^1(\cA,\E)^*$.\label{item.dualL1}
\end{enumerate}
\end{prop}

\begin{proof}
We take both items in turn.

\ref{item.dual} Since $\pi(\cA) \subseteq L^q(\cM,\E_{\mathsmaller{\cM}})$ is dense,
\[
\norm{A}_{L^p(\E_{\mathsmaller{\cM}})} = \sup\{|\E_{\mathsmaller{\cM}}[A\pi(b)]| : b \in \cA, \; \norm{b}_q = \norm{\pi(b)}_{L^q(\E_{\mathsmaller{\cM}})} \leq 1\} \qquad (A \in \cM).\pagebreak
\]
Taking $A = \pi(a)$ with $a \in \cA$, we get
\begin{align*}
    \norm{a}_p & = \norm{\pi(a)}_{L^p(\E_{\mathsmaller{\cM}})} = \sup\{|\E_{\mathsmaller{\cM}}[\pi(a)\pi(b)]| : b \in \cA, \; \norm{b}_q \leq 1\} \\
    & = \sup\{|\E_{\mathsmaller{\cM}}[\pi(ab)]| : b \in \cA, \; \norm{b}_q \leq 1\} = \sup\{|\E[ab]| : b \in \cA, \; \norm{b}_q \leq 1\}.
\end{align*}
As a result, $\cA \ni a \mapsto (b \mapsto \tilde{\E}[ab]) \in L^p(\cA,\E)^*$ extends to a linear isometry $T \colon L^q(\cA,\E) \to L^p(\cA,\E)^*$.
Identifying $L^p(\cA,\E) \cong L^p(\cM,\E_{\mathsmaller{\cM}})$ via $\iota_p$ and $L^q(\cM,\E_{\mathsmaller{\cM}})^* \cong L^q(\cA,\E)^*$ via $\iota_q^*$, we conclude from the $\mathrm{W}^*$ theory that $T$ is surjective.

\ref{item.dualL1} Identifying $L^1(\cM,\E_{\mathsmaller{\cM}})^* \cong L^1(\cA,\E)^*$ via $\iota_1^*$, the claims of this item follow readily from the $\mathrm{W}^*$ theory and the $\sigma$-WOT density of $\cA$ in $\cM$.
\end{proof}

\begin{rem}\label{rem.Linfty}
Since $L^1(\cA,\E)^* \cong \cM$, it is conceptually appropriate to define $L^{\infty}(\cA,\E) \coloneqq \cM$.
For notational convenience, we do not do so in this paper;
we take $L^{\infty}(\cA,\E)$ to be $\cA$.
\end{rem}

We end this appendix by proving the $\mathrm{C}^*$ case of Proposition \ref{prop.condexp}, again taking the $\mathrm{W}^*$ case for granted.

\begin{proof}[Proof of Proposition \ref{prop.condexp}]
Let $\cB \subseteq \cA$ be a $\mathrm{C}^*$-subalgebra, and write $\cN$ for the $\sigma$-WOT closure (equivalently, WOT closure) of $\pi(\cB)$ in $\cM$.
Then $\iota_p$ restricts to an isometric isomorphism $L^p(\cB,\E) \to L^p(\cN,\E_{\mathsmaller{\cM}})$ for all $p \in [1,\infty)$.
The only non-obvious part of this statement is the surjectivity of the restrictions.
This is taken care of by the argument from the proof of Lemma \ref{lem.LpALpM}\ref{item.LpALpM}, which shows that $\pi(\cB)$ is dense in $L^p(\cN,\E_{\mathsmaller{\cM}})$.
Therefore, identifying $L^p(\cA,\E) \cong L^p(\cM,\E_{\mathsmaller{\cM}})$ and $L^p(\cB,\E) \cong L^p(\cN,\E_{\mathsmaller{\cM}})$ using $\iota_p$, the map $\E[\cdot \mid \cB] \coloneqq \E_{\mathsmaller{\cM}}[\cdot \mid \cN]$ satisfies all the desired properties.
\end{proof}

\section{Notation index}\label{sec.nota}

\begin{tabbing}
$\cP_I$, $\cP_I^*$ \qquad\qquad\qquad\qquad \= partitions and augmented partitions of interval $I$; Notas.\ \ref{nota.part} \& \ref{nota.augpart}, pp.\ \pageref{nota.part} \& \pageref{nota.augpart} \\

$1_S$ \> indicator function of $S$; p.\ \pageref{page.indicator} \\

$V(F:I)$ \> variation of $F$ on interval $I$; Nota.\ \ref{nota.nota}\ref{item.V1}, p.\ \pageref{item.V1} \\

$B_k(\cV_1 \times \cdots \times \cV_k;\cV)$ \> bounded real--$k$-linear maps $\cV_1 \times \cdots \times \cV_k \to \cV$; Nota.\ \ref{nota.nota}\ref{item.Bk}, p.\ \pageref{item.Bk} \\

$F(t\pm)$, $F_{\pm}$ \> left/right limit of $F$ at $t$, left/right limit function of $F$; Nota.\ \ref{nota.nota}\ref{item.l/rlim}, p.\ \pageref{item.l/rlim} \\

$L_{(\loc)}^p(\Om,\mu;\cV)$ \> (local) Bochner $L^p$ space; Nota.\ \ref{nota.nota}\ref{item.Bochner}, p.\ \pageref{item.Bochner} \\

$(\mathcal{A},(\mathcal{A}_t)_{t\ge 0},\E = \E_{\mathsmaller{\cA}})$ \> filtered $\mathrm{C}^*$- or $\mathrm{W}^\ast$-probability space; Def.\ \ref{def.filtr}, pp.\ \pageref{page.Cstarprob} \& \pageref{def.filtr}\\

$\tr_n$ \> normalized trace on $n\times n$ matrices; Ex.\ \ref{ex.randommatrices}, p.\ \pageref{ex.randommatrices} \\

$L^p(\mathcal{A},\E) = L^p(\E)$, \> noncommutative $L^p$ space; Nota.\ \ref{nota.Lp}, p.\ \pageref{nota.Lp} \\

$\norm{\cdot}_p = \norm{\cdot}_{L^p(\E)}$ \> noncommutative $L^p$ norm; Nota.\ \ref{nota.Lp}, p.\ \pageref{nota.Lp} \\

$\E[\,\cdot\mid \mathcal{B}]$ \> conditional expectation onto $\mathrm{C}^\ast$- or $\mathrm{W}^\ast$-subalgebra $\mathcal{B}$; Prop.\ \ref{prop.condexp}, p.\ \pageref{prop.condexp} \\

$B_k^{p_1,\ldots,p_k;p}$ \> bounded real--$k$-linear maps $L^{p_1}(\E_1)\times\cdots\times L^{p_k}(\E_k)\to L^{p}(\E)$; \\
\> Nota.\ \ref{nota.bddlin}, p.\ \pageref{nota.bddlin} \\

$\|\cdot\|_{p_1,\ldots,p_k;p}$ \> operator norm on $B_k^{p_1,\ldots,p_k;p}$; Nota.\ \ref{nota.bddlin}, p.\ \pageref{nota.bddlin} \\

$\mathbb{B}_k,\mathbb{B}_k(\cA^d;\cA^m),\mathbb{B}_k(\cA)$ \> multilinear maps that are bounded uniformly on tuples of $L^p$ spaces with \\ \> exponents satisfying a H\"{o}lder conjugate relation; Notas.\ \ref{nota.bddlin}, \ref{nota.Bbkmult}, \& \ref{nota.tens}\ref{item.BbkA}, \\
\> pp.\ \pageref{nota.bddlin}, \pageref{nota.Bbkmult}, \& \pageref{item.BbkA} \\

$\vertiii{\cdot}_k$ \> norm on $\mathbb{B}_k$; Nota.\ \ref{nota.bddlin}, p.\ \pageref{nota.bddlin} \\

$\mathbb{C}\langle \mathbf{x}\rangle$, $\mathbb{C}^\ast\langle\mathbf{x}\rangle$ \> noncommutative ($\ast$-)polynomials in $\x = (x_1,\ldots,x_n)$; Nota.\ \ref{nota.TrPoly}, p.\ \pageref{nota.TrPoly} \\

$\mathrm{TrP} (\mathbf{x})$, $\mathrm{TrP}^\ast(\mathbf{x})$, $\mathrm{TrP}^\ast_n$ \> trace ($\ast$-)polynomials in $\x = (x_1,\ldots,x_n)$; Nota.\ \ref{nota.TrPoly}, p.\ \pageref{nota.TrPoly} \\

$\mathrm{TrP}^\ast_{n,k,d}$, $\mathrm{TrP}^{\ast,\mathbb{C}}_{n,k,d}$  \> trace $\ast$-polynomials in $(\x,\y_1,\ldots,\y_k)$ that are ($\C$--)$k$-linear in $(\mathbf{y}_1,\ldots,\mathbf{y}_k)$, \\
\> where $\x = (x_1,\ldots,x_n)$, $\mathbf{y}_j = (y_{j,1},\ldots,y_{j,d_j})$, and $d = (d_1,\ldots,d_k)$; \\
\> Nota.\ \ref{nota.lintrpoly}, p.\ \pageref{nota.lintrpoly} \\

$\mathrm{ev}^n_{\mathsmaller{(\mathcal{A},\E)}}$, $\mathrm{ev}^{n,m,k,d}_{\mathsmaller{(\mathcal{A},\E)}}$, $P_{\mathsmaller{(\mathcal{A},\E)}}$ \> evaluation maps for spaces of trace $\ast$-polynomials; Nota.\ \ref{nota.eval}, p.\ \pageref{nota.eval} \\

$[1,\infty\rangle$ \> either $[1,\infty]$ or $[1,\infty)$; Conv.\ \ref{conv.cond}, p.\ \pageref{conv.cond} \\

$\mathcal{F}_{k,t}^{p_1,\ldots,p_k;p}$, $\mathcal{F}_{k,t}$, $\mathcal{F}_t$ \> induced filtrations on spaces of multilinear maps; Def.\ \ref{def.adapt}\ref{item.klinadap1}--\ref{item.klinadap2}, p.\ \pageref{item.klinadap1} \\

$C_a(\mathbb{R}_+;L^p(\E))$ \> $L^p$-continuous, adapted processes; Def.\ \ref{def.adapt}\ref{item.Lpadap}, p.\ \pageref{item.Lpadap} \\

$\mathcal{T}_t$, $\mathcal{T}_{k,t}$, $\mathcal{T}_{m,k,d,t}$, etc. \> closure of evaluations of trace polynomials at arguments from the filtration \\
\> at time $t$; Nota.\ \ref{nota.tens}\ref{item.Tfilt}, p.\ \pageref{item.Tfilt} \\

$\#_k$, $\#_k^{\mathsmaller{\E}}$ \> alternating (expectation) multiplication maps on $k$-fold tensor products; \\
\> Nota.\ \ref{nota.tens}\ref{item.hash}--\ref{item.Ehash}, p.\ \pageref{item.hash} \\

$\mathbb{FV}^p = \mathbb{FV}_{\mathsmaller{\cA}}^p$ \> $L^p$-continuous $L^p$-finite variation processes; Def.\ \ref{def.processes}\ref{item.FV}, p.\ \pageref{item.FV} \\

$\M^p = \M_{\mathsmaller{\cA}}^p$, $\tbM^p = \tbM_{\mathsmaller{\cA}}^p$ \> $L^p$-continuous martingales, closure of $\M^{\infty}$ in $\M^p$; Def.\ \ref{def.processes}\ref{item.mart}, p.\ \pageref{item.mart} \\

$X^t$ \> process $X$ stopped at time $t$; Nota.\ \ref{nota.stop}, p.\ \pageref{nota.stop}\\

$X^{\mathrm{m}}$, $X^{\mathrm{fv}}$ \> martingale part of $X$, FV part of $X$; Cor.\ \ref{cor.uniquedecomp}, p.\ \pageref{cor.uniquedecomp}\\

$F^{(\Pi,\xi)}$, $F^{\Pi}$ \> step-function approximations of $F$ associated to (augmented) partitions; \\
\> Nota.\ \ref{nota.augpart}, p.\ \pageref{nota.augpart}\\

$\ell_{(\loc)}^{\infty}(\R_+;\cV)$  \> (locally) bounded functions $\R_+ \to \cV$;
pp.\ \pageref{page.ell} \& \pageref{thm.QC1} \\

$\nu_F(\d t) = \norm{\d F(t)}_{\cV}$ \> variation measure of $F \colon I \to \cV$; Thm.\ \ref{thm.LSint} \& Lem.\ \ref{lem.elemintbd}, p.\ \pageref{thm.LSint} \\

$\EP^{p;q}$, $\EP$ \> elementary predictable processes; Def.\ \ref{def.EP}, p.\ \pageref{def.EP} \\

$\kappa_X$ \> measure associated to $L^2$-decomposable process $X$; Nota.\ \ref{nota.kappa}, p.\ \pageref{nota.kappa} \\

$\cI(X)$, $\tilde{\cI}(X)$ \> stochastically $X$-integrable processes; Def.\ \ref{def.integrands}, p.\ \pageref{def.integrands} \\

$\norm{\cdot}_{X,t}$ \> seminorm on $\cI(X)$; Def.\ \ref{def.integrands}, p.\ \pageref{def.integrands} \\

$\into H[\d X] = I_X(H)$ \> (stochastic) integral of $H$ against $X$; Nota.\ \ref{nota.EPint} \& Thm.\ \ref{thm.stochint}, pp.\ \pageref{nota.EPint} \& \pageref{thm.stochint} \\

$\mathbb{L}^p\text{-}\lim$ \> locally uniform $L^p$-limit; Nota.\ \ref{nota.bLplim}, p.\ \pageref{nota.bLplim} \\

LCLB, LLLB  \> left-continuous and locally bounded, left-limited and locally bounded; \\
\> Ex.\ \ref{ex.LLLB}, p.\ \pageref{ex.LLLB} \\

$\mathrm{RS}_{\Pi}^{X,Y}(\Lambda)$ \> quadratic Riemann--Stieltjes sum; Nota.\ \ref{nota.qRS}, p.\ \pageref{nota.qRS} \\

$\mathrm{Q}_0$ \> adapted, bilinear processes $\Lambda \colon \R_+ \to \mathbb{B}_2$ that are left-continuous with locally \\
\> bounded variation w.r.t.\ $\vertiii{\cdot}_2$; Nota.\ \ref{nota.Q}, p.\ \pageref{nota.Q} \\

$\mathrm{Q}$ \> closure of $\mathrm{Q}_0$ in $\ell_{\loc}^{\infty}(\R_+;B_2^{2,2;1})$; Thm.\ \ref{thm.QC1}, p.\ \pageref{thm.QC1} \\

$\llbracket X, Y \rrbracket^{\Lambda} = \into \Lambda[\d X, \d Y]$ \> $\Lambda$-quadratic covariation of $X$ and $Y$; Defs.\ \ref{def.QC1} \& \ref{def.QC2}, pp.\ \pageref{def.QC1} \& \pageref{def.QC2} \\

$\kappa_{M,N}$ \> the measure $(\kappa_M + \kappa_N)/2$; Lem.\ \ref{lem.QCL1bd}, p.\ \pageref{lem.QCL1bd} \\

$\mathcal{Q}$, $\mathcal{Q}(X,Y)$ \> equivalence classes in $L_{\loc}^1(\R_+,\kappa_{M,N};B_2^{2,2;1})$ of elements of $\mathrm{Q}$, closure of $\mathcal{Q}$ \\
\> in $L_{\loc}^1(\R_+,\kappa_{M,N};B_2^{2,2;1})$, where $M$ is the  martingale part of $X$ and $N$ is \\ 
\> the martingale part of $Y$; Nota.\ \ref{nota.QXY}, p.\ \pageref{nota.QXY} \\

$\cA_{\beta}$, $\cB_{\gamma}$ \> fixed element of $\{\cA,\cA_{\sa}\}$, fixed element of $\{\cB,\cB_{\sa}\}$; \S\ref{sec.NCIF}, p.\ \pageref{sec.NCIF} \\

$D^kF$ \> $k^{\text{th}}$ Fréchet derivative of $F$; \S\ref{sec.NCIF}, p.\ \pageref{sec.NCIF} \\

$C_a^{k,\ell}(\cU;\cC)$, $C_a^k(\cU;\cC)$ \> adapted $C^{k,\ell}$ maps, adapted $C^k$ maps; Def.\ \ref{def.adapCkl}, p.\ \pageref{def.adapCkl} \\

$\partial_{\mathsmaller{\otimes}}^kp$ \> tensor noncommutative derivative of polynomial $p \in \C[\lambda]$; Nota.\ \ref{nota.ncder}, p.\ \pageref{nota.ncder} \\

$W_k(\R)$ \> $k^{\text{th}}$ Wiener space; Ex.\ \ref{def.Wk}, p.\ \pageref{def.Wk} \\

$f_{\mathsmaller{\cA}} \colon \cA_{\sa} \to \cA$ \> operator function induced by the scalar function $f$; Ex.\ \ref{ex.Wiener}, p.\ \pageref{ex.Wiener} \\

$\partial_{x_i}P$ \> algebraic derivative of the trace $\ast$-polynomial $P$ in the indeterminate $x_i$; \\
\> Lem.\ \ref{lem.partial}, p.\ \pageref{lem.partial} \\

$\partial^k P$ \> algebraic $k^{\text{th}}$ derivative of the trace $\ast$-polynomial $P$; Nota.\ \ref{nota.TrPolyder}, p.\ \pageref{nota.TrPolyder} \\

$C_{\E}(\cU;\mathbb{B}_k(\cA_{\gamma}^d;\cA^m))$ \> trace continuous maps $\cU \to \mathbb{B}_k(\cA_{\gamma}^d;\cA^m)$, where $\cA_{\beta}^n \in \{\cA^n,\cA_{\sa}^n\}$, $\cU \subseteq \cA_{\beta}^n$ \\
\> is open, $\cA_{\gamma}^d \in \{\cA^d,\cA_{\sa}^d\}$, and $d = (d_1,\ldots,d_k)$; Def.\ \ref{def.tracecontCk}\ref{item.tracecont}, p.\ \pageref{item.tracecont} \\

$C_{\E}^k(\cU;\cA^m)$ \> trace $C^k$ maps $\cU \to \cA^m$, where $\cU \subseteq \cA_{\beta}^n \in \{\cA^n,\cA_{\sa}^n\}$ is open; \\
\> Def.\ \ref{def.tracecontCk}\ref{item.traceCk}, p.\ \pageref{item.traceCk} \\

$\ell^{\infty}(\Xi,\sG)$ \> bounded $\sG$/$\cB_{\C}$-measurable functions $\Xi \to \C$; p.\ \pageref{page.bddmeas} \\

$I^{a_1,\ldots,a_{k+1}}\varphi$ \> multiple operator integral (MOI); Thm.\ \ref{thm.babyMOI}, p.\ \pageref{thm.babyMOI} \\

$f^{[k]}$ \> $k^{\text{th}}$ divided difference of the scalar function $f$; Def.\ \ref{def.divdiff}, p.\ \pageref{def.divdiff} \\

$NC^k(\R)$ \> noncommutative $C^k$ functions $\R \to \C$; Def.\ \ref{def.NCk}, p.\ \pageref{def.NCk}
\end{tabbing}

\hfill

\begin{ack}
\phantomsection
\addcontentsline{toc}{section}{Acknowledgments}
We are grateful to Michael Anshelevich, Guillaume C\'{e}bron, Nicolas Gilliers, Dimitri Shlyakhtenko, and Roland Speicher for inspiring conversations.
We extend special thanks to Bruce Driver, colleague and mentor, whose insights helped resolve many key technical issues throughout this work and who helped us realize that Proposition \ref{prop.clCk} holds.
Moreover, it was his initial intuition that ``all It\^o formulas are created equal'' that led us down the path to this paper in the first place.
\end{ack}

\phantomsection
\addcontentsline{toc}{section}{References}
\bibliographystyle{amsplain}
\bibliography{NCSC.bib}
\end{document}